\documentclass[11pt]{amsart}%
\usepackage{amsmath}%
\usepackage{amssymb}%
\usepackage{amscd}%
\usepackage[abbrev]{amsrefs}%
\usepackage{array}
\usepackage{color}%
\usepackage{mathrsfs}%
\usepackage[all]{xy}%
\usepackage{stmaryrd}
\usepackage{comment}

\def\ol#1{\overline{#1}}% 		overline
\def\wh#1{\widehat{#1}}% 	wide hat
\def\wt#1{\widetilde{#1}}% 	wide tilde
\theoremstyle{plain}
    \newtheorem{theorem}{Theorem}[section]
    
    \newtheorem{proposition}[theorem]{Proposition}
    \newtheorem{lemma}[theorem]{Lemma}
    \newtheorem{corollary}[theorem]{Corollary}
      
      \newtheorem{proposition-definition}[theorem]{Proposition-Definition}
\theoremstyle{definition}
    \newtheorem{definition}[theorem]{Definition}
    
    \newtheorem{convention}[theorem]{Convention}
    
    \newtheorem{remark}[theorem]{Remark}
\def\Alphabet{A,B,C,D,E,F,G,H,I,J,K,L,M,N,O,P,Q,R,S,T,U,V,W,X,Y,Z}%  Capitalized Alphabet
\def\alphabet{a,b,c,d,e,f,g,h,i,j,k,l,m,n,o,p,q,r,s,t,u,v,w,x,y,z}%	lowercase alphabet
\def\endpiece{xxx}%									marks end of list
\def\makeAlphabet[#1]{\expandafter\makeA#1,xxx,}%		Ex. \makeAlphabet[A,B]
\def\makealphabet[#1]{\expandafter\makea#1,xxx,}%		Ex. \makealphabet[c,d]
\def\makeA#1,{\def\temp{#1}\ifx\temp\endpiece\else%
\mkbb{#1}\mkfrak{#1}\mkbf{#1}\mkcal{#1}\mkscr{#1}\mkbs{#1}\expandafter\makeA\fi}%
\def\makea#1,{\def\temp{#1}\ifx\temp\endpiece\else\mkfrak{#1}\mkbf{#1}\mkbs{#1}\expandafter\makea\fi}%
\def\mkbb#1{\expandafter\def\csname bb#1\endcsname{\mathbb{#1}}}%      Define bb
\def\mkfrak#1{\expandafter\def\csname fr#1\endcsname{\mathfrak{#1}}}%    Define frak
\def\mkbf#1{\expandafter\def\csname b#1\endcsname{\mathbf{#1}}}%           Define bold letters
\def\mkcal#1{\expandafter\def\csname c#1\endcsname{\mathcal{#1}}}%       Define calligraphy
\def\mkscr#1{\expandafter\def\csname s#1\endcsname{\mathscr{#1}}}%       Define script
\def\mkbs#1{\expandafter\def\csname bs#1\endcsname{{\boldsymbol{#1}}}}%       Define bold symbol
\def\makeop[#1]{\xmakeop#1,xxx,}%					Ex. \makeop[Hom,Spec]
\def\mkop#1{\expandafter\def\csname #1\endcsname{{\mathrm{#1}}}} % 
\def\xmakeop#1,{\def\temp{#1}\ifx\temp\endpiece\else\mkop{#1}\expandafter\xmakeop\fi}%
\def\makeup[#1]{\xmakeup#1,xxx,}%					Ex. \makeop[Hom,Spec]
\def\mkup#1{\expandafter\def\csname #1\endcsname{{\mathrm{#1}\,}}} % 
\def\xmakeup#1,{\def\temp{#1}\ifx\temp\endpiece\else\mkup{#1}\expandafter\xmakeup\fi}%
\makeAlphabet[\Alphabet]%				Define bb, frak, bf, cal for Capitalized Alphabet
\makealphabet[\alphabet]%  				Define frak and bf for uncapitalized alphabet
\makeop[id,Alt,End,Ext,Hom,Mod,Sym,Tor,dR,rig,HK,an,unip,Li,Tot,res,nr,sm,pol,Gr,Re,Rep,GL,abs,holim,Fil,mix,ab,MF,ad,st,crys,syn,Isoc,MIC,OC,Str,Gal,Syn,sp,ord,fin]% 		Homs
\makeup[Spec,Spwf,Sp,Ker,Coker,Im,Coim,Cone,Lie,Ad,holim,hocolim]
  \makeatletter
    
    \@addtoreset{equation}{section}
  \makeatother

\newcommand*{\sheafhom}{\cH\kern -.5pt om}
\newcommand{\et}{{\text{\'{e}t}}}

\begin{document}
%---the title------------------------------------------------------------------------------------------------------------------------
\title[Hyodo--Kato theory with syntomic coefficients]{Hyodo--Kato theory with syntomic coefficients}
\author[Kazuki Yamada]{Kazuki Yamada}

%\date{\today}
\address{Department of Mathematics, Gakushuin University, 1-5-1 Mejiro, Toshima-ku, Tokyo, JAPAN}
\email{k.yamada@gakushuin.ac.jp}

\begin{abstract}
	The purpose of this article is to establish theories concerning $p$-adic analogues of Hodge cohomology and Deligne--Beilinson cohomology with coefficients in variations of mixed Hodge structures.
	We first study log overconvergent $F$-isocrystals as coefficients of Hyodo--Kato cohomology.
	In particular, we prove a rigidity property of Hain--Zucker type for mixed log overconvergent $F$-isocrystals.
	In the latter half of the article, we give a new definition of syntomic coefficients as coefficients of $p$-adic Hodge cohomology and syntomic cohomology, and prove some fundamental properties concerning base change and admissibility.
	In particular, we will see that our framework of syntomic coefficients depends only on the choice of a branch of the $p$-adic logarithm, but not on the choice of a uniformizer of the base ring.
	The rigid analytic reconstruction of Hyodo--Kato theory studied by Ertl and the author plays a key role throughout this article.
\end{abstract}
\thanks{The author is supported by KLL 2019 Ph.D Program Research Grant of Keio University and KAKENHI Grant Numbers 18H05233, 22K13899, and 23KJ0332.}
\maketitle

%%%%%%%%%%%%%%%%%%%%%%%%%%%%%%%%%%%%%%%%%%%%%%%%%%%

%%%%%%%%%%%%%%%%%
%
\section*{Introduction}\label{sec: introduction}
%
%%%%%%%%%%%%%%%%%
It is classically known as a central result of the Hodge theory that the cohomology groups of complex algebraic varieties carry mixed Hodge structures, which are sometimes called Hodge cohomology.
Steenbrink and Zucker \cite{SZ} introduced the notion of a variation of mixed Hodge structures as a coefficient of Hodge cohomology.
The theory of variations of mixed Hodge structures is extended by Saito to the theory of mixed Hodge modules, and plays important roles in arithmetic geometry.
In particular, a variation of mixed Hodge structures is also thought as a coefficient of Deligne--Beilinson cohomology, which is defined as the derived Hom of the Hodge cohomology complex, and related with the special values of $L$-functions in the context of the Beilinson conjecture.
The purpose of this article is to establish a $p$-adic analogue of the theory of Hodge cohomology and Deligne--Beilinson cohomology with coefficients in variations of mixed Hodge structures.

We first review previous research.
Let $V$ be a complete discrete valuation ring of mixed characteristic $(0,p)$ with fraction field $K$ and perfect residue field $k$.
Let $L$ be the fraction field of the ring of Witt vectors $W:=W(k)$.
For any proper schemes $X$ over $V$ with semistable reduction, Hyodo and Kato \cite{HK} constructed a map
	\[\Psi_\pi^\crys\colon R\Gamma_\crys(X_k/W^0)\rightarrow R\Gamma_\dR(X_K/K)\]
which depends on the choice of a uniformizer $\pi\in V$, and proved that $\Psi_\pi^\crys$ induces a quasi-isomorphism after tensoring $K$ with $R\Gamma_\crys(X_k/W^0)$.
This map is called the (crystalline) Hyodo--Kato map.
Here we endow $X$ and $X_k$ with the semistable log structures, and $R\Gamma_\crys(X_k/W^0)$ denotes the log crystalline cohomology over $W$ equipped with the log structure associated to the monoid homomorphism $\bbN\rightarrow W; 1\mapsto 0$.

The rational cohomology groups $H^n_\crys(X_k/W^0)_L:=H^n_\crys(X_k/W^0)\otimes_WL$ are finite-dimensional $L$-vector spaces, endowed with a Frobenius-linear automorphism $\varphi$ and an $L$-linear endomorphism $N$ such that $N\varphi=p\varphi N$.
On the other hand, the de Rham cohomology groups $H^n_\dR(X_K/K)$ are endowed with finite descending filtrations $F^\bullet$.
As a consequence of the semistable comparison theorem \cites{Ts,Fa,Ni,Bei2}, those structures give an object called an admissible filtered $(\varphi,N)$-module through the identification via $\Psi_\pi^\crys$, which include informations as rich as the absolute Galois representation on the \'{e}tale cohomology group $H^n_\et(X_{\ol K},\bbQ_p)$, and is regarded as a $p$-adic analogue of Hodge cohomology.

We note that, in Beilinson's proof of the semistable comparison theorem, he introduced Beilinson--Hyodo--Kato cohomology for any varieties over $K$, which is also equipped with endomorphisms $\varphi$ and $N$ as above.
Using this, Nekov\'{a}\v{r} and Nizio{\l} defined in \cite{NN} syntomic cohomology for any varieties over $K$ as a $p$-adic analogue of Deligne--Beilinson cohomology.

Coefficients of $p$-adic Hodge cohomology and syntomic cohomology were considered by Tsuji \cite{Ts1} and Faltings \cite{Fa}, as filtered log $F$-crystals. 
D\'{e}glise and Nizio\l{} \cite{DN} also gave a definition of syntomic coefficients as modules over a motivic dg algebra, however the relation with the works of Tsuji and Faltings is not clear.

In the good reduction case, the rigid analytic version of syntomic cohomoloy was introduced by Gros \cite{Gros} and established by Besser \cite{Bes}. Moreover, the theory of its coefficients and the reinterpretation as extension groups were developed by Bannai \cites{Ban,Ban2}.
This syntomic cohomology plays important roles in computation of $p$-adic regulators \cite{Ban,Ban3,BK,BKT,Bes2,Bes3,Bes4,BdJ,BdJ2,Sp} and applied to researches of special values of $p$-adic $L$-functions \cites{BBdJR,BK2,BDR}.

In the semistable case, a rigid analytic reinterpretation of Hyodo--Kato theory was first studied by Gro\ss e-Kl\"{o}nne in \cite{GK3}, and developed by Ertl and the author in \cite{EY1}.
However his construction of the Hyodo--Kato map is combinatorially complicated, and not suitable for computation and generalization to cohomology with coefficients.

In \cite{EY}, Ertl and the author gave a very simple and direct construction of the rigid analytic Hyodo-Kato map $\Psi^\rig_{\pi,\log}$ by choosing a uniformizer $\pi\in V$ and a branch $\log\colon K^\times\rightarrow K$ of the $p$-adic logarithm.
This Hyodo--Kato map is in fact independent of the choice of a uniformizer, and compatible with the crystalline Hyodo--Kato map and Gro\ss e-Kl\"{o}nne's Hyodo--Kato map if $\log$ is taken to be $\log(\pi)=0$.
While $\Psi^\rig_{\pi,\log}$ loses information of the integral structure, it is useful for explicit computation.

In this article, we study coefficients of rigid analytic Hyodo--Kato theory using extensively the viewpoint developed in \cite{EY}.
Let $k^0$ be $\Spec k$ equipped with the log structure associated to the monoid homomorphism $\bbN\rightarrow k; 1\mapsto 0$.
For a fine log scheme $Y$ over $k^0$, we consider a log overconvergent $F$-isocrystal $(\sE,\Phi)$ over $W^\varnothing$, where $W^\varnothing$ is $\Spwf W$ equipped with the trivial log structure, as a coefficient of Hyodo--Kato cohomology.
The object $(\sE,\Phi)$ is defined as a pair of a sheaf $\sE$ on the log overconvergent site on $Y$ over $W^\varnothing$ and an isomorphism $\Phi\colon\sigma^*\sE\xrightarrow{\cong}\sE$, however, it is locally interpreted by a locally free sheaf with an integrable log connection on a $p$-adic analytic space.
The log connection is naturally extended to Kim--Hain complex (Definition \ref{def: Kim--Hain}), and we define the Hyodo--Kato cohomology $R\Gamma_\HK(Y,(\sE,\Phi))$ with coefficients in $(\sE,\Phi)$ as the (gluing of) cohomology of a kind of de Rham type complex.

It would be worth emphasizing that the log connection is taken over $W^\varnothing$ (not $W^0$), hence it implicitly includes the information of monodromy.
Indeed, there exists a canonical functor
	\begin{equation}\label{eq: isocrystal on point}
	\Mod_L^\fin(\varphi,N)\rightarrow F\Isoc^\dagger(Y/W^\varnothing)^\unip;\ M\mapsto M^a
	\end{equation}
from the category of $(\varphi,N)$-modules to the category of unipotent log overconvergent $F$-isocrystals on $Y$ over $W^\varnothing$, which is an equivalence if $Y=k^0$ (Proposition \ref{prop: isoc on point}).
Moreover we will show that this equivalence is compatible with the Hyodo--Kato cohomology in the sense that we have an isomorphism
	\begin{equation}
	H^n_\HK(Y,M^a)\cong H^n_\HK(Y)\otimes M
	\end{equation}
as $(\varphi,N)$-modules (Proposition \ref{prop: pseudo-const}).

In the complex case, a rigidity property for unipotent variation of mixed Hodge structures was studied by Hain and Zucker \cite{HZ}, and a $p$-adic analogue for non-logarithmic unipotent $F$-isocrystals was studied by Chiarellotto \cite{Ch}.
We prove the following rigidity property for unipotent log overconvergent $F$-isocrystals over $W^\varnothing$.

\begin{theorem}[Theorem \ref{thm: mixed Mod}]
	Let $Y$ be a connected strictly semistable log scheme over $k^0$ and assume that there exists a $k$-rational point $y$ in the smooth locus of $Y$.
	Let $\pi_1^\unip(Y,y)$ be the tannakian fundamental group of unipotent log overconvergent isocrystals. 
	Then there exists a canonical equivalence
		\[F\Isoc^\dagger(Y/W^\varnothing)^\mix\cong\Mod^\mix_L(\varphi,\wh\frU(\Lie\pi_1^\unip(Y,y))),\]
	between the category of mixed log overconvergent $F$-isocrystals and the category of finite-dimensional $L$-vector spaces equipped with a continuous left $\wh\frU(\Lie\pi_1^\unip(Y,y))$-action and a $\sigma$-semilinear automorphism $\varphi$ which is compatible with the Frobenius on $\wh\frU(\Lie\pi_1^\unip(Y,y))$.
\end{theorem}

Here $\wh\frU(\Lie\pi_1^\unip(Y,y))$ the completion of the universal enveloping algebra of the Lie algebra with respect to the augmentation ideal, which admits a natural Frobenius action induced from the absolute Frobenius on $Y$.
The above theorem asserts that, for every unipotent log overconvergent $F$-isocrystal, both of the Frobenius structure and the mixedness are determined by those at the fiber at a $k$-rational point in the smooth locus.

In the later half of this article, we will newly give a definition of syntomic coefficients, and study $p$-adic Hodge cohomology and syntomic cohomology with syntomic coefficients.
Choose a uniformizer $\pi\in V$.
Let $V^\sharp$ be $\Spwf V$ equipped with the canonical log structure (associated to the closed point) and $\tau_\pi\colon k^0\hookrightarrow V^\sharp$ the exact closed immersion which sends $\pi$ to the canonical generator of the log structure of $k^0$.
Let $\cX$ be a strictly semistable weak formal log scheme over $V^\sharp$, and $Y_\pi:=\cX\times_{V^\sharp,\tau_\pi}k^0$ the reduction of $\cX$ with respect to $\pi$.
A log overconvergent isocrystal $\sE$ on $Y_\pi$ over $W^\varnothing$ naturally defines a locally free sheaf $\sE_\dR$ with a log connection on $\cX_\bbQ$, the generic fiber of $\cX$.
The constructions and results in \cite{EY} for the trivial coefficient immediately extend to any log overconvergent $F$-isocrystal, and we obtain the Hyodo--Kato map
	\[\Psi_{\pi,\log}\colon R\Gamma_\HK(Y_\pi,(\sE,\Phi))\rightarrow R\Gamma_\dR(\cX,\sE_\dR):=R\Gamma(\cX_\bbQ,\sE_\dR\otimes\omega^\bullet_{\cX/V^\sharp,\bbQ})\]
for any choice of a branch $\log\colon K^\times\rightarrow K$ of the $p$-adic logarithm.
(We omit the upper subscript $\rig$ of $\Psi^\rig_{\pi,\log}$.)
This satisfies the following properties.

\begin{theorem}[Results in \S \ref{sec: HK map}]
	\label{thm: 2}
	\begin{enumerate}
	\item (Quasi-isomorphy). For any choice of $\pi$ and $\log$,
		\[\Psi_{\pi,\log,K}:=\Psi_{\pi,\log}\otimes 1\colon R\Gamma_\HK(Y_\pi,(\sE,\Phi))\otimes_LK\rightarrow R\Gamma_\dR(\cX,\sE_\dR)\]
		is a quasi-isomorphism if $\sE$ is unipotent.
	\item (Dependence on $\log$). For two choices $\log$ and $\log'$ of branches of the $p$-adic logarithm, we have
		\[\Psi_{\pi,\log',K}=\Psi_{\pi,\log,K}\circ\exp\left(-\frac{\log'(\xi)}{\mathrm{ord}_\pi(\xi)}\cdot N\right),\]
		where $\xi$ is any non-zero element of the maximal ideal of $V$ such that $\log(\xi)=0$.
	\item\label{eq: item-3} (Base change property and independence from $\pi$).
		Consider an inclusion $V\subset V'$ of complete discrete valuation rings of mixed characteristic $(0,p)$ and a commutative diagram
			\[\xymatrix{
			\cX'\ar[d]_f\ar[r]&V'^\sharp\ar[d]\\
			\cX\ar[r]&V^\sharp
			}\]
		where $\cX'$ and $\cX$ are strictly semistable over $V'^\sharp$ and $V^\sharp$, respectively.
		For uniformizers $\pi'\in V'$ and $\pi\in V$, there exists a canonical functor
			\[\ol f_{\pi,\pi'}^*\colon F\Isoc^\dagger(Y_\pi/W^\varnothing)\rightarrow F\Isoc^\dagger(Y'_{\pi'}/W'^\varnothing),\]
		where the ring $W'$ and the log scheme $Y'_{\pi'}=\cX'\times_{V'^\sharp,\tau_{\pi'}}k'^0$ are defined obviously.
		Moreover $f$ induces canonical morphisms
		\begin{eqnarray*}
		&&f^*_\HK\colon R\Gamma_\HK(\cX,(\sE,\Phi))_\pi\rightarrow R\Gamma_\HK(\cX',\overline{f}^*_{\pi,\pi'}(\sE,\Phi))_{\pi'},\\
		&&f^*_\dR\colon R\Gamma_\dR(\cX,\sE_\dR)\rightarrow R\Gamma_\dR(\cX',f^*\sE_\dR),
		\end{eqnarray*}
	for any $(\sE,\Phi)\in F\Isoc^\dagger(Y_\pi/W^\varnothing)$, which are compatible with the Hyodo--Kato maps in the sense that
		\begin{equation*}
		\Psi_{\pi',\log}\circ f_\HK^*=f_\dR^*\circ\Psi_{\pi,\log}.
		\end{equation*}
	
		In particular, the Hyodo--Kato map is independent of the choice of a uniformizer, up to canonical isomorphisms.
	\end{enumerate}
\end{theorem}

For a proper $\cX$, a (unipotent) syntomic coefficient $\frE=(\sE,\Phi,F^\bullet)$ on $\cX$ is defined as a pair of a unipotent log overconvergent $F$-isocrystal $(\sE,\Phi)$ on $Y_\pi$ over $W^\varnothing$ and a Hodge filtration $F^\bullet$ on $\sE_\dR$.
Then the $p$-adic Hodge cohomology $R\Gamma_{p\mathrm{H}}(\cX,\frE)_{\pi,\log}$ is defined as the patching of $R\Gamma_\HK(Y_\pi,(\sE,\Phi))$ and $R\Gamma_\dR(\cX,\sE_\dR)$ by $\Psi_{\pi,\log}$, and the syntomic cohomology $R\Gamma_\syn(\cX,\frE)_\pi$ is defined as a homotopy limit which mimics the absolute $p$-adic Hodge cohomology.
A syntomic coefficient $\frE$ is called admissible if the differentals of $R\Gamma_\dR(\cX,\sE_\dR)$ are strictly compatible with the Hodge filtration and if the $p$-adic Hodge cohomology groups $H^n_{p\mathrm{H}}(\cX,\frE)$ are admissible as filtered $(\varphi,N)$-modules.
Denote by $\Syn_\pi(\cX/V^\sharp)$ and $\Syn_\pi^\ad(\cX/V^\sharp)$ the exact categories of syntomic coefficients and admissible syntomic coefficients, respectively.
Then we will show the following properties.

\begin{theorem}[Results in \S \ref{sec: syn coh}]\label{thm: 3}
	\begin{enumerate}
	\item\label{item 1}(Base change property and independence from $\pi$).
		Let $f\colon\cX'\rightarrow\cX$ and $\pi,\pi'$ be as in Theorem \ref{thm: 2} \eqref{eq: item-3}.
		Then $f$ induces a canonical functor
			\[f^*_{\pi,\pi'}\colon\Syn_\pi(\cX/V^\sharp)\rightarrow\Syn_{\pi'}(\cX'/V'^\sharp)\]
		and canonical morphisms
			\begin{align*}
			&f^*_{p\mathrm{H}}\colon R\Gamma_{p\mathrm{H}}(\cX,\frE)_{\pi,\log}\rightarrow R\Gamma_{p\mathrm{H}}(\cX',f^*_{\pi,\pi'}\frE)_{\pi',\log},\\
			&f^*_\syn\colon R\Gamma_\syn(\cX,\frE)_\pi\rightarrow R\Gamma_\syn(\cX',f^*_{\pi,\pi'}\frE)_{\pi'}
			\end{align*}
		for any $\frE\in\Syn_\pi(\cX/V^\sharp)$.
			
		In particular, the category of syntomic coefficients, $p$-adic Hodge cohomology, and syntomic cohomology are independent from the choice of a uniformizer, up to canonical isomorphisms.
	\item (Properties and examples of admissible syntomic coefficients).
		\begin{enumerate}
		\item The admissibility of a syntomic coefficient is independent of the choice of $\log$.
		\item Fix a branch of the $p$-adic logarithm $\log\colon K^\times\rightarrow K$.
			One can define functors
			\begin{equation}\label{eq: intro MF Syn}
			\MF_K^\ad(\varphi,N)\rightarrow\Syn_\pi^\ad(\cX/V^\sharp);\ M\mapsto M(\pi,\log)^a_\pi
			\end{equation}
			for any uniformizer $\pi\in V$, which are
			\begin{itemize}
			\item compatible with each other through the equivalences in \eqref{item 1},
			\item equivalences if $\cX=V^\sharp$,
			\item compatible with the $p$-adic Hodge cohomology in the sense that we have
			\begin{equation}
			H^n_{p\mathrm{H}}(\cX,M(\pi,\log)^a_\pi)_{\pi,\log}\cong H^n_{p\mathrm{H}}(\cX)_{\pi,\log}\otimes M
			\end{equation}
			as filtered $(\varphi,N)$-modules.
			\end{itemize}
			Here $\MF_K^\ad(\varphi,N)$ denotes the category of admissible filtered $(\varphi,N)$-modules.
		\item Assume that $\cX$ is a weak completion of a proper strictly semistable log scheme.
			Then a syntomic coefficient on $\cX$ is admissible if it can be written as an iterated extension of syntomic coefficients of the form $M(\pi,\log)^a_\pi$ for $M\in\MF_K^\ad(\varphi,N)$.
		\end{enumerate}
	\item\label{eq: item-ext} (Interpretation as extension groups).
		\begin{enumerate}
		\item For $\frE,\frE'\in\Syn_\pi(\cX/V^\sharp)$, we have canonical isomrophisms
				\begin{eqnarray*}
				\Hom_{\Syn_\pi(\cX/V^\sharp)}(\frE',\frE)&\xrightarrow{\cong}& H^0_\syn(\cX,\sheafhom(\frE',\frE))_\pi\\
				\Ext^1_{\Syn_\pi(\cX/V^\sharp)}(\frE',\frE)&\xrightarrow{\cong}& H^1_\syn(\cX,\sheafhom(\frE',\frE))_\pi.
				\end{eqnarray*}
		\item For $\frE\in\Syn_\pi^\ad(\cX/V^\sharp)$, we may regard $R\Gamma_{p\mathrm{H}}(\cX,\frE)_{\pi,\log}$ as an object of the derived category $ D^b(\MF_K^\ad(\varphi,N))$, and there exists a canonical isomorphism
			\[R\Gamma_\syn(\cX,\frE)_\pi\cong R\Hom_{\MF_K^\ad(\varphi,N)}(L,R\Gamma_{p\mathrm{H}}(\cX,\frE)_{\pi,\log}).\]
		\end{enumerate}
	\end{enumerate}
\end{theorem}

Note that, if we associate syntomic coefficients to filtered $(\varphi,N)$-modules naively, then they are not compatible with inverse images.
The functor \eqref{eq: intro MF Syn} is defined by twisting the Hodge filtration by $\exp(-e\log(\pi)\cdot N)$, where $e$ is the ramification index of $K$ over $\bbQ_p$.

Because of Theorem \ref{thm: 3} \eqref{eq: item-ext}, we may regard the syntomic cohomology as an absolute $p$-adic Hodge cohomology (a $p$-adic analogue of the Deligne--Beilinson cohomology).

Note that we suppose the properness of $\cX$ in order to consider Hodge filtrations of the de Rham cohomology.
However we may consider the log structure associated to a horizontal divisor on $\cX$,
so we may deal with compactifiable cases.

One can also show that, if $\cX$ is smooth and $(\sE,\Phi)$ is induced from a non-logarithmic overconvergent $F$-isocrystal, then $\Psi_{\pi,\log}$ is independent of $\pi$ and $\log$, and coincides with the base change map of the non-logarithmic rigid cohomology (Proposition \ref{prop: smooth case}).
Therefore our theory of syntomic coefficients is compatible with the construction in \cite{Ban2} of the syntomic coefficients of Besser's rigid syntomic cohomology for smooth schemes.

Our definitions of syntomic coefficients and the Hyodo--Kato map are useful for direct computation.
For example, our construction would be useful for explicit computation of $p$-adic polylogarithms on semistable abelian varieties.

\tableofcontents

The contents of this article are as follows.
In Part \ref{part: Log isocrystal}, we will study cohomology theory for log schemes over a finite field.
In \S \ref{sec: weak formal schemes}, we will recall the notion of weak formal schemes which is slightly generalized in \cite{EY}, as a preparation.
In \S \ref{sec: Isoc}, we will introduce log rigid cohomology with coefficients in a log overconvergent isocrystal over general base.
From \S \ref{sec: absolute isoc}, we will focus on log overconvergent isocrystals over $W^\varnothing$, which will be considered as the coefficients of Hyodo--Kato cohomology in \S \ref{sec: HK coh}.
The tannakian fundamental group of log overconvergent isocrystals is studied in \S \ref{sec: tannakian}--\ref{sec: mixed}.
As an application, we will prove rigidity properties for the Frobenius structure.

In Part \ref{part: mixed char}, we will study the cohomology theory for (strictly semistable) weak formal schemes over a ring of mixed characteristic.
We first prepare some notions and notations concerning de Rham cohomology in \S \ref{sec: de rham}.
The main ingredient of the Hyodo--Kato theory is the Hyodo--Kato map which will be studied in \S \ref{sec: HK map}.
In particular, we will show that the Hyodo--Kato map is independent of the choice of a uniformizer and induces a quasi-isomorphism.
We will introduce syntomic coefficients in \S \ref{sec: syn} and their admissibility in \S \ref{sec: syn coh} by using the Hyodo--Kato map.
We will also define and study the $p$-adic Hodge cohomology and the syntomic cohomology with syntomic coefficients in \S \ref{sec: syn coh}.

%%%%%%%%%%%%%%%%%
%
\part{Theory in positive characteristic: Hyodo--Kato cohomology with coefficients in log overconvergent $F$-isocrystals}\label{part: Log isocrystal}
%
%%%%%%%%%%%%%%%%%
In the first half of this part, we will introduce log overconvergent ($F$-)isocrystals as coefficients of Hyodo--Kato cohomology.
A log overconvergent ($F$-)isocrystal is locally corresponds to a coherent locally free sheaf with an integrable log connection, passing through a Taylor isomorphism and a log stratification.
Using this identification, the Hyodo--Kato cohomology with coefficients will be defined as the hypercohomology of a de Rham complex associated to the log connection.
This construction is an overconvergent analogue of Shiho's log analytic cohomology in \cites{Shi1,Shi2,Shi3}.

In the latter half of this part, we will focus on log overconvergent ($F$-)isocrystals over $W^\varnothing$, and study the tannakian fundamental group of unipotent log overconvergent isocrystals over $W^\varnothing$.
Those results are based on the methods of Chiarellotto \cite{Ch} and Chiarellotto--Le Stum \cites{CL1,CL2} for unipotent and mixed ($F$-)isocrystals on smooth schemes over finite fields.

%%%%%%%%%%%%%%%%%%%%
\section{Weak formal schemes}\label{sec: weak formal schemes}
%%%%%%%%%%%%%%%%%%%%
In this section, we recall some basic definitions and propositions concerning weak formal schemes given in \cite{EY}.
Our definition of weak formal schemes was slightly generalized from the original definition in \cite{Me}, so that a weak formal scheme is not necessarily adic over the base.

Let $R$ be a noetherian ring with an ideal $I$.
We first recall the notion of weakly complete algebras.

\begin{definition}[{\cite{MW}}]
	Let $A$ be an $R$-algebra.
	The {\it weak completion} $A^\dagger$ of $A$ with respect to $(R,I)$ is the $R$-subalgebra of $\widehat{A}$ consisting of elements $\xi$ which can be written as
	\[\xi = \sum_{i\geq 0}P_i(x_1,\ldots,x_r),\]
where $x_1,\ldots,x_r\in A$, $P_i(X_1,\ldots,X_r)\in I^i\cdot R[X_1,\ldots,X_r]$, such that there exists a constant $c>0$ satisfying
	\[c(i + 1)\geq\mathrm{deg}P_i\]
for all $i\geq 0$.

	An $R$-algebra $A$ is said to be {\it weakly complete} with respect to $(R,I)$ if $A^\dagger = A$, and {\it weakly complete finitely generated (wcfg)} with respect to $(R,I)$ if there exists a surjection $R[t_1,\ldots,t_n]^\dagger \rightarrow  A$ over $R$. 
	A wcfg algebra is naturally regarded as a topological $R$-algebra for the $I$-adic topology.
\end{definition}

For any $n\geq 0$, we denote  by $R_{[n]}$ the polynomial algebra $R[s_1,\ldots,s_n]$ and by $I_{[n]}\subset R_{[n]}$  the ideal generated by $I$ and $s_1,\ldots,s_n$.

\begin{definition}[{\cite[Definition 1.3]{EY}}]\label{def: pwcfg}
	A topological $R$-algebra $A$ is {\it pseudo-weakly complete finitely generated (pseudo-wcfg)} with respect to $(R,I)$ if there exists an ideal of definition $J$ of the topology of $A$ and a finite generating system $f_1,\ldots,f_n\in A$ of $J$ such that $A$ is $I_{[n]}$-adically wcfg over $R_{[n]}$ with respect to the map
		\[R_{[n]} = R[s_1,\ldots,s_n] \rightarrow  A,\ s_i\mapsto f_i.\]
	This means that there exists an integer $k\geq 0$ and an $R_{[n]}$-linear surjection 
		\[\rho\colon R_{[n]}[t_1,\ldots,t_k]^{\dagger} \rightarrow  A,\]
	whose domain denotes the $I_{[n]}$-adic weak completion of $R_{[n]}[t_1,\ldots,t_n]$.
	In this case, we call $\rho$ a {\it representation} of $A$.
	A morphism of pseudo-wcfg algebras with respect to $(R,I)$ is a morphism of topological $R$-algebras.
\end{definition}

\begin{remark}
	By \cite[Corollary 1.5]{EY}, for a pseudo-wcfg algebra $A$, the condition in Definition \ref{def: pwcfg} holds for any ideal of definition $J$ and any generating system $f_1,\ldots,f_n\in A$ of $J$.
\end{remark}

Weak formal schemes in our context are defined as a gluing of weak formal spectra of pseudo-wcfg algebras.

\begin{definition}[{\cite[Definition 1.6]{EY}}]
	Let $A$ be a pseudo-wcfg algebra with a representation $R_{[n]}[t_1,\ldots,t_k]^\dagger \rightarrow  A$.
	For $f\in A$, we denote by $A_f^\dagger$ the $I_{[n]}$-adic weak completion of $A_f$.
	Then $A_f^\dagger$ is independent of the choice of representation. 
	
	We define $\Spwf A$ to be the ringed space whose underlying topological space is $\Spec A/J$ for some ideal of definition $J$ of $A$ and the structure sheaf $\cO_{\Spwf A}$ is  defined by
		\[\Gamma(\Spec (A/J)_{\overline{f}},\cO_{\Spwf A}): = A_f^\dagger,\]
	for $f\in A$ and $\overline{f}$ its image in $A/J$.
	
	Note that the underlying topological space of $\Spec A/J$ is independent of the choice of $J$, and the sheaf $\cO_{\Spwf A}$ is well-defined.
\end{definition}

\begin{definition}[{\cite[Definition 1.9]{EY}}]
	A {\it weak formal scheme} with respect to $(R,I)$ is a locally topologically ringed space which admits an open covering $\{U_i\}$ such that each $U_i$ is isomorphic to $\Spwf A$ for some pseudo-wcfg $R$-algebra $A$.
\end{definition}

This generalization allows us to define weak completion of a weak formal scheme along a locally closed weak formal subscheme.

\begin{proposition-definition}[{\cite[Proposition-Definition 1.27]{EY}}]
	The natural inclusion from the category of homeomorphic closed immersions of weak formal schemes to the category of immersions of weak formal schemes has a right adjoint of the form $(Z\hookrightarrow\cZ)\mapsto(Z\hookrightarrow \cZ_Z)$.
	We call $\cZ_Z$ the weak completion of $\cZ$ along $Z$.
\end{proposition-definition}

The smoothness and \'{e}taleness of morphisms of weak formal schemes are defined by the infinitesimal lifting property for first order thickenings as follows.
We also introduce the notion of strong smoothness and strong \'{e}taleness, which are convenient for discussions concerning (log) rigid cohomology.

\begin{definition}[{\cite[Definition 1.32]{EY}}]\label{def: smoothness}
	Let $f\colon\cZ'\rightarrow\cZ$ be a morphism of weak formal schemes.
	\begin{enumerate}
	\item We say $f$ is \textit{smooth} (resp.\ \textit{\'{e}tale}) if for any commutative diagram of weak formal schemes
	\begin{equation}\label{eq: test diagram}
	\xymatrix{
	\cY'\ar[r]^-i\ar[d]^-{\theta'}&\cY\ar[d]^\theta\\
	\cZ'\ar[r]^-f&\cZ}
	\end{equation}
	where $i$ is a closed immersion defined by a square zero ideal, there exists locally on $\cY$ (resp.\ uniquely) a morphism $g\colon\cY\rightarrow\cZ'$ such that $g\circ i=\theta'$ and $f\circ g=\theta$.
	\item We say $f$ is \textit{strongly smooth} (resp.\ \textit{strongly \'{e}tale}) if it is smooth and for any commutative diagram \eqref{eq: test diagram} where $i$ is a homeomorphic closed immersion and $\cY$ is adic over $\Spwf R$, there exists locally on $\cY$ (resp.\ uniquely) a morphism $g\colon\cY\rightarrow\cZ'$ such that $g\circ i=\theta'$ and $f\circ g=\theta$.
	\end{enumerate}
\end{definition}

\begin{remark}
	Suppose that $(R,I)$ is a discrete valuation ring with the maximal ideal, and weak formal schemes $\cZ$ and $\cZ'$ are adic over $\Spwf R$. 
	Then by \cite[Propositions 1.36 and 1.44]{EY}, for a morphism $f\colon \cZ'\rightarrow\cZ$, the following conditions are equivalent to each other:
	\begin{enumerate}
	\item $f$ is smooth (resp.\ \'{e}tale),
	\item $f$ is strongly smooth (resp.\ strongly \'{e}tale),
	\item The morphism $\wh{f}\colon\wh\cZ'\rightarrow\wh\cZ$ induced by $f$ via completion is smooth (resp.\ \'{e}tale) as a morphism of formal schemes.
	\end{enumerate}
\end{remark}

One may consider log structures on weak formal schemes and related notions (e.g.\ fineness) as well as those on schemes.
In this article, we always consider log structures as sheaves with respect to the Zariski topology, and deal only with fine log structures.
For a (weak formal) log scheme $\cZ$, we often denote its underlying (weak formal) scheme by $\cZ$ again and its log structure by $\cN_\cZ$.
The log smoothness, \'{e}taleness, strong log smoothness, and strong log \'{e}taleness are defined by lifting properties similarly to Definition \ref{def: smoothness} (see \cite[Definition 1.52]{EY}).
A morphism of weak formal log schemes $f\colon \cZ'\rightarrow\cZ$ is called \textit{strict} if $f^\ast\cN_\cZ\rightarrow\cN_{\cZ'}$ is an isomorphism.
For an immersion of weak formal log schemes, one can define the canonical exactification by using weak completion in the same way as \cite[Proposition-Definition 2.10, Corollary 2.11]{Shi3}:

\begin{proposition-definition}[{\cite[Proposition-Definition 1.50]{EY}}]\label{def: exactification}
	The natural inclusion from the category of homeomorphic exact closed immersions of fine weak formal log schemes to the category of immersions of fine weak formal log schemes has a right adjoint of the form $(Z\hookrightarrow\cZ)\mapsto(Z\hookrightarrow\cZ^{\mathrm{ex}}_Z)$.
	We call $\cZ^{\mathrm{ex}}_Z$ the {\it exactification} of $Z\hookrightarrow\cZ$.
\end{proposition-definition}

Next we recall the construction of a dagger space $\cZ_\bbQ$ associated to a weak formal scheme $\cZ$.
Assume that $R$ is a complete discrete valuation ring of mixed characteristic $(0,p)$ with the maximal ideal $I$, and let $K$ be the fraction field of $R$.
We note that, for a weak formal scheme $\cZ$ which is adic over $R$, one may associate a dagger space $\cZ_\bbQ$ by tensoring $K$ with the local coordinate rings.
The overconvergent analogue of Raynaud's theorem was proved by Langer and Muralidharan \cite[Theorem 1.1]{LM}.
We extend the definition of $\cZ_\bbQ$ for weak formal schemes which are not necessary adic over $R$, in the same way as \cite[(0.2.6)]{Ber2}.

\begin{proposition}[{\cite[Lemma 1.28 and Proposition 1.29]{EY}}]\label{prop: gen fib}
	Let $R$ and $K$ be as above.
	Let $A$ be a pseudo-wcfg algebra over $R$ and $J\subset A$ be an ideal of definition.
	Take a generating system $a_1,\ldots,a_r$ of $J$ and set
		\begin{align}\label{eq: gen fib}
		A\left[\frac{J}p\right]^\dagger&:=\left(A[x_1,\ldots,x_r]^\dagger/(px_1-a_1,\ldots,px_r-a_r)\right)/(\text{$p$-torsion}).
		\end{align}
	Then $A[\frac Jp]^\dagger$ is a wcfg algebra over $R$ independent of the choice of $a_1,\ldots,a_r$, and satisfies the following universality:
	Any morphism $\tau\colon A\rightarrow B$ of pseudo-wcfg algebras with $B$ $p$-torsion free and $\tau(J)\subset pB$ factors uniquely through a morphism $A[\frac Jp]^\dagger\rightarrow B$.
	
	For any $n\geq 0$, the morphism $A[\frac{J^{n+1}}p]^\dagger\rightarrow A[\frac{J^n}p]^\dagger$ induced by this universality induces an open immersion
	\[\Sp\left(A\left[\frac{J^n}p\right]^\dagger\otimes_RK\right)\hookrightarrow\Sp\left(A\left[\frac{J^{n+1}}p\right]^\dagger\otimes_RK\right)\]
	pf a Wierstra\ss domain.
	The dagger space
		\[\bigcup_{n\geq 0}\Sp\left(A\left[\frac{J^n}p\right]^\dagger\otimes_RK\right)\]
	is independent of the choice of $J$.
\end{proposition}

\begin{definition}[{\cite[Definition 1.30]{EY}}]\label{def: gen fib}
	Let $R$ and $K$ be as above and $\cZ$ a weak formal scheme over $R$.
	For an ideal of definition $\cJ$ of $\cZ$, the constructions \eqref{eq: gen fib} for affine open weak formal subschemes glue to each other by the universality, and define a weak formal scheme $\cZ[\frac{\cJ}p]^\dagger$ which is adic over $R$.
	
	If we fix an ideal of definition $\cJ$ and let $\cZ_n:=\cZ[\frac{\cJ^n}p]^\dagger$, then $\{\cZ_n\}_n$ form an inductive system whose transition maps induce open immersions $\cZ_{n,\bbQ}\hookrightarrow\cZ_{n+1,\bbQ}$.
	Let $\cZ_\bbQ:=\bigcup_{n\geq 1}\cZ_{n,\bbQ}$, which is a dagger space over $K$ and independent of the choice of $\cJ$ by Proposition \ref{prop: gen fib}.
	We call $\cZ_\bbQ$ the {\it generic fiber} of $\cZ$.
	For a morphism $f\colon\cZ'\rightarrow\cZ$, we often denote again by $f$ the morphism $\cZ'_\bbQ\rightarrow\cZ_\bbQ$ induced by $f$, if there is no afraid of confusion.
\end{definition}

For a weak formal scheme $\cZ$, the specialization maps $\mathrm{sp}\colon\cZ_{n,\bbQ}\rightarrow\cZ_n$ for all $n$ together induce a map $\mathrm{sp}\colon\cZ_\bbQ\rightarrow\cZ$.
For an $\cO_\cZ$-module $\cF$, we may associate an $\cO_{\cZ_\bbQ}$-module $\cF_\bbQ:=\mathrm{sp}^*\cF=\mathrm{sp}^{-1}\otimes_{\mathrm{sp}^{-1}\cO_\cZ}\cO_{\cZ_\bbQ}$.

If $f\colon\cZ'\rightarrow\cZ$ is a morphism of fine weak formal log schemes, then the log de Rham complex $\omega^1_{\cZ'/\cZ}$ is associated.
By \cite[Proposition 1.62]{EY} we have canonical isomorphisms $\omega^1_{\cZ'/\cZ,\bbQ}|_{\cZ'_{n,\bbQ}}\cong\omega^1_{\cZ'_n/\cZ}\otimes_RK$ for each $n$, where we endow $\cZ'_n$ with the pull-back log structure from $\cZ'$.
Consequently, the differentials $\omega^m_{\cZ'_n/\cZ}\rightarrow\omega^{m+1}_{\cZ'_n/\cZ}$ induce the global differential $\omega^m_{\cZ'/\cZ,\bbQ}\rightarrow\omega^{m+1}_{\cZ'/\cZ,\bbQ}$.
Thus we obtain a complex $\omega^\bullet_{\cZ'/\cZ,\bbQ}$ of sheaves on $\cZ'_\bbQ$.

%%%%%%%%%%%%%%
\section{Log rigid cohomology with coefficients}\label{sec: Isoc}
%%%%%%%%%%%%%%%

In this section, we will introduce the notions of log connections, log stratifications, Taylor isomorphisms, and log ovreconvergent isocrystals.
Then we will see that a log overconvergent isocrystal is locally interpreted by a log connection, and define log rigid cohomology with coefficients in log overconvergent isocrystals.

\begin{convention}
	According to \cite{EY}, we work with the following convention in the rest of the article:
	For a log scheme $Y$, we always assume that its underlying scheme is separated, locally of finite type, and admits an affine covering indexed by a countable set.
	For a weak formal log scheme $\cZ$, we always assume that there exists an ideal of definition $\cJ$ of $\cZ$ such that the closed log subscheme of $\cZ$ defined by $\cJ$ satisfies the above conditions.
\end{convention}

%%%%%%%%%%%%%%
\subsection{Log connections}\label{subsec: Log connection}
%%%%%%%%%%%%%%

Let $R$ be a discrete valuation ring of characteristic $(0,p)$ with the maximal ideal $I$ and residue field $k$.
In this section, all pseudo-wcfg algebras and weak formal schemes are considered as those with respect to $(R,I)$.
For any pseudo-wcfg algebra $A$, let $A^\varnothing$ be the weak formal log scheme whose underlying space is $\Spwf A$ and whose log structure is trivial.

\begin{definition}
	A {\it fine weak formal log scheme with Frobenius} is a pair $(\cZ,\phi)$ of a fine weak formal log scheme and an endomorphism $\phi$ on $\cZ$ which lifts the absolute Frobenius on $\cZ\times_{R^\varnothing}k^\varnothing$.
	A morphism $f\colon(\cZ',\phi')\rightarrow(\cZ,\phi)$ of fine weak formal log schemes with Frobenius is a morphism $f\colon\cZ'\rightarrow\cZ$ of weak formal log schemes such that $\phi\circ f=f\circ\phi'$.
	A morphism of fine weak formal log schemes with Frobenius is said to be {\it log smooth} (resp.\ \textit{strongly log smooth}) if it is so as a morphism of fine weak formal log schemes. 
\end{definition}

In what follows, we consider a fine weak formal log scheme $\cT$ or a fine weak formal log scheme with Frobenius $(\cT,\sigma)$ as a base.

\begin{definition}
	Let $\cZ$ be a fine weak formal log scheme over $\cT$. 
	\begin{enumerate}
	\item Let $\cE$ be a coherent locally free sheaf on $\cZ_\bbQ$.
		An {\it integrable log connection} on $\cE$ over $\cT$ is an $\cO_{\cT_\bbQ}$-linear map
			\[\nabla\colon \cE\rightarrow\cE\otimes_{\cO_{\cZ_\bbQ}}\omega^1_{\cZ/\cT,\bbQ}\]
		which satisfies
		\begin{itemize}
		\item $\nabla(f\alpha)=f\nabla(\alpha)+\alpha\otimes df$ for any local sections $\alpha\in\cE$ and $f\in\cO_{\cZ_\bbQ}$,
		\item $\nabla^1\circ\nabla=0$, where $\nabla^1\colon\cE\otimes\omega^1_{\cZ/\cT,\bbQ}\rightarrow\cE\otimes\omega^2_{\cZ/\cT,\bbQ}$ is induced from $\nabla$ by $\nabla^1(\alpha\otimes\eta):=\nabla(\alpha)\wedge\eta+\alpha\otimes d\eta$ for local sections $\alpha\in\cE$ and $\eta\in\omega^1_{\cZ/\cT,\bbQ}$.
		\end{itemize}	
	\item We define the category $\MIC(\cZ/\cT)$ as follows:
		\begin{itemize}
		\item An object of $\MIC(\cZ/\cT)$ is a pair $(\cE,\nabla)$ of a coherent locally free sheaf $\cE$ on $\cZ_\bbQ$ and an integrable log connection $\nabla$ on $\cE$ over $\cT$, 
		\item A morphism $f\colon(\cE',\nabla')\rightarrow(\cE,\nabla)$ in $\MIC(\cZ/\cT)$ is an $\cO_{\cZ_\bbQ}$-linear homomorphism $f\colon\cE'\rightarrow\cE$ which is compatible with the log connections.
		\end{itemize}
	\end{enumerate}
\end{definition}

For an object $(\cE,\nabla)\in\MIC(\cZ/\cT)$, we may define $\nabla^k\colon\cE\otimes\omega^k_{\cZ/\cT,\bbQ}\rightarrow\cE\otimes\omega^{k+1}_{\cZ/\cT,\bbQ}$ for any $k\in\bbN$, in a similar way as $\nabla^1$.
Then they define a complex $\cE\otimes\omega^\bullet_{\cZ/\cT,\bbQ}$.
We often denote $\nabla^k$ simply by $\nabla$.

For a commutative diagram
	\[\xymatrix{
	\cZ'\ar[d]^f\ar[r]&\cT'\ar[d]^g\\
	\cZ\ar[r]&\cT,
	}\]
we define a functor
	\begin{equation}
	(f,g)^*\colon\MIC(\cZ/\cT)\rightarrow\MIC(\cZ'/\cT')
	\end{equation}
by $(f,g)^*(\cE,\nabla):=(f^*\cE,f^*\nabla)$ where $f^*\nabla$ is the integrable log connection which maps a local section
	\[f^*\alpha=1\otimes\alpha\in\cO_{\cZ'_\bbQ}\otimes_{f^{-1}\cO_{\cZ_\bbQ}}f^{-1}\cE=f^*\cE\]
to the image of $\nabla(\alpha)$ under the map
	\[f^{-1}(\cE\otimes_{\cO_{\cZ_\bbQ}}\omega^1_{\cZ/\cT,\bbQ})=f^{-1}\cE\otimes_{f^{-1}\cO_{\cZ_\bbQ}}f^{-1}\omega^1_{\cZ/\cT,\bbQ}\xrightarrow{\id\otimes f^\sharp} f^{-1}\cE\otimes_{f^{-1}\cO_{\cZ_\bbQ}}\omega^1_{\cZ'/\cT',\bbQ}=f^*\cE\otimes_{\cO_{\cZ'_\bbQ}}\omega^1_{\cZ'/\cT',\bbQ},\]
where $f^\sharp\colon f^{-1}\omega^1_{\cZ/\cT,\bbQ}\rightarrow\omega^1_{\cZ'/\cT',\bbQ}$ is the natural map induced by $f$.

In particular, a morphism $f\colon\cZ'\rightarrow\cZ$ over $\cT$ induces a functor
	\begin{equation}
	f^*:=(f,\id_\cT)^*\colon\MIC(\cZ/\cT)\rightarrow\MIC(\cZ'/\cT).
	\end{equation}

In addition, for a fine weak formal log scheme with Frobenius $(\cZ,\phi)$ over $(\cT,\sigma)$, we have a functor
	\begin{equation}
	\phi^*:=(\phi,\sigma)^*\colon\MIC(\cZ/\cT)\rightarrow\MIC(\cZ/\cT).
	\end{equation}

\begin{definition}\label{def: FMIC}
	Let $(\cZ,\phi)$ be a fine weak formal log scheme with Frobenius over $(\cT,\sigma)$.
	We define the category $F\MIC((\cZ,\phi)/(\cT,\sigma))$ as follows:
	\begin{itemize}
	\item An object of $F\MIC((\cZ,\phi)/(\cT,\sigma))$ is a triple $(\cE,\nabla,\Phi)$ where $(\cE,\nabla)\in\MIC(\cZ/\cT)$ and $\Phi\colon\phi^*(\cE,\nabla)\xrightarrow{\cong}(\cE,\nabla)$ is an isomorphism in $\MIC(\cZ/\cT)$, which we call a {\it Frobenius structure} on $(\cE,\nabla)$.
	\item A morphism $f\colon(\cE',\nabla',\Phi')\rightarrow(\cE,\nabla,\Phi)$ in $F\MIC((\cZ,\phi)/(\cT,\sigma))$ is a morphism $f\colon(\cE',\nabla')\rightarrow(\cE,\nabla)$ in $\MIC(\cZ/\cT)$ which is compatible with the Frobenius structures.
	\end{itemize}
	
	For an object $(\cE,\nabla,\Phi)\in F\MIC((\cZ,\phi)/(\cT,\sigma))$, the composition $\cE\xrightarrow{\phi^*}\phi^*\cE\xrightarrow{\Phi}\cE$ and the action of $\phi$ on $\omega^\bullet_{\cZ/\cT,\bbQ}$ together induce a $\phi$-semilinear endomorphism on $\cE\otimes\omega^\bullet_{\cZ/\cT,\bbQ}$, which we often denote by $\varphi$.
\end{definition}

For a morphism $f\colon(\cZ',\phi')\rightarrow(\cZ,\phi)$ of fine weak formal log schemes with Frobenius over $(\cT,\sigma)$ and $(\cE,\nabla,\Phi)\in F\MIC((\cZ,\phi)/(\cT,\sigma))$, the isomorphism
	\begin{equation*}
	f^*\Phi\colon \phi'^*f^*\cE=f^*\phi^*\cE\rightarrow f^*\cE
	\end{equation*}
defines a Frobenius structure on $f^*(\cE,\nabla)$.
Therefore we obtain a functor
	\begin{equation}\label{eq: f*Phi}
	f^*\colon F\MIC((\cZ,\phi)/(\cT,\sigma))\rightarrow F\MIC((\cZ',\phi')/(\cT,\sigma)).
	\end{equation}

%%%%%%%%%%%%%%%%
\subsection{Log stratifications}\label{subsec: Log stratification}
%%%%%%%%%%%%%%%%
The definition of log infinitesimal neighborhood by Kato can be extended to weak formal log schemes.

\begin{proposition-definition}[An overconvergent analogue of {\cite[Remark 5.8]{Ka},\cite[Proposition-Definition 3.2.1]{Shi1}}]
	The natural inclusion from the category of exact closed immersions $\cZ'\hookrightarrow \cZ$ of fine weak formal log schemes such that $\cZ'$ is defined in $\cZ$ by an ideal $\cJ$ with $\cJ^{n+1}=0$ to the category of immersions of fine weak formal log schemes has a right adjoint of the form $(Z\mapsto \cZ)\hookrightarrow (Z\mapsto\cZ')$.
	We call this $\cZ'$ the $n$-th log infinitesimal neighborhood of $Z$ in $\cZ$.
\end{proposition-definition}

\begin{proof}
	For an immersion $Z\hookrightarrow\cZ$, let $\cJ$ be the ideal of the exactification $Z\hookrightarrow\cZ^{\mathrm{ex}}$.
	Then the $n$-th infinitesimal neighborhood of $Z$ in $\cZ$ is given as the exact closed weak formal subscheme of $\cZ^{\mathrm{ex}}$ defined by $\cJ^{n+1}$.
\end{proof}

For a fine weak formal log scheme $\cZ$ over $\cT$ and $n\in\bbN$, let $\cZ\hookrightarrow\cZ(n)$ be the exactification of the diagonal embedding of $\cZ$ into the $(n+1)$-fold product of $\cZ$ over $\cT$.
Let
	\begin{align}\label{eq: projections and diagonal}
	p_i\colon\cZ(1)\rightarrow\cZ \ \ (i=1,2), && p_{ij}\colon\cZ(2)\rightarrow\cZ(1)\ \  (1\leq i<j\leq 3),&& \Delta\colon\cZ\rightarrow\cZ(1)
	\end{align}
be the $i$-th projection, the $(i,j)$-th projection, and the diagonal morphism, respectively.

For a fine weak formal log scheme with Frobenius $(\cZ,\phi)$ over $(\cT,\sigma)$, the product of $\phi$ naturally induces $\phi(n)$ on $\cZ(n)$.
The morphisms $p_i$, $p_{i,j}$, and $\Delta$ are compatible with the Frobenius lifts.

For $n\in\bbN$, let $\cZ^n$ be the $n$-th infinitesimal neighborhood of $\cZ$ in $\cZ(1)$.
Note that the canonical morphism $\cZ_\bbQ\rightarrow\cZ^n_\bbQ$ is homeomorphic, and hence we can regard $\cO_{\cZ^n_\bbQ}$ as a sheaf on $\cZ_\bbQ$.
Moreover, we regard $\cO_{\cZ^n_\bbQ}$ as a left $\cO_{\cZ_\bbQ}$-module by the morphism $\rho_{1,n}^\sharp\colon\cO_{\cZ_\bbQ}\rightarrow\cO_{\cZ^n_\bbQ}$ induced by the composition $\rho_{1,n}\colon\cZ^n\rightarrow\cZ(1)\xrightarrow{p_1}\cZ$, and as a right $\cO_{\cZ_\bbQ}$-module by the morphism $\rho_{2,n}^\sharp\colon\cO_{\cZ_\bbQ}\rightarrow\cO_{\cZ^n_\bbQ}$ induced by the composition $\rho_{2,n}\colon\cZ^n\rightarrow\cZ(1)\xrightarrow{p_2}\cZ$.
Let $\Delta_n\colon\cZ^n\rightarrow\cZ(1)$ be the canonical morphism and $\Delta_n^\sharp\colon\Delta_n^{-1}\cO_{\cZ(1)_\bbQ}\rightarrow\cO_{\cZ^n_\bbQ}$ be the induced morphism.
When we consider Frobenius structures, let $\phi^n\colon\cZ^n\rightarrow\cZ^n$ be the morphism induced by $\phi(1)$ on $\cZ(1)$.
For $n,n'\in\bbN$, let
\[q_{n,n'}\colon\cZ^n\times_{\rho_{2,n},\cZ,\rho_{1,n'}}\cZ^{n'}\rightarrow\cZ(1)\]
be the morphism induced from the morphisms
\begin{align*}\cZ^n\times_{\rho_{2,n},\cZ,\rho_{1,n'}}\cZ^{n'}\rightarrow\cZ^n\xrightarrow{\rho_{1,n}}\cZ,&&
\cZ^n\times_{\rho_{2,n},\cZ,\rho_{1,n'}}\cZ^{n'}\rightarrow\cZ^{n'}\xrightarrow{\rho_{2,n'}}\cZ
\end{align*}
where the first morphism of each composition is the canonical projection.
Let
\[\delta_{n,n'}\colon\cZ^n\times_{\rho_{2,n},\cZ,\rho_{1,n'}}\cZ^{n'}\rightarrow\cZ^{n+n'}\]
be the morphism induced from $q_{n,n'}$ by the universality of $\cZ^{n+n'}$.

\begin{definition}
	Let $\cZ$ be a fine weak formal log scheme over $\cT$.
	\begin{enumerate}
	\item Let $\cE$ be a coherent locally free sheaf on $\cZ_\bbQ$.
		A {\it log stratification} on $\cE$ over $\cT$ is a family of $\cO_{\cZ^n_\bbQ}$-linear isomorphisms
			\[\epsilon_n\colon\cO_{\cZ^n_\bbQ}\otimes_{\cO_{\cZ_\bbQ}}\cE\xrightarrow{\cong}\cE\otimes_{\cO_{\cZ_\bbQ}}\cO_{\cZ^n_\bbQ}\]
		for $n\in\bbN$ satisfying the following conditions:
		\begin{itemize}
		\item $\epsilon_0=\id$,
		\item For $m>n$, $\epsilon_m$ modulo $\Ker(\cO_{\cZ^m_\bbQ}\rightarrow\cO_{\cZ^n_\bbQ})$ coincides with $\epsilon_n$,
		\item For any $n$ and $m$,
			\begin{eqnarray*}
			(\id\otimes\delta_{n,m}^*)\circ\epsilon_{n+m}=(\epsilon_n\otimes\id)\circ(\id\otimes\epsilon_{m})\circ(\delta_{n,m}^*\otimes\id)\colon\\
			\cO_{\cZ^{n+m}_\bbQ}\otimes\cE\rightarrow\cE\otimes\cO_{\cZ^n_\bbQ}\otimes\cO_{\cZ^m_\bbQ}
			\end{eqnarray*}
			holds.
		\end{itemize}
	\item We define the category $\Str(\cZ/\cT)$ as follows:
		\begin{itemize}
		\item An object of $\Str(\cZ/\cT)$ is a pair $(\cE,\{\epsilon_n\})$ of a coherent locally free sheaf $\cE$ on $\cZ_\bbQ$ and a log stratification $\{\epsilon_n\}$ on $\cE$ over $\cT$,
		\item A morphism $f\colon(\cE',\{\epsilon'_n\})\rightarrow(\cE,\{\epsilon_n\})$ in $\Str(\cZ/\cT)$ is an $\cO_{\cZ_\bbQ}$-linear homomorphism $f\colon\cE'\rightarrow\cE$ which is compatible with the log connections.
		\end{itemize}
	\end{enumerate}
\end{definition}

For a commutative diagram
	\[\xymatrix{
	\cZ'\ar[d]^f\ar[r]&\cT'\ar[d]^g\\
	\cZ\ar[r]&\cT,
	}\]
we define a functor
	\begin{equation}
	(f,g)^*\colon\Str(\cZ/\cT)\rightarrow\Str(\cZ'/\cT')
	\end{equation}
by $(f,g)^*(\cE,\{\epsilon_n\}):=(f^*\cE,\{f^*\epsilon_n\})$ where $f^*\epsilon_n$ is the composition
	\[\cO_{\cZ'^n_\bbQ}\otimes_{\cO_{\cZ'_\bbQ}}f^*\cE=f^{n,*}(\cO_{\cZ^n_\bbQ}\otimes_{\cO_{\cZ_\bbQ}}\cE)\xrightarrow{f^{n,*}(\epsilon_n)}f^{n,*}(\cE\otimes_{\cO_{\cZ_\bbQ}}\cO_{\cZ^n_\bbQ})=f^*\cE\otimes_{\cO_{\cZ'_\bbQ}}\cO_{\cZ'^n_\bbQ},\]
	and $f^n\colon\cZ'^n\rightarrow\cZ^n$ is induced by $(f,f)\colon\cZ'\times_{\cT'}\cZ'\rightarrow\cZ\times_\cT\cZ$.
	
In particular, a morphism $f\colon\cZ'\rightarrow\cZ$ over $\cT$ induces a functor
	\begin{equation}
	f^*:=(f,\id_\cT)^*\colon\Str(\cZ/\cT)\rightarrow\Str(\cZ'/\cT).
	\end{equation}

In addition, for a fine weak formal log scheme $(\cZ,\phi)$ over $(\cT,\sigma)$, we have a functor
	\begin{equation}
	\phi^*:=(\phi,\sigma)^*\colon\Str(\cZ/\cT)\rightarrow\Str(\cZ/\cT).
	\end{equation}

\begin{definition}
	Let $(\cZ,\phi)$ be a weak formal log scheme with Frobenius over $(\cT,\sigma)$.
	We define the category $F\Str((\cZ,\phi)/(\cT,\sigma))$ as follows:
	\begin{itemize}
	\item An object of $F\Str((\cZ,\phi)/(\cT,\sigma))$ is a triple $(\cE,\{\epsilon_n\},\Phi)$ where $(\cE,\{\epsilon_n\})\in\Str(\cZ/\cT)$ and $\Phi\colon\phi^*(\cE,\{\epsilon_n\})\xrightarrow{\cong}(\cE,\{\epsilon_n\})$ is an isomorphism in $\Str(\cZ/\cT)$, which we call a {\it Frobenius structure} on $(\cE,\{\epsilon_n\})$.
	\item A morphism $f\colon(\cE',\{\epsilon'_n\},\Phi')\rightarrow(\cE,\{\epsilon_n\},\Phi)$ in $F\Str((\cZ,\phi)/(\cT,\sigma))$ is a morphism $f\colon(\cE',\{\epsilon'_n\})\rightarrow(\cE,\{\epsilon_n\})$ in $\Str(\cZ/\cT)$ which is compatible with the Frobenius structures.
	\end{itemize}
\end{definition}

Note that, a morphism $f\colon(\cZ',\phi')\rightarrow(\cZ,\phi)$ of fine weak formal log schemes with Frobenius over $(\cT,\sigma)$ induces a functor
	\begin{equation}\label{eq: f*Phi Str}
	f^*\colon F\Str((\cZ,\phi)/(\cT,\sigma))\rightarrow F\Str((\cZ',\phi')/(\cT,\sigma)).
	\end{equation}
in a similar way as \eqref{eq: f*Phi}.

Next we study the correspondence between log connections and log stratifications.
For a fine weak formal log scheme $\cZ$ over $\cT$ and an element $x\in\Gamma(\cZ,\cN_\cZ)$, we denote by $u(x)\in\Gamma(\cZ(1),\cN_{\cZ(1)})$ the unique element such that $p_2^*(x)=p_1^*(x)u(x)$.
Note that $u(x)$ belongs to $\Gamma(\cZ(1),\cO_{\cZ(1)}^\times)\subset \Gamma(\cZ(1),\cN_{\cZ(1)})$.
For $n\in\bbN$, let $u^n(x)$ be the image of $u(x)$ in $\Gamma(\cZ^n,\cO_{\cZ^n}^\times)$.

\begin{proposition}[{An overconvergent analogue of \cite[Proposition 3.2.5]{Shi1}}]\label{prop: omega}
	Let $\cZ$ be a fine weak formal log scheme over $\cT$.
	There is a canonical isomorphism of $\cO_{\cZ_\bbQ}$-modules
		\begin{equation}\label{eq: omega}
		\Psi\colon\Ker(\cO_{\cZ^1_\bbQ}\rightarrow\cO_{\cZ_\bbQ})\cong\omega^1_{\cZ/\cT,\bbQ},
		\end{equation}
	which satisfies the following condition:
	For a section $x$ of $\cN_{\cZ}$, the image of $u^1(x)-1\in\Ker(\cO_{\cZ^1_\bbQ}\rightarrow\cO_{\cZ_\bbQ})$ 
	by $\Psi$ is equal to $d\log x$.
\end{proposition}

\begin{proof}
	For a pseudo-wcfg algebra $A$ and a finite $A$-module $M$, the $A$-algebra $A\oplus M$ defined by $(a,m)(a',m'):=(aa',am'+a'm)$ is also pseudo-wcfg.
	Thus for any coherent sheaf $\cE$ on $\cZ$, we may construct a weak formal scheme $\cZ\oplus\cE$ by gluing the above construction.
	Let $\rho\colon\cZ\oplus\cE\rightarrow\cZ$ and $\iota\colon\cZ\rightarrow\cZ\oplus\cE$ be the morphisms induced by the canonical injection and projection of the coordinate rings, respectively. 
	We endow $\cZ\oplus\cE$ with the pull-back log structure $\cN_{\cZ\oplus\cE}:=\rho^*\cN_\cZ$, and regard it as a weak formal log scheme over $\cT$.
	
	Let $\mathrm{Def}(\cZ/\cT,\cE)$ be the set of retractions over $\cT$ of $\iota$.
	Note that $\mathrm{Def}(\cZ/\cT,\cE)$ is non-empty, because it contains $\rho$.
	Then by \cite[Lemma 1.51]{EY} we see that there exists a natural simply transitive action of $\Hom_{\cO_\cZ}(\omega^1_{\cZ/\cT},\cE)$ on $\mathrm{Def}(\cZ/\cT,\cE)$.
	In particular, we have a bijection
	\begin{align*}
	\Hom_{\cO_\cZ}(\omega^1_{\cZ/\cT},\cE)\rightarrow\mathrm{Def}(\cZ/\cT,\cE);\ \alpha\mapsto \alpha\cdot \rho.
	\end{align*}
	Using this bijection, the statement follows by the same arguments as in \cite[\S 3.2]{Shi1}.
\end{proof}

\begin{remark}
	Let $(\cZ,\phi)$ be a smooth weak formal log scheme with Frobenius over $(\cT,\sigma)$.
	Since the functor $\Psi$ is canonical, it commutes with the natural action of $\phi$ on the both sides of \eqref{eq: omega}.
\end{remark}

\begin{lemma}\label{lem: OZn is free}
	If $\cZ$ is a log smooth weak formal log scheme over $\cT$, then locally on $\cZ$ there exist elements $x_1,\ldots,x_d\in\Gamma(\cZ,\cN_\cZ)$ such that
	the morphism of weak formal schemes
	\[\cZ(1)\rightarrow\cZ\times_{\Spwf R}\Spwf R\llbracket t_1,\ldots,t_d\rrbracket\]
	defined by $p_1\colon\cZ(1)\rightarrow\cZ$ and an association $t_i\mapsto u(x_i)-1$ is an isomorphism.
\end{lemma}

\begin{proof}
	We note that $p_1\colon\cZ(1)\rightarrow\cZ$ is log smooth and $p_1^*\cN_\cZ\rightarrow\cN_{\cZ(1)}$ is an isomorphism by construction, hence the underlying morphism of $p_1$ is smooth as a morphism of weak formal schemes \cite[Lemma 1.59]{EY}.
	Moreover the morphism $\cZ(1)\rightarrow\cZ\times_\cT\cZ$ is log \'{e}tale \cite[Proposition 1.56]{EY}.
	Therefore by \cite[Proposition 1.54]{EY} we have isomorphisms
	\[\Omega^1_{\cZ(1)/\cZ}=\omega^1_{\cZ(1)/\cZ}\cong\omega^1_{\cZ\times_\cT\cZ/\cZ}\cong\omega^1_{\cZ/\cT},\]
	where the sheaves of (log) differentials are with respect to the first projection.
	
	Since $\omega^1_{\cZ/\cT}$ is locally free \cite[Proposition 1.54]{EY}, one can find locally a basis of $\omega^1_{\cZ/\cT}$ of the form $d\log x_1,\ldots,d\log x_d$ with $x_i\in\Gamma(\cZ,\cN_\cZ)$.
	As $d\log p_1^*(x_i)=0$ in $\omega^1_{\cZ(1)/\cZ}$, we have $d\log p_2^*(x_i)=u(x)^{-1}du(x)$.
	Thus the elements $du(x_i)=d(u(x_i)-1)$ for $i=1,\ldots,d$ give a basis of $\Omega^1_{\cZ(1)/\cZ}$, hence the proposition follows from the strong fibration lemma \cite[Proposition 1.38]{EY}.
\end{proof}

\begin{proposition}\label{prop: conn str}
	\begin{enumerate}
	\item\label{item: conn str1} Let $\cZ$ be a log smooth weak formal log scheme over $\cT$.
		Then there exists a canonical equivalence of categories
			\begin{equation}\label{eq: MIC STR 1}\Str(\cZ/\cT)\xrightarrow{\cong}\MIC(\cZ/\cT).\end{equation}
	\item\label{item: conn str2} Let $(\cZ,\phi)$ be a log smooth weak formal log scheme with Frobenius over $(\cT,\sigma)$.
		Then there exists a canonical equivalence of categories
			\begin{equation}\label{eq: MIC STR 2}F\Str((\cZ,\phi)/(\cT,\sigma))\xrightarrow{\cong}F\MIC((\cZ,\phi)/(\cT,\sigma)).\end{equation}
	\end{enumerate}
\end{proposition}

\begin{proof}
	We associate a formal groupoid $\sZ$ to $\cZ$ over $\cT$ in the same way as \cite[Example 1.2.10]{Shi2}.
	By Lemma \ref{lem: OZn is free} and the equation $\delta_{n,m}^*u^{n+m}(x_i)=u^n(x_i)\otimes u^{m}(x_i)$ for any $n,m\in\bbN$, we see that the elements $(u^n(x_i)-1)$'s satisfy the condition of \cite[Definition 1.2.2]{Shi2}, i.e. $\sZ$ is differentially log smooth.
	Now \eqref{item: conn str1} follows by applying \cite[Porposition 1.2.7]{Shi2} to $\sZ$.
	Finally, \eqref{item: conn str2} follows immediately from \eqref{item: conn str1}.
\end{proof}

According to the constructions in \cite[\S 3.2]{Shi1}, the equivalence \eqref{eq: MIC STR 1} is described as follows:

Let $\cZ$ be a weak formal log scheme over $\cT$, and $\cE$ a coherent locally free sheaf on $\cZ_\bbQ$.
If a log stratification $\{\epsilon_n\}$ on $\cE$ is given, then for any local section $\alpha\in\cE$ we have
	\[\epsilon_1(1\otimes\alpha)-\alpha\otimes 1\in \Ker(\cE\otimes_{\cO_{\cZ_\bbQ}}\cO_{\cZ^1_\bbQ}\rightarrow\cE\otimes_{\cO_{\cZ_\bbQ}}\cO_{\cZ_\bbQ})=\cE\otimes_{\cO_{\cZ_\bbQ}}\Ker(\cO_{\cZ^1_\bbQ}\rightarrow\cO_{\cZ_\bbQ}),\]
and the log connection $\nabla$ is defined by
	\begin{equation}\label{eq: log connection}
	\nabla(\alpha):=(\mathrm{id}\otimes\Psi)(\epsilon_1(1\otimes\alpha)-\alpha\otimes 1).
	\end{equation}

Conversely, if $\cZ$ is log smooth, a log stratification is restored from an integrable connection as follows.
Considering locally, take elements $x_1,\ldots,x_d\in\cN_\cZ$ as in Lemma \ref{lem: OZn is free}, and 
let $\partial^{\log}_i\colon\cE\rightarrow\cE$ for $i=1,\ldots,d$ be the differential operators defined by
	\begin{equation}\label{eq: log diff op}
	\nabla(\alpha)=\sum_{i=1}^d\partial^{\log}_i(\alpha)d\log x_i.
	\end{equation}
	
\begin{proposition}\label{prop: taylor series}
	Let $\cZ$ be a log smooth weak formal log scheme over $\cT$ and $(\cE,\nabla)\in\MIC(\cZ/\cT)$.
	Suppose that there exist elements $x_1,\ldots,x_d\in\Gamma(\cZ,\cN_\cZ)$ as in Lemma \ref{lem: OZn is free}.
	Then the corresponding log stratification is determined by
	\begin{equation}
		\label{eq: taylor}\epsilon_n(1\otimes\alpha)
		=\sum_{\mathbf{k}=(k_i)_i\in\bbN^d}\frac{(\partial_1^{\log{}})^{k_1}\circ\cdots\circ(\partial_d^{\log{}})^{k_d}(\alpha)}{k_1!\cdots k_d!}\otimes(\log u(x_1))^{k_1}\cdots(\log u(x_d))^{k_d} \ \mathrm{mod }\  \cI^{n+1}
	\end{equation}
	for any $\alpha\in \cE$, where
	\[\log u(x):=\log(1+(u(x)-1))=\sum_{n\geq 1}\frac{(-1)^{n-1}}{n}(u(x)-1)^n,\]
and $\cI$ is the ideal of $\cZ_\bbQ$ in $\cZ(1)_\bbQ$.
\end{proposition}

To prove the proposition, we will prepare some notions following \cite[\S 3.2]{Shi1} and prove a few lemmas.
For a coherent sheaf on $\cZ$ and $n\in\bbN$, let
\[\cD\mathrm{iff}^n(\cE,\cE):=\sheafhom_{\cO_{\cZ_\bbQ}}(\cO_{\cZ^n_\bbQ}\otimes\cE,\cE).\]
For $f\in\cD\mathrm{iff}^n(\cE,\cE)$ and $g\in\cD\mathrm{iff}^m(\cE,\cE)$ we define the product $f\ast g\in\cD\mathrm{iff}^{n+m}(\cE,\cE)$ to be the composition
\[\cO_{\cZ^{n+m}_\bbQ}\otimes\cE\xrightarrow{\delta_{n,m}^*\otimes 1}\cO_{\cZ^n_\bbQ}\otimes\cO_{\cZ^m_\bbQ}\otimes\cE\xrightarrow{1\otimes g}\cO_{\cZ^n}\otimes\cE\xrightarrow{f}\cE.\]
Then
\[\cD\mathrm{iff}(\cE,\cE):=\varinjlim_n\cD\mathrm{iff}^n(\cE,\cE)\]
carries a natural $\varprojlim_n\cO_{\cZ^n}$-module structure and an $\cO_\cZ$-algebra structure given by $\ast$.

Suppose elements $x_1,\ldots,x_d\in\Gamma(\cZ,\cN_\cZ)$ as in Lemma \ref{lem: OZn is free} exist.
For $\mathbf{m}=(m_i)_i\in\bbN^d$, put $\xi_{\mathbf{m}}:=\prod_{i=1}^d(u(x_i)-1)^{m_i}\in\Gamma(\cZ(1)_{\bbQ},\cO_{\cZ(1)_\bbQ})$, and denote its image in $\Gamma(\cZ^n_{\bbQ},\cO_{\cZ^n_\bbQ})$ by $\xi_{\mathbf{m},n}$.
Then by Lemma \ref{lem: OZn is free}, the elements $\xi_{\mathbf{m},n}$ for $\lvert\mathbf{m}\rvert\leq n$ give a free basis of $\cO_{\cZ^n}$ over $\cO_\cZ$.
Let $\xi_{\mathbf{m},n}^*\in\cD\mathrm{iff}^n(\cO_{\cZ_\bbQ},\cO_{\cZ_\bbQ})$ be the dual basis.
Its image in $\cD\mathrm{iff}(\cO_{\cZ_\bbQ},\cO_{\cZ_\bbQ})$ is independent of $n\geq\lvert\mathbf{m}\rvert$, which we denote by $\xi_{\mathbf{m}}^*$.
For $i=1,\ldots,d$, let $\mathbf{1}_i\in\bbN^d$ be the elements whose $i$-th entry is $1$ and other entries are $0$.

\begin{lemma}\label{lem: dual basis}
	The product of $\cD\mathrm{iff}(\cO_{\cZ_\bbQ},\cO_{\cZ_\bbQ})$ is commutative, and we have equations
	\begin{align}
	\label{eq: dual basis 2}&\xi_{\mathbf{m}}^*=\xi_{m_1\mathbf{1}_1}^*\ast \xi_{m_2\mathbf{1}_2}^*\ast\cdots\ast\xi_{m_d\mathbf{1}_d}^*&&(\mathbf{m}=(m_i)_i\in\bbN^d),\\
	\label{eq: dual basis 1}&\xi_{k\mathbf{1}_i}^*=\frac{1}{k!}\xi_{\mathbf{1}_i}^*\ast (\xi_{\mathbf{1}_i}^*-1)\ast\cdots\ast(\xi_{\mathbf{1}_i}^*-k+1)&&(k\geq 1,\ 1\leq i\leq d).
	\end{align}
\end{lemma}
\begin{proof}
	Note that, for any $n,m\in\bbN$ and $i=1,\ldots,d$, we have
	\begin{align*}
	\delta_{n,m}^*(u^{n+m}(x_i)-1)&=u^n(x_i)\otimes u^m(x_i)-1\otimes 1\\
	&=(u^n(x_i)-1)\otimes (u^m(x_i)-1)+(u^n(x_i)-1)\otimes 1+1\otimes(u^m(x_i)-1).
	\end{align*}
	Therefore, for any $\mathbf{m},\mathbf{n},\mathbf{r}\in\bbN^d$, we have
	\begin{align}
	\label{eq: dual basis proof}&\xi_{\mathbf{m},\lvert\mathbf{m}\rvert}^*\ast \xi_{\mathbf{n},\lvert\mathbf{n}\rvert}^*(\xi_{\mathbf{r},\lvert\mathbf{m}\rvert+\lvert\mathbf{n}\rvert})\\
	\nonumber&=\xi_{\mathbf{m},\lvert\mathbf{m}\rvert}^*\circ(1\otimes\xi_{\mathbf{n},\lvert\mathbf{n}\rvert}^*)\circ\delta^*_{\lvert\mathbf{m}\rvert,\lvert\mathbf{n}\rvert}(\xi_{\mathbf{r},\lvert\mathbf{m}\rvert+\lvert\mathbf{n}\rvert})\\
	\nonumber&=\xi_{\mathbf{m},\lvert\mathbf{m}\rvert}^*\circ(1\otimes\xi_{\mathbf{n},\lvert\mathbf{n}\rvert}^*)\left(\prod_{i=1}^d\left(\sum_{\substack{a_i,b_i,c_i\in\bbN\\ a_i+b_i+c_i=r_i}}\frac{r_i!}{a_i!b_i!c_i}\left(u^{\lvert\mathbf{m}\rvert}(x_i)-1\right)^{a_i+b_i}\otimes\left(u^{\lvert\mathbf{n}\rvert}(x_i)-1\right)^{b_i+c_i}\right)\right)\\
	\nonumber&=\prod_{i=1}^d\sum_{(\sharp)}\frac{r_i!}{a_i!b_i!c_i!},
	\end{align}
	where the sum $(\sharp)$ is take over $a_i,b_i,c_i\in\bbN$ satisfying
	\begin{align*}
	a_i+b_i+c_i=r_i,&&a_i+b_i=m_i,&&b_i+c_i=n_i.
	\end{align*}
	This shows the commutativity of the product.
	
	If $m_i=0$ for some $i$, we have 
	\[\xi_{\mathbf{m},\lvert\mathbf{m}\rvert}^*\ast\xi_{k\mathbf{1}_i,k}^*(\xi_{\mathbf{r},\lvert\mathbf{m}\rvert+k})=\begin{cases}1&(\mathbf{r}=\mathbf{m}+k\mathbf{1}_i)\\ 0&(\text{otherwise})\end{cases}\]
	for any $k\in\bbN$ as a special case of \eqref{eq: dual basis proof},
	namely we have $\xi_{\mathbf{m}}^*\ast\xi_{k\mathbf{1}_i}^*=\xi_{\mathbf{m}+k\mathbf{1}_i}^*$.
	This implies the equation \eqref{eq: dual basis 2}.
	
	Again as a special case of \eqref{eq: dual basis proof}, for any $k\in\bbN$ and $i=1,\ldots,d$ we have
	\begin{equation}\label{eq: dual basis proof 1}
	\xi_{k\mathbf{1}_i}^*\ast \xi_{\mathbf{1}_i}^*=k\xi_{\mathbf{m}}^*+(k+1)\xi_{(k+1)\mathbf{1}_i}^*.
	\end{equation}
	Now the equation \eqref{eq: dual basis 1} follows again by induction on $k$.
	Indeed, \eqref{eq: dual basis 1} clearly holds for $k=1$.
	If \eqref{eq: dual basis 1} holds for $k=\ell$, then by \eqref{eq: dual basis proof 1} we have
	\begin{align*}
	\xi_{(\ell+1)\mathbf{1}_i}^*&
	=\frac{1}{\ell+1}\xi_{\ell\mathbf{1}_i}^*\ast\left(\xi_{\mathbf{1}_i}^*-\ell\right)=\frac{1}{(\ell+1)!}\xi_{\mathbf{1}_i}^*\ast(\xi_{\mathbf{1}_i}^*-1)\ast\cdots\ast(\xi_{\mathbf{1}_i}^*-\ell),
	\end{align*}
	hence \eqref{eq: dual basis 1} holds also for $k=\ell+1$.
\end{proof}

For positive integers $k,n$ with $k\leq n$, we define a rational number $q_{n,k}$ by the equation of polynomials
\begin{equation}\label{eq: qnk}
X(X-1)\cdots(X-n+1)=\sum_{k=1}^nq_{n,k}X^k.
\end{equation}

\begin{lemma}\label{lem: Taylor expansion}
	For any $k\geq 1$, we have an equation of formal power series
	\[(\log(1+X))^k:=\left(\sum_{n\geq 1}\frac{(-1)^{n-1}}{n}X^n\right)^k=\sum_{n\geq k}\frac{k!}{n!}q_{n,k}X^n.\]
\end{lemma}

\begin{proof}
	One can show that
	\[\frac{d^n}{dX^n}\left(\log(1+X)\right)^k=\frac{1}{(1+X)^n}\sum_{j=1}^nq_{n,j}k(k-1)\cdots(k-j+1)\left(\log(1+X)\right)^{k-j}\]
	by the induction on $n$.
	Thus we have
	\[\frac{d^n}{dX^n}\left(\log(1+X)\right)^k|_{X=0}=\begin{cases}
	k!q_{n,k}&(n\geq k)\\ 0&(n<k).
	\end{cases}\]
	This shows our assertion.
\end{proof}

\begin{proof}[Proof of Proposition \ref{prop: taylor series}]
	We follow the construction of \cite[\S 3.2]{Shi1}.
	Let $\epsilon_1\colon\cO_{\cZ^1_\bbQ}\otimes\cE\rightarrow\cE\otimes\cO_{\cZ^1_\bbQ}$ be the $\cO_{\cZ^1_\bbQ}$-linear homomorphism defined by $\epsilon_1(1\otimes\alpha)=(1\otimes\Psi)^{-1}(\nabla(\alpha))+\alpha\otimes 1$ for any $\alpha\in\cE$.
	Let $\varphi\colon\cD\mathrm{iff}(\cO_{\cZ_\bbQ},\cO_{\cZ_\bbQ})\rightarrow\cD\mathrm{iff}(\cE,\cE)$ be the order-preserving $\varprojlim_n\cO_{\cZ^n_\bbQ}$-linear ring homomorphism such that the induced morphism $\varphi^1\colon\cD\mathrm{iff}^1(\cO_{\cZ_\bbQ},\cO_{\cZ_\bbQ})\rightarrow\cD\mathrm{iff}^1(\cE,\cE)$ associates $(1\otimes f)\circ\epsilon_1$ to $f\in\cD\mathrm{iff}^1(\cO_{\cZ_\bbQ},\cO_{\cZ_\bbQ})$.
	Note that such $\varphi$ uniquely exists (see \cite[Proposition 3.2.9]{Shi1}).
	Then for any $\alpha\in\cE$ and $i=1,\ldots,d$ we have
	\begin{align}\label{eq: taylor series proof}
	\varphi^*(\xi_{\mathbf{1}_i}^*)(1\otimes\alpha)&=\varphi^{1,*}(\xi_{\mathbf{1}_i,1}^*)(1\otimes\alpha)=(1\otimes\xi_{\mathbf{1}_i,1}^*)\circ\epsilon_1(1\otimes\alpha)\\
	\nonumber&=(1\otimes\xi_{\mathbf{1}_i,1}^*)\left(\alpha\otimes 1+\sum_{j=1}^d\partial^{\log}_j(\alpha)\otimes (u(x_i)-1)\right)=\partial^{\log}_i(\alpha).\end{align}
	
	For $n\geq 1$, let
	\[\varphi^{n,*}\colon \sheafhom(\sheafhom(\cO_{\cZ^n_\bbQ}\otimes\cE,\cE),\cE)\rightarrow\sheafhom(\sheafhom(\cO_{\cZ^n_\bbQ},\cO_{\cZ_\bbQ}),\cE)\]
	be the homomorphism induced by $\varphi$.
	By Lemma \ref{lem: dual basis} and \eqref{eq: qnk}, for any $\mathbf{m}\in\bbN^d$ we have
	\[\xi_{\mathbf{m}}^*=\prod_{i=1}^d\left(\sum_{k_i\leq m_i}\frac{q_{m_i,k_i}}{m_i!}(\xi_{\mathbf{1}_i}^*)^{\ast k_i}\right)=\sum_{\substack{\mathbf{k},\mathbf{m}\in\bbN^d\\ \mathbf{k}\leq\mathbf{m}}}\frac{q_{m_1,k_1}\cdots q_{m_d,k_d}}{m_1!\cdots m_d!}(\xi_{\mathbf{1}_1}^*)^{\ast k_1}\ast\cdots\ast(\xi_{\mathbf{1}_d}^*)^{k_d},\]
	where the inequality $\mathbf{k}\leq\mathbf{m}$ means that $k_i\leq m_i$ for all $i$.
	This with \eqref{eq: taylor series proof} implies that
		\[\varphi^{n,*}(\xi_{\mathbf{m},n}^*)(1\otimes\alpha)=\sum_{\substack{\mathbf{k},\mathbf{m}\in\bbN^d\\ \mathbf{k}\leq\mathbf{m}}}\frac{q_{m_1,k_1}\cdots q_{m_d,k_d}}{m_1!\cdots m_d!}(\partial_1^{\log})^{k_1}\circ\cdots\circ(\partial^{\log}_d)^{k_d}(\alpha).\]
	
	As explained in the proof of \cite[Proposition 3.2.11]{Shi1}, $\epsilon_n$ is defined to be the composite
	\[\cO_{\cZ^n_\bbQ}\otimes\cE\rightarrow\sheafhom(\sheafhom(\cO_{\cZ^n_\bbQ}\otimes\cE,\cE),\cE)\xrightarrow{\varphi^{n,*}}\sheafhom(\sheafhom(\cO_{\cZ^n_\bbQ},\cO_{\cZ_\bbQ}),\cE)\cong\cE\otimes\cO_{\cZ^n_\bbQ},\]
	hence we have
	\begin{align}
	\label{eq: taylor proof 2}\epsilon_n(1\otimes\alpha)&=\sum_{\substack{\mathbf{m}\in\bbN^d\\ \lvert\mathbf{m}\rvert\leq n}}\varphi^{n,*}(\xi_{\mathbf{m},n}^*)(1\otimes\alpha)\otimes\xi_{\mathbf{m},n}\\
	\nonumber&=\sum_{\substack{\mathbf{k},\mathbf{m}\in\bbN^d\\ \lvert\mathbf{m}\rvert\leq n, \mathbf{k}\leq\mathbf{m}}}\frac{q_{m_1,k_1}\cdots q_{m_d,k_d}}{m_1!\cdots m_d!}(\partial_1^{\log})^{k_1}\circ\cdots\circ(\partial^{\log}_d)^{k_d}(\alpha)\otimes\xi_{\mathbf{m},n}.
	\end{align}
	
	On the other hand, by Lemma \ref{lem: Taylor expansion} the right hand side of \eqref{eq: taylor} is written as
	\begin{align*}
	&\sum_{\mathbf{k}\in\bbN^d}\frac{(\partial_1^{\log{}})^{k_1}\circ\cdots\circ(\partial_d^{\log{}})^{k_d}(\alpha)}{k_1!\cdots k_d!}\otimes\prod_{i=1}^d\sum_{m_i\geq k_i}\frac{k_i!}{m_i!}q_{m_i,k_i}(u(x_i)-1)^{m_i}\ \mathrm{mod}\ \cI^{n+1}\\
	&=\sum_{\substack{\mathbf{k},\mathbf{m}\in\bbN^d\\ \mathbf{k}\leq\mathbf{m}}}\frac{q_{m_1,k_1}\cdots q_{m_d,k_d}}{m_1!\cdots m_d!}(\partial_1^{\log})^{k_1}\circ\cdots\circ(\partial^{\log}_d)^{k_d}(\alpha)\otimes\xi_{\mathbf{m}}\ \mathrm{mod}\ \cI^{n+1}.
	\end{align*}
	This coincides with \eqref{eq: taylor proof 2} and we finish the proof.
\end{proof}

\begin{corollary}\label{cor: horizontal}
	Let $\cZ$ be a log smooth weak formal log scheme over $\cT$.
	Let $(\cE,\nabla)\in\MIC(\cZ/\cT)$ and denote by $\{\epsilon_n\}$ the log stratification corresponding to $\nabla$.
	A local section $\alpha\in\cE$ is horizontal with respect to $\nabla$ if and only if $\epsilon_1(1\otimes\alpha)=\alpha\otimes 1$.
\end{corollary}

\begin{proof}
	Immediately follows from \eqref{eq: log connection} or \eqref{eq: taylor}.
\end{proof}

%%%%%%%%%%%%%%%%%%%
\subsection{Taylor isomorphisms}\label{subsec: Taylor isomorphism}
%%%%%%%%%%%%%%%%%%%

\begin{definition}
	Let $\cZ$ be a weak formal log scheme over $\cT$.
	\begin{enumerate}
	\item Let $\cE$ be a coherent locally free sheaf on $\cZ_\bbQ$.
		A {\it Taylor isomorphism} on $\cE$ over $\cT$ is an $\cO_{\cZ(1)_\bbQ}$-linear isomorphism $\epsilon\colon p_2^*\cE\xrightarrow{\cong}p_1^*\cE$ satisfying
			\begin{align}\label{eq: cocycle condition}
			\Delta^*(\epsilon)=\mathrm{id}&&\text{and}&&p_{12}^*(\epsilon)\circ p_{23}^*(\epsilon)=p_{13}^*(\epsilon).
			\end{align}
	\item	We define the category $\Str^\dagger(\cZ/\cT)$	as follows:
		\begin{itemize}
		\item An object of $\Str^\dagger(\cZ/\cT)$ is a pair $(\cE,\epsilon)$ of a coherent locally free sheaf $\cE$ on $\cZ_\bbQ$ and a Taylor isomorphism $\epsilon$ on $\cE$ over $\cT$,
		\item A morphism $f\colon(\cE',\epsilon')\rightarrow(\cE,\epsilon)$ in $\Str^\dagger(\cZ/\cT)$ is an $\cO_{\cZ_\bbQ}$-linear homomorphism $f\colon\cE'\rightarrow\cE$ which is compatible with the Taylor isomorphisms.
		\end{itemize}
	\end{enumerate}
\end{definition}

For a commutative diagram of widenings
	\[\xymatrix{
	\cZ'\ar[d]^f\ar[r]&\cT'\ar[d]^g\\
	\cZ\ar[r]&\cT,
	}\]
we define a functor
	\begin{equation}
	(f,g)^*\colon\Str^\dagger(\cZ/\cT)\rightarrow\Str^\dagger(\cZ'/\cT')
	\end{equation}
by $(f,g)^*(\cE,\epsilon):=(f^*\cE,f^*\epsilon)$ where $f^*\epsilon$ is the composition
	\[p'^*_2f^*\cE=f(1)^*p_2^*\cE\xrightarrow{f(1)^*(\epsilon)}f(1)^*p_1^*\cE=p'^*_1f^*\cE,\]
and $f(1)\colon\cZ'(1)\rightarrow\cZ(1)$ is induced by $(f,f)\colon\cZ'\times_{\cT'}\cZ'\rightarrow\cZ\times_\cT\cZ$.

In particular, a morphism $f\colon\cZ'\rightarrow\cZ$ over $\cT$ induces a functor
	\begin{equation}
	f^*:=(f,\id_\cT)^*\colon\Str^\dagger(\cZ/\cT)\rightarrow\Str^\dagger(\cZ'/\cT).
	\end{equation}
	
In addition, for a weak formal log scheme with Frobenius $(\cZ,\phi)$ over $(\cT,\sigma)$, we have a functor
	\begin{equation}
	\phi^*:=(\phi,\sigma)^*\colon\Str^\dagger(\cZ/\cT)\rightarrow\Str^\dagger(\cZ/\cT).
	\end{equation}
		
\begin{definition}
	Let $(\cZ,\phi)$ be a weak formal log scheme with Frobenius over $(\cT,\sigma)$.
	We define the category $F\Str^\dagger((\cZ,\phi)/(\cT,\sigma))$ as follows:
	\begin{itemize}
	\item An object of $F\Str^\dagger((\cZ,\phi)/(\cT,\sigma))$ is a triple $(\cE,\epsilon,\Phi)$ where $(\cE,\epsilon)\in\Str^\dagger(\cZ/\cT)$ and $\Phi\colon\phi^*(\cE,\epsilon)\xrightarrow{\cong}(\cE,\epsilon)$ is an isomorphism in $\Str^\dagger(\cZ/\cT)$, which we call a {\it Frobenius structure} on $(\cE,\epsilon)$.
	\item A morphism $f\colon(\cE',\epsilon',\Phi')\rightarrow(\cE,\epsilon,\Phi)$ in $F\Str^\dagger((\cZ,\phi)/(\cT,\sigma))$ is a morphism $f\colon(\cE',\epsilon')\rightarrow(\cE,\epsilon)$ in $\Str^\dagger(\cZ/\cT)$ which is compatible with the Frobenius structures.
	\end{itemize}
\end{definition}

Note that, a morphism $f\colon(\cZ',\phi')\rightarrow(\cZ,\phi)$ of weak formal log schemes with Frobenius over $(\cT,\sigma)$ induces a functor
	\begin{equation}\label{eq: f*Phi Taylor}
	f^*\colon F\Str^\dagger((\cZ,\phi)/(\cT,\sigma))\rightarrow F\Str^\dagger((\cZ',\phi')/(\cT,\sigma)).
	\end{equation}
in a similar way as \eqref{eq: f*Phi}.

\begin{lemma}\label{lem: Taylor Str}
	\begin{enumerate}
	\item\label{item: Str ff} Let $\cZ$ be a weak formal log scheme over $\cT$.
		There exists a canonical fully faithful functor
		\begin{equation}\label{eq: Taylor Str 1}\Str^\dagger(\cZ/\cT)\rightarrow \Str(\cZ/\cT).\end{equation}
	\item\label{item: FStr ff} Let $(\cZ,\phi)$ be a weak formal log scheme with Frobenius over $(\cT,\sigma)$.
		There exists a canonical fully faithful functor
		\begin{equation}\label{eq: Taylor Str 2}
		F\Str^\dagger((\cZ,\phi)/(\cT,\sigma))\rightarrow F\Str((\cZ,\phi)/(\cT,\sigma)).
		\end{equation}
	\end{enumerate}
\end{lemma}

\begin{proof}
	A Taylor morphism on $\cZ(1)_\bbQ$ induces a log stratification on $\cZ^n$ by taking the inverse image by $\Delta_n\colon\cZ^n\rightarrow\cZ(1)$.
	This defines a faithful functor \eqref{eq: Taylor Str 1}.
	
	Let $\cJ$ be the ideal of the closed immersion $\cZ\hookrightarrow\cZ(1)$.
	Then there exists an ideal of definition $\cI$ of $\cZ(1)$ containing $\cJ$.
	Since the natural map from a wcfg algebra $A$ into its completion with respect to an ideal of definition is injective, we see that $\cO_{\cZ(1)}\rightarrow\varprojlim_m\cO_{\cZ(1)}/\cI^m$ is injective.
	Thus the map $\cO_{\cZ(1)}\rightarrow\varprojlim_m\cO_{\cZ(1)}/\cJ^m$ is also injective.
	This implies that the functor \eqref{eq: Taylor Str 1} is full, and \eqref{item: Str ff} holds.
	Moreover \eqref{item: FStr ff} also follows by the same reason.
\end{proof}

\begin{definition}
	\begin{enumerate}
	\item Let $\cZ$ be a weak formal log scheme over $\cT$.
		We define the category $\MIC^\dagger(\cZ/\cT)$ to be the essential image of the composite
		\[\Str^\dagger(\cZ/\cT)\xrightarrow{\eqref{eq: Taylor Str 1}}\Str(\cZ/\cT)\xrightarrow{\eqref{eq: MIC STR 1}}\MIC(\cZ/\cT).\] 
	\item Let $(\cZ,\phi)$ be a weak formal log scheme with Frobenius over $(\cT,\sigma)$.
		Similarly to the above, we define the category $F\MIC^\dagger((\cZ,\phi)/(\cT,\sigma))$ to be the essential image of the composite
		\[F\Str^\dagger((\cZ,\phi)/(\cT,\sigma))\xrightarrow{\eqref{eq: Taylor Str 2}}F\Str((\cZ,\phi)/(\cT,\sigma))\xrightarrow{\eqref{eq: MIC STR 2}} F\MIC((\cZ,\phi)/(\cT,\sigma)).\]
	\item An object $(\cE,\nabla)\in\MIC(\cZ/\cT)$ (resp.\  $(\cE,\nabla,\Phi)\in F\MIC((\cZ,\phi)/(\cT,\sigma))$) or its log connection $\nabla$ is said to be {\it overconvergent} if $(\cE,\nabla)\in\MIC^\dagger(\cZ/\cT)$ (resp.\  $(\cE,\nabla,\Phi)\in F\MIC^\dagger((\cZ,\phi)/(\cT,\sigma))$).
	\end{enumerate}
\end{definition}

By definition we obtain the following.

\begin{proposition}
	\begin{enumerate}
	\item Let $\cZ$ be a log smooth weak formal log scheme over $\cT$.
		Then there exists a canonical equivalence of categories
			\begin{equation}\label{eq: Taylor MIC 1}\Str^\dagger(\cZ/\cT)\cong\MIC^\dagger(\cZ/\cT).\end{equation}
	\item Let $(\cZ,\phi)$ be a log smooth weak formal log scheme with Frobenius over $(\cT,\sigma)$.
		Then there exists a canonical equivalence of categories
			\begin{equation}\label{eq: Taylor MIC 2}F\Str^\dagger((\cZ,\phi)/(\cT,\sigma))\cong F\MIC^\dagger((\cZ,\phi)/(\cT,\sigma)).\end{equation}
	\end{enumerate}
\end{proposition}

\begin{remark}
	Through the equivalences in the above proposition, a Taylor isomorphism is restored from the corresponding log connection by the right hand side of \eqref{eq: taylor} without modulo $\cI^{n+1}$.
\end{remark}

%%%%%%%%%%%%%%%%%%%%
\subsection{Log overconvergent isocrystals}\label{subsec: Log overconvergent isocrystal}
%%%%%%%%%%%%%%%%%%%%

\begin{definition}\label{def: widening}
	\begin{enumerate}
	\item  A {\it widening} is a triple $(Z,\cZ,i)$ where $Z$ is a fine log scheme over $k$, $\cZ$ is a weak formal log scheme with respect to $(R,I)$, and $i\colon Z\hookrightarrow\cZ$ is a homeomorphic exact closed immersion over $R$.
		A morphism of widenings $(Z',\cZ',i')\rightarrow(Z,\cZ,i)$ is a pair $(f_0,f)$ of morphisms $f_0\colon Z'\rightarrow Z$ and $f\colon\cZ'\rightarrow\cZ$ such that $f\circ i'=i\circ f_0$.
		We often denote $(f_0,f)$ by $f$ for simplicity.
	\item An {\it $F$-widening} is a quadruple $(Z,\cZ,i,\phi)$ where $(Z,\cZ,i)$ is a widening and $\phi\colon\cZ\rightarrow\cZ$ is a lift of the absolute Frobenius on $\cZ\times_{R^\varnothing}k^\varnothing$.
		For an $F$-widening $(Z,\cZ,i,\phi)$, the absolute Frobenius $F_Z$ on $Z$ and $\phi$ on $\cZ$ define a morphism $(F_Z,\phi)\colon(Z,\cZ,i)\rightarrow(Z,\cZ,i)$, which we denote again by $\phi$.
		A morphism of $F$-widenings $f\colon(Z',\cZ',i',\phi')\rightarrow(Z,\cZ,i,\phi)$ is a morphism $f\colon(Z',\cZ',i')\rightarrow(Z,\cZ,i)$ of widenings such that $f\circ\phi'=\phi\circ f$.	
	\end{enumerate}
\end{definition}

In what follows, we consider a widening $(T,\cT,\iota)$ or an $F$-widening $(T,\cT,\iota,\sigma)$ as a base.
For a morphism $f\colon(Z',\cZ',i')\rightarrow(Z,\cZ,i)$ of widenings, we often denote the induced morphism $\cZ'_\bbQ\rightarrow\cZ_\bbQ$ again by $f$, if there is no afraid of confusion.

\begin{proposition-definition}\label{prop-def: log overconvergent site}
	Let $Y$ be a fine log scheme over $T$.
	We define the a category $\OC(Y/\cT)=\OC(Y/(T,\cT,\iota))$ of $Y$ over $(T,\cT,\iota)$ as follows:
	\begin{itemize}
	\item An object of $\OC(Y/\cT)$ is a quintuple $(Z,\cZ,i,h,\theta)$ where $h\colon(Z,\cZ,i)\rightarrow(T,\cT,\iota)$ is a morphism of widenings, and $\theta\colon Z\rightarrow Y$ is a morphism over $T$.
	\item A morphism $f\colon(Z',\cZ',i',h',\theta')\rightarrow(Z,\cZ,i,h,\theta)$ in $\OC(Y/\cT)$ is a morphism $f\colon(Z',\cZ',i')\rightarrow(Z,\cZ,i)$ of widenings over $(T,\cT,\iota)$ such that $\theta\circ f_0=\theta'$.
	\end{itemize}
	
	We define a covering family in $\OC(Y/\cT)$ to be a family of morphisms $\{f_\lambda\colon (Z_\lambda,\cZ_\lambda,i_\lambda,h_\lambda,\theta_\lambda)\rightarrow(Z,\cZ,i,h,\theta)\}_\lambda$ such that
		\begin{itemize}
		\item The morphism $f_{\lambda}\colon\cZ_\lambda\rightarrow\cZ$ is strict for any $\lambda$,
		\item The family $\{f_{\lambda}\colon\cZ_{\lambda,\bbQ}\rightarrow\cZ_\bbQ\}_\lambda$ is an admissible covering, 
		\item The morphism $(f_{\lambda,0},i_\lambda)\colon Z_\lambda\rightarrow Z\times_\cZ\cZ_\lambda$ is an isomorphism for any $\lambda$.
		\end{itemize}
	Then this defines a Grothendieck topology on $\OC(Y/\cT)$.
	We call $\OC(Y/\cT)$ the {\it log overconvergent site}.
	
	An object $(Z,\cZ,i,h,\theta)$ of $\OC(Y/\cT)$ is said to be \textit{strongly log smooth} if $h\colon\cZ\rightarrow\cT$ is strongly log smooth.
\end{proposition-definition}

\begin{proof}
	To prove that coverings are stable by base change, it suffices to show the following claim:
	For a strict morphism $f\colon\cZ'\rightarrow\cZ$ and any $g\colon\cZ''\rightarrow\cZ$, there exists a canonical isomorohism $\cZ'_\bbQ\times_{\cZ_\bbQ}\cZ''_\bbQ\cong (\cZ'\times_\cZ\cZ'')_\bbQ$.
	We may suppose that $\cZ=\Spwf A$, $\cZ'=\Spwf B$, and $\cZ''=\Spwf C$ are affine.
	Let $J\subset A$, $J'\subset B$, and $J''\subset C$ be ideals of definition such that $JB\subset J'$ and $JC\subset J''$.
	As $f$ is strict, the underlying weak formal scheme of $\cZ'\times_\cZ\cZ''$ is given as the fiber product in the category of weak formal schemes.
	Therefore we have $\cZ'\times_\cZ\cZ''=\Spwf B\otimes_A^\dagger C$ and $L:=J'(B\otimes_A^\dagger C)+J''(B\otimes_A^\dagger C)$ is an ideal of definition of $B\otimes_A^\dagger C$.
	
	For integers $k\geq 0$, there exist natural homomorphisms
	\[B\left[\frac{J'^k}p\right]^\dagger\otimes_{A\left[\frac{J^k}p\right]^\dagger}^\dagger C\left[\frac{J''^k}p\right]^\dagger\rightarrow(B\otimes_A^\dagger C)\left[\frac{L^k}p\right]^\dagger,\]
	and they induce a morphism
	\begin{equation}\label{eq: site proof}
	\cZ'_\bbQ\times_{\cZ_\bbQ}\cZ''_\bbQ\rightarrow (\cZ'\times_\cZ\cZ'')_\bbQ.
	\end{equation}
	
	On the other hand, for integeres $k\geq 0$ we also have natural morphisms
	\[(B\otimes_A^\dagger C)\left[\frac{L^{2k}}{p}\right]^\dagger\rightarrow\left(B\left[\frac{J'^k}p\right]^\dagger\otimes_{A\left[\frac{J^k}p\right]^\dagger}^\dagger C\left[\frac{J''^k}p\right]^\dagger\right)/(\text{$p$-torsion}),\]
	and they induce a morphism $(\cZ'\times_\cZ\cZ'')_\bbQ\rightarrow \cZ'_\bbQ\times_{\cZ_\bbQ}\cZ''_\bbQ$.
	By construction it is easy to see that this morphism is the inverse of \eqref{eq: site proof}, and the claim is proved.
\end{proof}

\begin{remark}
	Unlike \cite[Definition 2.1.3]{Shi2}, for an object $(Z,\cZ,i,h,\theta)\in\OC(Y/\cT)$ we do not suppose that $\cZ$ is $p$-adic.
	Instead we define a covering family of $\OC(Y/\cT)$ by the topology of the associated dagger spaces.
\end{remark}

Note that an object $(Z,\cZ,i,h,\theta)$ of $\OC(Y/\cT)$ is expressed by a commutative diagram
	\[\xymatrix{
	&Z\ar[ld]_\theta\ar[d]^{h_0}\ar[r]^i&\cZ\ar[d]^{h}\\
	Y\ar[r]&T\ar[r]^\iota&\cT.
	}\]

\begin{proposition}
	The category $\OC(Y/\cT)$ has products and fiber products.
\end{proposition}

\begin{proof}
	The product $(Z,\cZ,i,h,\theta)\times(Z',\cZ',i',h',\theta')$ in $\OC(Y/\cT)$ is the object expressed by a commutative diagram
		\[\xymatrix{
		&Z\times_Y Z'\ar[ld]\ar[d]\ar[r]&(\cZ\times_\cT\cZ')_{Z\times_YZ'}\ar[d]\\
	Y\ar[r]&T\ar[r]^\iota&\cT,
		}\]
	where $(\cZ\times_\cT\cZ')_{Z\times_YZ'}$ is the exactification of $(i,i')\colon Z\times_YZ'\hookrightarrow\cZ\times_\cT\cZ'$.
	
	The fiber product $(Z,\cZ,i)\times_{(Z'',\cZ'',i'')}(Z',\cZ',i')$ in $\OC(Y/\cT)$ is the object expressed by a commutative diagram
		\[\xymatrix{
		&Z\times_{Z''} Z'\ar[ld]\ar[d]\ar[r]&\cZ\times_{\cZ''}\cZ'\ar[d]\\
	Y\ar[r]&T\ar[r]^\iota&\cT.
		}\]
\end{proof}

For a widening $(Z,\cZ,i)$, let $\{\cZ_n\}$ be an inductive system of $p$-adic weak formal schemes over $\cZ$ of Definition \ref{def: gen fib} given by using the ideal of $Z$ in $\cZ$ as an ideal of definition of $\cZ$.
We endow $\cZ_n$ with the pull-back log structure of $\cZ$.
Let $Z_n:=Z\times_\cZ\cZ_n$ and $i_n\colon Z_n\hookrightarrow\cZ_n$ the canonical exact closed immersion.
Then we obtain an inductive system $\{(Z_n,\cZ_n,i_n)\}_n$ of widenings over $(Z,\cZ,i)$ in a canonical way, which we call the family of {\it universal enlargements} of $(Z,\cZ,i)$.

Let $\sE$ be a sheaf on $\OC(Y/\cT)$.
Then for any object $(Z,\cZ,i,h,\theta)\in\OC(Y/\cT)$ we may associate a sheaf $\sE_\cZ=\sE_{(Z,\cZ,i,h,\theta)}$ on $\cZ_\bbQ$ as follows:
Let $\{(Z_n,\cZ_n,i_n)\}_n$ be the family of universal enlargements of $(Z,\cZ,i)$, and let $\frU\subset\cZ_{\bbQ}$ be an affinoid open subset.
As $\frU$ is quasi-compact, it is contained in $\cZ_{n,\bbQ}$ for some $n$.
Since $\cZ_n$ is $p$-adic, there by \cite[Lemma 4.7]{LM} exist an admissible blow-up $\cZ'\rightarrow\cZ_n$ and an open subset $\cU\subset\cZ'$ such that $\frU=\cU_\bbQ$.
Now we let
\begin{equation}\label{eq: induced sheaf on Z}
\Gamma(\frU,\sE_{\cZ}):=\Gamma((U,\cU,i_{\cU},h_\cU,\theta_\cU),\sE)
\end{equation}
where $\cU$ is endowed with the pull-back log structure of $\cZ_n$, $U:=\cU\times_{\cZ_n}Z_n$, and $i_\cU$, $h_\cU$, and $\theta_\cU$ are defined in the obvious way.
By definition of the topology of $\OC(Y/\cT)$,  that $\sE$ is a sheaf immediately implies that \eqref{eq: induced sheaf on Z} is well-defined and satisfies the sheaf condition for a covering family consisting of affinoid open subsets.
Since the topology of $\cZ_{\bbQ}$ is generated by affinoid open subsets, \eqref{eq: induced sheaf on Z} is extended to a sheaf $\sE_{\cZ}$ on $\cZ_{\bbQ}$.

We denote by $\sO_{Y/\cT}$ the sheaf on $\OC(Y/\cT)$ defined by $\Gamma((Z,\cZ,i,h,\theta),\sO_{Y/\cT}):=\Gamma(\cZ_\bbQ,\cO_{\cZ_\bbQ})$.
Then we have $\sO_{Y/\cT,\cZ}=\cO_{\cZ_\bbQ}$.

\begin{definition}\label{def: isocrystal}
	A {\it log overconvergent isocrystal} on $Y$ over $(T,\cT,\iota)$ is an $\sO_{Y/\cT}$-module $\sE$ on $\OC(Y/\cT)$, such that
	\begin{itemize}
	\item For any $(Z,\cZ,i,h,\theta)\in\OC(Y/\cT)$, the sheaf $\sE_\cZ$ on $\cZ_\bbQ$ is a coherent locally free $\cO_{\cZ_\bbQ}$-module,
	\item For any morphism $f\colon(Z',\cZ',i',h',\theta')\rightarrow(Z,\cZ,i,h,\theta)$, the natural morphism $f^*\sE_{\cZ}\rightarrow\sE_{\cZ'}$ is an isomorphism.
	\end{itemize}
	A morphism $f\colon\sE'\rightarrow\sE$ of log overconvergent isocrystals is an $\sO_{Y/\cT}$-linear homomorphism.
	We denote by
		\[\Isoc^\dagger(Y/\cT)=\Isoc^\dagger(Y/(T,\cT,\iota))\]
	the category of log overconvergent isocrystals on $Y$ over $(T,\cT,\iota)$.
\end{definition}

Let $\varrho\colon (T',\cT',\iota')\rightarrow(T,\cT,\iota)$ be a morphism of widenings and consider a commutative diagram
	\[\xymatrix{
	Y'\ar[r]\ar[d]^\rho&T'\ar[d]^{\varrho_0}\\
	Y\ar[r]&T.
	}\]
Then we may regard an object of $\OC(Y'/\cT')$ as an object of $\OC(Y/\cT)$ via $\rho$ and $\varrho$.
This defines a functor $\OC(Y'/\cT')\rightarrow\OC(Y/\cT)$, which is continuous and cocontinous.
Moreover the induced functor between the topoi preserves log overconvergent isocrystals, namely we obtain a functor
	\begin{equation}\label{eq: pull back isoc}
	(\rho,\varrho)^*\colon\Isoc^\dagger(Y/\cT)\rightarrow\Isoc^\dagger(Y'/\cT').
	\end{equation}
	
In particular, a morphism $f\colon Y'\rightarrow Y$ over $T$ induces a functor
	\begin{equation}
	f^*:=(f,\id_{(T,\cT,\iota)})^*\colon\Isoc^\dagger(Y/\cT)\rightarrow\Isoc^\dagger(Y'/\cT).
	\end{equation}

In addition, the absolute Frobenius $F_T$ on $T$ and a Frobenius lift $\sigma$ on $\cT$ together define an endomorphism $\sigma$ on $(T,\cT,\iota)$.
This and the absolute Frobenius $F_Y$ on $Y$ together define a functor
	\begin{equation}\label{eq: Frob of isoc}
	\sigma^*:=(F_Y,\sigma)^*\colon\Isoc^\dagger(Y/\cT)\rightarrow\Isoc^\dagger(Y/\cT).
	\end{equation}

\begin{definition}
	A {\it log overconvergent $F$-isocrystal} on $Y$ over $(T,\cT,\iota,\sigma)$ is a pair $(\sE,\Phi)$ of a log overconvergent isocrystal $\sE\in\Isoc^\dagger(Y/\cT)$ and an isomorphism $\Phi\colon \sigma^*\sE\xrightarrow{\cong}\sE$, which we call a {\it Frobenus structure} on $\sE$.
	A morphism $f\colon(\sE',\Phi')\rightarrow(\sE,\Phi)$ of log overconvergent $F$-isocrystals is an $\sO_{Y/\cT}$-linear homomorphism $f\colon\sE'\rightarrow\sE$ which is compatible with the Frobenius strucures.
	We denote by
		\[F\Isoc^\dagger(Y/\cT)=F\Isoc^\dagger(Y/(T,\cT,\iota,\sigma))\]
	the category of log overconvergnet $F$-isocrystals on $Y$ over $(T,\cT,\iota,\sigma)$.
\end{definition}

The structure sheaf $\sO_{Y/\cT}$ with a Frobenius structure given by the natural identification $\sigma^*\sO_{Y/\cT}=\sO_{Y/\cT}$ is a log overconvergent $F$-isocrystal, which we denote again by $\sO_{Y/\cT}$.

Let $\varrho\colon(T',\cT',\iota',\sigma')\rightarrow(T,\cT,\iota,\sigma)$ be a morphism of $F$-widenings and consider a commutative diagram
	\[\xymatrix{
	Y'\ar[r]\ar[d]^\rho&T'\ar[d]^{\varrho_0}\\
	Y\ar[r]&T.
	}\]
Then the functor $(\rho,\varrho)^*$ in \eqref{eq: pull back isoc} induces a functor
	\begin{equation}
	(\rho,\varrho)^*\colon F\Isoc^\dagger(Y/\cT)\rightarrow F\Isoc^\dagger(Y'/\cT'),
	\end{equation}
which associates to $(\sE,\Phi)\in F\Isoc^\dagger(Y/\cT)$ an object $(\rho,\varrho)^*\sE$ with the Frobenius structure
	\[(\rho,\varrho)^*\Phi\colon\sigma'^*\circ(\rho,\varrho)^*\sE=(\rho,\varrho)^*\circ\sigma^*\sE\xrightarrow[\cong]{(\rho,\varrho)^*(\Phi)}(\rho,\varrho)^*\sE.\]

We note that, for sheaves $\sE$ and $\sE'$ on $\OC(Y/\cT)$ and an object $(Z,\cZ,i,h,\theta)\in\OC(Y/\cT)$, we have $\sheafhom(\sE',\sE)_\cZ=\sheafhom(\sE'_\cZ,\sE_\cZ)$.
Indeed, if we take an object $(U,\cU,i_\cU,h_\cU,\theta_\cU)$ as in \eqref{eq: induced sheaf on Z} and write as $\sE(\cU):=\Gamma((U,\cU,i_\cU,h_\cU,\theta_\cU),\sE)$, then
we have
\[\Gamma(\frU,\sheafhom(\sE',\sE)_\cZ)=\Hom(\sE'(\cU),\sE(\cU))=\Gamma(\frU,\sheafhom(\sE'_\cZ,\sE_\cZ)).\]
	
\begin{definition}\label{def: tensor and hom}
	Let $Y$ be a fine log scheme over $T$.
	For $\sE,\sE'\in\Isoc^\dagger(Y/\cT)$, the tensor product and internal Hom as $\sO_{Y/\cT}$-modules define those as log overconvergent isocrystals, which we denote by $\sE\otimes\sE'$ and $\sheafhom(\sE,\sE')$.
	
For $(\sE,\Phi),(\sE',\Phi')\in F\Isoc^\dagger(Y/\cT)$, we define $(\sE,\Phi)\otimes(\sE',\Phi')$ by
	\begin{equation*}\label{eq: tensor phi}\sigma^*(\sE\otimes\sE')=\sigma^*\sE\otimes\sigma^*\sE'\xrightarrow{\Phi\otimes\Phi'}\sE\otimes\sE',
	\end{equation*}
and $\sheafhom((\sE',\Phi'),(\sE,\Phi))$ by
	\begin{eqnarray*}
	\label{eq: hom phi}
	\sigma^*\sheafhom(\sE',\sE)=\sheafhom(\sigma^*\sE',\sigma^*\sE)&\rightarrow&\sheafhom(\sE',\sE)\\
	\nonumber f&\mapsto&\Phi\circ f\circ\Phi'^{-1}.
	\end{eqnarray*}
We call $\sE^\vee:=\sheafhom(\sE,\sO_{Y/\cT})$ (resp.\ $(\sE,\Phi)^\vee:=\sheafhom((\sE,\Phi),\sO_{Y/\cT})$) the dual of $\sE$ (resp.\ $(\sE,\Phi)$).
\end{definition}

%%%%%%%%%%%%%%%
\subsection{Realization and log rigid cohomology with coefficients}\label{subsec: Realization}
%%%%%%%%%%%%%%%

Let $Y$ be a fine log scheme over $T$.
For an object $(Z,\cZ,i,h,\theta)\in\OC(Y/\cT)$ and $n\geq 0$, let $\cZ\hookrightarrow\cZ(n)$ be as in \S \ref{subsec: Log stratification}.
Let $i(n)\colon Z\hookrightarrow\cZ(n)$ be the base change of the diagonal embedding $Z\hookrightarrow\cZ(n)$.
Then there is a natural morphism $h(n)\colon(Z,\cZ(n),i(n))\rightarrow(T,\cT,\iota)$, and we have $(Z,\cZ(n),i(n),h(n),\theta)\in\OC(Y/\cT)$.

The morphisms $p_i$, $p_{i,j}$, $\Delta$ in \eqref{eq: projections and diagonal} induce morphisms in $\OC(Y/\cT)$, which we denote by the same symbols.
For $\sE\in\Isoc^\dagger(Y/\cT)$, by the definition of a log overconvergent isocrystal we obtain a canonical $\cO_{\cZ(1)_\bbQ}$-linear isomorphism
	\[\epsilon_\cZ=\epsilon_{(Z,\cZ,i,h,\theta)}\colon p_2^*\sE_\cZ\xrightarrow{\cong}\sE_{\cZ(1)}\xleftarrow{\cong}p_1^*\sE_\cZ,\]
which satisfies the conditions \eqref{eq: cocycle condition}.
Therefore we obtain a functor
	\begin{equation}\label{eq: realization Str}
	R_\cZ=R_{(Z,\cZ,i,h,\theta)}\colon \Isoc^\dagger(Y/\cT)\rightarrow\Str^\dagger(\cZ/\cT);\ \sE\mapsto(\sE_\cZ,\epsilon_\cZ).
	\end{equation}

For $(\sE,\Phi)\in F\Isoc^\dagger(Y/\cT)$, $(Z,\cZ,i,h,\theta)\in\OC(Y/\cT)$, and a Frobenius lift $\phi$ on $\cZ$ compatible with $\sigma$, the morphism $(F_Z,\phi)\colon(Z,\cZ,i,h\circ\phi,\theta\circ F_Z)\rightarrow(Z,\cZ,i,h,\theta)$ induces an isomorphism
	\[\phi^*\sE_{(Z,\cZ,i,h,\theta)}\xrightarrow{\cong}\sE_{(Z,\cZ,i,h\circ\phi,\theta\circ F_Z)}.\]
Since $h\circ\phi=\sigma\circ h$ and $\theta\circ F_Z=F_Y\circ\theta$, we have $\sE_{(Z,\cZ,i,h\circ\phi,\theta\circ F_Z)}=(\sigma^*\sE)_{(Z,\cZ,i,h,\theta)}$.
Thus we obtain an $\cO_{\cZ_\bbQ}$-linear isomorphism
	\[\Phi_{\cZ,\phi}=\Phi_{(Z,\cZ,i,h,\theta),\phi}\colon \phi^*\sE_{(Z,\cZ,i,h,\theta)}\xrightarrow{\cong}\sE_{(Z,\cZ,i,h\circ\phi,\theta\circ F_Z)}=(\sigma^*\sE)_{(Z,\cZ,i,h,\theta)}\xrightarrow[\Phi]{\cong}\sE_{(Z,\cZ,i,h,\theta)},\]
which is compatible with $\epsilon_\cZ$.
Therefore we obtain a functor
	\begin{equation}\label{eq: realization FStr}
	R_{\cZ,\phi}=R_{(Z,\cZ,i,h,\theta),\phi}\colon F\Isoc^\dagger(Y/\cT)\rightarrow F\Str^\dagger((\cZ,\phi)/(\cT,\sigma));\ (\sE,\Phi)\mapsto(\sE_\cZ,\epsilon_\cZ,\Phi_{\cZ,\phi}).
	\end{equation}
We call the functors $R_\cZ$ and $R_{\cZ,\phi}$ the {\it realization functors} at $(Z,\cZ,i,h,\theta)$ (with respect to $\phi$).

When $\cZ$ is log smooth over $\cT$, then we often denote the composites
\begin{align*}
&\Isoc^\dagger(Y/\cT)\xrightarrow{R_\cZ}\Str^\dagger(\cZ/\cT)\xrightarrow{\eqref{eq: Taylor MIC 1}}\MIC^\dagger(\cZ/\cT),\\
&F\Isoc^\dagger(Y/\cT)\xrightarrow{R_\cZ}\Str^\dagger((\cZ,\phi)/(\cT,\sigma))\xrightarrow{\eqref{eq: Taylor MIC 1}}\MIC^\dagger((\cZ,\phi)/(\cT,\sigma))
\end{align*}
again by $R_\cZ=R_{(Z,\cZ,i,h,\theta)}$ and $R_{\cZ,\phi}=R_{(Z,\cZ,i,h,\theta),\phi}$, respectively, and call them the {\it realization functors}.

\begin{proposition}\label{prop: tensor hom realization}
	Let $Y$ be a fine log scheme over $T$, and let $(Z,\cZ,i,h,\theta)\in\OC(Y/\cT)$.
	Suppose that $\cZ$ is log smooth over $\cT$.
	For $\sE,\sE'\in\Isoc^\dagger(Y/\cT)$, denote by $\nabla$ and $\nabla'$ the log connections on $\cZ_\bbQ$ induced by $\sE$ and $\sE'$, respectively.
	Then the realizations of $\sE\otimes\sE'$ and $\sheafhom(\sE',\sE)$ at $(Z,\cZ,i,h,\theta)$ are given by the coherent locally free sheaves $\sE_\cZ\otimes_{\cO_{\cZ_\bbQ}}\sE'_\cZ$ and $\sheafhom_{\cO_{\cZ_\bbQ}}(\sE'_\cZ,\sE_\cZ)$ with the log connections
		\begin{align}
		\label{eq: tensor nabla}\sE_\cZ\otimes_{\cO_{\cZ_\bbQ}}\sE'_\cZ&\rightarrow\sE_\cZ\otimes_{\cO_{\cZ_\bbQ}}\sE'_\cZ\otimes_{\cO_{\cZ_\bbQ}}\omega^1_{\cZ/\cT,\bbQ}\\
		\nonumber\alpha\otimes\alpha'&\mapsto\nabla(\alpha)\otimes\alpha'+\alpha\otimes\nabla'(\alpha'),\\
		\label{eq: sheafhom nabla}\sheafhom_{\cO_{\cZ_\bbQ}}(\sE'_\cZ,\sE_\cZ)&\rightarrow\sheafhom_{\cO_{\cZ_\bbQ}}(\sE'_\cZ,\sE_\cZ\otimes_{\cO_{\cZ_\bbQ}}\omega^1_{\cZ/\cT,\bbQ})=\sheafhom_{\cO_{\cZ_\bbQ}}(\sE'_\cZ,\sE_\cZ)\otimes_{\cO_{\cZ_\bbQ}}\omega^1_{\cZ/\cT,\bbQ}\\
		\nonumber f&\mapsto\nabla\circ f-(f\otimes\id)\circ\nabla'.
		\end{align}
	respectively.
	Moreover let $\phi$ be a Frobenius lift on $\cZ$.
	Then for any $(\sE,\Phi),(\sE',\Phi')\in F\Isoc^\dagger(Y/\cT)$, the realizations of $(\sE,\Phi)\otimes(\sE',\Phi')$ and $\sheafhom((\sE',\Phi'),(\sE,\Phi))$ at $(Z,\cZ,i,h,\theta)$ with respect to $\phi$ are given by the above log connections and the Frobenius structures defined in a similar manner as in Definition \ref{def: tensor and hom}.
\end{proposition}

\begin{proof}
	Considering locally, we may take  a basis $d\log x_1,\ldots,d\log x_d$ of $\omega^1_{\cZ/\cT,\bbQ}$ and define differential operators $\partial_j\colon\sE_\cZ\rightarrow\sE_\cZ$ and $\partial'_j\colon\sE'_\cZ\rightarrow\sE'_\cZ$ as in \eqref{eq: log diff op}. (We omit the superscript $\log$ in the notation of log differential operators.)
	We denote the Taylor isomorphisms and the log stratifications on $\sE_\cZ$, $\sE'_\cZ$, $(\sE\otimes\sE')_\cZ$ by $\epsilon$, $\epsilon'$, $\wt\epsilon$, and $\{\epsilon_n\}$, $\{\epsilon'_n\}$, $\{\wt\epsilon_n\}$, respectively.
	
	As $\wt\epsilon$ is given by
	\[\epsilon\otimes\epsilon'\colon p_2^*(\sE_\cZ\otimes\sE'_\cZ)=p_2^*\sE_\cZ\otimes p_2^*\sE'_\cZ\rightarrow p_1^*\sE_\cZ\otimes p_1^*\sE_\cZ=p_1^*(\sE_\cZ\otimes\sE'_\cZ),\]
	the log stratification $\wt\epsilon_n$ is written as the composite
	\[\cO_{\cZ^n_\bbQ}\otimes\sE_\cZ\otimes\sE'_\cZ\xrightarrow{\epsilon_n\otimes 1}\sE_\cZ\otimes\cO_{\cZ^n_\bbQ}\otimes\sE'_\cZ\xrightarrow{1\otimes\epsilon'_n}\sE_\cZ\otimes\sE'_\cZ\otimes\cO_{\cZ^n_\bbQ}.\]
	Note that, for any section $x\in\cN_\cZ$ and positive integer $r$, the image of $(\log u(x))^r$ in $\cO_{\cZ^1_\bbQ}$ is equal to $u^1(x)-1$ if $r=1$ and vanishes if $r>1$.
	Thus by Proposition \ref{prop: taylor series} we have
	\begin{align*}
	\wt\epsilon_1(1\otimes\alpha\otimes\alpha')&=(1\otimes\epsilon'_1)\left(\left(\alpha\otimes 1+\sum_{i=1}^d\partial_i(\alpha)\otimes (u^1(x_i)-1)\right)\otimes\alpha'\right)\\
	&=\alpha\otimes\left(\alpha'\otimes 1+\sum_{j=1}^d\partial'_j(\alpha')\otimes(u^1(x_j)-1)\right)+\sum_{i=1}^d\partial_i(\alpha)\otimes\alpha'\otimes(u^1(x_i)-1)\\
	&=\alpha\otimes\alpha' \otimes 1+\sum_{i=1}^d\left(\partial_i(\alpha)\otimes\alpha'+\alpha\otimes\partial'_j(\alpha')\right)\otimes(u^1(x_j)-1).
	\end{align*}
	Then it follows by \eqref{eq: log connection} that the log connection $\wt\nabla$ on $(\sE\otimes\sE')_\cZ=\sE_\cZ\otimes\sE'_\cZ$ is given by
	\begin{align*}
	\wt\nabla(\alpha\otimes\alpha')&=\sum_{i=1}^d\left(\partial_i(\alpha)\otimes\alpha'+\alpha\otimes\partial'_j(\alpha')\right)\otimes d\log x_i
	=\nabla(\alpha)\otimes\alpha'+\alpha\otimes\nabla'(\alpha').
	\end{align*}
	
	Next, we denote by $\wh\nabla$ and $\{\wh{\epsilon}_n\}$ the log connection \eqref{eq: sheafhom nabla} and the corresponding log stratification, respectively.
	Let $\breve\epsilon$ and $\{\breve\epsilon_n\}$ be the Taylor isomorphism and the log stratification on $\sheafhom(\sE',\sE)_\cZ$, respectively.
	Then
	\[\breve\epsilon\colon p_2^*\sheafhom(\sE',\sE)_\cZ=\sheafhom(p_2^*\sE'_\cZ,p_2^*\sE_\cZ)\rightarrow\sheafhom(p_1^*\sE'_\cZ,p_1^*\sE_\cZ)=p_1^*\sheafhom(\sE',\sE)_\cZ\]
	is given by $\psi\mapsto \epsilon\circ \psi\circ\epsilon'^{-1}$.
	For $f\in\sheafhom(\sE',\sE)_\cZ$, denote by $\breve\epsilon_1(1\otimes f)$ the image of $1\otimes f$ by the composite
	\[\cO_{\cZ^1_\bbQ}\otimes\sheafhom(\sE',\sE)_\cZ\xrightarrow{\breve\epsilon_1}\sheafhom(\sE',\sE)_\cZ\otimes\cO_{\cZ^1_\bbQ}=\sheafhom_{\cO_{\cZ^1_\bbQ}}(\sE'_\cZ\otimes\cO_{\cZ^1_\bbQ},\sE\otimes\cO_{\cZ^1_\bbQ})\]
	Then $\breve\epsilon_1$ is characterized as the unique $\cO_{\cZ^1_\bbQ}$-linear homomorphism satisfying $\breve{\epsilon}_1(1\otimes f)\circ\epsilon'_1=\epsilon_1\circ (1\otimes f)$ for any $f$.
	Thus it suffices to show the equality
	\begin{equation}\label{eq: epsilon_1 claim}
	\wh\epsilon_1(1\otimes f)\circ\epsilon'_1=\epsilon_1\circ(1\otimes f).
	\end{equation}
	
	Let $\wh\partial_1,\ldots,\wh\partial_d$ be the differential operators on $\sheafhom(\sE'_\cZ,\sE_\cZ)$ induced by $\wh\nabla$.
	As
	\[\wh\nabla(f)(\alpha')=\sum_{i=1}^d\partial_i\circ f(\alpha')d\log x_i-\sum_{i=1}^df\circ\partial'_i(\alpha')d\log x_i\]
	for any $\alpha'\in\sE'$, we have $\wh\partial_i(f)=\partial_i\circ f-f\circ\partial'_i$.
	Thus we have
	\begin{align*}
	\wh\epsilon_1(1\otimes f)\circ\epsilon'_1(1\otimes\alpha')&=\left(f\otimes 1+\sum_{i=1}^d\wh\partial_i(f)\otimes(u^1(x_i)-1)\right)\left(\alpha'\otimes 1+\sum_{i=1}^d\partial'_i(\alpha')\otimes (u^1(x_i)-1)\right)\\
	&=f(\alpha')\otimes 1+\sum_{i=1}^d\wh\partial_i(f)(\alpha')\otimes (u^1(x_i)-1)+\sum_{i=1}^df\circ\partial'_i(\alpha')\otimes(u^1(x_i)-1)\\
	&=f(\alpha')\otimes 1+\sum_{i=1}^d\partial_i\circ f(\alpha')\otimes(u^1(x_i)-1)\\
	&=\epsilon_1\circ(1\otimes f)(1\otimes\alpha')
	\end{align*}
	for any $\alpha'\in\sE'$, hence we get \eqref{eq: epsilon_1 claim} as claimed.
	
	Finally, the descriptions of Frobenius structures are obvious.
\end{proof}

In order to show the equivalence between log overconvergent isocrystals and local data of coherent locally free sheaves with log stratifications (or log connections), we first prepare some propositions.
The following proposition is a log overconvergent analogue of \cite[Proposition 2.2.17]{Ber2}.

\begin{proposition}\label{prop: isoc indep 1}
	\begin{enumerate}
	\item\label{item: isoc indep 1} Let $f_j\colon(Z',\cZ',i')\rightarrow(Z,\cZ,i)$ for $j=1,2$ be morphisms of widenings over $(T,\cT,\iota)$ such that the morphisms $f_k:=f_{1,k}=f_{2,k}\colon Z'\rightarrow Z$.
		Then there exists a canonical natural equivalence $\varepsilon_{f_1,f_2}\colon f_2^*\cong f_1^*$ between two functors $\Str^\dagger(\cZ/\cT)\rightarrow\Str^\dagger(\cZ'/\cT)$.
		Moreover, if $(\cE,\epsilon)$ is an object of $\Str^\dagger(\cZ/\cT)$ and $\alpha$ is a horizontal section of $\cE$ with respect to the corresponding log connection, then we have an equality
			\begin{equation}\label{eq: horizontal section}
			\varepsilon_{f_1,f_2} (f_2^*(\alpha))=f_1^*(\alpha).
			\end{equation}
		If three morphisms $f_j\colon(Z',\cZ',i')\rightarrow(Z,\cZ,i)$ are given, we have an equality
			\begin{equation}\label{eq: epsilon compatible}\varepsilon_{f_2,f_3}\circ\varepsilon_{f_1,f_2}=\varepsilon_{f_1,f_3}\colon f_3^*\cong f_1^*.
			\end{equation}
	\item\label{item: isoc indep 2} Let $f_j\colon(Z',\cZ',i',\phi')\rightarrow(Z,\cZ,i,\phi)$ for $j=1,2$ be morphisms of $F$-widenings over $(T,\cT,\iota,\sigma)$ such that the morphisms $f_k:=f_{1,k}=f_{2,k}\colon Z'\rightarrow Z$.
		Then there exists a canonical natural equivalence $\varepsilon_{f_1,f_2}\colon f_2^*\cong f_1^*$ between two functors $F\Str^\dagger((\cZ,\phi)/(\cT,\sigma))\rightarrow F\Str^\dagger((\cZ',\phi')/(\cT,\sigma))$.
		Moreover the equalities \eqref{eq: horizontal section} and \eqref{eq: epsilon compatible} also hold.
	\end{enumerate}
\end{proposition}
	
\begin{proof}
	We first prove \eqref{item: isoc indep 1}.
	Let $g\colon\cZ'\rightarrow\cZ(1)$ be the morphism induced by $(f_1,f_2)\colon\cZ'\rightarrow\cZ\times_{\cT}\cZ$.
	Let $(\cE,\epsilon)\in \Str^\dagger(\cZ/\cT)$.
	Since $p_j\circ g=f_j$ ($j=1,2$), there is an isomorphism of $\cO_{\cZ'_\bbQ}$-modules
		\[\varepsilon_{f_1,f_2}:=g^*(\epsilon)\colon f_2^*\cE=g^*p_2^*\cE\xrightarrow{\cong} g^*p_1^*\cE=f_1^*\cE.\]
	It is enough to show that $\varepsilon_{f_1,f_2}$ is compatible with $f_1^*\epsilon$ and $f_2^*\epsilon$, that is the commutativity of the following diagram:
		\[\xymatrix{
		p'^*_2f_2^*\cE\ar@{=}[r]\ar[d]_{p'^*_2(\varepsilon_{f_1,f_2})=(g\circ p'_2)^*(\epsilon)} &
		f_2(1)^*p_2^*\cE\ar[rr]^{f_2(1)^*(\epsilon)}&&
		f_2(1)^*p_1^*\cE\ar@{=}[r]&
		p'^*_1f_2^*\cE\ar[d]^{p'^*_1(\varepsilon_{f_1,f_2})=(g\circ p'_1)^*(\epsilon)}\\
		p'^*_2f_1^*\cE\ar@{=}[r]&
		f_1(1)^*p_2^*\cE\ar[rr]^{f_1(1)^*(\epsilon)}&&
		f_1(1)^*p_1^*\cE\ar@{=}[r]&
		p'^*_1f_1^*\cE,
		}\]
	where $p'_j\colon \cZ'(1)\rightarrow\cZ'$ ($j=1,2$) are the canonical projections.
	Let $h_1\colon\cZ'(1)\rightarrow\cZ(2)$ be the morphism induced by $(f_1,g)\colon\cZ'\times_\cT\cZ'\rightarrow\cZ\times_\cT\cZ\times_\cT\cZ$, and let $h_2\colon\cZ'(1)\rightarrow\cZ(2)$ be the morphism induced by $(g,f_2)\colon\cZ'\times_\cT\cZ'\rightarrow\cZ\times_\cT\cZ\times_\cT\cZ$.
	Then we have
		\begin{align*}
		f_1(1)=p_{12}\circ h_1,&&f_2(1)=p_{23}\circ h_2,&&g\circ p'_1=p_{12}\circ h_2,&&g\circ p'_2=p_{23}\circ h_1.
		\end{align*}
	These equations and the cocycle condition \eqref{eq: cocycle condition} give
		\begin{eqnarray*}
		f_1(1)^*(\epsilon)\circ(g\circ p'_2)^*(\epsilon)
		&=&(p_{12}\circ h_1)^*(\epsilon)\circ(p_{23}\circ h_1)^*(\epsilon)
		=h_1^*(p_{12}^*(\epsilon)\circ p_{23}^*(\epsilon))
		=(p_{13}\circ h_1)^*(\epsilon),\\
		(g\circ p'_1)^*(\epsilon)\circ f_2(1)^*(\epsilon)
		&=&(p_{12}\circ h_2)^*(\epsilon)\circ(p_{23}\circ h_2)^*(\epsilon)
		=h_2^*(p_{12}^*(\epsilon)\circ p_{23}^*(\epsilon))
		=(p_{13}\circ h_2)^*(\epsilon).
		\end{eqnarray*}
	Since $p_{13}\circ h_1=p_{13}\circ h_2$, we obtained the desired commutativity.
		
	The equality \eqref{eq: horizontal section} follows immediately from Corollary \ref{cor: horizontal}.
	
	Finally assume that another morphism $f_3\colon(Z',\cZ',i')\rightarrow(Z,\cZ,i)$ is given.
	If we let $G\colon\cZ'\rightarrow\cZ(2)$ be the morphism induced by $(f_1,f_2,f_3)\colon\cZ'\rightarrow\cZ\times_\cT\cZ\times_\cT\cZ$, then we have $\varepsilon_{f_i,f_j}=G^*p_{i,j}^*(\epsilon)$ for any $i$ and $j$.
	Therefore the cocycle condition \eqref{eq: cocycle condition} implies \eqref{eq: epsilon compatible}.
	
	To prove \eqref{item: isoc indep 2},we have to show that $\varepsilon_{f_1,f_2}$ is compatible with the Frobenius structures, that is, the following diagram commutes for any object $(\cE,\epsilon,\Phi)\in F\Str^\dagger((\cZ,\phi)/(\cT,\sigma))$:
	\[\xymatrix{
	\phi'^*f_2^*\cE\ar[d]_-{f_2^*(\Phi)}\ar[rr]^-{\phi'^*(\varepsilon_{f_1,f_2})}&&\phi'^*f_1^*\cE\ar[d]^-{f_1^*(\Phi)}\\
	f_2^*\cE\ar[rr]^-{\varepsilon_{f_1,f_2}}&&f_1^*\cE.
	}\]
	By definition of $\varepsilon_{f_1,f_2}$, this diagram is written as
	\[\xymatrix{
	g^*p_2^*\phi^*\cE\ar[d]_-{g^*p_2^*(\Phi)}\ar[rr]^-{g^*(\phi^*\epsilon)}&&g^*p_1^*\phi^*\cE\ar[d]^-{g^*p_1^*(\Phi)}\\
	g^*p_2^*\cE\ar[rr]^-{g^*(\epsilon)}&&g^*p_1^*\cE.
	}\]
	The commutativity of this diagram follows from the fact that $\Phi\colon\phi^*\cE\rightarrow\cE$ defines a morphism in $\Str^\dagger(\cZ/\cT)$.
\end{proof}

The following is an adaptation of \cite[Proposition 5.2.6]{Shi1} to our setting.

\begin{proposition}\label{prop: realization}
	Let $Y$ be a fine log scheme over $T$.
	Assume that there exists a strongly log smooth object $(Y,\cZ,i,h,\id_Y)\in\OC(Y/\cT)$ whose first entry is $Y$.
	Then the realization functor $R_\cZ$ is an equivalence of categories.
	
	If moreover there exists a lift of Frobenius $\phi$ on $\cZ$ which is compatible with $\sigma$, then the realization functor $R_{\cZ,\phi}$ is also an equivalence of categories.
\end{proposition}

\begin{proof}
	We construct a quasi-inverse of $R_\cZ$ as follows.
	Let $(\cE,\epsilon)\in \Str^\dagger(\cZ/\cT)$ and $(Z',\cZ',i',h',\theta')\in\OC(Y/\cT)$.
	Let $\cJ$ be an ideal of definition of $\cZ'$.
	For any integer $n\geq 0$, let $\cZ'_n:=\cZ'[\frac{\cJ^n}p]^\dagger$ and $Z'_n:=Z'\times_{\cZ'}\cZ'_n$.
	We denote by $(Z'_n,\cZ'_n,i'_n,h'_n,\theta'_n)\in\OC(Y/\cT)$ the object induced from $(Z',\cZ',i',h',\theta')$.
	
	Since $\cZ$ is strongly log smooth over $\cT$, by \cite[Proposition 1.61]{EY} there locally on $\cZ'_n$ exists a morphism $f\colon \cZ'_n\rightarrow\cZ$ which makes the following diagram commutative:
		\[\xymatrix{
		Z'_n\ar[r]^-{i'_n}\ar[d]&\cZ'_n\ar[d]\ar[ld]^-f\\
		\cZ\ar[r]&\cT,
		}\]
	where the left vertical morphism is the composite $Z'_n\xrightarrow{\theta'_n} Y\xrightarrow{i}\cZ$.
	
	Then the definition $\cE_{(Z'_n,\cZ'_n,i'_n,h'_n,\theta'_n)}:=f^*\cE$ is well-defined.
	Indeed, if there are two morphisms $f_j\colon\cZ'_n\rightarrow\cZ$ for $j=1,2$ making the above diagram commutative, then by Proposition \ref{prop: isoc indep 1} we obtain a canonical isomorphism $\varepsilon_{f_1,f_2}\colon f^*_2\cE\cong f^*_1\cE$ satisfying the compatibility \eqref{eq: epsilon compatible}.
	
	Thus we may define a sheaf $\cE_{(Z',\cZ',i',h',\theta')}$ on $\cZ'_\bbQ$ by gluing $\cE_{(Z'_n,\cZ'_n,i'_n,h'_n,\theta'_n)}$ for all $n$, and a sheaf $\sE$ on $\OC(Y/\cT)$ by setting $\Gamma(\sE,(Z',\cZ',i',h',\theta')):=\Gamma(\cZ'_\bbQ,\cE_{(Z',\cZ',i',h',\theta')})$.
	This gives a quasi-inverse of $R_\cZ$.
	
	If moreover a lift of Frobenius $\phi$ on $\cZ$ is given, the isomorphism
		\[(\sigma^*\sE)_{(Z',\cZ',i',h',\theta')}=\sE_{(Z',\cZ',i',\sigma\circ h',F_Y\circ\theta')}=f^*\phi^*\cE\xrightarrow{f^*(\Phi)}f^*\cE=\sE_{(Z',\cZ',i',h',\theta')}\]
	induces a Frobenius structure on $\sE$, and we get a quasi-inverse of $R_{\cZ,\phi}$.
\end{proof}

Combining Proposition \ref{prop: conn str}, Lemma \ref{lem: Taylor Str}, and Proposition \ref{prop: realization}, we obtain the following.

\begin{corollary}\label{cor: realization of isoc}
	Let $Y$ be a fine log scheme over $T$.
	If there exists a strongly log smooth object $(Y,\cZ,i,h,\id_Y)\in\OC(Y/\cT)$ whose first entry is $Y$, then we have a commutative diagram
		\[\xymatrix{
		\Isoc^\dagger(Y/\cT)\ar[r]^-{R_\cZ}_-\cong & \Str^\dagger(\cZ/\cT)\ar[r]_-\cong\ar[d] & \MIC^\dagger(\cZ/\cT)\ar[d]\\
		& \Str(\cZ/\cT)\ar[r]_-\cong & \MIC(\cZ/\cT),
		}\]
	where the horizontal functors are equivalences and the vertical functors are fully faithful.
	
	If moreover a Frobenius lift $\phi$ on $\cZ$ compatible with $\sigma$ exists, then we also have a commutative diagram
	\[\xymatrix{
		F\Isoc^\dagger(Y/\cT)\ar[r]^-{R_{\cZ,\phi}}_-\cong & F\Str^\dagger((\cZ,\phi)/(\cT,\sigma))\ar[r]_-\cong\ar[d] & F\MIC^\dagger((\cZ,\phi)/(\cT,\sigma))\ar[d]\\
		& F\Str((\cZ,\phi)/(\cT,\sigma))\ar[r]_-\cong & F\MIC((\cZ,\phi)/(\cT,\sigma)),
		}\]
	where the horizontal functors are equivalences and the vertical functors are fully faithful.
\end{corollary}

In general, an object $(Y,\cZ,i,h,\id_Y)$ as in Proposition \ref{prop: realization} would not exist globally on $Y$, hence we need local constructions.

\begin{definition}
	Let $Y$ be a fine log scheme over $T$,
	A {\it local embedding datum} for $Y$ over $(T,\cT,\iota)$ is a a finite family $(Z_\lambda,\cZ_\lambda,i_\lambda,h_\lambda,\theta_\lambda)_{\lambda\in\Lambda}$ of strongly log smooth objects in $\OC(Y/\cT)$ such that $\{\theta_\lambda\colon Z_\lambda\rightarrow Y\}_\lambda$ is a Zariski covering by strict open immersions.
	
	A {\it local embedding $F$-datum} for $Y$ over $(T,\cT,\iota,\sigma)$ is a finite family $(Z_\lambda,\cZ_\lambda,i_\lambda,h_\lambda,\theta_\lambda,\phi_\lambda)_{\lambda\in\Lambda}$ consisting of a local embedding datum $(Z_\lambda,\cZ_\lambda,i_\lambda,h_\lambda,\theta_\lambda)_{\lambda\in\Lambda}$ and Frobenius lifts $\phi_\lambda$ on $\cZ_\lambda$ which are compatible with $\sigma$.
\end{definition}

We note that at least one local embedding datum exists in the following case.

\begin{proposition}
	A local embedding ($F$-)datum exists for any fine log scheme over $T$.
\end{proposition}

\begin{proof}
	Since open immersion is strongly log \'{e}tale, we may suppose that $\cT$ is affine and admits a chart.
	(Note that log structures in this article are Zariski sheaves.)
	Then the assertion without Frobenius lift is given in \cite[Lemma 2.6]{EY}.
	A Frobenius lift also can be obtain by using the construction of its proof, after shrinking $\cT$ so that $\sigma$ admits a chart.
\end{proof}

Let $Y$ be a fine log scheme over $T$ and $(Z_\lambda,\cZ_\lambda,i_\lambda,h_\lambda,\theta_\lambda)_{\lambda\in\Lambda}$ be a local embedding datum for $Y$.
For $m\in\bbN$ and $\underline{\lambda}=(\lambda_0,\ldots,\lambda_m)\in\Lambda^{m+1}$, let
	\begin{equation}\label{eq: product of embedding data}
	(Z_{\underline{\lambda}},\cZ_{\underline{\lambda}},i_{\underline{\lambda}},h_{\underline{\lambda}},\theta_{\underline{\lambda}}):=(Z_{\lambda_0},\cZ_{\lambda_0},i_{\lambda_0},h_{\lambda_0},\theta_{\lambda_0})\times\cdots\times(Z_{\lambda_m},\cZ_{\lambda_m},i_{\lambda_m},h_{\lambda_m},\theta_{\lambda_m}),
	\end{equation}
where the product is taken in $\OC(Y/\cT)$.
If Frobenius lifts $\phi_\lambda$ are given, let $\phi_{\underline{\lambda}}$ be the Frobenius lift on $\cZ_{\underline{\lambda}}$ induced by $\phi_{\lambda_0}\times\cdots\times\phi_{\lambda_m}$ on $\cZ_{\lambda_0}\times_\cT\cdots\times_\cT\cZ_{\lambda_m}$.
We denote by
	\begin{align*}
	&\mathrm{pr}_j\colon(Z_{\underline{\lambda}},\cZ_{\underline{\lambda}},i_{\underline{\lambda}},h_{\underline{\lambda}},\theta_{\underline{\lambda}})\rightarrow(Z_{\lambda_j},\cZ_{\lambda_j},i_{\lambda_j},h_{\lambda_j},\theta_{\lambda_j})&&\text{for }\ 0\leq j\leq m\\
	&\mathrm{pr}_{j,k}\colon(Z_{\underline{\lambda}},\cZ_{\underline{\lambda}},i_{\underline{\lambda}},h_{\underline{\lambda}},\theta_{\underline{\lambda}})\rightarrow(Z_{\lambda_j},\cZ_{\lambda_j},i_{\lambda_j},h_{\lambda_j},\theta_{\lambda_j})\times(Z_{\lambda_k},\cZ_{\lambda_k},i_{\lambda_k},h_{\lambda_k},\theta_{\lambda_k})&&\text{for }\ 0\leq j<k\leq m,
	\end{align*}
the canonical projections.

\begin{definition}\label{def: MIC descent}
	Let $Y$ be a fine log scheme over $T$ and  $(Z_\lambda,\cZ_\lambda,i_\lambda,h_\lambda,\theta_\lambda)_{\lambda\in\Lambda}$ a local embedding datum for $Y$.
	We define the category $\MIC^\dagger((\cZ_\lambda)_{\lambda\in\Lambda}/\cT)$ as follows:
	\begin{itemize}
	\item An object of $\MIC^\dagger((\cZ_\lambda)_{\lambda\in\Lambda}/\cT)$ is a family $\{(\cE_\lambda,\nabla_\lambda),\rho_{\lambda_0,\lambda_1}\}$ of $(\cE_\lambda,\nabla_\lambda)\in\MIC^\dagger(\cZ_\lambda/\cT)$ for $\lambda\in\Lambda$ and isomorphisms $\rho_{\lambda_0,\lambda_1}\colon \mathrm{pr}_1^*(\cE_{\lambda_1},\nabla_{\lambda_1})\xrightarrow{\cong}\mathrm{pr}_0^*(\cE_{\lambda_0},\nabla_{\lambda_0})$ for $(\lambda_0,\lambda_1)\in\Lambda^2$, such that for any $(\lambda_0,\lambda_1,\lambda_2)\in\Lambda^3$ the diagram
			\[\xymatrix{
			\mathrm{pr}_2^*(\cE_{\lambda_2},\nabla_{\lambda_2})\ar[rr]^{\mathrm{pr}^*_{0,2}(\rho_{\lambda_0,\lambda_2})}\ar[rd]_{\mathrm{pr}^*_{1,2}(\rho_{\lambda_1,\lambda_2})} && \mathrm{pr}_0^*(\cE_{\lambda_0},\nabla_{\lambda_0})\\
			& \mathrm{pr}_1^*(\cE_{\lambda_1},\nabla_{\lambda_1})\ar[ru]_{\mathrm{pr}_{0,1}^*(\rho_{\lambda_0,\lambda_1})}&
			}\]
		commutes.
	\item A morphism $f\colon \{(\cE'_\lambda,\nabla'_\lambda),\rho'_{\lambda_0,\lambda_1}\}\rightarrow\{(\cE_\lambda,\nabla_\lambda),\rho_{\lambda_0,\lambda_1}\}$ in $\MIC^\dagger((\cZ_\lambda)_{\lambda\in\Lambda}/\cT)$ is a family $f=\{f_\lambda\}_\lambda$ of morphisms $f_\lambda\colon(\cE'_\lambda,\nabla'_\lambda)\rightarrow(\cE_\lambda,\nabla_\lambda)$ in $\MIC^\dagger(\cZ_\lambda/\cT)$ which are compatible with $\rho'_{\lambda_0,\lambda_1}$ and $\rho_{\lambda_0,\lambda_1}$.
	\end{itemize}
	
	We define the category $\Str^\dagger((\cZ_\lambda)_{\lambda\in\Lambda}/\cT)$ in a similar manner.
	
	Moreover, for a local embedding $F$-datum $(Z_\lambda,\cZ_\lambda,i_\lambda,h_\lambda,\theta_\lambda,\phi_\lambda)_{\lambda\in\Lambda}$ for $Y$, we define the categories $F\MIC^\dagger((\cZ_\lambda,\phi_\lambda)_{\lambda\in\Lambda}/(\cT,\sigma))$ and $F\Str^\dagger((\cZ_\lambda,\phi_\lambda)_{\lambda\in\Lambda}/(\cT,\sigma))$ in a similar manner as above.
\end{definition}

Since log overconvergent isocrystals satisfy descent with respect to $\{Z_\lambda\}_\lambda$ and there are canonical equivalences
	\begin{eqnarray*}
	&&\Isoc^\dagger(Z_{\underline{\lambda}}/\cT)\xrightarrow{\cong}\Str^\dagger(\cZ_{\underline{\lambda}}/\cT)\xrightarrow{\cong}\MIC^\dagger(\cZ_{\underline{\lambda}}/\cT),\\
	&&F\Isoc^\dagger(Z_{\underline{\lambda}}/\cT)\xrightarrow{\cong}F\Str^\dagger((\cZ_{\underline{\lambda}},\phi_{\underline{\lambda}})/(\cT,\sigma))\xrightarrow{\cong}F\MIC^\dagger((\cZ_{\underline{\lambda}},\phi_{\underline{\lambda}})/(\cT,\sigma)),
	\end{eqnarray*}
we obtain the following corollary to Proposition \ref{prop: realization}.

\begin{corollary}\label{cor: descent realization}
	Let $Y$ be a fine log scheme over $T$.
	For a local embedding datum $(Z_\lambda,\cZ_\lambda,i_\lambda,h_\lambda,\theta_\lambda)_{\lambda\in\Lambda}$ for $Y$, there exist canonical equivalences
		\begin{equation}
		\Isoc^\dagger(Y/\cT)\xrightarrow{\cong}\Str^\dagger((\cZ_\lambda)_{\lambda\in\Lambda}/\cT)\xrightarrow{\cong}\MIC^\dagger((\cZ_\lambda)_{\lambda\in\Lambda}/\cT).
		\end{equation}
	For a local embedding $F$-datum $(Z_\lambda,\cZ_\lambda,i_\lambda,h_\lambda,\theta_\lambda,\phi_\lambda)_{\lambda\in\Lambda}$, there exist canonical equivalences
		\begin{equation}\label{eq: descent realization}
		F\Isoc^\dagger(Y/\cT)\xrightarrow{\cong} F\Str^\dagger((\cZ_\lambda,\phi_\lambda)_{\lambda\in\Lambda}/(\cT,\sigma))\xrightarrow{\cong} F\MIC^\dagger((\cZ_\lambda,\phi_\lambda)_{\lambda\in\Lambda}/(\cT,\sigma)).
		\end{equation} 
\end{corollary}

Let $(Z_\lambda,\cZ_\lambda,i_\lambda,h_\lambda,\theta_\lambda)_{\lambda\in\Lambda}$ be a local embedding datum for $Y$.
Then
	\begin{equation}\label{eq: simplicial system}
	(Z_m,\cZ_m,i_m,h_m,\theta_m):=\coprod_{\underline{\lambda}\in\Lambda^{m+1}}(Z_{\underline{\lambda}},\cZ_{\underline{\lambda}},i_{\underline{\lambda}},h_{\underline{\lambda}},\theta_{\underline{\lambda}})
	\end{equation}
for $m\geq 0$ give a simplicial object $(Z_\bullet,\cZ_\bullet,i_\bullet,h_\bullet,\theta_\bullet)$ in $\OC(Y/\cT)$.
In particular we have a simplicial widening $(Z_\bullet,\cZ_\bullet,i_\bullet)$ over $(T,\cT,\iota)$.
For an object $\sE\in\Isoc^\dagger(Y/\cT)$, we have its realizations $(\sE_{\cZ_m},\nabla_{\cZ_m})\in\MIC^\dagger(\cZ_m/\cT)$ with isomorphisms $f_\tau^*(\sE_{\cZ_m},\nabla_{\cZ_m})\xrightarrow{\cong}(\sE_{\cZ_n},\nabla_{\cZ_n})$ where $f_\tau\colon(Z_n,\cZ_n,i_n)\rightarrow(Z_m,\cZ_m,i_m)$ is induced by a morphism $\tau\colon\{0,\ldots,m\}\rightarrow\{0,\ldots,n\}$ in the simplex category.
Then the complexes $R\Gamma(\cZ_{m,\bbQ},\sE_{\cZ_m}\otimes\omega^\star_{\cZ_m/\cT,\bbQ})$ form a cosimplicial complex $R\Gamma(\cZ_{\bullet,\bbQ},\sE_{\cZ_\bullet}\otimes\omega^\star_{\cZ_\bullet/\cT,\bbQ})$. Note that, since the underlying Grothendieck topological space of a dagger space is defined to be that of a rigid analytic space, we may consider Godement resolution by using prime filters instead of usual points (see \cite{vdPS} and \cite[\S 3]{CCM}), and we may regard $R\Gamma(\cZ_{m,\bbQ},\sE_{\cZ_m}\otimes\omega^\star_{\cZ_m/\cT,\bbQ})$ as complexes.

Similarly, if $(Z_\lambda,\cZ_\lambda,i_\lambda,h_\lambda,\theta_\lambda,\phi_\lambda)_{\lambda\in\Lambda}$ is a local embedding $F$-datum, $\phi_{\underline{\lambda}}$ give a Frobenius lift $\phi_m$ on $\cZ_m$.
Then we obtain a simplicial $F$-widening $(Z_\bullet,\cZ_\bullet,i_\bullet,\phi_\bullet)$ over $(T,\cT,\iota,\sigma)$.

\begin{definition}
	Let $Y$ be a fine log scheme over $T$.
	\begin{enumerate}
	\item For an object $\sE\in\Isoc^\dagger(Y/\cT)$, we define the {\it log rigid cohomology} $R\Gamma_\rig(Y/\cT,\sE)$ of $Y$ over $(T,\cT,\iota)$ with coefficients in $\sE$ to be the complex associated to the cosimplicial complex $R\Gamma(\cZ_{\bullet,\bbQ},\sE_{\cZ_\bullet}\otimes\omega^\star_{\cZ_\bullet/\cT,\bbQ})$ as above, by taking a local embedding datum $(Z_\lambda,\cZ_\lambda,i_\lambda,h_\lambda,\theta_\lambda)_{\lambda\in\Lambda}$.
		Note that $R\Gamma_\rig(Y/\cT,\sE)$ is a complex of $\Gamma(\cT_\bbQ,\cO_{\cT_\bbQ})$-modules.
	\item For an object $(\sE,\Phi)\in F\Isoc^\dagger(Y/\cT)$, take a local embedding $F$-datum $(Z_\lambda,\cZ_\lambda,i_\lambda,h_\lambda,\theta_\lambda,\phi_\lambda)_{\lambda\in\Lambda}$.
		Then the endomorphisms $\varphi_m$ on $\sE_{\cZ_m}\otimes\omega^\bullet_{\cZ_m/\cT,\bbQ}$ defined as in Definition \ref{def: FMIC} induce a $\sigma$-semilinear endomorphism $\varphi$ on $R\Gamma_\rig(Y/\cT,\sE)$.
		We denote by $R\Gamma_\rig(Y/\cT,(\sE,\Phi))$ the pair of $R\Gamma_\rig(Y/\cT,\sE)$ and $\varphi$, and call it the {\it log rigid cohomology} of $Y$ over $(T,\cT,\iota,\sigma)$ with coefficients in $(\sE,\Phi)$.
	\end{enumerate}
	For the trivial coefficient $\sO_{Y/\cT}$, we denote $R\Gamma_\rig(Y/\cT):=R\Gamma_\rig(Y/\cT,\sO_{Y/\cT})$.
\end{definition}

\begin{proposition}\label{prop: log rig coh indep}
	The log rigid cohomology $R\Gamma_\rig(Y/\cT,\sE)$ and $R\Gamma_\rig(Y/\cT,(\sE,\Phi))$ are independent of the choice of a local embedding ($F$-)data up to canonical quasi-isomorphisms.
\end{proposition}

\begin{proof}
	This is given by a standard argument of rigid cohomology theory.
	More precisely, assume that another family $(Z'_{\lambda'},\cZ'_{\lambda'},i'_{\lambda'},h'_{\lambda'},\theta'_{\lambda'})_{\lambda'\in\Lambda'}$ is given.
	Let $(Z''_{\bullet,\bullet},\cZ''_{\bullet,\bullet},i''_{\bullet,\bullet},h'_{\bullet,\bullet},\theta''_{\bullet,\bullet})$ be the bisimplicial object given by
		\[(Z''_{m,n},\cZ''_{m,n},i''_{m,n},h'_{m,n},\theta''_{m,n}):=(Z_m,\cZ_m,i_m,h_m,\theta_m)\times(Z'_n,\cZ'_n,i'_n,h'_n,\theta'_n).\]
	Then we obtain a bicosimplicial complex $R\Gamma(\cZ''_{\bullet,\bullet,\bbQ},\sE_{\cZ''_{\bullet,\bullet}}\otimes\omega^\star_{\cZ''_{\bullet,\bullet}/\cT,\bbQ})$
	with natural morphisms
		\[R\Gamma(\cZ_{\bullet,\bbQ},\sE_{\cZ_\bullet}\otimes\omega^\star_{\cZ_\bullet/\cT,\bbQ})\rightarrow R\Gamma(\cZ''_{\bullet,\bullet,\bbQ},\sE_{\cZ''_{\bullet,\bullet}}\otimes\omega^\star_{\cZ''_{\bullet,\bullet}/\cT,\bbQ})\leftarrow R\Gamma(\cZ'_{\bullet,\bbQ},\sE_{\cZ'_\bullet}\otimes\omega^\star_{\cZ'_\bullet/\cT,\bbQ}).\]
	We will prove the first morphism is a quasi-isomorphism on the associated complexes, then the second morphism is also a quasi-isomorphism by the same reason.
	(We regard (bi)cosimplicial complexes as a complex by considering the associated complexes.)
	
	It is sufficient to show that the morphism
		\begin{equation}\label{eq: m bullet}
		R\Gamma(\cZ_{m,\bbQ},\sE_{\cZ_m}\otimes\omega^\star_{\cZ_m/\cT,\bbQ})\rightarrow R\Gamma(\cZ''_{m,\bullet,\bbQ},\sE_{\cZ''_{m,\bullet}}\otimes\omega^\star_{\cZ''_{m,\bullet}/\cT,\bbQ})
		\end{equation}
	is a quasi-isomorphism for any $m\in\bbN$.
	Let $\cU_{m,\lambda'}\subset\cZ_m$ be the open weak formal log subscheme whose underlying topological space is $Z_m\cap Z_{\lambda'}\subset Z_m$ and whose log structure is the pull-back of $\cN_{\cZ_m}$.
	Then the above morphism \eqref{eq: m bullet} factors through $R\Gamma(\cU_{m,\bullet,\bbQ},\sE_{\cU_{m,\bullet}}\otimes\omega^\star_{\cU_{m,\bullet}/\cT,\bbQ})$.
	By the cohomological descent with respect to an admissible covering of a dagger space,
		\[R\Gamma(\cZ_{m,\bbQ},\sE_{\cZ_m}\otimes\omega^\star_{\cZ_m/\cT,\bbQ})\rightarrow R\Gamma(\cU_{m,\bullet,\bbQ},\sE_{\cU_{m,\bullet}}\otimes\omega^\star_{\cU_{m,\bullet}/\cT,\bbQ})\]
	is a quasi-isomorphism.
	By the strong fibration lemma \cite[Proposition 1.38]{EY}, the canonical morphism $\cZ''_{m,n,\bbQ}\rightarrow\cU_{m,n,\bbQ}$, which we denote by $f_{m,n}$, is a relative open polydisk.
	As $\sE_{\cZ''_{m,n}}=f_{m,n}^*\sE_{\cU_{m,n}}$, one can see that
		\[R\Gamma(\cU_{m,\bullet,\bbQ},\sE_{\cU_{m,\bullet}}\otimes\omega^\star_{\cU_{m,\bullet}/\cT,\bbQ})\rightarrow R\Gamma(\cZ''_{m,\bullet,\bbQ},\sE_{\cZ''_{m,\bullet}}\otimes\omega^\star_{\cZ''_{m,\bullet}/\cT,\bbQ})\]
	is also a quasi-isomorphism, by the same arguments as the proof of \cite[Proposition 2.5]{EY}.
	
	It is straight forward that these constructions are compatible with Frobenius structures.
	Therefore we proved the assertion.
\end{proof}

\begin{remark}\label{rem: affinoid}
	Let $(Z_\lambda,\cZ_\lambda,i_\lambda,h_\lambda,\theta_\lambda)_{\lambda\in\Lambda}$ be a local embedding datum of a fine log scheme $Y$ over $(T,\cT,\iota)$.
	Define objects $(Z_{\underline{\lambda}},\cZ_{\underline{\lambda}},i_{\underline{\lambda}},h_{\underline{\lambda}},\theta_{\underline{\lambda}})\in \OC(Y/\cT)$ for $\underline{\lambda}\in\Lambda^{m+1}$ and a simplicial object $(Z_\bullet,\cZ_\bullet,i_\bullet,h_\bullet,\theta_\bullet)$ as \eqref{eq: product of embedding data} and \eqref{eq: simplicial system}, respectively.
	By construction of the log rigid cohomology, there exists a spectral sequence
	\begin{equation}\label{eq: ss for open}
	E_1^{i,j}=H^j_\rig(Z_i/\cT,\sE)\cong\prod_{\underline{\lambda}\in\Lambda^{i+1}}H^j_\rig(Z_{\underline{\lambda}}/\cT,\sE)\Rightarrow H^{i+j}_\rig(Y/\cT,\sE).
	\end{equation}
	This is useful for reducing arguments concerning log rigid cohomology to the embeddable case.
	However, even if each $\cZ_\lambda$ is affine and $p$-adic, $\cZ_{\underline{\lambda},\bbQ}$ in general is not an affinoid space, hence the cohomology might not be computed by the global sections. 
		
	For this reason, the following description will also be used:
	We take a local embedding datum $(Z_\lambda,\cZ_\lambda,i_\lambda,h_\lambda,\theta_\lambda)_{\lambda\in\Lambda}$ so that $Z_\lambda$ and $\cZ_\lambda$ are affine for all $\lambda\in\Lambda$.
	Let $\cJ_{\underline{\lambda}}:=\Ker(\cO_{\cZ_{\underline{\lambda}}}\rightarrow i_{\underline{\lambda},\ast}\cO_{Z_{\underline{\lambda}}})$ and $\cZ_{\underline{\lambda},n}:=\cZ_{\underline{\lambda}}[\frac{\cJ_{\underline{\lambda}}^n}p]^\dagger$ for each $\underline{\lambda}$ and $n\geq 0$.
	Since $\cZ_{\underline{\lambda},n,\bbQ}$ is an affinoid space, the cohomology of a coherent sheaf is computed by the global section.
	For $m,n\geq 0$, let $\cZ_{m,n}:=\coprod_{\underline{\lambda}\in\Lambda^{m+1}}\cZ_{\underline{\lambda},n}$.
	Then for any $\sE\in\Isoc^\dagger(Y/\cT)$ we have
	\begin{equation*}
		R\Gamma_\rig(Y/\cT)\cong \holim_{m,n}\Gamma(\cZ_{m,n,\bbQ},\sE_{\cZ_{m,n}}\otimes\omega^\bullet_{\cZ_{m,n}/\cT,\bbQ}),
		\end{equation*}
	which is represented as the total complex of the triple complex
	\begin{equation}\label{eq: homotopy limit}
	\prod_{n\in\bbN}\Gamma(\cZ_{\bullet,n,\bbQ},\sE_{\cZ_{\bullet,n}}\otimes\omega^\star_{\cZ_{\bullet,n}/\cT,\bbQ})\xrightarrow{\eth}\prod_{n\in\bbN}\Gamma(\cZ_{\bullet,n,\bbQ},\sE_{\cZ_{\bullet,n}}\otimes\omega^\star_{\cZ_{\bullet,n}/\cT,\bbQ})\end{equation}
	where the first differential is induced by the log connection, the second differential is induced by the face maps of simplicial weak formal log schemes $\cZ_{\bullet,n}$, and the third differential $\eth$ maps $(\eta_{i,m,n})_n\in\prod_{n\in\bbN}\Gamma(\cZ_{m,n,\bbQ},\sE_{\cZ_{m,n}}\otimes\omega^i_{\cZ_{m,n}/\cT,\bbQ})$ to $(\eta_{i,m,n}-\eta_{i,m,n+1}|_{\cZ_{m,n,\bbQ}})_n$ for any $i,m\in\bbN$.
	\end{remark}

Next we discuss the functoriality of the log rigid cohomology.
Let $\varrho\colon (T',\cT',\iota')\rightarrow(T,\cT,\iota)$ be a morphism of widenings, and consider a commutative diagram
	\begin{equation}\label{eq: diag functoriality}
	\xymatrix{
	Y'\ar[d]^\rho\ar[r] & T'\ar[d]^{\varrho_0}\\
	Y\ar[r] & T.
	}\end{equation}
For an object $\sE\in\Isoc^\dagger(Y/\cT)$, we define a morphism
	\begin{equation}\label{eq: pull back log rig coh}
	(\rho,\varrho)^*\colon R\Gamma_\rig(Y/\cT,\sE)\rightarrow R\Gamma_\rig(Y'/\cT',(\rho,\varrho)^*\sE)
	\end{equation}
in the derived category as follows:
Let $(Z_\lambda,\cZ_\lambda,i_\lambda,h_\lambda,\theta_\lambda)_{\lambda\in\Lambda}$ and $(Z'_{\lambda'},\cZ'_{\lambda'},i'_{\lambda'},h'_{\lambda'},\theta'_{\lambda'})_{\lambda'\in\Lambda'}$ be local embedding data for $Y$ and $Y'$ respectively.
For $m,n\in\bbN$, let $Z''_{m,n}:=Z_m\times_YZ'_n$ and let $\cZ''_{m,n}$ be the exactification of $(i_m,i'_n)\colon Z''_{m,n}\hookrightarrow\cZ_m\times_\cT\cZ'_n$.
Then they form a bisimplicial object $(Z''_{\bullet,\bullet},\cZ''_{\bullet,\bullet},i''_{\bullet,\bullet},h''_{\bullet,\bullet},\theta''_{\bullet,\bullet})$ in $\OC(Y'/\cT')$ and a bicosimplicial complex $R\Gamma(\cZ''_{\bullet,\bullet,\bbQ},((\rho,\varrho)^*\sE)_{\cZ''_{\bullet,\bullet}}\otimes\omega^\star_{\cZ''_{\bullet,\bullet}/\cT',\bbQ})$ with morphisms
	\[R\Gamma(\cZ_{\bullet,\bbQ},\sE_{\cZ_\bullet}\otimes\omega^\star_{\cZ_\bullet/\cT,\bbQ})\rightarrow R\Gamma(\cZ''_{\bullet,\bullet,\bbQ},((\rho,\varrho)^*\sE)_{\cZ''_{\bullet,\bullet}}\otimes\omega^\star_{\cZ''_{\bullet,\bullet}/\cT',\bbQ}) \leftarrow R\Gamma(\cZ'_{\bullet,\bbQ},((\rho,\varrho)^*\sE)_{\cZ'_\bullet}\otimes\omega^\star_{\cZ'_\bullet/\cT',\bbQ}).\]
The second morphism is a quasi-isomorphism by the same reason as the proof of Proposition \ref{prop: log rig coh indep}, and therefore this induces a morphism \eqref{eq: pull back log rig coh} in the derived category.

Note that, for an $F$-widening $(T,\cT,\iota,\sigma)$, a fine log scheme $Y$ over $T$, and an object $(\sE,\Phi)\in F\Isoc^\dagger(Y/\cT)$, the endomorphism $\varphi$ on $R\Gamma_\rig(Y/\cT,(\sE,\Phi))$ coincides with the composite
\[R\Gamma_\rig(Y/\cT,\sE)\xrightarrow{(\sigma,F_Y)^*}R\Gamma(Y/\cT,\sigma^*\sE)\xrightarrow{\Phi}R\Gamma_\rig(Y/\cT,\sE),\]
where the first morphism is defined by the above construction.

Next, let $\varrho\colon (T',\cT',\iota',\sigma')\rightarrow(T,\cT,\iota,\sigma)$ be a morphism of $F$-widenings, and consider a commutative diagram \eqref{eq: diag functoriality}.
Then for an object $(\sE,\Phi)\in F\Isoc^\dagger(Y/\cT)$, the morphism \eqref{eq: pull back log rig coh} is compatible with Frobenius structures.
Hence we have a morphism 
	\begin{equation}\label{eq: pull back log rig coh F}
	(\rho,\varrho)^*\colon R\Gamma_\rig(Y/\cT,(\sE,\Phi))\rightarrow R\Gamma_\rig(Y'/\cT',(\rho,\varrho)^*(\sE,\Phi)).
	\end{equation}

\begin{proposition}\label{prop: base change coh}
	Let $(T,\cT,\iota)$ be a widening, $Y$ a quasi-compact fine log scheme over $T$, and $\varrho\colon\cT'\rightarrow\cT$ a morphism of weak formal log schemes.
	Consider a Cartesian diagram
		\[\xymatrix{
		Y'\ar[d]^-\rho\ar[r] & T'\ar[d]^-{\varrho_{0}}\ar[r]&\cT'\ar[d]^-{\varrho}\\
		Y\ar[r] & T\ar[r] & \cT.
		}\]
	Suppose that the underlying weak formal schemes of $\cT$ and $\cT'$ are weak formal spectra of complete discrete valuation rings with fraction fields denoted by $K$ and $K'$, respectively.
	
	Then for any $\sE\in\Isoc^\dagger(Y/\cT)$ the map \eqref{eq: pull back log rig coh} induces a quasi-isomorphism
		\[(\rho,\varrho)^*\colon R\Gamma_\rig(Y/\cT,\sE)\otimes_{K}K'\xrightarrow{\cong}R\Gamma_\rig(Y'/\cT',(\rho,\varrho)^*\sE).\]
		
	Moreover let $\sigma$ and $\sigma'$ be Frobenius lifts on $\cT$ and $\cT'$ which are compatible each other via $\varrho$.
	Then for any $(\sE,\Phi)\in F\Isoc^\dagger(Y/\cT)$ the map \eqref{eq: pull back log rig coh F} induces a quasi-isomorphism
		\[(\rho,\varrho)^*\colon R\Gamma_\rig(Y/\cT,(\sE,\Phi))\otimes_{K}K'\xrightarrow{\cong}R\Gamma_\rig(Y'/\cT',(\rho,\varrho)^*(\sE,\Phi))\]
	which is compatible with the Frobenius structures.
\end{proposition}

\begin{proof}
	Take a local embedding datum $(Z_\lambda,\cZ_\lambda,i_\lambda,h_\lambda,\theta_\lambda)_{\lambda\in\Lambda}$ for $Y$ over $\cT$ such that $Z_\lambda$ and $\cZ_\lambda$ are affine for all $\lambda\in\Lambda$ and $\Lambda$ is finite.
	Then it induces a local embedding datum for $Y'$ over $\cT'$ by taking base change by $\varrho\colon \cT'\rightarrow\cT$, and the morphism on cohomologies induced by the natural projection coincides with \eqref{eq: pull back log rig coh}.
	Put $\sE':=(\rho,\varrho)^\ast\sE$.
	For $i,n\geq 0$ and $\underline{\lambda}\in\Lambda^{i+1}$ we let $\cZ'_{\underline{\lambda},n}:=\cZ_{\underline{\lambda},n}\times_\cT\cT'$ where $\cZ_{\underline{\lambda},n}$ is defined as in Remark \ref{rem: affinoid}.
	Then by the description of log rigid cohomology in Remark \ref{rem: affinoid}, it suffices to prove that the map $H^j(\cZ'_{\underline{\lambda},n,\bbQ},\sE'_{\cZ'_{\underline{\lambda},n}}\otimes\omega^\bullet_{\cZ'_{\underline{\lambda},n}/\cT',\bbQ})\rightarrow H^j(\cZ_{\underline{\lambda},n,\bbQ},\sE_{\cZ_{\underline{\lambda},n}}\otimes\omega^\bullet_{\cZ_{\underline{\lambda},n}/\cT,\bbQ})$ is an isomorphism for any $\underline{\lambda}$, $n$ and $j$.
	But this is clear, because $\cZ'_{\underline{\lambda},n}$ and $\cZ_{\underline{\lambda},n}$ are affine and we have $\Gamma(\cZ'_{\underline{\lambda},n,\bbQ},\sE_{\cZ'_{\underline{\lambda},n}}\otimes\omega^j_{\cZ'_{\underline{\lambda},n}/\cT',\bbQ})\cong \Gamma(\cZ_{\underline{\lambda},n,\bbQ},\sE_{\cZ_{\underline{\lambda},n}}\otimes\omega^j_{\cZ_{\underline{\lambda},n}/\cT,\bbQ})\otimes_KK'$ by \cite[Proposition 1.54(4)]{EY}.
	The second assertion is proved similarly.
\end{proof}

%%%%%%%%%%%%%
\section{Log overconvergent isocrystals over $W^\varnothing$}\label{sec: absolute isoc}
%%%%%%%%%%%

Let $k$ be a perfect field of characteristic $p>0$, $W=W(k)$ the ring of Witt vectors of $k$, and $L$ the fraction field of $W$.
From now on, we consider $W$ and its ideal $pW$ as the base.
Namely, all weak formal (log) schemes in this and subsequent sections are with respect to $(W,pW)$.

Let $\iota\colon k^\varnothing\hookrightarrow W^\varnothing$ be the natural exact closed immersion.
Let $\mathcal{S}$ be $\Spwf W\llbracket s\rrbracket$ endowed with the log structure associated to the map $\bbN\rightarrow W\llbracket s\rrbracket;\ 1\mapsto s$, and $\tau\colon k^0\hookrightarrow\cS$ the exact closed immersion defined by the ideal $(p,s)$.
Let $W^0$ be the exact closed weak formal log subscheme of $\cS$ defined by the ideal generated by $s$, and let $\tau_0\colon k^0\hookrightarrow W^0$ be the exact closed immersion induced by $\tau\colon k^0\hookrightarrow\cS$.
We denote by $\sigma$ the Frobenius endomorphism on $W^\varnothing$ and its extensions to $\cS$ and $W^0$ defined by $s\mapsto s^p$.

We note that $\cS$ is strongly log smooth over $W^\varnothing$.
For any fine weak formal log scheme $\cZ$ over $\cS$ (in particular for any fine log scheme over $k^0$), we again denote by $s\in\cN_\cZ$ the image of $s\in\cN_{\cS}$.

In this section we will study fundamental properties of log overconvergent ($F$-)isocrystals over $W^\varnothing$ on strictly semistable log schemes over $k^0$, and introduce some classes of log overconvergent ($F$-)isocrystals.

\begin{definition}\label{def: strictly semistable}
	\begin{enumerate}
	\item A log scheme $Y$ over $k^0$ is said to be {\it strictly semistable} if Zariski locally there exist integers $n\geq 1,m\geq 0$ and a strict log smooth morphism
			\begin{equation}\label{eq: ss morp}Y\rightarrow(\Spec k[x_1,\ldots,x_n,y_1,\ldots,y_m]/(x_1\cdots x_n),\cN)\end{equation}
		where the log structure $\cN$ is induced by the map
			\[\bbN^n\oplus\bbN^m\rightarrow k[x_1,\ldots,x_n,y_1,\ldots,y_m]/(x_1\cdots x_n);\ ((k_i),(\ell_j))\mapsto x_1^{k_1}\cdots x_n^{k_n}y_1^{\ell_1}\cdots y_m^{\ell_m}.\]
		We call the closed subset
		\[\{y\in Y\mid \text{$\exists \alpha\in\cN_{Y,y}$ $\forall \beta\in\cN_{Y,y}$ $\alpha\beta$ is not contained in the image of $\Gamma(k^0,\cN_{k^0})$}\}\]
		with the reduced structure the {\it horizontal divisor} of $Y$.
		Namely the horizontal divisor is locally defined by $y_1\cdots y_m$ via \eqref{eq: ss morp}.
	\item A weak formal log scheme $\cZ$ over $\cS$ or a weak formal log scheme with Frobenius $(\cZ,\phi)$ over $(\cS,\sigma)$ is said to be {\it strictly semistable} if Zariski locally on $\cZ$ there exist integers $n\geq 1$, $m\geq 0$ and a strict strongly log smooth morphism
			\begin{equation}\label{eq: ss morp S}
			\cZ\rightarrow(\Spwf W\llbracket s\rrbracket[x_1,\ldots,x_n,y_1,\ldots,y_m]^\dagger/(s-x_1\cdots x_n),\cN)
			\end{equation}
		where the log structure $\cN$ is induced by the map
			\[\bbN^n\oplus\bbN^m\rightarrow W\llbracket s\rrbracket[x_1,\ldots,x_n,y_1,\ldots,y_m]^\dagger/(s-x_1\cdots x_n);\ ((k_i),(\ell_j))\mapsto x_1^{k_1}\cdots x_n^{k_n}y_1^{\ell_1}\cdots y_m^{\ell_m}.\]
		We call the closed subset
		\[\{z\in \cZ\mid \text{$\exists \alpha\in\cN_{\cZ,z}$ $\forall \beta\in\cN_{\cZ,z}$ $\alpha\beta$ is not contained in the image of $\Gamma(\cS,\cN_{\cS})$}\}\]
		with the reduced structure the {\it horizontal divisor} of $\cZ$.
		Namely the horizontal divisor is locally defined by $y_1\cdots y_m$ via \eqref{eq: ss morp S}.
	\item Let $Y$ be a strictly semistable log scheme over $k^0$.
		An object $(Z,\cZ,i,h,\theta)\in\OC(Y/\cS)$ is said to be {\it strictly semistable} if $\cZ$ is so.
	\end{enumerate}
\end{definition}

We first show that the existence of Frobenius structure implies the overconvergence of an integrable log connection.

\begin{theorem}\label{thm: Frobenius overconvergence}
	Let $(\cZ,\phi)$ be a a strictly semistable weak formal log scheme over $(\cS,\sigma)$.
	Then any object $(\cE,\nabla,\Phi)\in F\MIC((\cZ,\phi)/(W^\varnothing,\sigma))$ is overconvergent.
	In other words, we have an equality $F\MIC^\dagger((\cZ,\phi)/(W^\varnothing,\sigma))= F\MIC((\cZ,\phi)/(W^\varnothing,\sigma))$.
\end{theorem}

\begin{proof}
	Considering locally, we may assume that $\cZ=\Spwf A$ is affine and there exist integers $a,b,c,d$ and an \'{e}tale morphism
	\[\cZ\rightarrow\cZ':=\Spwf W\llbracket s\rrbracket[x_1,\ldots,x_a,y_1,\ldots,y_b,z_1,\ldots,z_c]^\dagger/(x_1\cdots x_a-s)\]
	such that the log structure of $\cZ$ is induced by the map 
	\[\bbN^a\oplus\bbN^b\rightarrow A, \ (\beta,\gamma)\mapsto x^\beta y^\gamma=x_1^{\beta_1}\cdots x_a^{\beta_a}y_1^{\gamma_1}\cdots y_b^{\gamma_b}.\]
	For any $1\leq i\leq a$ and $(\alpha_j)_j\in\{0,1\}^c$, let
	\[\cZ_{i,j}:=\cZ\times_{\cZ'}\Spwf W\llbracket x_i\rrbracket\left[x_1,\ldots,x_a,y_1,\ldots,y_b,z_1,\ldots,z_c,\frac{1}{z_1-\alpha_1},\ldots,\frac{1}{z_c-\alpha_c}\right]^\dagger.\]
	Then the family $\{\cZ_{i,j,\bbQ}\}_{i,j}$ is an admissible covering of $\cZ_\bbQ$, and the morphism
	\begin{align*}
	\cZ_{i,j}\rightarrow \Spwf W[x_1,\ldots,x_a,y_1,\ldots,y_b,z'_1,\ldots,z'_c]^\dagger
	\end{align*}
	defined by $z'_j\mapsto z_j-\alpha_j$ is strict \'{e}tale, if we consider the log structure generated by $x_1,\ldots,x_a$, $y_1,\ldots,y_b$, $z'_1,\ldots,z'_c$ on the codomain.
	Therefore,  after replacing symbols, we may assume that there exist an integer $n$ and a strict \'{e}tale morphism
	\[\cZ=\Spwf A\rightarrow\Spwf W[x_1,\ldots,x_m]^\dagger\]
	where we consider the log structure generated by $x_1,\ldots,x_m$ on the codomain.
	Then $\omega^1_{\cZ/W^\varnothing,\bbQ}$ is a free $\cO_{\cZ_\bbQ}$-module generated by $d\log x_1,\ldots, d\log x_m$.
	
	Since $\cZ_\bbQ$ can be covered by admissible open subsets of the form $\cU_\bbQ$ for a $p$-adic weak formal scheme $\cU$ over $\cZ$, we may moreover suppose that $\cE$ is free and there exist an integer $n$ and a surjection $A':=W[X_1,\ldots,X_m,Y_1,\ldots,Y_n]^\dagger\rightarrow A$ which maps $X_s$ to $x_s$ for any $s=1,\ldots,m$.
	We denote by $y_t\in A$ the image of $Y_t$ for $t=1,\ldots,n$.
	
	From now on, we follow the method of Berthelot's proof of the non-logarithmic version of the theorem \cite[Th\'{e}or\`{e}me 2.5.7]{Ber2}.
	Let $\partial^{\log}_s\colon\cE\rightarrow\cE$ for $s=1,\ldots,m$ be the differential operators as in \eqref{eq: log diff op}.
	For a basis $e_1,\ldots,e_r$ of $\cE$, if we write as
		\[\partial^{\log{}}_s(e_j)=\sum_{i=1}^r\xi_{i,j}^{(s)}e_i\]
	with $\xi_{i,j}^{(s)}\in A_{\bbQ}:=A\otimes_{\bbZ}\bbQ=\Gamma(\cZ_\bbQ,\cO_{\cZ_\bbQ})$, then the log connection $\nabla$ is uniquely determined by the matrices $B_s:=(\xi_{i,j}^{(s)})_{i,j}\in M_r(A_\bbQ)$ for $s=1,\ldots,m$.
	
	We first claim that there exists a basis $e_1,\ldots,e_r$ such that $B_s\in M_r(A)$ for any $s$.
	Indeed, for any choice of a basis there exists an integer $c\geq 0$ such that $B_s\in M_r(p^{-c}A)$.
	We denote again by $\phi$ the endomorphism on $A$ induced by $\phi$ on $\cZ$.
	Since $\phi$ preserves log structures and $\phi$ modulo $p$ is the absolute Frobenius, we may write as $\phi(x_s)=x_s^pg_s$ with $g_s\in 1+pA$ for each $s=1,\ldots,m$.
	Therefore we have
		\[\phi(d\log x_s)=pd\log x_s+d\log g_s=pd\log x_s+g_s^{-1}d(g_s-1)\]
	and this can be divided by $p$.
	Let $\phi^*e_j:=1\otimes e_j\in\cO_{\cZ_\bbQ}\otimes_{\phi^{-1}\cO_{\cZ_\bbQ}}\phi^{-1}\cE=\phi^*\cE$.
	Then by
		\[\phi^*\nabla(\phi^*e_j)=\sum_{s=1}^m\sum_{i=1}^r\phi(\xi_{i,j}^{(s)})\phi^*e_i\otimes\phi(d\log x_s),\]
	the matrices defining $\phi^*\nabla$ with respect to the basis $\phi^*e_1,\ldots,\phi^*e_r$ belong to $M_r(p^{1-c}A)$.
	Thus the matrices defining $\nabla$ with respect to the basis $\Phi(\phi^*e_1),\ldots,\Phi(\phi^*e_r)$ also belong to $M_r(p^{1-c}A)$.
	By repeating this argument we proved the claim.

	We will show that the series
		\begin{equation}\label{eq: taylor isom}
		\epsilon(1\otimes\zeta):=\sum_{\bsk=(k_1,\ldots,k_m)\in\bbN^m}\frac{(\partial^{\log}_1)^{k_1}\circ\cdots\circ(\partial_n^{\log})^{k_m}(\zeta)}{k_1!\cdots k_m!}\otimes(\log u(x_1))^{k_1}\cdots(\log u(x_m))^{k_m}
		\end{equation}
	for any section $\zeta\in\cE$ is well-defined as an element of $\cE\otimes\cO_{\cZ(1)_\bbQ}$.
	Note that, if we put $A(1):=\Gamma(\cZ(1),\cO_{\cZ(1)})$, there exists a surjection
		\[A'(1):=W\llbracket X'_1,\ldots,X'_m,Y'_1,\ldots,Y'_n\rrbracket[X_1,\ldots,X_m,Y_1,\ldots,Y_n]^\dagger\rightarrow A(1)\]
	which maps $X_s\mapsto p_1^*x_s$,  $Y_t\mapsto p_1^*y_t$, $X'_s\mapsto u(x_s)-1$, and $Y'_t\mapsto u(y_t)-1$.
	For $\nu=a^{-1}b\in\bbQ_{>0}$ with $a,b\in\bbN_{>0}$, let
		\begin{align*}
		A'(1)_{b/a}:=A'(1)\left[\frac{X'^a_1}{p^b},\ldots,\frac{X'^a_m}{p^b},\frac{Y'^a_1}{p^b},\ldots,\frac{Y'^a_n}{p^b}\right]^\dagger
		&&\text{and}&&A(1)_{b/a}:=A(1)\otimes^\dagger_{A'(1)}A'(1)_{b/a}.
		\end{align*}
	Note that $A(1)_{\nu,\bbQ}:=A(1)_{b/a}\otimes_{\bbZ}\bbQ$ is independent of the choice of $a$ and $b$.

	We first show that $\epsilon$ given by \eqref{eq: taylor isom} is well-defined on $\frZ_{\nu}:=\Sp A(1)_{\nu,\bbQ}\subset\cZ(1)_\bbQ$ for some $\nu\in\bbQ_{>0}$.
	We may assume that $B_1,\ldots,B_m\in M_r(A)$ by taking an appropriate basis.
	For $\nu=a^{-1}b\in\bbQ_{>0}$ and real numbers $\eta<1$ and $\lambda>1$, let $\|\cdot\|_{\eta,\lambda}$ be the norm on $A'(1)_{\nu,\bbQ}:=A'(1)_{b/a}\otimes_{\bbZ}\bbQ$ defined by
		\begin{eqnarray*}
		\left\|\sum a_{\alpha,\alpha',\beta,\beta'}X^\alpha X'^{\alpha'} Y^\beta Y'^{\beta'} \right\|_{\eta,\lambda}:=\sup\lvert a_{\alpha,\alpha',\beta,\beta'}\rvert\eta^{\lvert\alpha'\rvert+\lvert\beta'\rvert}\lambda^{\lvert\alpha\rvert+\lvert\alpha'\rvert+\lvert\beta\rvert+\lvert\beta'\rvert}\in\bbR_{\geq 0}\cup\{\infty\}.
		\end{eqnarray*}
	We denote again by $\|\cdot\|_{\eta,\lambda}$ the quotient norm on $A(1)_{\nu,\bbQ}$.
	Now it suffices that there exist some $\eta<1$ and $\lambda>1$ such that
	\[\left\|\frac{(\partial^{\log{}}_1)^{k_1}\circ\cdots\circ(\partial_n^{\log{}})^{k_m}(\zeta)}{k_1!\cdots k_m!}\otimes(\log u(x_1))^{k_1}\cdots(\log u(x_m))^{k_m}\right\|_{\eta,\lambda}\rightarrow 0\hspace{10pt}\left(\lvert \bsk\rvert\rightarrow0\right),\]
	which implies that \eqref{eq: taylor isom} converges as an element of $A(1)_{\lambda^{-1}\eta,\bbQ}$.
	
	For this, it suffices to claim that, for any $\rho>1$, there exists a $\lambda>1$ such that
	\begin{equation}\label{eq: norm claim}\|(\partial^{\log}_1)^{k_1}\circ\cdots\circ(\partial_n^{\log})^{k_m}(\zeta)\|_{\eta,\lambda}\leq \rho^{\lvert\mathbf{k}\rvert}\|\zeta\|_{\eta,\lambda}
	\end{equation}
	for any $\bsk=(k_1,\ldots,k_m)$.
	Indeed, choose arbitrary $\rho>1$ and $\eta<(\rho p)^{-1}$, and take a $\lambda$ satisfying the claim.
	Then, since $p^i\lvert i\rvert\geq p^{1/(p-1)}$ for any $i\geq 1$, we have
		\begin{equation*}\label{eq: log u-1}
		\|\log u(x_s)\|_{\eta,\lambda}\leq \max_{i\geq 1}\frac{\eta^i}{\lvert i\rvert}<\max_{i\geq 1}\frac{1}{\rho^ip^i\lvert i\rvert}\leq \rho^{-1}p^{1/(p-1)}
		\end{equation*}
	for $s=1,\ldots,m$.
	Thus we obtain
	\[\left\|\frac{(\partial^{\log{}}_1)^{k_1}\circ\cdots\circ(\partial_n^{\log{}})^{k_m}(\zeta)}{k_1!\cdots k_m!}\otimes(\log u(x_1))^{k_1}\cdots(\log u(x_m))^{k_m}\right\|_{\eta,\lambda}<
	\frac{p^{\lvert\bsk\rvert/(p-1)}\|\zeta\|_{\eta,\lambda}}{\lvert k_1!\cdots k_m!\rvert},\]
	which converges to zero when $\lvert\bsk\rvert\rightarrow 0$.
		
	Let
		\[\zeta=\sum_{j=1}^r\zeta_je_j\in \cE=\bigoplus_{j=1}^r A_\bbQ e_j.\]
	Then by
		\[\partial^{\log}_s(\zeta)=\sum_{i=1}^r\left(\sum_{j=1}^r\xi_{i,j}^{(s)}\zeta_j+\partial^{\log}_s(\zeta_i)\right)e_i,\]
	we obtain
		\begin{equation}\label{eq: par zeta}
		\|\partial^{\log}_s(\zeta)\|_{\eta,\lambda}\leq \max_{i,j}\{\|\xi_{i,j}^{(s)}\zeta_j\|_{\eta,\lambda},\|\partial^{\log}_s(\zeta_i)\|_{\eta,\lambda}\},
		\end{equation}
	where we denote $\|p_1^*\theta\|_{\eta,\lambda}$ simply by $\|\theta\|_{\eta,\lambda}$ for any $\theta\in A_{\bbQ}$.
	Since
		\[\partial^{\log}_s(x^\alpha y^\beta)=\alpha_sx^\alpha y^\beta+\sum_{t=1}^n\beta_tx^\alpha y^{\beta-1_t}\partial^{\log}_s(y_t),\]
	we have
	\begin{equation}\label{eq: norm partial}
	\|\partial_s^{\log}(\zeta_i)\|_{\eta,\lambda}\leq\max_t\{\|\zeta_i\|_{\eta,\lambda}, \lambda^{-1}\|\partial_s^{\log}(y_t)\zeta_i\|_{\eta,\lambda}\}.
	\end{equation}
	As $\partial^{\log}_s(y_t), \xi^{(s)}_{i,j}\in A$, by \cite[Lemme 2.5.3]{Ber2} there exists a $\lambda>1$ such that
		\[\|\partial_s^{\log}(y_t)\|_{\eta,\lambda},\|\xi^{(s)}_{i,j}\|_{\eta,\lambda}\leq \rho.\]
	Therefore by \eqref{eq: par zeta} and \eqref{eq: norm partial} we see that
		\begin{equation*}\label{eq: par zeta 2}
		\|\partial^{\log}_s(\zeta)\|_{\eta,\lambda}\leq\rho\|\zeta\|_{\eta,\lambda}.
		\end{equation*}
	This holds for arbitrary $\zeta$ and $\lambda$ does not depend on $\zeta$.
	So we obtain the claim \eqref{eq: norm claim}, and we see that $\epsilon$ is defined on $\frZ_\nu=\Sp A(1)_{\nu,\bbQ}$ for some $\nu\in\bbQ_{>0}$.
	
	Finally we will see that $\epsilon$ extends to whole $\cZ(1)_\bbQ$ by using Frobenius.
	We denote again by $\phi(1)$ the endomorphism on $A(1)$ induced by $\phi(1)$ on $\cZ(1)$.
	If we write $\phi(x_s)=x_s^pg_s$ with $g_s\in 1+pA$ as before, we have
		\begin{eqnarray}
		\nonumber\phi(1)(1-u(x_s))&=&1-u(x_s)u(x_s)^p=1-u(g_s)+u(g_s)(1-u(x_s)^p)\\
		\nonumber&=&p_1^*(g_s^{-1})\left(p_1^*(g_s)-p_2^*(g_s)\right)\\
		\label{eq: Frob 1-u}&&-u(g_s)\left((u(x_s)-1)^p+p(u(x_s)-1)\sum_{j=0}^{p-2}\frac{1}{j+1}\binom{p-1}{j}(u(x_s)-1)^j\right).
		\end{eqnarray}
	If we put $g_s=1+pg'_s$, we have
		\[p_1^*(g_s)-p_2^*(g_s)=p(p_1^*(g'_s)-p_2^*(g'_s)),\]
	and $p_1^*(g'_s)-p_2^*(g'_s)$ belongs to the ideal generated by $1-u(x_q)$ and $1-u(y_t)$ for all $1\leq q\leq m$ and $1\leq t\leq n$.
	Thus we can write
	\[\phi(1)(1-u(x_s))=\cA(1-u(x_s))^p+p\sum_{q=1}^m\cB_s(1-u(x_q))+p\sum_{t=1}^n\cC_t(1-u(y_t))\]
	for some $\cA,\cB_q,\cC_t\in A(1)$.
	Then for $\nu=a^{-1}b$ with $a,b\in\bbN_{>0}$, $\phi(1)\left(\frac{(1-u(x_s))^a}{p^b}\right)$ is written as a sum of elements of the form
	\[\cD:=p^{\ell-b}(1-u(x_s))^{p(a-\ell)}(1-u(x_{i_1}))\cdots (1-u(x_{i_\nu}))(1-u(y_{j_1}))\cdots (1-u(y_{j_{\mu}}))\]
	with $0\leq\ell\leq a$ and $\nu+\mu=\ell$.
	
	Let $a'=pa-(p-1)\ell$ and $b'=b-\ell$.
	Then we have
	\begin{align*}
	\frac{b}a\leq\frac{p}{p-1}&\Rightarrow\frac{b'}{a'}=\frac{b-\ell}{pa-(p-1)\ell}\leq\frac b{pa},\\
	\frac{b}a\geq\frac{p}{p-1}&\Rightarrow\frac{b'}{a'}=\frac{b-\ell}{pa-(p-1)\ell}\leq\frac b{a}-1.
	\end{align*}
	If we take $c,d\in\bbN_{>0}$ with
	\[\frac{d}c=\nu':=\max\left\{p^{-1}\nu,\nu-1\right\},\]
	then we have
	\begin{align*}
	\cD^{c}=&p^{a'd-b'c}\left(\frac{(1-u(x_s))^c}{p^d}\right)^{p(a-\ell)}\frac{(1-u(x_{i_1}))^{c}}{p^d}\cdots\frac{(1-u(x_{i_\nu}))^{c}}{p^d}\frac{(1-u(y_{j_1}))^{c}}{p^d}\cdots\frac{(1-u(y_{j_\mu}))^{c}}{p^d}.
	\end{align*}
	As $a'd-b'c\geq 0$, we see that $\phi$ induces a map $A(1)_{bc/ac}\rightarrow A(1)_{d/c}$, hence $\phi(1)_{\frZ_{\nu'}}$ factors through $\frZ_{\nu}$.
	
	Choose $\nu_0$ such that $\epsilon$ converges on $\frZ_{\nu_0}$, and define $\nu_N$ for $N\in\bbN$ recursively by
	\[\nu_{N+1}:=\max\{p^{-1}\nu_{N},\nu_N-1\}.\]
	Note that $p^{-\nu_N}\rightarrow 1-0$ when $N\rightarrow\infty$, hence we have $\bigcup_{N\in\bbN}\frZ_{\nu_N}=\cZ(1)_\bbQ$.
	By the above argument, $\phi(1)$ induces $\phi_{N}\colon\frZ_{\nu_{N+1}}\rightarrow\frZ_{\nu_N}$.
	
	If $\epsilon$ is defined on $\frZ_{\nu_N}$, then the composition
		\begin{eqnarray*}
		(p_2^*\cE)|_{\frZ_{\nu_{N+1}}}&\xrightarrow[\cong]{p_2^*(\Phi^{-1})}&(p_2^*\phi^*\cE)|_{\frZ_{\nu_{N+1}}}=(\phi(1)^*p_2^*\cE)|_{\frZ_{\nu_{N+1}}}=\phi_N^*(p_2^*\cE|_{\frZ_{\nu_N}})\\
		&\xrightarrow[\cong]{\phi_N^*(\epsilon)}&\phi_N^*(p_1^*\cE|_{\frZ_{\nu_N}})=(\phi(1)^*p_1^*\cE)|_{\frZ_{\nu_{N+1}}}=(p_1^*\phi^*\cE)|_{\frZ_{\nu_{N+1}}}\\
		&\xrightarrow[\cong]{p_1^*(\Phi)}&(p_1^*\cE)|_{\frZ_{\nu_{N+1}}}
		\end{eqnarray*}
	is an extension of $\epsilon$ to $\frZ_{\nu_{N+1}}$.
	Thus $\epsilon$ extends to whole $\cZ(1)_\bbQ$, and we finish the proof.
\end{proof}

In the non-logarithmic situation, the category of overconvergent isocrystals is abelian, since any coherent sheaf with integrable connection is locally free.
However this is not true for log overconvergent isocrystals in general.
In order to ensure abelianness, we introduce the notion of log overconvergent isocrystals with nilpotent residues, following Ogus \cite{Og} (see also \cite{Ke}).

\begin{definition}
	Let $Y$ be a strictly semistable log scheme over $k^0$.
	We say an object $\sE\in\Isoc^\dagger(Y/W^\varnothing)$ has {\it nilpotent residues} if for any strictly semistable object $(Z,\cZ,i,h,\theta)\in\OC(Y/\cS)$ the realization of $\sE$ at $(Z,\cZ,i,h',\theta)\in\OC(Y/W^\varnothing)$, where $h'$ is defined to be the composition of $h$ with the natural morphism $\cS\rightarrow W^\varnothing$, has nilpotent residues with respect to the generators of the log structure, in the sense of \cite[Definition 2.3.9]{Ke}.
	More precisely, Zariski locally on $\cZ$ we take a strict strongly log smooth morphism
		\[\cZ\rightarrow(\Spwf W\llbracket s\rrbracket[x_1,\ldots,x_n,y_1,\ldots,y_m]^\dagger/(s-x_1\cdots x_n),\cN)\]
	as in Definition \ref{def: strictly semistable}.
	Let $v$ be one of $x_i$ or $y_j$, and $\frD\subset\cZ_\bbQ$ be the closed dagger subspace defined by one of $v$.
	If we let $\Omega^1_{\cZ_\bbQ}$ the sheaf of (non-logarithmic) differentials on $\cZ_\bbQ$ and $I:=\{x_1,\ldots,x_n,y_1,\ldots,y_m\}\setminus\{v\}$, we have
		\[\Coker\Bigl(\sE_\cZ\otimes\bigl(\Omega^1_{\cZ_\bbQ}\oplus\bigoplus_{z\in I}\cO_{\cZ_\bbQ}d\log z\bigr)\rightarrow\sE_\cZ\otimes\omega^1_{\cZ/W^\varnothing,\bbQ}\Bigr)=\sE_\cZ\otimes\cO_\frD\cdot d\log v.\]
	Then the composition of $\nabla_{\cZ}\colon \sE_\cZ\rightarrow\sE_\cZ\otimes\omega^1_{\cZ/W^\varnothing,\bbQ}$ with the projection $\sE_\cZ\otimes\omega^1_{\cZ/W^\varnothing,\bbQ}\rightarrow\sE_\cZ\otimes\cO_\frD\cdot d\log v$ induces a morphism $\sE_\cZ\otimes\cO_\frD\rightarrow\sE_\cZ\otimes\cO_\frD\cdot d\log v$.
	Identifying $\sE_\cZ\otimes\cO_\frD\cdot d\log v$ with $\sE_\cZ\otimes\cO_\frD$, we obtain an endomorphism $\res_v$ on $\sE_\cZ\otimes\cO_\frD$, which we call the {\it residue} along $\frD$.
	We say $\sE$ has nilpotent residues if $\res_v$ is nilpotent for any $v$.
	
	We denote by $\Isoc^\dagger(Y/W^\varnothing)^\nr$ the full subcategory of $\Isoc^\dagger(Y/W^\varnothing)$ consisting of log overconvergent isocrystals with nilpotent residues.
\end{definition}

\begin{proposition}\label{prop: nr abelian}
	Let $Y$ be a strictly semistable log scheme over $k^0$.
	The categories $\Isoc^\dagger(Y/W^\varnothing)^\nr$ and $F\Isoc^\dagger(Y/W^\varnothing)$ are abelian.
\end{proposition}

\begin{proof}
	To prove $\Isoc^\dagger(Y/W^\varnothing)^\nr$ is abelian, it suffices to show that for any morphism $f\colon\sE'\rightarrow\sE$ in $\Isoc^\dagger(Y/W^\varnothing)^\nr$ and any object $(Z,\cZ,i,h,\theta)\in\OC(Y/W^\varnothing)$ the kernel and cokernel of $f_\cZ\colon \sE'_\cZ\rightarrow\sE_\cZ$ are locally free.
	Since $Y$ is locally embeddable into a strictly semistable weak formal log scheme over $\cS$, by Proposition \ref{prop: realization} it is enough to consider the case when $(Z,\cZ,i,h,\theta)$ is induced from a strictly semistable object in $\OC(Y/\cS)$.
	The completions of the kernel and cokernel are locally free sheaves on the rigid analytic space $\wh\cZ_\bbQ$ by \cite[Lemma 3.2.14]{Ke}.
	Thus the kernel and cokernel themselves are also locally free, because the completion $\wh{A}$ of a dagger affinoid $A$ is faithfully flat over $A$ \cite[Theorem 1.7(1)]{GK1}.
	Since a Frobenius structure implies the nilpotence of residues, $F\Isoc^\dagger(Y/W^\varnothing)$ is also abelian.
\end{proof}

\begin{definition}
	Let $Y$ be a strictly semistable log scheme over $k^0$.
	\begin{enumerate}
	\item An object $\sE\in \Isoc^\dagger(Y/W^\varnothing)^\nr$ is said to be {\it unipotent} if $\sE$ is an iterated extension of $\sO_{Y/W^\varnothing}$.
		We denote by $\Isoc^\dagger(Y/W^\varnothing)^\unip$ the full subcategory of $\Isoc^\dagger(Y/W^\varnothing)^\nr$ consisting of unipotent objects.
	\item An object $(\sE,\Phi)\in F\Isoc^\dagger(Y/W^\varnothing)$ is said to be {\it unipotent} if $\sE$ is unipotent.
		We denote by $F\Isoc^\dagger(Y/W^\varnothing)^\unip$ the full subcategory of $F\Isoc^\dagger(Y/W^\varnothing)$ consisting of unipotent objects.
	\end{enumerate}
\end{definition}

\begin{proposition}\label{prop: unipotence}
	Suppose that $Y$ is connected.
	Then the unipotence is closed under extension, subobject, quotient, tensor product, internal Hom, and dual in the categories $\Isoc^\dagger(Y/W^\varnothing)^\nr$ and $F\Isoc^\dagger(Y/W^\varnothing)$.
	In particular the categories $\Isoc^\dagger(Y/W^\varnothing)^\unip$ are $F\Isoc^\dagger(Y/W^\varnothing)^\unip$ are abelian.
\end{proposition}

\begin{proof}
	The assertion for $\Isoc^\dagger(Y/W^\varnothing)^\nr$ follows by the same proof as \cite[Proposition 2.3.2]{CL1}, and immediately implies the assertion for $F\Isoc^\dagger(Y/W^\varnothing)$.
\end{proof}

Next we recall the definition of $(\varphi,N)$-modules and some related notions, and then give the definition of constant and pseudo-constant log overconvergent ($F$-)isocrystals.

\begin{definition}\label{def: phi N module}
	\begin{enumerate}
	\item We define the category $\Mod_L(N)$ as follows:
		\begin{itemize}
		\item An object of $\Mod_L(N)$ is an $L$-vector space $M$ equipped with an $L$-linear endomorphism $N$ called the {\it monodromy operator}.
		\item A morphism $f\colon M'\rightarrow M$ is an $L$-linear map $f\colon M'\rightarrow M$ compatible with the monodromy operators.
		\end{itemize}
	\item Let $\Mod^\fin_L(N)$ be the full subcategory of $\Mod_L(N)$ consisting of objects having finite dimension and nilpotent monodromy operator.
		We regard the category $\Mod^\fin_L$ of finite-dimensional $L$-vector spaces as the full subcategory of $\Mod^\fin_L(N)$ consisting objects whose monodromy operators are zero maps.
	\item We define the category $\Mod_L(\varphi,N)$ as follows:
		\begin{itemize}
		\item An object of $\Mod_L(\varphi,N)$ is an $L$-vector space $M$ equipped with a $\sigma$-semilinear endomorphism $\varphi$ called the {\it Frobenius operator} and an $F$-linear endomorphism on $M$ called the {\it monodromy operator}, such that $N\varphi=p\varphi N$.
		\item A morphism $f\colon M'\rightarrow M$ is an $L$-linear map $f\colon M'\rightarrow M$ compatible with the Frobenius operators and the monodromy operators.
		\end{itemize}
		Let $\Mod_L(\varphi)$ be the full subcategory of $\Mod_L(\varphi,N)$ consisting of objects whose monodromy operators are zero maps.
	\item Let $\Mod^\fin_L(\varphi,N)$ be the full subcategory of $\Mod_L(\varphi,N)$ consisting of objects having finite dimension and bijective Frobenius operator.
		Note that $N\varphi=p\varphi N$ implies the nilpotence of $N$.
		Let $\Mod^\fin_L(\varphi):=\Mod_L^\fin(\varphi,N)\cap\Mod_L(\varphi)$.
	\item For $M,M'\in\Mod_L(\varphi,N)$, the tensor product is defined to be the space $M\otimes_LM'$ with the Frobenius operator $\varphi\otimes\varphi'$ and the monodromy operator $N\otimes\id_{M'}+\id_M\otimes N'$.
		For $M,M'\in\Mod_L^\fin(\varphi,N)$, the internal Hom is defined to be the space $\Hom_L(M',M)$ with the Frobenius operator associating $\varphi\circ f\circ \varphi'^{-1}$ to $f\in\Hom_L(M',M)$ and the monodromy operator associating $N\circ f-f\circ N'$ to $f\in\Hom_L(M',M)$.
		Tensor products in $\Mod_L(\varphi)$ and $\Mod_L(N)$ and internal Hom in $\Mod^\fin_L(\varphi)$ and $\Mod_L(N)$ are defined in a similar manner.
	\item For $r\in\bbZ$, let $L(r)$ be an object of $\Mod_L^\fin(\varphi)$ given by the vector space $L$ and the Frobenius operator $p^{-r}\sigma$.
		For any object $M\in\Mod_L(\varphi,N)$ and $r\in\bbZ$, we denote $M(r):=M\otimes L(r)$.
	\end{enumerate}
\end{definition}

\begin{lemma}\label{lem: unipotence of M}
	Any object of $\Mod_L^\fin(N)$ can be written as an iterated extension of objects in $\Mod_L^\fin$. 
\end{lemma} 

\begin{proof}
	Let $M$ be an object of $\Mod_L^\fin(N)$.
	As the monodromy operator $N$ on $M$ is nilpotent, one can take a basis $e_1,\ldots,e_d$ of $M$ such that the representation matrix of $N$ is strictly upper triangular.
	Then the subspaces $M_i:=\bigoplus_{j\leq i}Le_j$ for $i=1,\ldots,d$ are stable under $N$, and we have $N=0$ on each $M_i/M_{i-1}$.
	This shows the assertion.
\end{proof}

To define a log overconvergent $F$-isocrystal associated to a $(\varphi,N)$-module, we prepare some sheaves.
Let $Y$ be a fine log scheme over $k^0$.
For an object $(Z,\cZ,i,h,\theta)\in\OC(Y/W^\varnothing)$, let $\cJ_\cZ\subset\cO_\cZ$ be the ideal of $Z$ in $\cZ$.
For an element $\alpha\in\Gamma(Z,\cN_Z)$, we let $\cN_\cZ^\alpha\subset\cN_\cZ$ be the subsheaf consisting of preimages of $\alpha$. 
Since $i$ is a homeomorphic exact closed immersion, it follows that $1+\cJ_\cZ$ acts on $\cN_\cZ^\alpha$ freely and transitively.
In particular, we denote by $s\in\Gamma(Z,\cN_Z)$ the image of the canonical generator of $\cN_{k^0}$, then the subsheaf $\cN_\cZ^s\subset\cN_\cZ$ is defined. 

We define a sheaf $\sJ_\cZ=\sJ_{(Z,\cZ,i,h,\theta)}$ on $\cZ_\bbQ$ by setting $\Gamma(\frU,\sJ_\cZ):=\varprojlim_{\cU}\Gamma(\cU,\cJ_\cZ\cO_{\cU})$ for any quasi-compact admissible open subspace $\frU\subset\cZ_\bbQ$, where the limit is taken over all admissible blow-ups between $p$-adic weak formal schemes $\cU$ over $\cZ$ with $\cU_\bbQ=\frU$.
Similarly we define a sheaf $\sN^\alpha_\cZ=\sN^\alpha_{(Z,\cZ,i,h,\theta)}$ on $\cZ_\bbQ$ for any $\alpha\in\Gamma(Z,\cN_Z)$ by $\Gamma(\cU_\bbQ,\sN^\alpha_\cZ):=\varprojlim_\cU\Gamma(\frU,\cN_\cU^\alpha)$.

\begin{definition}\label{def: pseudo-constant}
	Let $Y$ be a fine log scheme over $k^0$.
	\begin{enumerate}
	\item For $M\in\Mod^{\mathrm{fin}}_L(N)$, we define an object $M^a\in\Isoc^\dagger(Y/W^\varnothing)^\unip$ by giving sheaves $M^a_\cZ$ for $(Z,\cZ,i,h,\theta)\in\OC(Y/W^\varnothing)$ as follows:
		We first define an $\cO_{\cZ_\bbQ}$-linear action of $1+\sJ_\cZ$ on $M\otimes_L\cO_{\cZ_\bbQ}$ by
		\[g\cdot(m\otimes f):=\exp(-\log g\cdot N)(m\otimes f)=\sum_{n\geq 0}N^n(m)\otimes\frac{1}{n!}(-\log g)^nf\]
		for local sections $f\in\cO_{\cZ_\bbQ}$, $g\in 1+\sJ_\cZ$ and $m\in M$.
		We define a sheaf $M^a_\cZ$ on $\cZ_\bbQ$ to be the quotient of $\sN_\cZ^s\times(M\otimes_L\cO_{\cZ_\bbQ})$ by the diagonal action of $1+\sJ_\cZ$.
	Then $M^a_\cZ$ is non-canonically isomorphic to $M\otimes_L\cO_{\cZ_\bbQ}$, and in particular a coherent free $\cO_{\cZ_\bbQ}$-module.
	Thus $M^a_\cZ$ is locally free and moreover unipotent by Lemma \ref{lem: unipotence of M}.
		
		If a Frobenius operator $\varphi$ on $M$ is given, we may define a Frobenius structure
		\begin{eqnarray*}
		\Phi\colon (\sigma^*M^a)_\cZ&=&(1+\sJ_\cZ)\backslash(\sN_\cZ^{s^p}\times(M\otimes_{L,\sigma}\cO_{\cZ_\bbQ}))\\
		&\xrightarrow{\cong}&(1+\sJ_\cZ)\backslash(\sN_\cZ^s\times(M\otimes_L\cO_{\cZ_\bbQ}))=M^a_\cZ
		\end{eqnarray*}
		as follows:
		Choosing a section $\alpha\in\sN_\cZ^s$, a section of $(\sigma^*M^a)_\cZ$ is written as a sum of sections of the form $(\alpha^p,m\otimes f)$ for some $m\in M$ and $f\in\cO_{\cZ_\bbQ}$.
		The Frobenius structure is defined by
		\[\Phi(\alpha^p,m\otimes f):= (\alpha,\varphi(m)\otimes f).\]
		This is independent of the choice of $\alpha$.
		Indeed, if we let $\alpha'=g\alpha$ with $g\in 1+\sJ_\cZ$, we have
		\begin{eqnarray*}
		\Phi(\alpha'^p,m\otimes f)&=&\Phi(\alpha^p,\exp(\log g^p\cdot N)(m\otimes f))=(\alpha,\varphi\exp(\log g^p\cdot N)(m\otimes f))\\
		&=&(\alpha,\exp(\log g\cdot N)(\varphi(m)\otimes f)=(\alpha',\varphi(m)\otimes f).
		\end{eqnarray*}
	\item An object $\sE\in\Isoc^\dagger(Y/W^\varnothing)$ is said to be {\it constant} (resp.\ {\it pseudo-constant}) if it is isomorphic to $M^a$ for some $M\in\Mod^\fin_L$ (resp.\ $M\in\Mod^\fin_L(N)$).
		An object $(\sE,\Phi)\in F\Isoc^\dagger(Y/W^\varnothing)$ is said to be {\it constant} (resp.\ {\it pseudo-constant}) if it is isomorphic to $M^a$ for some $M\in\Mod^\fin_L(\varphi)$ (resp.\ $M\in\Mod^\fin_L(\varphi,N)$).
	\end{enumerate}
\end{definition}

\begin{proposition}\label{prop: realization pseudo-const}
	Let $Y$ be a fine log scheme over $k^0$, $M\in\Mod^{\mathrm{fin}}_L(N)$, and $(Z,\cZ,i,h,\theta)\in\OC(Y/W^\varnothing)$.
	Choose a morphism $\gamma\colon\cZ\rightarrow\cS$ making the following diagram commutative:
	\[\xymatrix{
	Z\ar@{^(->}[r]^-i\ar[d]&\cZ\ar[d]_-\gamma\ar[rd]&\\
	k^0\ar@{^(->}[r]^-\tau&\cS\ar[r]&W^\varnothing,
	}\]
	where the left vertical morphism is the composition of $\theta\colon Z\rightarrow Y$ and the structure morphism $Y\rightarrow k^0$.
	Then $\gamma$ induces an isomorphism
		\begin{equation}\label{eq: realization Ma}
		(M\otimes_L\cO_{\cZ_\bbQ},\nabla)\cong(M^a_\cZ,\nabla_\cZ),
		\end{equation}
	where the right hand side is the realization of $M^a$ at $(Z,\cZ,i,h,\theta)$, and the left hand side is an object of $\MIC^\dagger(\cZ/W^\varnothing)$ defined by
		\begin{eqnarray*}
		\nabla\colon M\otimes_L\cO_{\cZ_{\bbQ}}&\rightarrow& M\otimes_L\omega^1_{\cZ/W^\varnothing,\bbQ}\\
		m\otimes f&\mapsto& m\otimes df+N(m)\otimes fd\log s.
		\end{eqnarray*}
	
	If a Frobenius structure $\varphi$ on $M$ is given and if $\cZ$ admits a Frobenius lift $\phi$ which is compatible with $\sigma$ on $\cS$, then the corresponding Frobenius structure on $ (M\otimes_L\cO_{\cZ_\bbQ},\nabla)$ is given by
	\[\phi^*(M\otimes_L\cO_{\cZ_\bbQ})=M\otimes_{L,\sigma}\cO_{\cZ_\bbQ}\xrightarrow{\cong}M\otimes_L\cO_{\cZ_\bbQ};\ m\otimes f\mapsto\varphi(m)\otimes f.\]
\end{proposition}

\begin{proof}
	Let $s_\gamma\in\cN_\cZ^s$ be the image of $s\in\cN_\cS$ under the homomorphism induced by $\gamma$.
	This defines a  global section of $\sN_\cZ^s$, which we denote again by $s_\gamma$.
	The isomorphism \eqref{eq: realization Ma} is given by
		\[M\otimes_L\cO_{\cZ_\bbQ}\rightarrow (1+\sJ_\cZ)\backslash(\sN_\cZ^s\times(M\otimes_L\cO_{\cZ_\bbQ}))=M^a_\cZ;\ m\otimes f\mapsto (\gamma,m\otimes f).\]
	
	Let $\cZ(1)$ be the homeomorphic exactification of $Z\hookrightarrow\cZ\times_{W^\varnothing}\cZ$ and $p_j\colon\cZ(1)\rightarrow\cZ$ for $j=1,2$ the canonical projection.
	 Since
	 \[(p_2^*s,m\otimes 1)=\left(p_1^*s,\exp\left(-\log\frac{p_1^*s}{p_2^*s}\cdot N\right)(m\otimes 1)\right),\]
	the Taylor isomorphism on $M\otimes_L\cO_{\cZ_\bbQ}$ induced by \eqref{eq: realization Ma} is given by
	\[M\otimes_L\cO_{\cZ(1)_\bbQ}\rightarrow M\otimes_L\cO_{\cZ(1)_\bbQ};\ m\otimes 1\mapsto \exp\left(-\log\frac{p_1^*s}{p_2^*s}\cdot N\right)(m\otimes 1).\]
	By \eqref{eq: log connection} we see that this Taylor isomorphism induces a connection given by $\nabla(m\otimes 1)=N(m)\otimes d\log s$.
	
	The correspondence between Frobenius structures is obvious.
\end{proof}

\begin{proposition}\label{prop: pseudo-const tensor}
	Let $Y$ be a fine log scheme over $k^0$.
	The functors $\Mod^\fin_L(N)\rightarrow\Isoc^\dagger(Y/W^\varnothing)^\unip$ and $\Mod^\fin_L(\varphi,N)\rightarrow F\Isoc^\dagger(Y/W^\varnothing)^\unip$ defined by $M\mapsto M^a$ preserve tensor product and internal Hom.
\end{proposition}

\begin{proof}
	By Corollary \ref{cor: descent realization}, it suffices to confirm the compatibilities after taking realization at an object $(Z,\cZ,i,h,\theta)\in\OC(Y/W^\varnothing)$.
	Then the assertion follows from the descriptions of Proposition \ref{prop: realization pseudo-const} and Proposition \ref{prop: tensor hom realization}.
	For example, choosing a morphism $\cZ\rightarrow\cS$ as in Proposition \ref{prop: realization pseudo-const}.
	For $M_1,M_2\in\Mod^\fin_L(N)$, denote by $\nabla_1$, $\nabla_2$, $\nabla$, and $\nabla'$ the the log connections on $M^a_{1,\cZ}$, $M^a_{2,\cZ}$, $\Hom(M_1,M_2)^a_\cZ$, and $\sheafhom(M_{1,\cZ}^a,M_{2,\cZ}^a)$, respectively.
	Regrading them as a map
	\[\Hom(M_1,M_2)\otimes_L\cO_{\cZ_\bbQ}\rightarrow\Hom(M_1,M_2)\otimes_L\omega^1_{\cZ/W^\varnothing,\bbQ}\cong\Hom(M_1\otimes_L\cO_{\cZ_\bbQ},M_2\otimes_L\omega^1_{\cZ/W^\varnothing,\bbQ}),\]
	we have 
	\begin{align*}
	\nabla(f\otimes 1)(m\otimes 1)&=N(f)(m)\otimes d\log s=(N\circ f(m)-f\circ N(m))\otimes d\log s,\\
	\nabla'(f\otimes 1)(m\otimes 1)&=\nabla_2(f(m)\otimes 1)-(f\otimes 1)(\nabla_1(m\otimes 1))\\
	&=N\circ f(m)\otimes d\log s-f\circ N(m)\otimes d\log s
	\end{align*}
	for any $f\in \Hom(M_1,M_2)$ and $m\in M_1$.
	Thus we have $\nabla=\nabla'$.
	Other compatibilities can be confirmed by similar computations.
\end{proof}

Next we see that a unipotent log overconvergent ($F$-)isocrystal admits a canonical filtration whose graded pieces are constant if it is connected and admits a rational point.
For a strictly semistable log scheme $Y$ over $k^0$, let $Y^\sm$ be the smooth locus of the underlying scheme of $Y\setminus D$ over $k$, where $D$ is the horizontal divisor of $Y$ (see Definition \ref{def: strictly semistable}).

\begin{lemma}\label{lem: H0 smooth}
	Let $\frZ$ be a connected dagger space smooth over $L$ with an $L$-rational point $z\in\frZ$.
	Then we have $H^0_\dR(\frZ/L):=H^0(\frZ,\Omega^\bullet_{\frZ/L})= L$.
\end{lemma}

\begin{proof}
	It is clear that $L\subset H^0_\dR(\frZ/L)$.
	We have isomorphisms
	\begin{align*}
	\widehat{\cO}_{\frZ,z}\cong L\llbracket x_1,\ldots,x_m\rrbracket,&&
	\wh\Omega^1_{\frZ/L,z}\cong \bigoplus_{i=1}^mL\llbracket x_1,\ldots,x_m\rrbracket dx_i
	\end{align*}
	where $\wh\cdot$ indicates the completion of with respect to the maximal ideal $\frm_z\subset\cO_{\frZ,z}$ and $m=\dim\frZ$.
	Thus there exists a commutative diagram
	\[\xymatrix{
	0\ar[r] & H^0_\dR(\frZ/L)\ar[r]&\Gamma(\frZ,\cO_\frZ)\ar[r]^-d\ar[d]&\Gamma(\frZ,\Omega^1_{\frZ/L})\ar[d]\\
	0\ar[r] &L \ar[r]&L\llbracket x_1,\ldots,x_m\rrbracket\ar[r]^-{\wh d}&\bigoplus_{i=1}^mL\llbracket x_1,\ldots,x_m\rrbracket dx_i
	}\]
	where the both lines are exact and the vertical maps are injectove.
	This shows that $H^0_\dR(\frZ/L)\subset L$.
\end{proof}

\begin{lemma}\label{lem: H0}
	Let $Y$ be a strictly semistable log scheme over $k^0$.
	If $Y$ is geometrically connected over $k$, then we have $H^0_\rig(Y/W^\varnothing)= L$.
\end{lemma}

\begin{proof}
	By \cite[Proposition 3.32]{EY} we may suppose that $D=\emptyset$.
	We first consider the case when $Y$ is quasi-compact and there exists an object $(Y,\cZ,i,h,\theta)\in\OC(Y/W^\varnothing)$ whose first entry is $Y$.
	Then by Proposition \ref{prop: base change coh} we may suppose that there exists a $k$-rational point $y$ of $Y$.
	Then $\Spec k$ equipped with the pull-back log structure by $y\colon \Spec k\hookrightarrow Y$ is isomorphic to $(k^0)^\ell$, the self product of $k^0$ over $k^\varnothing$ for some $\ell\geq 1$.
	Let $\cZ'_y$ be the exactification of the closed immersion $(k^0)^\ell\hookrightarrow\cS^\ell\times_{W^\varnothing}\cZ$ induced by the canonical closed immersion $(k^0)^\ell\hookrightarrow\cS^\ell$ and $y\colon (k^0)^\ell\hookrightarrow\cZ$.
	The natural projection $\cZ'_y\rightarrow\cS^\ell$ is strict and log smooth, hence the underlying morphism of weak formal schemes is also smooth.
	Thus by \cite[Proposition 1.38]{EY} we have an isomorphism $\cZ'_y\cong \cS^\ell\times\Spwf W\llbracket x_1,\ldots,x_m\rrbracket$.
	Thus there exists an $L$-rational point on $\cZ'_{y,\bbQ}$, and it induces an $L$-rational point of $\cZ_\bbQ$ via the natural projection $\cZ'_{y,\bbQ}\rightarrow\cZ_\bbQ$.
	Then we have
	\[L=H^0_\dR(\cZ_\bbQ/L)=H^0_\rig(Y/W^\varnothing)\]
	where the equalities follow from Lemma \ref{lem: H0 smooth} and the fact that $\Omega^1_{\cZ_\bbQ/L}\rightarrow\omega^1_{\cZ/W^\varnothing,\bbQ}$ is injective, respectively.

	For the general case, take a local embedding datum $(Z_\lambda,\cZ_\lambda,i_\lambda,h_\lambda,\theta_\lambda)_{\lambda\in\Lambda}$ such that $Z_\lambda$ is quasi-compact and geometrically connected.
	Then we have
	\begin{align*}
	H^0_\rig(Y/W^\varnothing)&\cong\Ker\left(\prod_{\lambda_0\in\Lambda}H^0_\rig(Z_{\lambda_0}/W^\varnothing)\rightarrow \prod_{(\lambda_0,\lambda_1)\in\Lambda^2}H^0_\rig(Z_{(\lambda_0,\lambda_1)}/W^\varnothing)\right).
	\end{align*}
	Since $H^0_\rig(Z_{\lambda_0}/W^\varnothing)$ and $H^0_\rig(Z_{(\lambda_0,\lambda_1)}/W^\varnothing)$ are isomorphic to $L$ if $Z_{(\lambda_0,\lambda_1)}\neq\emptyset$, we obtain the equality $H^0_\rig(Y/W^\varnothing)=L$.
\end{proof}

\begin{lemma}\label{lem: nonzero inj}
	Let $Y$ be a strictly semistable log scheme over $k^0$ and let $\sE\in\Isoc^\dagger(Y/W^\varnothing)^\unip$.
	Then $H^0_\rig(Y/W^\varnothing,\sE)$ is non-zero.
	If $Y$ is geometrically connected over $k$, then the natural map
		\[H^0_\rig(Y/W^\varnothing,\sE)\otimes_L\sO_{Y/W^\varnothing}\rightarrow\sE\]
	is injective.
\end{lemma}

\begin{proof}
	The assertions for $\sE=\sO_{Y/W^\varnothing}$ follow by Lemma \ref{lem: H0}.	
	For general case, we use the induction on the rank of $\sE$.
	Take an exact sequence $0\rightarrow \sF\rightarrow\sE\rightarrow\sO_{Y/W^\varnothing}\rightarrow 0$, where $\sF$ is unipotent.
	Since $H^0_\rig(Y/W^\varnothing,\sF)\rightarrow H^0_\rig(Y/W^\varnothing,\sE)$ is injective, the first assertion follows from the hypothesis of the induction.
	Moreover in a commutative diagram	
		\[\xymatrix{
		0\ar[d] & 0\ar[d]\\
		H^0_\rig(Y/W^\varnothing,\sF)\otimes_L\sO_{Y/W^\varnothing}\ar[r]\ar[d] & \sF\ar[d]\\
		H^0_\rig(Y/W^\varnothing,\sE)\otimes_L\sO_{Y/W^\varnothing}\ar[r]\ar[d] & \sE\ar[d]\\
		H^0_\rig(Y/W^\varnothing,\sO_{Y/W^\varnothing})\otimes_L\sO_{Y/W^\varnothing}\ar[r] & \sO_{Y/W^\varnothing}\ar[d]\\
		& 0,
		}\]
	the upper and bottom horizontal maps are injective by the hypothesis of the induction.
	Chasing the diagram we see that the middle horizontal map is also injective.
\end{proof}

The following is an analogue of \cite[Proposition 2.3.5]{CL1}.

\begin{proposition-definition}\label{prop: filtration}
	Let $Y$ be a geometrically connected strictly semistable log scheme over $k^0$.
	Then for any object $\sE\in\Isoc^\dagger(Y/W^\varnothing)^\unip$ there exists a unique increasing filtration $\Fil_\bullet$ on $\sE$ such that $\Fil_{-1}\sE=0$ and the natural maps
		\begin{eqnarray}
		\label{eq: Gr of Fil} H^0_\rig(Y/W^\varnothing,\sE)\otimes_L\sO_{Y/W^\varnothing}&\rightarrow&\Gr^\Fil_0\sE,\\
		\nonumber H^0_\rig(Y/W^\varnothing,\sE/\Fil_n\sE)\otimes_L\sO_{Y/W^\varnothing}&\rightarrow& \Gr^\Fil_{n+1}\sE \hspace{10pt} \text{for }n\geq 0
		\end{eqnarray}
	are isomorphisms.
	We call this the {\it natural filtration} on $\sE$.
\end{proposition-definition}

\begin{proof}
	The uniqueness is clear.
	We prove the existence by the induction on the rank of $\sE$.
	Let $\sE':=H^0_\rig(Y/W^\varnothing,\sE)\otimes_L\sO_{Y/W^\varnothing}$.
	Then clearly $\sE'$ has nilpotent residues, and by Lemma \ref{lem: nonzero inj} $\sE'$ is a non-zero subobject of $\sE$ in $\Isoc^\dagger(Y/W^\varnothing)^\nr$.
	By Proposition \ref{prop: unipotence} $\sE'':=\sE/\sE'$ is unipotent, and its rank is smaller than $\sE$.
	Therefore $\sE''$ has the required filtration.
	For any $n\geq 0$ we define $\Fil_{n+1}\sE$ to be the pull-back of $\Fil_n\sE''\subset\sE''$ by the surjection $\sE\rightarrow\sE''$.
	Since the induced map $\sE/\Fil_{n+1}\sE\rightarrow\sE''/\Fil_n\sE''$ is an isomorphism, we have isomorphisms
		\begin{eqnarray*}
		H^0_\rig(Y/W^\varnothing,\sE/\Fil_{n+1}\sE)\otimes_L\sO_{Y/W^\varnothing}&\xrightarrow{\cong}& H^0_\rig(Y/W^\varnothing,\sE''/\Fil_n\sE'')\otimes_L\sO_{Y/W^\varnothing}\\
		&\xrightarrow{\cong}&\Gr^\Fil_n\sE''\xleftarrow{\cong}\Gr^\Fil_{n+1}\sE,
		\end{eqnarray*}
	which commute with the natural map $H^0_\rig(Y/W^\varnothing,\sE/\Fil_{n+1}\sE)\otimes_L\sO_{Y/W^\varnothing}\rightarrow\Gr^\Fil_{n+1}\sE$.
\end{proof}

\begin{proposition}\label{prop: unip const}
	Let $Y$ be a geometrically connected strictly semistable log scheme over $k^0$.
	Then for any object $(\sE,\Phi)\in F\Isoc^\dagger(Y/W^\varnothing)^\unip$, the natural filtration $\Fil_\bullet$ is stable under $\Phi$, and $\Gr^\Fil_n(\sE,\Phi)$ is constant for any $n$.
\end{proposition}

\begin{proof}
	Since $\Fil_\bullet$ is functorial, it is stable under Frobenius action.
	Moreover the isomorphisms \eqref{eq: Gr of Fil} are compatible with the Frobenius action.
	Thus $\Gr^\Fil_n(\sE,\Phi)$ is constant.
\end{proof}

%%%%%%%%%%%%%%%%%%%%%%%%%%%%%%%%%
\section{Hyodo--Kato cohomology with coefficients}\label{sec: HK coh}
%%%%%%%%%%%%%%%%%%%%%%%%%%%%%%%%%
In this section we will define the Hyodo--Kato cohomology with coefficients in log overconvergent ($F$-)isocrystals over $W^\varnothing$.
In fact the construction in \cite{EY} for the trivial coefficient is naturally generalized to nontrivial coefficients.
The following is an enhancement of a construction due to Kim and Hain \cite[pp. 1259--1260]{KH}.

\begin{definition}[{\cite[Definition 3.2]{EY}}]\label{def: Kim--Hain}
	Let $\cZ$ be a weak formal log scheme over $\cS$.
	We define the {\it Kim--Hain complex } $\omega^\bullet_{\cZ/W^\varnothing,\bbQ}[u]$ to be the CDGA on $\cZ_\bbQ$ generated by $\omega^\bullet_{\cZ/W^\varnothing,\bbQ}$ and degree zero elements $u^{[i]}$ for $i\geq 0$ with multiplication
		\begin{equation}
			u^{[i]}\wedge u^{[j]}=\frac{(i+j)!}{i!j!}u^{[i+j]}
		\end{equation}
	and relations
		\begin{align}
			du^{[i+1]}=-d\log s\cdot u^{[i]}&&\text{and}&&u^{[0]}=1.
		\end{align}
	
	For an object $(\cE,\nabla)\in \MIC^\dagger(\cZ/W^\varnothing)$, we denote again by $\nabla$ the composition
		 \begin{align*}
		 \cE\xrightarrow{\nabla}\cE\otimes\omega^1_{\cZ/W^\varnothing,\bbQ}\rightarrow\cE\otimes\omega^1_{\cZ/W^\varnothing,\bbQ}[u],
		 \end{align*}
	which defines a complex $\cE\otimes\omega^\bullet_{\cZ/W^\varnothing,\bbQ}[u]$.
	We define an $\cO_{\cZ_\bbQ}$-linear endomorphism $N$ on $\cE\otimes\omega_{\cZ/W^\varnothing,\bbQ}^\bullet[u]$ 
	by
		\[N(u^{[i]}):= u^{[i-1]}.\]
	When we consider a Frobenius structure, i.e. when $(\cE,\nabla,\Phi)\in F\MIC^\dagger((\cZ,\phi)/(W^\varnothing,\sigma))$, we extend the $\phi$-semilinear endomorphism $\varphi$ on $\cE\otimes\omega^\bullet_{\cZ/W^\varnothing,\bbQ}$ defined as in Definition \ref{def: FMIC} to $\cE\otimes\omega^\bullet_{\cZ/W^\varnothing,\bbQ}[u]$
	by
		\[\varphi(u^{[i]}):=p^iu^{[i]}.\]
\end{definition}

Now we define the Hyodo--Kato cohomology and the absolute Hyodo--Kato cohomology with coefficients.
For an abelian category $\cM$, we denote its bounded derived category by $ D^b(\cM)$.

\begin{definition}\label{def: HK coh}
	Let $Y$ be a fine log scheme over $k^0$.
	\begin{enumerate}
	\item Take a local embedding datum $(Z_\lambda,\cZ_\lambda,i_\lambda,h_\lambda,\theta_\lambda)_{\lambda\in\Lambda}$ for $Y$ over $(k^0,\cS,\tau)$.
	Similarly to \eqref{eq: product of embedding data} and \eqref{eq: simplicial system}, for $m\geq 0$ and $\underline{\lambda}=(\lambda_0,\ldots,\lambda_m)\in\Lambda^{m+1}$ we define
			\begin{align}
			\label{eq: product of embedding data HK}
			&(Z_{\underline{\lambda}},\cZ_{\underline{\lambda}},i_{\underline{\lambda}},h_{\underline{\lambda}},\theta_{\underline{\lambda}}):=(Z_{\lambda_0},\cZ_{\lambda_0},i_{\lambda_0},h_{\lambda_0},\theta_{\lambda_0})\times\cdots\times(Z_{\lambda_m},\cZ_{\lambda_m},i_{\lambda_m},h_{\lambda_m},\theta_{\lambda_m}),\\
			\label{eq: simplicial system HK}
			&(Z_m,\cZ_m,i_m,h_m,\theta_m):=\coprod_{\underline{\lambda}\in\Lambda^{m+1}}(Z_{\underline{\lambda}},\cZ_{\underline{\lambda}},i_{\underline{\lambda}},h_{\underline{\lambda}},\theta_{\underline{\lambda}}),
			\end{align}
		where the products are taken in $\OC(Y/\cS)$ (not in $\OC(Y/W^\varnothing)$).
		Then we obtain a simplicial object $(Z_\bullet,\cZ_\bullet,i_\bullet,h_\bullet,\theta_\bullet)$ of $\OC(Y/\cS)$, which is also regarded as a simplicial object of $\OC(Y/W^\varnothing)$.
		Hence for any object $\sE\in \Isoc^\dagger(Y/W^\varnothing)$, we have a cosimplicial complex $R\Gamma(\cZ_{\bullet,\bbQ},\sE_{\cZ_\bullet}\otimes\omega^\star_{\cZ_\bullet/W^\varnothing,\bbQ}[u])$.
		Note that, by using Godement resolution we may obtain transition morphisms as morphisms of complexes (not in the derived sense).
		We define the {\it Hyodo--Kato cohomology} $R\Gamma_\HK(Y,\sE)$ of $Y$ with coefficients in $\sE$ to be the complex associated with the cosimplicial complex $R\Gamma(\cZ_{\bullet,\bbQ},\sE_{\cZ_\bullet}\otimes\omega^\star_{\cZ_\bullet/W^\varnothing,\bbQ}[u])$,
		which is equipped with the endomorphism induced by $N$ in Definition \ref{def: Kim--Hain}.
	\item Let $(\sE,\Phi)\in F\Isoc^\dagger(Y/W^\varnothing)$.
		By taking Frobenius lifts $\phi_\lambda$ on the above $\cZ_\lambda$, the endomorphisms $\varphi_m$ on $\sE_{\cZ_m}\otimes\omega^\bullet_{\cZ_m/W^\varnothing,\bbQ}[u]$ defined in Definition \ref{def: Kim--Hain} induce a $\sigma$-semilinear endomorphism $\varphi$ on $R\Gamma_\HK(Y,\sE)$.
		Therefore we obtain an object $ D^b(\Mod_L(\varphi,N))$, which we denote by $R\Gamma_\HK(Y,(\sE,\Phi))$.
	\item For an object $(\sE,\Phi)\in F\Isoc^\dagger(Y/W^\varnothing)$, we define the {\it absolute Hyodo--Kato cohomology} $R\Gamma_\abs(Y/W^\varnothing,(\sE,\Phi))$ of $Y$ with coefficients in $(\sE,\Phi)$ to be the mapping cone
			\begin{equation}
			R\Gamma_\abs(Y,(\sE,\Phi)):=\Cone\left(R\Gamma_\rig(Y/W^\varnothing,(\sE,\Phi))\xrightarrow{\varphi-1}R\Gamma_\rig(Y/W^\varnothing,(\sE,\Phi))\right)[-1].
			\end{equation}
	\end{enumerate}
	For the trivial coefficient $\sO_{Y/W^\varnothing}$, we denote
	\begin{align*}
	R\Gamma_\HK(Y):=R\Gamma_\HK(Y,\sO_{Y/W^\varnothing})&&\text{and}&&R\Gamma_\abs(Y):=R\Gamma_\abs(Y,\sO_{Y/W^\varnothing}).
	\end{align*}
\end{definition}

\begin{proposition}\label{prop: indep HK}
	Let $Y$ be a fine log scheme over $k^0$.
	For any $\sE\in\Isoc^\dagger(Y/W^\varnothing)$, the cohomology $R\Gamma_\HK(Y,\sE)$ is independent of the choice of a local embedding datum over $(k^0,\cS,\tau)$ up to canonical quasi-isomorphisms.
	For any $(\sE,\Phi)\in F\Isoc^\dagger(Y/W^\varnothing)$, the cohomologies $R\Gamma_\HK(Y,(\sE,\Phi))$ and $R\Gamma_\abs(Y,(\sE,\Phi))$ are independent of the choice of local embedding $F$-datum over $(k^0,\cS,\tau,\sigma)$ and local embedding $F$-datum over $(k^\varnothing,W^\varnothing,\iota,\sigma)$ up to canonical quasi-isomorphisms, respecticely.
\end{proposition}

\begin{proof}
	The independence of $R\Gamma_\HK(Y,\sE)$ follows by the same arguments as Proposition \ref{prop: log rig coh indep} (see also \cite[Proposition 3.6]{EY}).
	The independence of $R\Gamma_\abs(Y,(\sE,\Phi))$ also follows from that of $R\Gamma_\rig(Y/W^\varnothing,(\sE,\Phi))$.
\end{proof}

\begin{remark}\label{rem: rig W}
	Let $Y$, $\sE$, $(Z_\lambda,\cZ_\lambda,i_\lambda,h_\lambda,\theta_\lambda)_{\lambda\in\Lambda}$, and $(Z_\bullet,\cZ_\bullet,i_\bullet,h_\bullet,\theta_\bullet)$ be as in Definition \ref{def: HK coh}.
	Note that the log rigid cohomology $R\Gamma_\rig(Y/W^\varnothing,\sE)$ is computed by using $(Z_\bullet,\cZ_\bullet,i_\bullet,h_\bullet,\theta_\bullet)$.
	Indeed, let $(Z'_\bullet,\cZ'_\bullet,i'_\bullet,h'_\bullet,\theta'_\bullet)$ be the simplicial object of $\OC(Y/W^\varnothing)$ given by taking the products in $\OC(Y/W^\varnothing)$ instead of $\OC(Y/\cS)$.
	The log rigid cohomology $R\Gamma_\rig(Y/W^\varnothing,\sE)$ is a priori defined to be the complex associated to the cosimplicial complex $R\Gamma(\cZ'_{\bullet,\bbQ},\sE_{\cZ'_\bullet}\otimes\omega^\star_{\cZ'_\bullet/W^\varnothing,\bbQ})$.
	Then there exists a canonical morphism
		\[R\Gamma_\rig(Y/W^\varnothing,\sE)=R\Gamma(\cZ'_{\bullet,\bbQ},\sE_{\cZ'_\bullet}\otimes\omega^\star_{\cZ'_\bullet/W^\varnothing,\bbQ})\rightarrow R\Gamma(\cZ_{\bullet,\bbQ},\sE_{\cZ_\bullet}\otimes\omega^\star_{\cZ_\bullet/W^\varnothing,\bbQ}),\]
	which is a quasi-isomorphism since $R\Gamma(\cZ'_{m,\bbQ},\sE_{\cZ'_m}\otimes\omega^\star_{\cZ'_m/W^\varnothing,\bbQ})$ and $R\Gamma(\cZ_{m,\bbQ},\sE_{\cZ_m}\otimes\omega^\star_{\cZ_m/W^\varnothing,\bbQ})$ both compute $R\Gamma_\rig(Z_m/W^\varnothing,\sE)$ for each $m\in\bbN$.
\end{remark}

\begin{proposition}\label{prop: HK coh total}
	Let $Y$ be a fine log scheme over $k^0$, $\sE$ an object of $\Isoc^\dagger(Y/W^\varnothing)$, and $(Z_\lambda,\cZ_\lambda,i_\lambda,h_\lambda,\theta_\lambda)_{\lambda\in\Lambda}$ a local embedding datum for $Y$ over $(k^0,\cS,\tau)$ such that $Z_\lambda$ and $\cZ_\lambda$ are affine for all $\lambda\in\Lambda$.
	Define objects $(Z_{\underline{\lambda}},\cZ_{\underline{\lambda}},i_{\underline{\lambda}},h_{\underline{\lambda}},\theta_{\underline{\lambda}})$ of $\OC(Y/\cS)$ and a simplicial object $(Z_\bullet,\cZ_\bullet,i_\bullet,h_\bullet,\theta_\bullet)$ of $\OC(Y/\cS)$ as in Definition \ref{def: HK coh}.
	Moreover define weak formal log schemes $\cZ_{\underline{\lambda},n}=\cZ_{\underline{\lambda}}[\frac{\cJ_{\underline{\lambda}}^n}p]^\dagger$ and $\cZ_{m,n}:=\coprod_{\underline{\lambda}\in\Lambda^{m+1}}\cZ_{\underline{\lambda},n}$ as in Remark \ref{rem: affinoid}.
	Then $R\Gamma_\HK(Y,\sE)$ is represented by the total complex of the triple complex
	\[\prod_{n\in\bbN}\Gamma(\cZ_{\bullet,n,\bbQ},\sE_{\cZ_{\bullet,n}}\otimes\omega^\star_{\cZ_{\bullet,n}/W^\varnothing,\bbQ}[u])\xrightarrow{\eth}\prod_{n\in\bbN}\Gamma(\cZ_{\bullet,n,\bbQ},\sE_{\cZ_{\bullet,n}}\otimes\omega^\star_{\cZ_{\bullet,n}/W^\varnothing,\bbQ}[u])\]
	where the first differential is the log connection on $\sE_{\cZ_{\bullet,n}}\otimes\omega^\star_{\cZ_{\bullet,n}/W^\varnothing}[u]$, the second differential is induced by the face maps of the simplicial weak formal log schemes $\cZ_{\bullet,n}$, and the third differential $\eth$ maps an element $(\eta_{i,m,n})_n\in\prod_{n\in\bbN}\Gamma(\cZ_{m,n,\bbQ},\sE_{\cZ_{m,n}}\otimes\omega^i_{\cZ_{m,n}/W^\varnothing,\bbQ}[u])$ to $(\eta_{i,m,n}-\eta_{i,m,n+1}|_{\cZ_{m,n,\bbQ}})_n$ for any $i,m\in\bbN$.
\end{proposition}

\begin{proof}
	 Note that $\sE_{\cZ_{m,n}}\otimes\omega^i_{\cZ_{m,n}/W^\varnothing,\bbQ}[u]$ is written as a direct sum of coherent shaves, and that $\cZ_{m,n,\bbQ}$ is a disjoint union of affinoid spaces.
	 Thus by \cite[Lemma 3.5]{EY} the cohomology of $\sE_{\cZ_{m,n}}\otimes\omega^\star_{\cZ_{m,n}/W^\varnothing,\bbQ}[u]$ is computed by global sections.
	 Using this, the assertion follows similarly to Remark \ref{rem: affinoid}.
\end{proof}

The following relations between Hyodo--Kato cohomology and log rigid cohomology over $W^\varnothing$ are deduced immediately from the definition.

\begin{proposition}
	Let $Y$ be a fine log scheme over $k^0$ and $\sE\in\Isoc^\dagger(Y/W^\varnothing)$.
	Then there exists a canonical quasi-isomorphism
		\begin{equation}\label{eq: cone of N}
		R\Gamma_\rig(Y/W^\varnothing,\sE)\xrightarrow{\cong}\Cone\left(R\Gamma_\HK(Y,\sE)\xrightarrow{N}R\Gamma_\HK(Y,\sE)\right)[-1].
		\end{equation}
	If we consider a Frobenius structure $\Phi$ on $\sE$, then \eqref{eq: cone of N} is compatible with Frobenius operators, and we have a canonical quasi-isomrophism
		\begin{align}\label{eq: abs HK holim}
		R\Gamma_\abs(Y,(\sE,\Phi))&\xrightarrow{\cong}\Tot\left[\begin{xy}
		(0,6)*{R\Gamma_\HK(Y,(\sE,\Phi))}="A", (35,6)*{ R\Gamma_\HK(Y,(\sE,\Phi))}="C",
		(0,-6)*{R\Gamma_\HK(Y,(\sE,\Phi))}="B", (35,-6)*{R\Gamma_\HK(Y,(\sE,\Phi))}="D",
		\ar^-{\varphi-1} "B";"A", \ar^-N "A";"C", \ar^-N "B";"D", \ar_-{p\varphi-1} "D";"C"
		\end{xy}\right]
		\end{align}
		where we regard the lower left term as degree $(0,0)$ with respect to the horizontal and vertical differentials.
\end{proposition}

\begin{proof}
	For any object $(Z,\cZ,i,h,\theta)\in\OC(Y/\cS)$, the sequence
		\[0\rightarrow\sE_{\cZ}\otimes\omega^\bullet_{\cZ/W^\varnothing,\bbQ}\rightarrow\sE_{\cZ}\otimes\omega^\bullet_{\cZ/W^\varnothing,\bbQ}[u]\xrightarrow{N}\sE_{\cZ}\otimes\omega^\bullet_{\cZ/W^\varnothing,\bbQ}[u]\rightarrow0\]
	is exact.
	Therefore we obtain \eqref{eq: cone of N} by Remark \ref{rem: rig W}, and this implies \eqref{eq: abs HK holim}.
\end{proof}

For a morphism $f\colon Y'\rightarrow Y$ between fine log schemes over $k^0$ and an object $(\sE,\Phi)\in F\Isoc^\dagger(Y/W^\varnothing)$, the pull-back of log rigid cohomology induces
	\begin{equation}
	f^*\colon R\Gamma_\abs(Y,(\sE,\Phi))\rightarrow R\Gamma_\abs(Y',f^*(\sE,\Phi)).
	\end{equation}
In addition for $\sE\in\Isoc^\dagger(Y/W^\varnothing)$ (resp.\  $(\sE,\Phi)\in F\Isoc^\dagger(Y/W^\varnothing)$) we may define a morphism
	\begin{align}
	f^*\colon R\Gamma_\HK(Y,\sE)\rightarrow R\Gamma_\HK(Y',f^*\sE)&&\left(\text{resp.\  }f^*\colon R\Gamma_\HK(Y,(\sE,\Phi))\rightarrow R\Gamma_\HK(Y',f^*(\sE,\Phi))\right)
	\end{align}
by a similar construction as \eqref{eq: pull back log rig coh F}.

Let $Y$ be a fine log scheme over $k^0$.
For $(\sE,\Phi)\in F\Isoc^\dagger(Y/W^\varnothing)$, by abusing notation, we again denote by $(\sE,\Phi)$ the objects of $F\Isoc^\dagger(Y/\cS)$ and $F\Isoc^\dagger(Y/W^0)$ defined as the pull-back of $(\sE,\Phi)$ by the commutative diagram
	\[\xymatrix{
	& k^0\ar[r]^-{\tau_0}\ar@{=}[d] & W^0\ar[d]\\
	Y\ar[r]\ar[ru]\ar[rd] & k^0\ar[r]^-\tau \ar[d]& \cS\ar[d]\\
	& k^\varnothing\ar[r]^-\iota & W^\varnothing.
	}\]
For any object $(Z,\cZ,i,h,\theta)\in\OC(Y/\cS)$, natural morphisms $\sE_{\cZ}\otimes\omega^\bullet_{\cZ/W^\varnothing,\bbQ}\rightarrow\sE_{\cZ}\otimes\omega^\bullet_{\cZ/\cS,\bbQ}$ extend to $\sE_{\cZ}\otimes\omega^\bullet_{\cZ/W^\varnothing,\bbQ}[u]\rightarrow\sE_{\cZ}\otimes\omega^\bullet_{\cZ/\cS,\bbQ}$ by $u^{[i]}\mapsto 0$ for $i>0$.
Thus we obtain a canonical morphism
	\begin{equation*}\label{eq: rig to HK}
	R\Gamma_\HK(Y,(\sE,\Phi))\rightarrow R\Gamma_\rig(Y/\cS,(\sE,\Phi)),
	\end{equation*}
which makes the following diagram commutative
	\begin{equation}\label{eq: diag hk rig}
	\xymatrix{
	& R\Gamma_\HK(Y,(\sE,\Phi)) \ar[rd]\ar[dd] & \\
	R\Gamma_\rig(Y/W^\varnothing,(\sE,\Phi)) \ar[ru]\ar[rd] & & R\Gamma_\rig(Y/\cS,(\sE,\Phi))\ar[ld]\\
	& R\Gamma_\rig(Y/W^0,(\sE,\Phi)) &.
	}\end{equation}
There also exists a diagram of the same form for any $\sE\in\Isoc^\dagger(Y/W^\varnothing)$ (without Frobenius structure).

Assume that $Y$ is geometrically connected and strictly semistable.
We note that an object $(\sE,\Phi)\in F\Isoc^\dagger(Y/W^\varnothing)^\unip$ is not an iterated extension of $\sO_{Y/W^\varnothing}$ in $F\Isoc^\dagger(Y/W^\varnothing)^\unip$, in general.
However Proposition \ref{prop: unip const} asserts that $(\sE,\Phi)$ is an iterated extension of constant objects.
Since cohomology with constant coefficients is written as the tensor product $H^n_\HK(Y,M^a)=H^n_\HK(Y)\otimes M$, some properties for $(\sE,\Phi)$ are implied from that for $\sO_{Y/W^\varnothing}$.
Precisely, putting
	\[M_i:=H^0_\rig(Y/W^\varnothing,(\sE,\Phi)/\Fil_{i-1}(\sE,\Phi))\]
for $i\geq 0$, we have $\Gr_i^\Fil(\sE,\Phi)=M_i^a$. 
Therefore $\Fil_\bullet$ induces spectral sequences
	\begin{eqnarray}
	\label{eq: ss for Fil}&&E_1^{-i,i+n}=M_i\otimes H^n_\HK(Y) \Rightarrow H^n_\HK(Y,(\sE,\Phi)),\\
	\nonumber&&E_1^{-i,i+n}=M_i\otimes H^n_\rig(Y/\cT) \Rightarrow H^n_\rig(Y/\cT,(\sE,\Phi))\hspace{15pt}\text{for }\cT=W^\varnothing,W^0,\cS,
	\end{eqnarray}
which are compatible with each other via the morphisms in the diagram \eqref{eq: diag hk rig}.
Of course we may consider similar spectral sequences for any $\sE\in\Isoc^\dagger(Y/W^\varnothing)^\unip$ (without Frobenius structure).

\begin{theorem}\label{thm: HK rig comparison}
	Let $Y$ be a strictly semistable log scheme over $k^0$.
	\begin{enumerate}
	\item\label{item: finiteness} If $Y$ is quasi-compact, then the cohomology groups $H^n_\HK(Y,\sE)$, $H^n_\rig(Y/W^\varnothing,\sE)$, $H^n_\rig(Y/W^0,\sE)$, and $H_\abs^n(Y,\sE)$ for any $\sE\in \Isoc^\dagger(Y/W^\varnothing)^\unip$ are finite-dimensional $L$-vector spaces.
		Moreover $H^n_\HK(Y,\sE)$ is an object of $\Mod^\fin_L(N)$, i.e. the monodromy operator is nilpotent.
	\item\label{item: Frob bij} For any $(\sE,\Phi)\in F\Isoc^\dagger(Y/W^\varnothing)^\unip$, the Frobenius operators on $H^n_\HK(Y,(\sE,\Phi))$, $H^n_\rig(Y/W^\varnothing,(\sE,\Phi))$, and $H^n_\rig(Y/W^0,(\sE,\Phi))$ are bijective.
		In particular, $H_\HK^n(Y,(\sE,\Phi))$ is an object of $\Mod^\fin_L(\varphi,N)$ if $Y$ is quasi-compact.
	\item\label{item: HK and rig} For any $\sE\in \Isoc^\dagger(Y/W^\varnothing)^\unip$, the morphism 
			\begin{equation}\label{eq: HK rig qis}
			R\Gamma_\HK(Y,\sE) \rightarrow R\Gamma_\rig(Y/W^0,\sE)
			\end{equation}
		defined by $u^{[i]}\mapsto 0$ for $i>0$ is a quasi-isomorphism.
		When we consider a Frobenius structure $\Phi$ on $\sE$, this is compatible with the Frobenius operators.
	\end{enumerate}
\end{theorem}

\begin{proof}
	That \eqref{eq: HK rig qis} is a quasi-isomorphism immediately follows from \cite[Corollary 3.37]{EY} since $\sE$ is unipotent.
	The compatibility with the Frobenius operators is clear by construction, hence \eqref{item: HK and rig} is proved.
	
	Next we prove \eqref{item: finiteness} and \eqref{item: Frob bij} for the case that $Y$ is quasi-compact.
	Due to Proposition \ref{prop: base change coh}, we may assume that $Y$ is geometrically connected, by taking base change and a connected component.
	Moreover, using the spectral sequence \eqref{eq: ss for Fil}, we may reduce to the case of the trivial coefficients.
	In this case the assertions follow from \cite[Theorem 5.3]{GK3} via \eqref{eq: HK rig qis}, \eqref{eq: cone of N}, and \eqref{eq: abs HK holim}.
	Note that the nilpotence of the monodromy operator for the trivial coefficients follows from the finiteness of the cohomology and the relation $N\varphi=p\varphi N$.
	
	Finally, \eqref{item: Frob bij} for the general case follows from the quasi-compact case by using \eqref{eq: ss for open}.
\end{proof}

\begin{remark}
	We may define monodromy operator on $H^n_\rig(Y/W^0,\sE)$ in the same way as \cite[Definition 3.34]{EY}.
	One can show that this monodromy operator is compatible with that on $H^n_\HK(Y,\sE)$ through \eqref{eq: HK rig qis}, in the same way as \cite[Proposition 3.35]{EY}.
\end{remark}

Next we see that the extension groups of log overconvergent ($F$-)isocrystals can be computed as cohomology groups.

\begin{proposition}\label{prop: Ext abs HK}
	Let $Y$ be a strictly semistable log scheme over $k^0$.
	\begin{enumerate}
	\item For any $\sE,\sE'\in \Isoc^\dagger(Y/W^\varnothing)^\unip$, we have canonical isomorphisms
		\begin{eqnarray}
		\nonumber &&H^{0}_\rig(Y/W^\varnothing,\sheafhom(\sE',\sE))\xrightarrow{\cong}\Hom_{\Isoc^\dagger(Y/W^\varnothing)^\nr}(\sE',\sE)\xrightarrow{\cong}\Hom_{\Isoc^\dagger(Y/W^\varnothing)^\unip}(\sE',\sE),\\
		\label{eq: Ext and coh 1} &&H^1_\rig(Y/W^\varnothing,\sheafhom(\sE',\sE))\xrightarrow{\cong}\Ext^1_{\Isoc^\dagger(Y/W^\varnothing)^\nr}(\sE',\sE)\xrightarrow{\cong}\Ext^1_{\Isoc^\dagger(Y/W^\varnothing)^\unip}(\sE',\sE).
		\end{eqnarray}
	\item For any $(\sE,\Phi),(\sE',\Phi')\in F\Isoc^\dagger(Y/W^\varnothing)^\unip$, we have canonical isomorphisms
		\begin{eqnarray}
		\nonumber H_\abs^0(Y,\sheafhom((\sE',\Phi'),(\sE,\Phi)))&\xrightarrow{\cong}&\Hom_{F\Isoc^\dagger(Y/W^\varnothing)}((\sE',\Phi'),(\sE,\Phi))\\
		\nonumber&\xrightarrow{\cong}&\Hom_{F\Isoc^\dagger(Y/W^\varnothing)^\unip}((\sE',\Phi'),(\sE,\Phi)),\\
		\label{eq: Ext and coh} H_\abs^1(Y,\sheafhom((\sE',\Phi'),(\sE,\Phi)))&\xrightarrow{\cong}&\Ext^1_{F\Isoc^\dagger(Y/W^\varnothing)}((\sE',\Phi'),(\sE,\Phi))\\
		\nonumber&\xrightarrow{\cong}&\Ext^1_{F\Isoc^\dagger(Y/W^\varnothing)^\unip}((\sE',\Phi'),(\sE,\Phi)).
		\end{eqnarray}
	\end{enumerate}
\end{proposition}

\begin{proof}
	We only prove \eqref{eq: Ext and coh}, and the other isomorphisms can be proved more easily by the same way.
	The second isomorphism of \eqref{eq: Ext and coh} is given by Proposition \ref{prop: unipotence}.
	Let $(\sG,\wt\Phi):=\sheafhom((\sE',\Phi'),(\sE,\Phi))$.
	Take a local embedding datum $(Z_\lambda,\cZ_\lambda,i_\lambda,h_\lambda,\theta_\lambda)_{\lambda\in\Lambda}$ with Frobenius lifts $\phi_\lambda$ on $\cZ_\lambda$ such that $Z_\lambda$ and $\cZ_\lambda$ are affine for all $\lambda\in\Lambda$.
	For $m\geq 0$, $n\geq 1$ and $\underline{\lambda}\in\Lambda^{m+1}$, we define the weak formal log schemes $\cZ_{\underline{\lambda}}$, $\cZ_{\underline{\lambda},n}=\cZ_{\underline{\lambda}}[\frac{\cJ^n_{\underline{\lambda}}}p]^\dagger$ and $\cZ_{m,n}=\coprod_{\underline{\lambda}\in\Lambda^{m+1}}\cZ_{\underline{\lambda},n}$ as in Remark \ref{rem: affinoid}.
	Then by Remark \ref{rem: affinoid}, $R\Gamma_\abs(Y,(\sG,\Psi))$ is represented by the total complex of the quadruple complex
	\[\xymatrix{
	\prod_{n\in\bbN}\Gamma(\cZ_{\bullet,n,\bbQ},\sG_{\cZ_{\bullet,n}}\otimes\omega^\star_{\cZ_{\bullet,n}/W^\varnothing,\bbQ})\ar[r]^-\eth
	&\prod_{n\in\bbN}\Gamma(\cZ_{\bullet,n,\bbQ},\sG_{\cZ_{\bullet,n}}\otimes\omega^\star_{\cZ_{\bullet,n}/W^\varnothing,\bbQ})\\
	\prod_{n\in\bbN}\Gamma(\cZ_{\bullet,n,\bbQ},\sG_{\cZ_{\bullet,n}}\otimes\omega^\star_{\cZ_{\bullet,n}/W^\varnothing,\bbQ})\ar[r]^-\eth\ar[u]^-{\varphi-1}
	&\prod_{n\in\bbN}\Gamma(\cZ_{\bullet,n,\bbQ},\sG_{\cZ_{\bullet,n}}\otimes\omega^\star_{\cZ_{\bullet,n}/W^\varnothing,\bbQ})\ar[u]^-{\varphi-1}
	}\]
	where the first differential which we denote by $\wt\nabla$ is induced by the log connection of $\sG_{\cZ_{\bullet,n}}\otimes\omega^\star_{\cZ_{\bullet,n}/W^\varnothing,\bbQ}$, the second differential is induced by the face maps of the simplicial weak formal log schemes $\cZ_{\bullet,n}$, the third differential $\eth$ is defined as in Remark \ref{rem: affinoid}, and the fourth differential is $\varphi-1$.
	We consider that the lower left term of the above diagram has degree $0$ for the third and fourth differentials.
	Then a $1$-cocycle of this total complex is a family of elements
	\begin{align*}
	&(\eta_{\lambda,n})_{\lambda,n}\in\prod_{\lambda\in\Lambda}\prod_{n\in\bbN}\Gamma(\cZ_{\lambda,n,\bbQ},\sG_{\cZ_{\lambda,n}}\otimes\omega^1_{\cZ_{\lambda,n}/W^\varnothing,\bbQ})=\prod_{\lambda\in\Lambda}\prod_{n\in\bbN}\Hom(\cE'_{\cZ_{\lambda,n}},\cE_{\cZ_{\lambda,n}}\otimes\omega^1_{\cZ_{\lambda,n}/W^\varnothing,\bbQ}),\\
	&(f_{\underline{\lambda},n})_{\underline{\lambda},n}\in\prod_{\underline{\lambda}\in\Lambda^2}\prod_{n\in\bbN}\Gamma(\cZ_{\underline{\lambda},n,\bbQ},\sG_{\cZ_{\underline{\lambda},n}})=\prod_{\underline{\lambda}\in\Lambda^2}\prod_{n\in\bbN}\Hom(\cE'_{\cZ_{\underline{\lambda},n}},\cE_{\cZ_{\underline{\lambda},n}}),\\
	&(g_{\lambda,n})_{\lambda,n}\in\prod_{\lambda\in\Lambda}\prod_{n\in\bbN}\Gamma(\cZ_{\lambda,n,\bbQ},\sG_{\cZ_{\lambda,n}})=\prod_{\lambda\in\Lambda}\prod_{n\in\bbN}\Hom(\cE'_{\cZ_{\lambda,n}},\cE_{\cZ_{\lambda,n}}),\\
	&(h_{\lambda,n})_{\lambda,n}\in\prod_{\lambda\in\Lambda}\prod_{n\in\bbN}\Gamma(\cZ_{\lambda,n,\bbQ},\sG_{\cZ_{\lambda,n}})=\prod_{\lambda\in\Lambda}\prod_{n\in\bbN}\Hom(\cE'_{\cZ_{\lambda,n}},\cE_{\cZ_{\lambda,n}}),
	\end{align*}
	such that
	\begin{align}
	\label{eq: E1}&\wt\nabla(\eta_{\lambda,n})=0
	&&(\forall \lambda\in\Lambda, \forall n\in\bbN),\\
	\label{eq: E2}&\mathrm{pr}^\ast_1\eta_{\lambda_1,n}-\mathrm{pr}_0^\ast\eta_{\lambda_0,n}=\wt\nabla(f_{(\lambda_0,\lambda_1),n})
	&&(\forall(\lambda_0,\lambda_1)\in\Lambda^2, \forall n\in\bbN),\\
	\label{eq: E3}&\eta_{\lambda,n}-\eta_{\lambda,n+1}|_{\cZ_{\lambda,n,\bbQ}}=\wt\nabla(g_{\lambda,n})
	&&(\forall\lambda\in\Lambda, \forall n\in\bbN),\\
	\label{eq: E4}&(\varphi-1)(\eta_{\lambda,n})=\wt\nabla(h_{\lambda,n})
	&&(\forall\lambda\in\Lambda,\forall n\in\bbN),\\
	\label{eq: E5}&\mathrm{pr}_{1,2}^\ast f_{(\lambda_1,\lambda_2),n}-\mathrm{pr}_{0,2}^\ast f_{(\lambda_0,\lambda_2),n}+\mathrm{pr}_{0,1}^\ast f_{(\lambda_0,\lambda_1),n}=0
	&&(\forall(\lambda_0,\lambda_1,\lambda_2)\in\Lambda^3,\forall n\in\bbN),\\
	\label{eq: E6}&f_{(\lambda_0,\lambda_1),n}-f_{(\lambda_0,\lambda_1),n+1}|_{\cZ_{(\lambda_0,\lambda_1),n,\bbQ}}=\mathrm{pr}_1^\ast g_{\lambda_1,n}-\mathrm{pr}_0^\ast g_{\lambda_0,n}
	&&(\forall(\lambda_0,\lambda_1)\in\Lambda^2, \forall n\in\bbN),\\
	\label{eq: E7}&(\varphi-1)(f_{(\lambda_0,\lambda_1),n})=\mathrm{pr}_1^\ast h_{\lambda_1,n}-\mathrm{pr}_0^\ast h_{\lambda_0,n}
	&&(\forall(\lambda_0,\lambda_1)\in\Lambda^2, \forall n\in\bbN),\\
	\label{eq: E8}&(\varphi-1)(g_{\lambda,n})=h_{\lambda,n}-h_{\lambda,n+1}|_{\cZ_{\lambda,n,\bbQ}}
	&&(\forall\lambda\in\Lambda, \forall n\in\bbN),
	\end{align}
	where $\mathrm{pr}_l\colon \cZ_{(\lambda_0,\lambda_1),n,\bbQ}\rightarrow \cZ_{\lambda_l,n,\bbQ}$ for $l=0,1$ and $\mathrm{pr}_{k,l}\colon \cZ_{(\lambda_0,\lambda_1,\lambda_2),n,\bbQ}\rightarrow\cZ_{(\lambda_k,\lambda_l),n,\bbQ}$ for $0\leq k<l\leq 2$ denote the natural projections.
	
	For $\lambda\in\Lambda$ and $n\in\bbN$, let $\cF_{\lambda,n}:=\sE_{\cZ_{\lambda,n}}\oplus\sE'_{\cZ_{\lambda,n}}$ and define a log connection $\nabla_{\lambda,n}$ and a Frobenius structure $\Psi_{\lambda,n}$ on $\cF_{\lambda,n}$ by
	\begin{align*}
	&\nabla_{\lambda,n}(\alpha,\beta):=(\nabla(\alpha)+\eta_{\lambda,n}(\beta),\nabla'(\beta)),\\
	&\Psi_{\lambda,n}(\alpha,\beta):=(\Phi(\alpha)+h_{\lambda,n}\circ\Phi'(\beta),\Phi'(\beta))
	\end{align*}
	where $\nabla$ and $\nabla'$ denote the log connections of $\sE_{\cZ_{\lambda,n}}$ and $\sE'_{\cZ_{\lambda,n}}$, respectively.
	Then \eqref{eq: E1} shows that $\nabla_{\lambda,n}$ is integrable and \eqref{eq: E4} shows that $\Psi_{\lambda,n}$ is compatible with $\nabla_{\lambda,n}$.
	We define an isomorphism
	\begin{equation}\label{eq: Ext Weierstrass}\varsigma_{\lambda,n}\colon \cF_{\lambda,n}=\sE_{\cZ_{\lambda,n}}\oplus\sE'_{\cZ_{\lambda,n}}\xrightarrow{\cong}\sE_{\cZ_{\lambda,n}}\oplus\sE'_{\cZ_{\lambda,n}}=\cF_{\lambda,n+1}|_{\cZ_{\lambda,n,\bbQ}}\end{equation}
	by $\varsigma_{\lambda,n}(\alpha,\beta):=(\alpha+g_{\lambda,n}(\beta),\beta)$.
	This is compatible with the log connections and Frobenius structures by \eqref{eq: E3} and \eqref{eq: E8}, respctively.
	Consequently $(\cF_{\lambda,n},\nabla_{\lambda,n},\Psi_{\lambda,n})$ for all $n$ glue with each other and define an extension of $\sE'_{\cZ_\lambda}$ by $\sE_{\cZ_\lambda}$ in $F\MIC^\dagger((\cZ_\lambda,\phi_\lambda)/(W^\varnothing,\sigma))$, which we denote by $(\cF_\lambda,\nabla_\lambda,\Psi_\lambda)$.
	
	For $(\lambda_0,\lambda_1)\in\Lambda^2$, we define an isomorphism
	\[\rho_{\lambda_0,\lambda_1,n}\colon \mathrm{pr}_1^\ast\cF_{\lambda_1,n}\cong\sE_{\cZ_{(\lambda_0,\lambda_1),n}}\oplus\sE'_{\cZ_{(\lambda_0,\lambda_1),n}}\xrightarrow{\cong}\sE_{\cZ_{(\lambda_0,\lambda_1),n}}\oplus\sE'_{\cZ_{(\lambda_0,\lambda_1),n}}\cong\mathrm{pr}_0^\ast\cF_{\lambda_0,n}\]
	by $\rho_{\lambda_0,\lambda_1,n}(\alpha,\beta):=(\alpha+f_{(\lambda_0,\lambda_1),n}(\beta),\beta)$.
	This is compatible with the log connections, Frobenius structures and the isomorphisms \eqref{eq: Ext Weierstrass} by \eqref{eq: E2}, \eqref{eq: E7}, and \eqref{eq: E6}, respectively.
	Moreover we have
	\[\mathrm{pr}_{0,2}^\ast\rho_{\lambda_0,\lambda_2,n}=\mathrm{pr}_{0,1}^\ast\rho_{\lambda_0,\lambda_1,n}\circ\mathrm{pr}_{1,2}^\ast\rho_{\lambda_1,\lambda_2,n}\]
	by \eqref{eq: E5}.
	Varying $n$, we obtain isomorphisms
	\[\rho_{\lambda_0,\lambda_1}\colon\mathrm{pr}_1^\ast(\cF_{\lambda_1},\nabla_{\lambda_1},\Psi_{\lambda_1})\xrightarrow{\cong}\mathrm{pr}_0^\ast(\cF_{\lambda_0},\nabla_{\lambda_0},\Psi_{\lambda_0})\]
	for all $(\lambda_0,\lambda_1)\in\Lambda^2$, satisfying
	\[\mathrm{pr}_{0,2}^\ast\rho_{\lambda_0,\lambda_2}=\mathrm{pr}_{0,1}^\ast\rho_{\lambda_0,\lambda_1}\circ\mathrm{pr}_{1,2}^\ast\rho_{\lambda_1,\lambda_2}\]
	for any $(\lambda_0,\lambda_1,\lambda_2)\in\Lambda^3$.
	
	Thus the family $\{(\cF_\lambda,\nabla_\lambda,\Phi_\lambda),\rho_{\lambda_0,\lambda_1}\}$ defines an object of $F\MIC^\dagger((\cZ_\lambda,\phi_\lambda)_{\lambda\in\Lambda}/(W^\varnothing,\sigma))$ (see Definition \ref{def: MIC descent}), and hence an object $(\sF,\Psi)$ of $F\Isoc^\dagger(Y/W^\varnothing)$ via the equivalence of Corollary \ref{cor: descent realization}.
	By construction $(\sF,\Phi)$ is an extension of $(\sE',\Phi')$ by $(\sE,\Phi)$.
	This defines a map
	\[H^1_\abs(Y,\sheafhom((\sE',\Phi'),(\sE,\Phi)))\rightarrow \Ext^1_{F\Isoc^\dagger(Y/W^\varnothing)}((\sE',\Phi'),(\sE,\Phi)).\]
	One can follow this construction backwards and define the inverse map.
	Thus we obtain the first isomorphism of \eqref{eq: Ext and coh}.
\end{proof}

According to Proposition \ref{prop: Ext abs HK} and the philosophy of six operations, the following proposition would be thought of as the adjunction between the inverse image and the direct image.

\begin{proposition}\label{prop: Leray}
	Let $Y$ be a strictly semistable log scheme over $k^0$.
	\begin{enumerate}
	\item Let $\sE\in \Isoc^\dagger(Y/W^\varnothing)^\unip$.
		There exists a quasi-isomorphism
		\begin{equation}\label{eq: RHom 1}
		R\Hom_{\Mod(N)}(L,R\Gamma_\HK(Y,\sE))\cong R\Gamma_\rig(Y/W^\varnothing,\sE)
		\end{equation}
	and a spectral sequence
		\begin{equation}\label{eq: Leray 1}
		E_2^{i,j}=\Ext^i_{\Mod_L(N)}(L,H^j_\HK(Y,\sE))\Rightarrow H^{i+j}_\rig(Y/W^\varnothing,\sE).
		\end{equation}
	\item Let $(\sE,\Phi)\in F\Isoc^\dagger(Y/W^\varnothing)^\unip$.
		There exists a quasi-isomorphism
		\begin{equation}\label{eq: RHom}
		R\Hom_{\Mod(\varphi,N)}(L,R\Gamma_\HK(Y,(\sE,\Phi)))\cong R\Gamma_\abs(Y,(\sE,\Phi))
		\end{equation}
	and a spectral sequence
		\begin{equation}\label{eq: Leray}
		E_2^{i,j}=\Ext^i_{\Mod_L(\varphi,N)}(L,H^j_\HK(Y,(\sE,\Phi)))\Rightarrow H_\abs^{i+j}(Y,(\sE,\Phi)).
		\end{equation}
	\end{enumerate}
\end{proposition}

\begin{proof}
	In the proof of \cite[Lemma 2.5]{DN}, it has proved that $R\Hom$ in $ D^b(\Mod_L(\varphi,N))$ can be computed by a homotopy limit like the right hand side of \eqref{eq: abs HK holim}.
	Thus we have \eqref{eq: RHom}, and this implies \eqref{eq: Leray}.
	One can see \eqref{eq: RHom 1} and \eqref{eq: Leray 1} in a similar way.
\end{proof}

\begin{remark}\label{rem: dg eq}
	In fact, the full subcategory of $ D^b(\Mod_L(\varphi,N))$ consisting of objects whose cohomology groups lie in $\Mod^\fin_L(\varphi,N)$ is canonically equivalent to $ D^b(\Mod^\fin_L(\varphi,N))$ \cite[Proposition 3.11]{EY2}.
	Thus, if $Y$ is quasi-compact, then $R\Gamma_\HK(Y,\sE)$ and $R\Gamma_\HK(Y,(\sE,\Phi))$ for unipotent coefficients are regarded as objects of $D^b(\Mod^\fin_L(N))$ and $D^b(\Mod^\fin_L(\varphi,N))$, respectively.
\end{remark}

\begin{corollary}\label{cor: unip isoc extend}
	Let $Y$ be a strictly semistable log scheme over $k^0$, $D$ the horizontal divisor of $Y$, and $U:=Y\setminus D\subset Y$.
	Then the canonical functors
		\begin{eqnarray}
		\label{eq: compactification}&&\Isoc^\dagger(Y/W^\varnothing)^\unip\rightarrow \Isoc^\dagger(U/W^\varnothing)^\unip,\\
		\nonumber&&F\Isoc^\dagger(Y/W^\varnothing)^\unip\rightarrow F\Isoc^\dagger(U/W^\varnothing)^\unip
		\end{eqnarray}
	are equivalences of categories.
\end{corollary}

\begin{proof}
	We prove the statement for log overconvergent $F$-isocrystals, and another follows in the same way.
	We first note that the homomorphism $H^n_\HK(Y)\rightarrow H^n_\HK(U)$ is an isomorphism for any $n\geq 0$ by \cite[Proposition 3.32]{EY}.
	Hence by using the spectral sequence \eqref{eq: ss for Fil} we see that
		\[H^n_\HK(Y,\sheafhom((\sE',\Phi'),(\sE,\Phi)))\rightarrow H^n_\HK(U,\sheafhom((\sE',\Phi'),(\sE,\Phi)))\]
	is also an isomorphism for any $(\sE,\Phi),(\sE',\Phi')\in F\Isoc^\dagger(Y/W^\varnothing)^\unip$.
	Furthermore by Proposition \ref{prop: Leray} the homomorphism
		\begin{equation}\label{eq: abs hk isom}
		H_\abs^n(Y,\sheafhom((\sE',\Phi'),(\sE,\Phi)))\rightarrow H_\abs^n(U,\sheafhom((\sE',\Phi'),(\sE,\Phi)))
		\end{equation}
	is an isomorphism.
	In the case $n=0$, this shows that \eqref{eq: compactification} is fully faithful by Proposition \ref{prop: Ext abs HK}.
	
	Next we show that any object $(\sE,\Phi)\in F\Isoc^\dagger(U/W^\varnothing)^\unip$ uniquely extends to an object of $F\Isoc^\dagger(Y/W^\varnothing)^\unip$ by the induction on the rank of $\sE$.
	The case of rank one is clear, since $\sE$ is constant.
	In general, by Proposition \ref{prop: unip const} there exists a short exact sequence
		\[0\rightarrow(\sF,\Psi)\rightarrow(\sE,\Phi)\rightarrow(\sF',\Psi')\rightarrow 0,\]
	where $(\sF,\Psi)$ is unipotent and $(\sF',\Psi')$ is non-zero and constant.
	Then $(\sF',\Psi')$ and by the induction hypothesis $(\sF,\Psi)$ uniquely extend to object of  $F\Isoc^\dagger(Y/W^\varnothing)^\unip$, which we denote again by $(\sF',\Psi')$ and $(\sF,\Psi)$, respectively.
	It suffices to prove that the canonical map
		\[\Ext^1_{F\Isoc^\dagger(Y/W^\varnothing)}((\sF',\Psi'),(\sF,\Psi))\rightarrow\Ext^1_{F\Isoc^\dagger(U/W^\varnothing)}((\sF',\Psi'),(\sF,\Psi))\]
	is an isomorphism.
	This holds again by Proposition \ref{prop: Ext abs HK} and the isomorphism \eqref{eq: abs hk isom} above.
\end{proof}

We note that the log connections associated to pseudo-constant log overconvergent ($F$-)isocrystals are nontrivial in general.
However we may prove that the cohomology with pseudo-constant coefficients is written as a tensor product as well as that with constant coefficients.

\begin{proposition}\label{prop: pseudo-const}
	Let $Y$ be a stritly semistable log scheme over $k^0$.
	For any $M\in\Mod^\fin_L(\varphi,N)$ and $i\geq 0$, we have a canonical isomorphism in $\Mod^\fin_L(\varphi,N)$
		\[H^i_\HK(Y,M^a)\cong M\otimes H^i_\HK(Y).\]
\end{proposition}

\begin{proof}
	 We first note that there exist canonical isomorphisms
		\begin{equation}\label{eq: HK coh of M}
		H^i_\HK(Y,M^a)\xrightarrow{\cong} H^i_\rig(Y/W^0,M^a)\cong M\otimes H^i_\rig(Y/W^0)\xleftarrow{\cong} M\otimes H^i_\HK(Y),
		\end{equation}
	which are compatible with the Frobenius operators.
	Here the first and the third isomorphisms are given by Theorem \ref{thm: HK rig comparison}, and the second isomorphism follows from the fact that the restriction of $M^a$ to $F\Isoc^\dagger(Y/W^0)$ is constant.
	Thus it remains to prove the composition of \eqref{eq: HK coh of M} is compatible with the monodromy operators.
	
	We take a local embedding datum $(Z_\lambda,\cZ_\lambda,i_\lambda,h_\lambda,\theta_\lambda)_{\lambda\in\Lambda}$ indexed by a finite set $\Lambda$ such that $Z_\lambda$ and $\cZ_\lambda$ are affine.
	Define the weak formal log schemes $\cZ_{m,n}$ as in Proposition \ref{prop: HK coh total} and let $\cY_{m,n}:=\cZ_{m,n}\times_\cS W^0$.
	Then by Remark \ref{rem: affinoid}, Proposition \ref{prop: realization pseudo-const} and Proposition \ref{prop: HK coh total}, we have
	\begin{align}
	\label{eq: rig Tot 2}&R\Gamma_\rig(Y/W^0,M^a)\cong\Tot\left[\prod_{n\in\bbN}\Gamma(\cY_{\bullet,n,\bbQ},M\otimes\omega^\star_{\cY_{\bullet,n}/W^0,\bbQ})\xrightarrow{\eth}\prod_{n\in\bbN}\Gamma(\cY_{\bullet,n,\bbQ},M\otimes\omega^\star_{\cY_{\bullet,n}/W^0,\bbQ})\right],\\
	\label{eq: HK Tot 1}&R\Gamma_\HK(Y)\cong\Tot\left[\prod_{n\in\bbN}\Gamma(\cZ_{\bullet,n,\bbQ},\omega^\star_{\cZ_{\bullet,n}/W^\varnothing,\bbQ}[u])\xrightarrow{\eth}\prod_{n\in\bbN}\Gamma(\cZ_{\bullet,n,\bbQ},\omega^\star_{\cZ_{\bullet,n}/W^\varnothing,\bbQ}[u])\right],\\
	\label{eq: HK Tot 2}&R\Gamma_\HK(Y,M^a)\cong\Tot\left[\prod_{n\in\bbN}\Gamma(\cZ_{\bullet,n,\bbQ},M\otimes\omega^\star_{\cZ_{\bullet,n}/W^\varnothing,\bbQ}[u])\xrightarrow{\eth}\prod_{n\in\bbN}\Gamma(\cZ_{\bullet,n,\bbQ},M\otimes\omega^\star_{\cZ_{\bullet,n}/W^\varnothing,\bbQ}[u])\right],
	\end{align}
	where the log connection on $M\otimes\omega^\star_{\cZ_{\bullet,n}/W^\varnothing}[u]$ is given by $\nabla(m\otimes f)=m\otimes df+N(m)\otimes fd\log s$.
	
	We consider an element of $M\otimes H^i_\HK(Y)$ of the form $\mu\otimes \alpha$ with $\mu\in M$ and $\alpha\in H^i_\HK(Y)$.
	The cohomology class $\alpha$ is represented via \eqref{eq: HK Tot 1} by a family
	\[\Bigl(\sum_{j\geq 0}\alpha_{j,m,n}u^{[j]},\sum_{j\geq 0}\beta_{j,m,n}u^{[j]}\Bigr)_{m,n}\]
	with
	\begin{align*}
	\alpha_{j,m,n}\in\Gamma(\cZ_{m,n,\bbQ},\omega^{i-m}_{\cZ_{m,n}/W^\varnothing,\bbQ}),&&
	\beta_{j,m,n}\in\Gamma(\cZ_{m,n,\bbQ},\omega^{i-m-1}_{\cZ_{m,n}/W^\varnothing,\bbQ}),
	\end{align*}
	such that
	\begin{align*}
	&d(\sum_{j\geq 0}\alpha_{j,m,n}u^{[j]})+(-1)^{i-m-1}\partial(\sum_{j\geq 0}\alpha_{j,m-1,n}u^{[j]})=0,\\
	&d(\sum_{j\geq 0}\beta_{j,m,n}u^{[j]})+(-1)^{i-m}\partial(\sum_{j\geq 0}\beta_{j,m-1,n}u^{[j]})+(-1)^i(\sum_{j\geq 0}\alpha_{j,m,n}u^{[j]}-\sum_{j\geq 0}\alpha_{j,m,n+1}u^{[j]}|_{\cZ_{m,n,\bbQ}})=0
	\end{align*}
	for all $m$ and $n$, where $\partial$ denotes the second differential map of the right hand side of \eqref{eq: HK Tot 1} induced by the face maps of the weak formal log scheme $\cZ_{\bullet,n}$.
	These two conditions are equivalent to the equations
	\begin{align}
	\label{eq: cocycle alpha}&d\alpha_{j,m,n}-d\log s\wedge\alpha_{j+1,m,n}+(-1)^{i-m-1}\partial\alpha_{j,m-1,n}=0,\\
	\label{eq: cocycle beta}&d\beta_{j,m,n}-d\log s\wedge\beta_{j+1,m,n}+(-1)^{i-m}\partial\beta_{j,m-1,n}+(-1)^i\alpha_{j,m,n}+(-1)^{i+1}\alpha_{j,m,n+1}|_{\cZ_{m,n,\bbQ}}=0
	\end{align}
	for all $j$, $m$, and $n$.
	
	Let
	\begin{align*}
	\gamma_{j,m,n}&:=\sum_{\ell=0}^j\binom{j}{\ell}N^\ell(\mu)\otimes\alpha_{j-\ell,m,n}\in \Gamma(\cZ_{m,n,\bbQ},M\otimes\omega^{i-m}_{\cZ_{m,n}/W^\varnothing,\bbQ}),\\
	\delta_{j,m,n}&:=\sum_{\ell=0}^j\binom{j}{\ell}N^\ell(\mu)\otimes\beta_{j-\ell,m,n}\in \Gamma(\cZ_{m,n,\bbQ},M\otimes\omega^{i-m-1}_{\cZ_{m,n}/W^\varnothing,\bbQ}).
	\end{align*}
	Then we have
	\begin{align}
	\label{eq: cocycle gamma}&\nabla\gamma_{j,m,n}-d\log s\wedge\gamma_{j+1,m,n}+(-1)^{i-m-1}\partial\gamma_{j,m-1,n}\\
	\nonumber&=\sum_{\ell=0}^j\binom{j}{\ell}N^\ell(\mu)\otimes d\alpha_{j-\ell,m,n}
	+\sum_{\ell=0}^j\binom{j}{\ell}N^{\ell+1}(\mu)\otimes d\log s\wedge\alpha_{j-\ell,m,n}\\
	\nonumber&\hspace{15pt}-\sum_{\ell=0}^{j+1}\binom{j+1}{\ell}N^\ell(\mu)\otimes d\log s\wedge\alpha_{j-\ell+1,m,n}
	+(-1)^{i-m-1}\sum_{\ell=0}^j\binom j\ell N^\ell(\mu)\otimes\partial\alpha_{j-\ell,m-1,n}\\
	\nonumber&=\sum_{\ell=0}^j\binom{j}{\ell}N^{\ell+1}(\mu)\otimes d\log s\wedge\alpha_{j-\ell,m,n}-\sum_{\ell=0}^{j+1}\binom{j+1}{\ell}N^\ell(\mu)\otimes d\log s\wedge\alpha_{j-\ell+1,m,n}\\
	\nonumber&\hspace{15pt}+\sum_{\ell=0}^j\binom j\ell N^\ell(\mu)\otimes d\log s\wedge\alpha_{j-\ell+1,m,n}\\
	\nonumber&=\sum_{\ell=1}^j\left(\binom j{\ell-1}-\binom{j+1}\ell+\binom j\ell\right)N^\ell(\mu)\otimes d\log s\wedge\beta_{j-\ell,m,n}\\
	\nonumber&=0,
	\end{align}
	where the second equality follows from \eqref{eq: cocycle alpha}.
	Noting that
	\[(-1)^i\gamma_{j,m,n}+(-1)^{i+1}\gamma_{j,m,n+1}|_{\cZ_{m,n,\bbQ}}=\sum_{\ell=0}^j\binom j\ell N^\ell(\mu)\otimes((-1)^i\alpha_{j-\ell,m,n}+(-1)^{i+1}\alpha_{j-\ell,m,n+1}|_{\cZ_{m,n,\bbQ}}),\]
	a similar calculation with the use of \eqref{eq: cocycle beta} instead of \eqref{eq: cocycle alpha} shows that
	\begin{align}
	\label{eq: cocycle delta}
	&\nabla\delta_{j,m,n}-d\log s\wedge\delta_{j+1,m,n}+(-1)^{i-m}\partial\delta_{j,m-1,n}+(-1)^i\gamma_{j,m,n}+(-1)^{i+1}\gamma_{j,m,n+1}|_{\cZ_{m,n,\bbQ}}=0.
	\end{align}
	Now \eqref{eq: cocycle gamma} and \eqref{eq: cocycle delta} together indicate that the family
	\[\Bigl(\sum_{j\geq 0}\gamma_{j,m,n}u^{[j]},\sum_{j\geq 0}\delta_{j,m,n}u^{[j]}\Bigr)_{m,n}\]
	defines an $i$-cocycle of $R\Gamma_\HK(Y,M^a)$ via the description \eqref{eq: HK Tot 2}.
	We denote by $\gamma\in H^i_\HK(Y,M^a)$ its cohomology class.
	
	Then the isomorphisms of \eqref{eq: HK coh of M} map $\gamma\in H^i_\HK(Y,M^a)$ and $\mu\otimes\alpha\in M\otimes H^i_\HK(Y)$ to the same element of $H^i_\rig(Y/W^0,M^a)=M\otimes H^i_\rig(Y/W^0)$ represented via \eqref{eq: rig Tot 2} by the family $(\mu\otimes\overline{\alpha}_{0,m,n},\mu\otimes\overline{\beta}_{0,m,n})_{m,n}=\mu\otimes(\overline{\alpha}_{0,m,n},\overline{\beta}_{0,m,n})_{m,n}$, where the overline indicates the image in $\omega^{i-m}_{\cY_{m,n}/W^0,\bbQ}$ of a log differential form of $\omega^{i-m}_{\cZ_{m,n}/W^\varnothing,\bbQ}$.
	Thus $\mu\otimes\alpha$ and $\gamma$ correspond to each other by \eqref{eq: HK coh of M}.
	
	Since the monodromy operator of the Hyodo--Kato cohomology is defined by $u^{[j]}\mapsto u^{[j-1]}$, $N(\gamma)$ is represented by the family
	\[\Bigl(\sum_{j\geq 0}\gamma_{j+1,m,n}u^{[j]},\sum_{j\geq 0}\delta_{j+1,m,n}u^{[j]}\Bigr)_{m,n},\]
	which is mapped to the $i$-cocycle
	\begin{align}\label{eq: monodromy computation}
	(\overline{\gamma}_{1,m,n},\overline{\delta}_{1,m,n})_{m,n}&=(\mu\otimes\overline{\alpha}_{1,m,n}+N(\mu)\otimes\overline{\alpha}_{0,m,n},\mu\otimes\overline{\beta}_{1,m,n}+N(\mu)\otimes\overline{\beta}_{0,m,n})_{m,n}\\
	\nonumber&=\mu\otimes(\overline{\alpha}_{1,m,n},\overline{\beta}_{1,m,n})_{m,n}+N(\mu)\otimes(\overline{\alpha}_{0,m,n},\overline{\beta}_{0,m,n})_{m,n}
	\end{align}
	of $R\Gamma_\rig(Y/W^0,M^a)=M\otimes R\Gamma_\rig(Y/W^0)$.
	
	Similarly the images of $\alpha$ and $N(\alpha)$ in $H^i_\rig(Y/W^0)$ are represented by the cocycles $(\overline{\alpha}_{0,m,n},\overline{\beta}_{0,m,n})_{m,n}$ and $(\overline{\alpha}_{1,m,n},\overline{\beta}_{1,m,n})_{m,n}$, respectively, hence \eqref{eq: monodromy computation} shows that $N(\gamma)\in H^i_\HK(Y,M^a)$ and $N(\mu\otimes\alpha)=\mu\otimes N(\alpha)+N(\mu)\otimes\alpha\in M\otimes H^i_\HK(Y)$ correspond to the same element of $H^i_\rig(Y/W^0,M^a)=M\otimes H^i_\rig(Y/W^0)$.
	This finishes the proof of the compatibility of \eqref{eq: HK coh of M} with the monodromy operators.
\end{proof}

\begin{proposition}\label{prop: isoc on point}
	The functor
		\begin{align*}
		(\cdot)^a\colon \Mod^\fin_L(N)\rightarrow \Isoc^\dagger(k^0/W^\varnothing)^\unip&&(\text{resp.\  }(\cdot)^a\colon \Mod^\fin_L(\varphi,N)\rightarrow F\Isoc^\dagger(k^0/W^\varnothing)^\unip)
		\end{align*}
	is an equivalence of categories with a quasi-inverse $H^0_\HK(k^0,\cdot)$.\end{proposition}

\begin{proof}
	We prove the statement for unipotent log overconvergent $F$-isocrystals.
	We have $H^0_\HK(k^0,M^a)=M$ by Proposition \ref{prop: pseudo-const}, hence $(\cdot)^a$ is fully faithful.
	We will show that any object $(\sE,\Phi)\in F\Isoc^\dagger(k^0/W^\varnothing)^\unip$ lies in the essential image of $(\cdot)^a$ by the induction on the rank of $\sE$.
	The case of rank one is clear since $(\sE,\Phi)$ must be constant.
	If the rank is positive, then there exists a short exact sequence
		\[0\rightarrow(\sF,\Psi)\rightarrow(\sE,\Phi)\rightarrow(\sF',\Psi')\rightarrow 0,\]
	where $(\sF,\Psi)$ is unipotent and $(\sF',\Psi')$ is non-zero constant.
	Then by the induction hypothesis we have $(\sF,\Psi)=M^a$ and $(\sF',\Psi')=M'^a$ for some $M,M'\in\Mod^\fin_L(\varphi,N)$.
	It suffices to show that the map
		\begin{equation}\label{eq: map of Ext}
		\Ext^1_{\Mod^\fin_L(\varphi,N)}(M',M)\rightarrow\Ext^1_{F\Isoc^\dagger(k^0/W^\varnothing)}(M'^a,M^a)
		\end{equation}
	induced by $(\cdot)^a$ is surjective.
	Since we have $H^1_\HK(k^0)=0$ by direct computation, we have
		\[H^1_\HK(k^0,\Hom(M',M)^a)=H^1_\HK(k^0)\otimes\Hom(M',M)=0\]
	by Proposition \ref{prop: pseudo-const}.
	Therefore by Proposition \ref{prop: Leray}, Proposition \ref{prop: pseudo-const tensor}, Proposition \ref{prop: Ext abs HK}, and again Proposition \ref{prop: pseudo-const}, we obtain
		\begin{eqnarray*}
		\Ext^1_{F\Isoc^\dagger(k^0/W^\varnothing)}(M'^a,M^a)&\cong& H_\abs^1(k^0,\sheafhom(M'^a,M^a))=H_\abs^1(k^0,\Hom(M',M)^a)\\
		&\cong&\Ext^1_{\Mod^\fin_L(\varphi,N)}(L,H^0_\HK(k^0,\Hom(M',M)^a))\\
		&=&\Ext^1_{\Mod^\fin_L(\varphi,N)}(L,\Hom(M',M))\\
		&\cong&\Ext^1_{\Mod^\fin_L(\varphi,N)}(M',M).
		\end{eqnarray*}
	Thus the map \eqref{eq: map of Ext} is an isomorphism as desired.
\end{proof}

%%%%%%%%%%%%%%%%%
\section{Tannakian fundamental group}\label{sec: tannakian}
%%%%%%%%%%%%%%%%%

In this section, we will introduce the tannakian fundamental group of unipotent log overconvergent ($F$-)isocrystals, and see its rigidity property.
We recall that, for a strictly semistable log scheme $Y$ over $k^0$ with the horizontal divisor $D$, we denote by $Y^\sm$ the smooth locus of $Y\setminus D$.
Let
	\[P(Y):=\coprod_{k'/k}Y^\sm(k')\]
where $k'$ runs through all finite extensions of $k$.
For $y=\Spec k_y\rightarrow Y^\sm$ in $P(Y)$ we denote by $W_y$ the ring of Witt vectors of $k_y$ and $L_y$ the fraction field of $W_y$.
We endow $y$ with the pull-back log structure of $Y$.
Then $y$ is isomorphic to $k_y^0:=k^0\times_{k^\varnothing}k_y^\varnothing$.
For an object $\sE\in \Isoc^\dagger(Y/W^\varnothing)^\unip$ (resp.\  $(\sE,\Phi)\in F\Isoc^\dagger(Y/W^\varnothing)^\unip$), we denote by $\sE_y\in\Isoc^\dagger(y/W_y^\varnothing)^\unip$ (resp.\  $(\sE,\Phi)_y=(\sE_y,\Phi_y)\in F\Isoc^\dagger(y/W_y^\varnothing)^\unip$) the pull-back by the commutative diagram
	\[\xymatrix{
	y\ar[r]\ar[d] &k_y^\varnothing\ar[r]\ar[d] & W_y^\varnothing\ar[d]\\
	Y\ar[r] & k^\varnothing\ar[r] & W^\varnothing.
	}\]
We denote by
	\begin{equation}\label{eq: fiber functor}
	\eta_y\colon\Isoc^\dagger(Y/W^\varnothing)^\unip\rightarrow\Isoc^\dagger(y/W_y^\varnothing)^\unip\xrightarrow{\cong}\Mod^\fin_{L_y}(N)\rightarrow\Mod^\fin_{L_y}
	\end{equation}
the composition of the pull-back, $H^0_\HK(y,-)$, and the forgetful functor.
Note that $H^0_\HK(y,-)$ is an equivalence of categories by Proposition \ref{prop: isoc on point}.
We also denote by
	\begin{equation}\label{eq: F fiber functor}
	\eta_y\colon F\Isoc^\dagger(Y/W^\varnothing)^\unip\rightarrow\Mod^\fin_{L_y}
	\end{equation}
the functor given in the same way as \eqref{eq: fiber functor} with forgetting Frobenius structure.

\begin{proposition}\label{prop: tannakian}
	Let $Y$ be a geometrically connected strictly semistable log scheme over $k^0$ and $y\in P(Y)$.
	Then the functors \eqref{eq: fiber functor} and \eqref{eq: F fiber functor} are faithful exact tensor functors.
	
	Moreover, the category $\Isoc^\dagger(Y/W^\varnothing)^\unip$ (resp.\  $F\Isoc^\dagger(Y/W^\varnothing)^\unip$) is a tannakian category over $L$ (resp.\  $\bbQ_p$) with a fiber functor $\eta_y$.
	If $y\in Y^\sm(k)$, then $\Isoc^\dagger(Y/W^\varnothing)^\unip$ is a neutral tannakian category over $L$.
\end{proposition}

\begin{proof}
	The exactness and the compatibility with the tensor structures are clear from definition of functors.
	Let $\sE,\sE'\in\Isoc^\dagger(Y/W^\varnothing)^\unip$ and $\sF:=\sheafhom(\sE',\sE)$.
	If we put $M:=H^0_\rig(Y/W^\varnothing,\sF)$, then we have $\Fil_0\sF=M^a$.
	By Proposition \ref{prop: Ext abs HK} we have a commutative diagram
		\[\xymatrix{
		\Hom_{\Isoc^\dagger(Y/W^\varnothing)^\unip}(\sE',\sE)\ar[r]^-\cong\ar[d] & H^0_\rig(Y/W^\varnothing,\sF)\ar[d] & H^0_\rig(Y/W^\varnothing,\Fil_0\sF)\ar[l]_-\cong\ar[d] & M\ar@{=}[l]\ar[d]\\
		\Hom_{\Isoc^\dagger(y/W_y^\varnothing)^\unip}(\sE'_y,\sE_y)\ar[r]^-\cong & H^0_\rig(y/W_y^\varnothing,\sF_y) & H^0_\rig(y/W_y^\varnothing,(\Fil_0\sF)_y)\ar[l] & M\otimes_LL_y\ar@{=}[l].
		}\]
	The middle horizontal map in the bottom is injective by Proposition \ref{prop: isoc on point}, hence the left vertical map is also injective.
	This implies that \eqref{eq: fiber functor} is faithful.
	
	Let $(\sE,\Phi),(\sE',\Phi')\in F\Isoc^\dagger(Y/W^\varnothing)^\unip$ and $\sF:=\sheafhom((\sE',\Phi'),(\sE,\Phi))$.
	Then
	\[\Hom_{F\Isoc^\dagger(Y/W^\varnothing)^\unip}((\sE',\Phi'),(\sE,\Phi))=H^0_\abs(Y,\cF)=H^0_\rig(Y/W^\varnothing,\cF)^{\varphi=1}\]
	is clearly a subspace of $H^0_\rig(Y/W^\varnothing,\cF)$.
	Thus \eqref{eq: fiber functor} is also faithful.
	
	By Lemma \ref{lem: H0} we have
		\begin{eqnarray*}
		&&\End_{\Isoc^\dagger(Y/W^\varnothing)}(\sO_{Y/W^\varnothing})\cong H^0_\rig(Y/W^\varnothing)=L,\\
		&&\End_{F\Isoc^\dagger(Y/W^\varnothing)}(\sO_{Y/W^\varnothing})\cong H_\abs^0(Y/W^\varnothing)=\bbQ_p,
		\end{eqnarray*}
	hence we see that $\Isoc^\dagger(Y/W^\varnothing)^\unip$ (resp.\  $F\Isoc^\dagger(Y/W^\varnothing)^\unip$) is a tannakian category over $L$ (resp.\  $\bbQ_p$).
\end{proof}

As a corollary, we obtain the rigidity of Frobenius structures.

\begin{corollary}
	Let $Y$ be a geometrically connected strictly semistable log scheme over $k^0$ and $\sE\in \Isoc^\dagger(Y/W^\varnothing)^\unip$.
	Let $\Phi$ and $\Phi'$ be Frobenius structures on $\sE$.
	If there exists $y\in P(Y)$ such that $\Phi_y=\Phi'_y$, then we have $\Phi=\Phi'$.
\end{corollary}

\begin{proof}
	Since $\Phi'\circ\Phi^{-1}$ is a morphism in $\Isoc^\dagger(Y/W^\varnothing)^\unip$ and $\eta_y(\Phi'\circ\Phi^{-1})=\id$, the assertion follows from Proposition \ref{prop: tannakian} .
\end{proof}

\begin{definition}
	Let $Y$ be a connected strictly semistable log scheme over $k^0$ and assume that there exists a $k$-rational point $y\in Y^\sm(k)$.
	We denote by $\pi_1^\unip(Y/W^\varnothing,y)$ the tannakian fundamental group of the neutral tannakian category $\Isoc^\dagger(Y/W^\varnothing)^\unip$ with respect to the fiber functor $\eta_y$.
	This is an affine group scheme over $L$.
\end{definition}

We first see the Frobenius $\sigma$ on $W^\varnothing$ induces an automorphism on $\pi^\unip(Y/W^\varnothing,y)$.

\begin{lemma}\label{lem: Frob isom}
	Let $Y$ be a strictly semistable log scheme over $k^0$ and $\sE\in\Isoc^\dagger(Y/W^\varnothing)^\unip$.
	Then the natural maps
		\begin{align*}
		H^n_\HK(Y,\sE)\rightarrow H^n_\HK(Y,\sigma^*\sE) && \text{and} && H^n_\rig(Y/W^\varnothing,\sE)\rightarrow H^n_\rig(Y/W^\varnothing,\sigma^*\sE)
		\end{align*}
	are isomorphisms for any $n$.
\end{lemma}

\begin{proof}
	By using the induction on the rank of $\sE$, it suffices to show the assertion for $\sE=\sO_{Y/W^\varnothing}$.
	In this case, the required bijectivity follows from Theorem \ref{thm: HK rig comparison}.
\end{proof}

\begin{lemma}
	Let $Y$ be a strictly semistable log scheme over $k^0$.
	The functor
		\begin{equation}\label{eq: sigma pi}
		\sigma^*\colon \Isoc^\dagger(Y/W^\varnothing)^\unip\rightarrow\Isoc^\dagger(Y/W^\varnothing)^\unip
		\end{equation}
	induced by \eqref{eq: Frob of isoc} is an equivalence of categories.
	Moreover for any $y\in Y^\sm(k)$ there exists a commutative diagram
		\begin{equation}
		\label{eq: sigma on unip}\xymatrix{
		\Isoc^\dagger(Y/W^\varnothing)^\unip \ar[r]^-{\sigma^*}\ar[d]^-{\eta_y} & \Isoc^\dagger(Y/W^\varnothing)^\unip\ar[d]^-{\eta_y}\\
		\Mod_{L}^\fin\ar[r]^-{\sigma^*} & \Mod^\fin_{L},
		}\end{equation}
	where the lower horizontal arrow is defined by $M\mapsto M\otimes_{L,\sigma}L$.
\end{lemma}

\begin{proof}
	By Lemma \ref{lem: Frob isom}, the assertion follows in the same way as the proof of Corollary \ref{cor: unip isoc extend}.
	The commutativity of the diagram follows from construction.
\end{proof}

\begin{proposition}
	Let $Y$ be a connected strictly semistable log scheme over $k^0$ and assume that there exists a $k$-rational point $y\in Y^\sm(k)$.
	Then the diagram \eqref{eq: sigma on unip} induces an automorphism
		\begin{equation}\label{eq: sigma pi1}
		\sigma_*\colon \pi_1^\unip(Y,y)\rightarrow\pi_1^\unip(Y,y)
		\end{equation}
	as a group scheme, which makes the following diagram commutative:
		\begin{equation}\label{eq: diag sigma pi}
		\xymatrix{
		\pi_1^\unip(Y,y)\ar[r]^-{\sigma_*}\ar[d] & \pi_1^\unip(Y,y)\ar[d]\\
		\Spec L\ar[r]^{\sigma} & \Spec L.
		}\end{equation}
	In addition, this induces a $\sigma^{-1}$-semilinear automorphism
		\begin{equation}\label{eq: sigma Lie}
		\sigma_*\colon \Lie\pi_1^\unip(Y,y)\rightarrow\Lie\pi_1^\unip(Y,y).
		\end{equation}
\end{proposition}

\begin{proof}
	The morphism \eqref{eq: sigma pi1} is given by the functoriality of tannakian fundamental groups.
	Let
		\[\pi'_1:=\Spec L\times_{\sigma,\Spec L}\pi_1^\unip(Y,y)\]
	and regard it as a group scheme over $L$ by the first projection.
	Then the diagram \eqref{eq: diag sigma pi} induces an isomorphism $f\colon \pi_1^\unip(Y,y)\rightarrow\pi'_1$ over $L$ and hence an $L$-linear isomorphism
		\[f\colon\Lie\pi_1^\unip(Y,y)\rightarrow\Lie\pi'_1.\]
	Since the base change by $\sigma$ induces an isomorphism $\pi_1^\unip(Y,y)(R)\rightarrow\pi'_1(R\otimes_{L,\sigma}L)$ for any $L$-algebra $R$, we have a $\sigma$-semilinear isomorphsim
		\[g\colon\Lie\pi_1^\unip(Y,y)\rightarrow\Lie\pi'_1.\]
	Therefore $\sigma_*:=g^{-1}\circ f$ is a $\sigma^{-1}$-semilinear automorphism.
\end{proof}

The following construction based on \cite[Part I, \S 1]{Wi} and \cite[Chap.\,II, \S 2.2]{Ch} is used to state the rigidity of log overconvergent ($F$-)isocrystals.
For an associative $L$-algebra $A$ which is not necesarily commutative, the product $[a,b]:=ab-ba$ makes $A$ a Lie algebra, which we denote by $[A]$.
For any finite-dimensional Lie algebra $\frg$ over $L$, let $\frU(\frg)$ be the universal enveloping algebra.
Namely $\frU(\frg)$ is an associative $L$-algebra equipped with a Lie algebra homomorphism $\frg\rightarrow [\frU(\frg)]$, such that for any associative $L$-algebra $A$ the natural map $\Hom(\frU(\frg),A)\rightarrow\Hom(\frg,[A])$ is bijective.
Then the zero map $\frg\rightarrow [L]$ induces an $L$-algebra homomorphism $\frU(\frg)\rightarrow L$, which we call the {\it argumentation}.
We call its kernel the {\it argumentation ideal} of $\frU(\frg)$.
We denote by $\widehat{\frU}(\frg)$ the completion of $\frU(\frg)$ with respect to the argumentation ideal.

For a pro-unipotent algebraic group $\cG=\varprojlim_{n\in\bbN}\cG_n$, let
	\begin{equation}
	\widehat{\frU}(\Lie\cG):=\varprojlim_{n\in\bbN}\widehat{\frU}(\Lie\cG_n).
	\end{equation}
We define the argumentation ideal $\fra\subset\widehat{\frU}(\Lie\cG)$ to be $\fra:=\varprojlim_{n\in\bbN}\widehat{\fra}_n$, where $\widehat{\fra}_n$ is the ideal of $\widehat{\frU}(\Lie\cG_n)$ generated by the argumentation ideal of $\frU(\Lie\cG_n)$.
The argumentations of $\frU(\Lie\cG_n)$ extend to $\widehat{\frU}(\Lie\cG_n)$ and induce the argumentation $\widehat{\frU}(\Lie\cG)\rightarrow L$.

Let $Y$ be a connected strictly semistable log scheme over $k^0$ and assume that there exists a $k$-rational point $y\in Y^\sm(k)$.
Since $\Isoc^\dagger(Y/W^\varnothing)^\unip$ is a nilpotent tannakian category, $\pi_1^\unip(Y,y)$ is a pro-unipotent algebraic group over $L$.
Since $\fra/\fra^2\cong\Lie\pi_1^\unip(Y,y)^\ab$, we have natural isomorphisms
	\begin{eqnarray}
	\nonumber H^1_\rig(Y/W^\varnothing)&\cong& \Ext^1_{\Isoc^\dagger(Y/W^\varnothing)^\unip}(\sO_{Y/W^\varnothing},\sO_{Y/W^\varnothing})\cong\Ext^1_{\Rep_L(\pi_1^\unip(Y,y))}(L,L)\\
	\label{eq: H1 and Lie}&\cong&\Hom(\pi_1^\unip(Y,y),\bbG_a)\cong\Hom(\pi_1^\unip(Y,y)^\ab,\bbG_a)\underset{(*)}{\cong}\Hom(\Lie\pi_1^\unip(Y,y)^\ab,[L]),\\
	\nonumber &\cong&(\fra/\fra^2)^\vee
	\end{eqnarray}
where $\vee$ denotes the dual space, and $(*)$ is an isomorphism since $\pi_1^\unip(Y,y)^\ab$ is pro-unipotent.

\begin{proposition}\label{prop: complete}
	Let $Y$ be a connected strictly semistable log scheme over $k^0$ and assume that there exists a $k$-rational point $y\in Y^\sm(k)$.
	The ring $\widehat{\frU}(\Lie\pi_1^\unip(Y,y))$ is $\fra$-adically complete.
	Moreover the $\fra$-adic topology coincides with the project limit topology.
\end{proposition}

\begin{proof}
	By the isomorphisms \eqref{eq: H1 and Lie} and Theorem \ref{thm: HK rig comparison}, the $L$-vector space $\fra/\fra^2$ has finite dimension.
	This implies the assertion by the same proof as \cite[Lemma 2.4]{Ch}.
\end{proof}

The automorphism $\sigma_*$ in \eqref{eq: sigma Lie} induces a $\sigma^{-1}$-semilinear automorphism on $\widehat{\frU}(\Lie\pi_1^\unip(Y,y))$, which we denote again by $\sigma_*$.

\begin{definition}
	Let $Y$ be a connected strictly semistable log scheme over $k^0$ and assume that there exists a $k$-rational point $y\in Y^\sm(k)$.
	\begin{enumerate}
	\item Let $\Mod^\fin_L(\widehat{\frU}(\Lie\pi_1^\unip(Y,y)))$ be the category of finite-dimensional $L$-vector spaces equipped with left $\widehat{\frU}(\Lie\pi_1^\unip(Y,y))$-module structure such that the action of $\widehat{\frU}(\Lie\pi_1^\unip(Y,y))$ is $L$-linear and continuous with respect to the discrete topology on the module.
	\item Let $\Mod^\fin_L(\varphi,\widehat{\frU}(\Lie\pi_1^\unip(Y,y)))$ be the category of pairs $(M,\varphi)$ of an object $M\in \Mod^\fin_L(\widehat{\frU}(\Lie\pi_1^\unip(Y,y)))$ and a $\sigma$-semilinear automorphism $\varphi$ on $M$ such that the following diagram commutes:
			\[\xymatrix{
			\widehat{\frU}(\Lie\pi_1^\unip(Y,y))\ar[r]^-{\rho_M}\ar[d]^-{\sigma_*^{-1}}&\End_L(M)\ar[d]^-{\Ad\varphi}\\
			\widehat{\frU}(\Lie\pi_1^\unip(Y,y))\ar[r]^-{\rho_M}&\End_L(M),
			}\]
		where $\rho_M$ is the homomorphism defining the module structure, and $\Ad\varphi$ is defined by $f\mapsto\varphi\circ f\circ\varphi^{-1}$.
	\end{enumerate}
\end{definition}

For any $\sE\in\Isoc^\dagger(Y/W^\varnothing)^\unip$, by the definition of the tannakian fundamental group we have a natural morphism of affine group schemes
	\[\pi_1^\unip(Y,y)\rightarrow\GL(\eta_y(\sE)),\]
which induces a morphism of Lie algebras
	\[\Lie\pi_1^\unip(Y,y)\rightarrow[\End_L(\eta_y(\sE))].\]
Moreover this induces an $L$-algebra homomorphism
	\[\widehat{\frU}(\Lie\pi_1^\unip(Y,y))\rightarrow\End_L(\eta_y(\sE)),\]
namely a left action of $\widehat{\frU}(\Lie\pi_1^\unip(Y,y))$ on $\eta_y(\sE)$.

\begin{proposition}
	Let $Y$ be a connected strictly semistable log scheme over $k^0$ and assume that there exists a $k$-rational point $y\in Y^\sm(k)$.
	Then by the construction above, the fiber functor $\eta_y$ induces canonical equivalences of categories
		\begin{eqnarray}
		\label{eq: Isoc Mod}\Isoc^\dagger(Y/W^\varnothing)^\unip&\xrightarrow{\cong}&\Mod_L^\fin(\widehat{\frU}(\Lie\pi_1^\unip(Y,y))),\\
		\label{eq: FIsoc Mod}F\Isoc^\dagger(Y/W^\varnothing)^\unip&\xrightarrow{\cong}&\Mod_L^\fin(\varphi,\widehat{\frU}(\Lie\pi_1^\unip(Y,y))).
		\end{eqnarray}
\end{proposition}

\begin{proof}
	Since $L$ is of characteristic zero, the equivalence \eqref{eq: Isoc Mod} is given by \cite[IV \S 2, Corollary 4.5.b]{DG}.
	Then the autofunctor $\sigma^*$ on $\Isoc^\dagger(Y,y)^\unip$ corresponds to the autofunctor on $\Mod_L(\widehat{\frU}(\Lie\pi_1^\unip(Y,y)))$ given by twisting the action as
	\[\rho_{\sigma^*M}:=\rho_M\circ \sigma_*\colon\widehat{\frU}(\Lie\pi_1^\unip(Y,y))\rightarrow\End_L(M)\]
	for any $M\in\Mod_L(\widehat{\frU}(\Lie\pi_1^\unip(Y,y)))$, where $\rho_M$ denotes the algebra homomorphism defining the module sructure of $M$.
	Thus a Frobenius structure $\Phi\colon\sigma^*\sE\xrightarrow{\cong}\sE$ corresponds to an isomorphism $\varphi\colon\sigma^*\eta_y(\sE)\xrightarrow{\cong}\eta_y(\sE)$ such that
		\[\varphi\circ(\rho_{\eta_y(\sE)}\circ\sigma_*(g))(m)=\rho_{\eta_y(\sE)}(g)\circ\varphi(m)\]
	for any $m\in\eta_y(\sE)$ and $g\in\widehat{\frU}(\Lie\pi_1^\unip(Y,y))$.
	This gives the equivalence \eqref{eq: FIsoc Mod}.
\end{proof}

%%%%%%%%%%%%%%%%%%%%%%%%%%%%
\section{Mixed $F$-isocrystals}\label{sec: mixed}
%%%%%%%%%%%%%%%%%%%%%%%%%%%%

In this section, we introduce the notion of mixed log overconvergent $F$-isocrystals, and prove their rigidity property.

Assume that $k$ is a finite field of order $q=p^d$.
Since $\sigma^d=\id_L$ on $L$, for any $M\in\Mod^\fin_L(\varphi,N)$ the automorphism $\varphi^d$ is $L$-linear.
A {\it Weil number relative to $k$} is an algebraic number $\alpha$ whose archimedean absolute value with respect to any embedding into $\bbC$ is $q^{n/2}$ for some $n\in\bbZ$.
We call $n$ the {\it weight} of $\alpha$.

\begin{definition}\label{def: mixed mod}
	Let $M\in\Mod^\fin_L(\varphi,N)$.
	We say $M$ is {\it pure of weight $n$} (resp.\  {\it mixed}) relative to $k$ if all eigenvalues of $\varphi^d$ are Weil numbers of weight $n$ (resp.\  of several weights) relative to $k$.
	We denote by $\Mod_L^\mix(\varphi,N)$ (resp.\  $\Mod_L^\mix(\varphi)$) the full subcategory of $\Mod_L^\fin(\varphi,N)$ (resp.\  $\Mod_L^\fin(\varphi)$) consisting of mixed objects relative to $k$.
\end{definition}

\begin{remark}\label{rem: mixed mod tannakian}
	\begin{enumerate}
	\item Mixed objects are closed under extension, subobject, quotient, tensor product, internal Hom, and dual in the category $\Mod_L^\fin(\varphi,N)$.
		In particular, $\Mod_L^\mix(\varphi,N)$ is a tannakian category over $\bbQ_p$.
	\item\label{item: remark mixed} Let $k'$ be a finite extension of $k$ of degree $r$, and $L'$ the fraction field of the ring of Witt vectors of $k'$.
		Then an algebraic number $\alpha$ is a Weil number of weight $n$ relative to $k$ if and only if $\alpha^r$ is a Weil number of weight $n$ relative to $k'$.
		Thus an object $M\in\Mod_L^\fin(\varphi,N)$ is pure of weight $n$ (resp.\ mixed) relative to $k$ if and only if $M\otimes_LL'$ is pure of weight $n$ (resp.\ mixed) relative to $k'$.
	\end{enumerate}
\end{remark}

If $M$ is mixed relative to $k$, there exists a unique increasing filtration $W_\bullet$ on $M$ by subobjects such that $\Gr_n^WM$ is pure of weight $n$.
Indeed $W_nM$ is given by the sum of generalized eigenspaces of the eigenvalues being Weil number weight $\leq n$ relative to $k$.
The relation $N\varphi=p\varphi N$ implies that $N$ preserves $W_\bullet$.
We call this the {\it weight filtration} on $M$.

\begin{definition}\label{def: mixed isoc}
	Let $Y$ be a strictly semistable log scheme over $k^0$ and let $(\sE,\Phi)\in F\Isoc^\dagger(Y/W^\varnothing)^\unip$.
	\begin{enumerate}
	\item We say $(\sE,\Phi)$ is {\it pure of weight $n$} if for any $y\in P(Y)$ the object $H^0_\HK(y,(\sE_y,\Phi_y))\in\Mod^\fin_{L_y}(\varphi,N)$ is pure of weight $n$ relative to $k_y$.
	\item We say $(\sE,\Phi)$ is {\it mixed} if there exists a finite increasing filtration $W_\bullet$ on $(\sE,\Phi)$ by subobjects, such that $\Gr^W_n(\sE,\Phi)$ is pure of weight $n$ for any $n\in\bbZ$.
		We call such a filtration a {\it weight filtration}.
		We denote by $F\Isoc^\dagger(Y/W^\varnothing)^\mix$ the full subcategory of $F\Isoc^\dagger(Y/W^\varnothing)^\unip$ consisting of mixed objects.
	\end{enumerate}
\end{definition}

\begin{proposition}\label{prop: mixed isoc tannakian}
	Mixed objects are closed under subobject, quotient, tensor product, internal Hom, and dual in the category $F\Isoc^\dagger(Y/W^\varnothing)^\unip$.
	In particular $F\Isoc^\dagger(Y/W^\varnothing)^\mix$ is a tannakian category over $\bbQ_p$.
\end{proposition}

\begin{proof}
	Let $(\sE',\Phi')$ be a subobject (resp.\  quotient) of $(\sE,\Phi)$ in $F\Isoc^\dagger(Y/W^\varnothing)^\unip$, and assume that $(\sE,\Phi)$ is mixed with a weight filtration $W_\bullet$.
	We denote again by $W_\bullet$ the induced filtration on $\sE'$.
	Then $\Phi'$ preserves $W_\bullet$, and $\Gr^W_n(\sE',\Phi')$ is a subobject (resp.\  quotient) of $\Gr^W_n(\sE,\Phi)$.
	Therefore $H^0_\HK(y,(\Gr^W_n(\sE',\Phi'))_y)$ is a subobject (resp.\  quotient) of $H^0_\HK(y,(\Gr^W_n(\sE,\Phi))_y)$ and hence pure of weight $n$.
	Thus $(\sE',\Phi')$ is also mixed.
	
	If two objects $(\sF,\Psi),(\sF',\Psi')\in F\Isoc^\dagger(Y/W^\varnothing)^\mix$ are given, their weight filtrations induce the weight filtrations on $(\sF,\Psi)\otimes(\sF',\Psi')$ and $\sheafhom((\sF',\Psi'),(\sF,\Psi))$.
\end{proof}

In order to prove that a weight filtration on a log overconvergent $F$-isocrystal is unique if it exists, we prepare some assertions.

\begin{lemma}\label{lem: weight ss}
	Let $Y$ be a strictly semistable log scheme over $k^0$, $D$ the horizontal divisor of $Y$, and $\Upsilon_Y$ the set of irreducible components of $Y$.
	For a subset $I\subset\Upsilon_Y$, let $Y_I^\infty$ be the log scheme whose underlying scheme is $Y_I:=\bigcap_{i\in I}Y_i$ and whose log structure is associated to the divisor $Y_I\cap D$ (not the pull-back structure of $Y$).
	For $r\geq 1$, let $Y^{\infty,(r)}:=\coprod_{\substack{I\subset \Upsilon_Y\\ \lvert I\rvert=r}}Y_I^\infty$.
	Then there exists a spectral sequence
		\begin{equation}\label{eq: weight ss}
		E_1^{-i,i+n}=\bigoplus_{j\geq\max\{0,-i\}}H^{n-i-2j}_\rig(Y^{\infty,(i+2j+1)}/W^\varnothing)(-i-j)\Rightarrow H^n_\rig(Y/W^0)\cong H^n_\HK(Y).
		\end{equation}
	Here $(-i-j)$ denotes the twist of $\varphi$-modules (see Definition \ref{def: phi N module}).
\end{lemma}

\begin{proof}
	This is given by repeating the same construction as \cite[\S 5 (4)]{GK3}, or applying \cite[\S 5 (4)]{GK3} for $U:=Y\setminus D$. (Note that we have $H_\HK^n(Y)\cong H_\HK^n(U)$ by \cite[Proposition 3.32]{EY}.)
\end{proof}

\begin{proposition}\label{prop: coh is mixed}
	Let $Y$ be a strictly semistable log scheme over $k^0$.
	\begin{enumerate}
	\item For any object $(\sE,\Phi)\in F\Isoc^\dagger(Y/W^\varnothing)^\mix$, the cohomology groups $H^n_\HK(Y,(\sE,\Phi))$ and $H^n_\rig(Y/W^\varnothing,(\sE,\Phi))$ are mixed relative to $k$ for any $n$.
	\item For the trivial coefficient, the weights of $H^n_\rig(Y/W^\varnothing)$ and $H^n_\HK(Y)$ lie in the interval $[0,2n]$, that is, we have
			\[\Gr^W_iH^n_\rig(Y/W^\varnothing)=\Gr^W_iH^n_\HK(Y)=0\]
		for any $i\notin [0,2n]$.
	\end{enumerate}
\end{proposition}

\begin{proof}
	By Proposition \ref{prop: mixed isoc tannakian} and the spectral sequence \eqref{eq: ss for Fil}, it suffices to consider the case $(\sE,\Phi)=\sO_{Y/W^\varnothing}$.
	For the spectral sequence \eqref{eq: weight ss}, we have $E_1^{-i,i+n}=0$ if $i\notin[-n,n]$.
	Moreover $E_1^{-i,i+n}$ are mixed whose weights lie in $[n+i,2n]$ by \cite[Chap.\,I, Theorem 2.2]{Ch} .
	Thus the weights of $H^n_\HK(Y)$ lie in $[0,2n]$.
	Since \eqref{eq: cone of N} induces an exact sequence
		\[H_\HK^{n-1}(Y)(-1)\rightarrow H^n_\rig(Y/W^\varnothing)\rightarrow H^n_\HK(Y),\]
	we see that the weights of $H^n_\rig(Y/W^\varnothing)$ also lie in $[0,2n]$.
\end{proof}

\begin{lemma}\label{lem: weight H0}
	Let $Y$ be a strictly semistable log scheme over $k^0$ and let $(\sE,\Phi)\in F\Isoc^\dagger(Y/W^\varnothing)^\unip$. 
	If $(\sE,\Phi)$ is pure object of weight $n$, then $H^0_\HK(Y,(\sE,\Phi))$ and $H^0_\rig(Y/W^\varnothing,(\sE,\Phi))$ are pure of weight $n$ relative to $k$.
	If $(\sE,\Phi)$ is mixed with a weight filtration $W_\bullet$, then we have
		\begin{align*}
		W_nH^0_\HK(Y,(\sE,\Phi))=H^0_\HK(Y,W_n(\sE,\Phi))&&\text{and}&&W_nH^0_\rig(Y/W^\varnothing,(\sE,\Phi))=H^0_\rig(Y/W^\varnothing,W_n(\sE,\Phi)),
		\end{align*}
	where $W_n$ on the left hand side are the weight filtrations canonically defined for objects of $\Mod_L^\mix(\varphi,N)$.
\end{lemma}

\begin{proof}
	The first assertion is reduced to the case when $Y$ is geometrically connected over $k$, by Remark \ref{rem: mixed mod tannakian} \eqref{item: remark mixed}.
	Then the assertion follows from the spectral sequence \eqref{eq: ss for Fil} and the fact that $H^0_\HK(Y)=H^0_\rig(Y/W^\varnothing)$ is pure of weight $0$ relative to $k$.
	
	To prove the second assertion, assume that $(\sE,\Phi)$ is mixed with a weight filtration $W_\bullet$, and let $W'_nH^0_\HK(Y,(\sE,\Phi)):=H^0_\HK(Y,W_n(\sE,\Phi))$ for $n\in\bbZ$. 
	Then $\Gr^{W'}_nH^0_\HK(Y,(\sE,\Phi))\subset H^0_\HK(Y,\Gr^W_n(\sE,\Phi))$, and they are pure of weight $n$ relative to $k$ by the first assertion.
	Thus $W'$ gives the weight filtration on $H^0_\HK(Y,(\sE,\Phi))$.
	The equality for the weight filtration on $H^0_\rig(Y/W^\varnothing,(\sE,\Phi))$ follows in the same way.
\end{proof}

Now we may prove the uniqueness of the weight filtration of a mixed log overconvergent $F$-isocrystal.

\begin{proposition}\label{prop: W unique}
	Let $Y$ be a strictly semistable log scheme over $k^0$ and let $(\sE,\Phi)\in F\Isoc^\dagger(Y/W^\varnothing)^\unip$.
	A weight filtration $W_\bullet$ on $(\sE,\Phi)$ is unique if it exists. 
\end{proposition}

\begin{proof}
	Assume that $W^1_\bullet$ and $W^2_\bullet$ are weight filtrations on $(\sE,\Phi)$.
	Let $\widetilde{W}^1_\bullet$ and $\widetilde{W}^2_\bullet$ be the filtrations on $\sheafhom(\sE,\sE)$ defined by
		\begin{eqnarray*}
		\widetilde{W}^1_n\sheafhom(\sE,\sE)&:=&\left\{f\in\sheafhom(\sE,\sE)\mid f(W^1_m\sE)\subset W^2_{m+n}\sE\text{ for any }m\in\bbZ\right\},\\
		\widetilde{W}^2_n\sheafhom(\sE,\sE)&:=&\left\{f\in\sheafhom(\sE,\sE)\mid f(W^2_m\sE)\subset W^1_{m+n}\sE\text{ for any }m\in\bbZ\right\}.
		\end{eqnarray*}
	By Proposition \ref{prop: mixed isoc tannakian}, both of $W^1_\bullet$ and $W^2_\bullet$ are weight filtrations making $\sheafhom((\sE,\Phi),(\sE,\Phi))$ mixed.
	The identity on $\sE$ defines an element
		\[\id_\sE\in W_0H^0_\rig(Y/W^\varnothing,\sheafhom((\sE,\Phi),(\sE,\Phi)))=H^0_\rig(Y/W^\varnothing,\widetilde{W}^i_0\sheafhom((\sE,\Phi),(\sE,\Phi)))\]
	for $i=1,2$, since it is fixed by the Frobenius action.
	Here the equality is given by Lemma \ref{lem: weight H0}.
	The image of $\id_\sE$ by the morphism
		\[H^0_\rig(Y/W^\varnothing,\widetilde{W}^i_0\sheafhom((\sE,\Phi),(\sE,\Phi)))\otimes\sO_{Y/W^\varnothing}\rightarrow\widetilde{W}^i_0\sheafhom((\sE,\Phi),(\sE,\Phi))\] 
	is also given by the identity on $\sE$.
	Thus by the definition of $\widetilde{W}^i_0$ we see that $W^1_\bullet\subset W^2_\bullet$ and $W^2_\bullet\subset W^1_\bullet$.
\end{proof}

\begin{corollary}\label{cor: W strict}
	Let $Y$ be a strictly semistable log scheme of finite type over $k^0$.
	Any morphism in $F\Isoc^\dagger(Y/W^\varnothing)^\mix$ is strictly compatible with the weight filtrations.
\end{corollary}

\begin{proof}
	Let $f\colon(\sE',\Phi')\rightarrow(\sE,\Phi)$ be a morphism in $F\Isoc^\dagger(Y/W^\varnothing)^\mix$.
	We define filtrations $W^1_\bullet$ and $W^2_\bullet$ on $\Im f$ by
		\begin{align*}
		W^1_n\Im f:=\Im(W_n\sE'\rightarrow\sE)&&\text{and}&&W^2_n\Im f:=W_n\sE\cap\Im f.
		\end{align*}
	Then $\Gr^W_n\sE'\rightarrow\Gr^{W^1}_n\Im f$ is surjective, and $\Gr^{W^2}_n\Im f\rightarrow\Gr^W_n\sE$ is injective.
	Thus $W^1$ and $W^2$ both define weight filtrations on $\Im f$.
	By Proposition \ref{prop: W unique} we have $W^1_\bullet=W^2_\bullet$.
\end{proof}

\begin{proposition}\label{prop: bc isoc}
	Let $Y$ be a strictly semistable log scheme over $k^0$, and let $(\sE,\Phi)\in F\Isoc^\dagger(Y/W^\varnothing)^\unip$.
	Let $k'$ be a finite extension of $k$, $W'$ the ring of Witt vectors of $k'$, and $Y':=Y\times_{W^\varnothing}W'^\varnothing$.
	Then $(\sE,\Phi)$ is mixed if and only if its pull-back $(\rho,\varrho)^*(\sE,\Phi)\in F\Isoc^\dagger(Y'/W'^\varnothing)^\unip$ with respect to the commutative diagram
		\[\xymatrix{
		Y'\ar[r]\ar[d]& k'^\varnothing\ar[r]\ar[d] & W'^\varnothing\ar[d]\\
		Y \ar[r] & k^\varnothing\ar[r] & W^\varnothing
		}\]
	is mixed.
\end{proposition}

\begin{proof}
	The ``only if'' part is obvious.
	To see ``if'' part,  by extending $k'$ we may suppose $k'$ is Galois over $k$.
	Let $W'=W(k')$ and $L'=\mathrm{Frac}W'$.
	Since the weight filtration is canonical, $W_\bullet$ on $(\sE',\Phi')$ is stable under the action of $\Gal(k'/k)=\Gal(L'/L)$.
	Then one can see that $W_\bullet$ descend to a filtration on $(\sE,\Phi)$ by Galois descent.
	
	Indeed, to obtain a log overconvergent isocrystal, it is enough to define coherent sheaves on $\cZ_\bbQ$ for all objects $(Z,\cZ,i,h,\theta)\in\OC(Y/W^\varnothing)$ where $\cZ$ is $p$-adic and affine.
	Note that coherent modules over $\cZ_\bbQ$ are equivalent to finite modules over the global section, and base change of dagger affinoid is given by usual tensor product $\otimes_LL'$ (without any weak completion).
	Then we may apply Galois descent for modules over $L$-algebras.
\end{proof}

Next we study the weight filtration of the tannakian fundamental group.

\begin{proposition}
	Let $Y$ be a connected strictly semistable log scheme over $k^0$ and assume that there exists a $k$-rational point $y\in Y^\sm(k)$.
	The object of $\Mod_L(\varphi)$ given by the $L$-vector space $\widehat{\frU}(\Lie\pi_1^\unip(Y,y))$ equipped with $\sigma_*^{-1}$ is a projective limit of objects of $\Mod_L^\mix(\varphi)$.
	There exists a canonical increasing filtration $W_\bullet$ on $\widehat{\frU}(\Lie\pi_1^\unip(Y,y))$ such that $\Gr^W_n\widehat{\frU}(\Lie\pi_1^\unip(Y,y))$ is a projective limit of objects of $\Mod_L^\mix(\varphi)$ which are pure of weight $n$ relative to $k$.
	We have $W_0\widehat{\frU}(\Lie\pi_1^\unip(Y,y))=\widehat{\frU}(\Lie\pi_1^\unip(Y,y))$.
\end{proposition}

\begin{proof}
	We first note that $\sigma_*$ induces an automorphism on $\fra/\fra^2$, and we have a commutative diagram
	\begin{equation}\label{eq: diag sigma}
	\xymatrix{
	H^1_\rig(Y/W^\varnothing)\ar[r]^-\cong\ar[d]^-{\varphi^d} & (\fra/\fra^2)^\vee\ar[d]^-{(\sigma_*^d)^\vee}\\
	H^1_\rig(Y/W^\varnothing)\ar[r]^-\cong & (\fra/\fra^2)^\vee,
	}\end{equation}
where the horizontal maps are given in \eqref{eq: H1 and Lie} and $(\sigma^d_*)^\vee$ is defined by $f\mapsto f\circ\sigma_*^d$.

	Since for any $k\geq 0$ there exists a surjection $(\fra/\fra^2)^{\otimes k}\rightarrow\fra^k/\fra^{k+1}$ which is compatible with $\sigma_*^{-1}$, we see that $\fra^k/\fra^{k+1}$ and hence $\widehat{\frU}(\Lie\pi_1^\unip(Y,y))/\fra^k$ belong to $\Mod_L^\mix(\varphi)$.
	By Proposition \ref{prop: complete} we have
		\[\widehat{\frU}(\Lie\pi_1^\unip(Y,y))=\varprojlim_{k\in\bbN}\widehat{\frU}(\Lie\pi_1^\unip(Y,y))/\fra^k.\]
	The weight filtration $W_\bullet$ on $\widehat{\frU}(\Lie\pi_1^\unip(Y,y))$ is given by
		\[W_n\widehat{\frU}(\Lie\pi_1^\unip(Y,y)):=\varprojlim_{k\in\bbN}W_n\left(\widehat{\frU}(\Lie\pi_1^\unip(Y,y))/\fra^k\right).\]
	By Proposition \ref{prop: coh is mixed} and the diagram \eqref{eq: diag sigma}, the weights of $\fra/\fra^{2}$ lie in the interval $[-2,0]$.
	This implies $W_0\widehat{\frU}(\Lie\pi_1^\unip(Y,y))=\widehat{\frU}(\Lie\pi_1^\unip(Y,y))$.
\end{proof}

\begin{definition}
	Let $Y$ be a connected strictly semistable log scheme over $k^0$ and assume that there exists a $k$-rational point $y\in Y^\sm(k)$.
	Let $\Mod_L^\mix(\varphi,\widehat{\frU}(\Lie\pi_1^\unip(Y,y)))$ be the full subcategory of $\Mod^\fin_L(\varphi,\widehat{\frU}(\Lie\pi_1^\unip(Y,y)))$ consisting of objects which are mixed relative to $k$.
\end{definition}

The following theorem for the rigidity of mixed log overconvergent $F$-isocrystals is regarded as a $p$-adic analogue of \cite[Theorem 1.6]{HZ}.

\begin{theorem}\label{thm: mixed Mod}
	Let $Y$ be a connected strictly semistable log scheme over $k^0$ and assume that there exists a $k$-rational point $y\in Y^\sm(k)$.
	Then the fiber functor $\eta_y$ induces a canonical equivalence of categories
		\begin{equation}
		F\Isoc^\dagger(Y/W^\varnothing)^\mix\xrightarrow{\cong} \Mod_L^\mix(\varphi,\widehat{\frU}(\Lie\pi_1^\unip(Y,y))).
		\end{equation}
	In other words, an object $(\sE,\Phi)\in F\Isoc^\dagger(Y/W^\varnothing)^\unip$ is mixed if and only if $\eta_y(\sE,\Phi)$ is mixed.
\end{theorem}

\begin{proof}
	The necessity is clear by definition.
	We will show that an object $(\sE,\Phi)\in F\Isoc^\dagger(Y/W^\varnothing)^\unip$ is mixed if $M:=\eta_y(\sE,\Phi)$ is mixed.
	Since the homomorphism
		\[\widehat{\frU}(\Lie\pi_1^\unip(Y,y))\rightarrow\End_L(M)\]
	is compatible with the Frobenius operators, it induces
		\begin{equation}\label{eq: W0 U}
		\widehat{\frU}(\Lie\pi_1^\unip(Y,y))=W_0\widehat{\frU}(\Lie\pi_1^\unip(Y,y))\rightarrow W_0\End_L(M).
		\end{equation}
	Therefore the action of $\widehat{\frU}(\Lie\pi_1^\unip(Y,y))$ on $M$ preserves the weight filtration on $M$.
	Hence we may define a subobject $W_n(\sE,\Phi)\subset(\sE,\Phi)$  for $n\in\bbN$ to be the object which corresponds to $W_nM$ via the equivalence \eqref{eq: FIsoc Mod}.
	By Proposition \ref{prop: unip const}, the graded quotient $\Gr^\Fil_m\Gr^W_n(\sE,\Phi)$ is a constant object associated to a pure $\varphi$-module of weight $n$ relative to $k$.
	This shows that $W_\bullet$ is the weight filtration on $(\sE,\Phi)$. 
\end{proof}

\begin{corollary}
	Mixed objects are closed under extension in $F\Isoc^\dagger(Y/W^\varnothing)^\unip$.
\end{corollary}

\begin{remark}
	In the non-logarithmic situation, the weights of the first rigid cohomology are $1$ or $2$, and hence $\Gr^W_0$ of the universal enveloping algebra is just $L$.
	This implies that any pure object is constant (see the proof of \cite[Theorem 3.4.1]{Ch}).
	
	The analogy of this breaks down In our situation.
	Namely, the weights of $H^1_\rig(Y/W^\varnothing)$ are $0$, $1$, or $2$, and $\Gr^W_0\wt\frU(\Lie\pi_1^\unip(Y,y))$ can have infinite dimension.
	This causes a pure object $(\sE,\Phi)\in F\Isoc^\dagger(Y/W^\varnothing)^\mix$ can be non-constant.
	For example, if we consider a strictly semistable reduction of a Tate curve as in \cite[\S 5]{EY}, $H^1_\abs(Y/W^\varnothing)\cong\Ext^1_{F\Isoc^\dagger(Y/W^\varnothing)^\unip}(\sO_{Y/W^\varnothing},\sO_{Y/W^\varnothing})$ is a two dimensional $\bbQ_p$-vector space.
	Extension classes in
		\[L/(1-\sigma)(L)\cong\Im(H^0_\rig(Y/W^\varnothing)\rightarrow H^1_\abs(Y/W^\varnothing))\]
	are represented by pure constant objects, however other extension classes are represented by pure non-constant objects.
\end{remark}

%%%%%%%%%%%%%%%%%%%%%%%
%
\part{Theory in mixed characteristic: $p$-adic Hodge cohomology and syntomic cohomology with syntomic coefficients}\label{part: mixed char}
%
%%%%%%%%%%%%%%%%%%%%%%%

In this part, we study the Hyodo--Kato map, introduce the category of syntomic coefficients, and give the definition of the $p$-adic Hodge cohomology and the syntomic cohomology with syntomic coefficients.

Let $V$ be a complete discrete valuation ring of mixed characteristic $(0,p)$ with maximal ideal $\frm$, residue field $k$, and fraction field $K$.
Assume that $k$ is algebraic over $\bbF_p$.
Let $W=W(k)$ be the ring of Witt vectors of $k$, and $L$ the fraction field of $W$.
Let $V^\sharp$ be the weak formal scheme $\Spwf V$ endowed with the canonical log structure.
Namely we have $\Gamma(V^\sharp,\cN_{V^\sharp})=V\setminus\{0\}$.
Let $\cS=(\Spwf W\llbracket s\rrbracket,\cN_\cS)$ and $\tau\colon k^0\hookrightarrow\cS$ be as in \S \ref{sec: absolute isoc}.

For a choice of a uniformizer $\pi\in V$, let $j_\pi\colon V^\sharp\hookrightarrow\cS$ be the exact closed immersion defined by $s\mapsto \pi$, and let $\tau_\pi\colon k^0\hookrightarrow V^\sharp$ be the unique exact closed immersion such that $\tau=j_\pi\circ\tau_\pi$.

%%%%%%%%%%%%%%%%%%%%%%%%%%%%%%%%%
\section{De Rham cohomology}\label{sec: de rham}
%%%%%%%%%%%%%%%%%%%%%%%%%%%%%%%%%
We first prepare some notion concerning the de Rham cohomology of strictly semistable weak formal log schemes over $V^\sharp$.

\begin{definition}\label{def: strictly semistable over V}
	A weak formal log scheme $\cX$ over $V^\sharp$ is said to be {\it strictly semistable} if Zariski locally there exist a uniformizer $\pi\in V$, integers $n\geq 1$, $m\geq 0$ and a strict smooth morphism
		\begin{equation}\label{eq: semistable over V}\cX\rightarrow(\Spwf V[x_1,\ldots,x_n,y_1,\ldots,y_m]^\dagger/(\pi-x_1\cdots x_n),\cN)\end{equation}
	where the log structure $\cN$ is induced by the map
		\[\bbN^n\oplus\bbN^m\rightarrow V[x_1,\ldots,x_n,y_1,\ldots,y_m]^\dagger/(\pi-x_1\cdots x_n);\ ((k_i),(\ell_j))\mapsto x_1^{k_1}\cdots x_n^{k_n}y_1^{\ell_1}\cdots y_m^{\ell_n}.\]
		We call the closed subset
		\[\{x\in\cX\mid \text{$\exists\alpha\in\cN_{\cX,x}$ $\forall\beta\in\cN_{\cX,x}$ $\alpha\beta$ is not contained in the image of $\Gamma(V^\sharp,\cN_{V^\sharp})$}\}\]
		with reduced structure the {\it horizontal divisor} of $\cX$.
		Namely the horizontal divisor is locally defined by $y_1\cdots y_m$ via \eqref{eq: semistable over V}.
		Note that the condition that $\cX$ is strictly semistable is independent of the choice of $\pi$.
\end{definition}

\begin{definition}
	Let $\cX$ be a strictly semistable weak formal log scheme over $V^\sharp$.
	We say an object $(\cE,\nabla)\in\MIC^\dagger(\cX/V^\sharp)$ has {\it nilpotent residues} if it has nilpotent residues with respect to the horizontal divisor in the sense of \cite[Definition 2.3.9]{Ke}.	
	More precisely, Zariski locally on $\cX$ we take a strict smooth morphism
		\[\cX\rightarrow(\Spwf V[x_1,\ldots,x_n,y_1,\ldots,y_m]^\dagger/(\pi-x_1\cdots x_n),\cN)\]
	as in Definition \ref{def: strictly semistable over V}.
	For $j=1,\ldots,m$, $\frD_j\subset\cX_\bbQ$ be the closed dagger space defined by $y_j$.
	If we let $\Omega^1_{\cX_\bbQ}$ the sheaf of (non-logarithmic) differentials on $\cX_\bbQ$, then we have
		\[\Coker\Bigl(\cE\otimes\bigl(\Omega^1_{\cX_\bbQ}\oplus\bigoplus_{i\neq j}\cO_{\cX_\bbQ}d\log y_i\bigr)\rightarrow\cE\otimes\omega^1_{\cX/V^\sharp,\bbQ}\Bigr)=\cE\otimes\cO_{\frD_j}\cdot d\log y_j.\]
	The composition of $\nabla\colon \cE\rightarrow\cE\otimes\omega^1_{\cX/V^\sharp,\bbQ}$ with the projection $\cE\otimes\omega^1_{\cX/V^\sharp,\bbQ}\rightarrow\cE\otimes\cO_{\frD_j}\cdot d\log y_j$ induces a morphism $\cE\otimes\cO_{\frD_j}\rightarrow\cE\otimes\cO_{\frD_j}\cdot d\log y_j$.
	Identifying $\cE\otimes\cO_{\frD_j}\cdot d\log y_j$ with $\cE\otimes\cO_{\frD_j}$, we obtain an endomorphism $\res_{\frD_j}$ on $\cE\otimes\cO_{\frD_j}$, which we call the {\it residue} along $\frD_j$.
	We denote by $\MIC^\dagger(\cX/V^\sharp)^\nr$ the full subcategory of $\MIC^\dagger(\cX/V^\sharp)$ consisting of objects having nilpotent residues.
\end{definition}

\begin{proposition}
	The category $\MIC^\dagger(\cX/V^\sharp)^\nr$ is abelian.
\end{proposition}

\begin{proof}
	The proposition follows by the same reason as Proposition \ref{prop: nr abelian}.
\end{proof}

\begin{definition}
	Let $\cX$ be a strictly semistable weak formal log scheme over $V^\sharp$.
	An object $(\cE,\nabla)\in\MIC^\dagger(\cX/V^\sharp)^\nr$ is said to be {\it unipotent} if it is an iterated extension of $\cO_{\cX_\bbQ}$.
	We denote by $\MIC^\dagger(\cX/V^\sharp)^\unip$ the full subcategory of $\MIC^\dagger(\cX/V^\sharp)^\nr$ consisting of unipotent objects.
\end{definition}

\begin{proposition}
	The category $\MIC^\dagger(\cX/V^\sharp)^\unip$ is abelian.
\end{proposition}

\begin{proof}
	The proposition follows by the same reason as Proposition \ref{prop: unipotence}.
\end{proof}

\begin{proposition}\label{prop: dR finite}
	Let $\cX$ be a strictly semistable weak formal log scheme over $V^\sharp$.
	Assume that $\cX$ is quasi-compact and adic over $V^\sharp$.
	For $(\cE,\nabla)\in\MIC^\dagger(\cX/V^\sharp)^\unip$, the de Rham cohomology groups
		\[H^n_\dR(\cX,(\cE,\nabla)):=H^n(\cX_\bbQ,\cE\otimes\omega^\bullet_{\cX/V^\sharp,\bbQ})\]
	are finite-dimensional $K$-vector spaces.
\end{proposition}

\begin{proof}
	Since $\cE$ is unipotent, we may reduce to the case $\cE=\cO_{\cX_\bbQ}$.
	By counting log poles along the horizontal divisor, we may define a filtration on $\omega^\bullet_{\cX/V^\sharp,\bbQ}$.
	The $E_1$-page of the spectral sequence associated to this filtration consists of de Rham cohomology of strictly semistable weak formal log schemes over $V^\sharp$ with empty horizontal divisor.
	Hence we may assume that the horizontal divisor of $\cX$ is empty.
	In this case we have $\omega^\bullet_{\cX/V^\sharp,\bbQ}=\Omega^\bullet_{\cX_\bbQ}$, the usual de Rham complex of $\cX_\bbQ$.
	Therefore the finiteness is given by \cite[Theorem A]{GK1.5}.
\end{proof}

We say a weak formal log scheme over $V^\sharp$ is {\it proper} if the underlying weak formal scheme is proper over $V$.

\begin{definition}
	Let $\cX$ be a proper strictly semistable weak formal log scheme adic over $V^\sharp$ and $(\cE,\nabla)\in\MIC^\dagger(\cX/V^\sharp)^\unip$.
	A {\it Hodge filtration} on $(\cE,\nabla)$ is a finite descending filtration on $\cE$ by coherent subsheaves satisfying the Griffith transversality
		\[\nabla(F^n\cE)\subset F^{n-1}\cE\otimes\omega^1_{\cX/V^\sharp,\bbQ}.\]
\end{definition}

For $(\cE,\nabla)$ with a Hodge filtration $F^\bullet$, the complex $\cE\otimes\omega^\bullet_{\cX/V^\sharp,\bbQ}$ is filtered by $F^\bullet$ and the stupid filtration on $\omega^\bullet_{\cX/V^\sharp,\bbQ}$.
Namely the filtration defined by
	\[F^n(\cE\otimes\omega^k_{\cX/V^\sharp,\bbQ}):=F^{n-k}\cE\otimes\omega^k_{\cX/V^\sharp,\bbQ}\]
is preserved by $\nabla$.
This induces a filtration $F^\bullet$ on $R\Gamma_\dR(\cX,\cE)=R\Gamma(\cX_\bbQ,\cE\otimes\omega^\bullet_{\cX/V^\sharp,\bbQ})$.

%%%%%%%%%%%%%%%%%%%%%%%%
\section{Hyodo--Kato map}\label{sec: HK map}
%%%%%%%%%%%%%%%%%%%%%%%%
In this section, we will extend the construction of the Hyodo--Kato map in \cite{EY} to cohomology with coefficients in log overconvergent isocrystal.
The Hyodo--Kato map is defined by choosing a uniformizer of $V$ and a branch of the $p$-adic logarithm over $K$, but in fact it depends only on the choice of a branch of the $p$-adic logarithm. 

Let $\pi\in V$ be a uniformizer.
For a strictly semistable log scheme $Y$ over $k^0$, we denote
	\[\Isoc^\dagger_\pi(Y/V^\sharp):=\Isoc^\dagger(Y/(k^0,V^\sharp,\tau_\pi)).\]
	
\begin{proposition}
	Let $Y$ be a quasi-compact strictly semistable log scheme over $k^0$.
	For $\sE\in\Isoc^\dagger_\pi(Y/V^\sharp)^\unip$, the log rigid cohomology groups $H^n_\rig(Y/V^\sharp,\sE)$ are finite-dimensional $K$-vector spaces.
\end{proposition}

\begin{proof}
	Considering locally, $R\Gamma_\rig(Y/V^\sharp,\sE)_\pi$ is described as the de Rham cohomology of a strictly semistable weak formal log scheme adic over $V^\sharp$ with unipotent coefficients.
	Hence the assertion follows by Proposition \ref{prop: dR finite}.
\end{proof}

The diagram
	\[\xymatrix{
	Y\ar[r]\ar[rd] &k^0\ar[d]\ar[r]^-{\tau_\pi} & V^\sharp\ar[d]\\
	& k^\varnothing\ar[r]^\iota & W^\varnothing
	}\]
induces the pull-back functor
	\begin{equation}\label{eq: pull back to V}\Isoc^\dagger(Y/W^\varnothing)\rightarrow\Isoc^\dagger_\pi(Y/V^\sharp).
	\end{equation}
For an object $\sE\in \Isoc^\dagger(Y/W^\varnothing)$, we often denote its pull-back to $\Isoc^\dagger_\pi(Y/V^\sharp)$ again by $\sE$, and its cohomology by $R\Gamma_\rig(Y/V^\sharp,\sE)_\pi:=R\Gamma_\rig(Y/(k^0,V^\sharp,\tau_\pi),\sE)$.

We recall the notion of a branch of the $p$-adic logarithm.
Let $\boldsymbol{\mu}\subset W^\times\subset V^\times$ be the image of $k^\times$ under the Teichm\"{u}ller map.
Then there is a decomposition $V^\times=\boldsymbol{\mu}(1+\frm)$.

\begin{definition}\label{def: branch}
	A {\it branch of the $p$-adic logarithm} over $K$ is a homomorphism $\log\colon K^\times\rightarrow K$ whose restriction to $1+\frm$ is given by the series
		\[\log (1+x)=-\sum_{n\geq 1}\frac{(-x)^n}{n}\hspace{10pt}(x\in\frm).\]
	For an element $\xi\in\frm\setminus\{0\}$ we denote by $\log_\xi$ the unique branch of the $p$-adic logarithm over $K$ such that $\log_\xi(\xi)=0$.
	More precisely, for any uniformizer $\pi\in V$ the element $\xi$ can be written as $\xi=\pi^mvw$ for some $m\geq 1$, $v\in 1+\frm$, and $w\in\boldsymbol{\mu}$.
	Thus if we set $\log_\xi(\pi):=\sum_{n\geq 1}\frac{(1-v)^n}{mn}$, it extends to a group homomorphism $\log_\xi\colon K^\times\rightarrow K$ such that $\log _\xi(\xi)=0$.
\end{definition}

\begin{remark}
	Note that any branch of the $p$-adic logarithm $\log$ over $K$ can be written in the form $\log_\xi$ for some $\xi\in\frm\setminus\{0\}$.
	Indeed, $\exp(-p^k\log(p))$ is well-defined as an element of $1+\frm$ for a sufficiently large integer $k$.
	Then one can see that $\log=\log_\xi$ for $\xi:=p^{p^k}\exp(-p^k\log(p))$.
\end{remark}

Now we are ready to define the Hyodo--Kato map.
Let $\cX$ be a strictly semistable weak formal log scheme over $V^\sharp$ and let $Y_\pi:=\cX\times_{V^\sharp,\tau_\pi}k^0$.
Composing the pull-back \eqref{eq: pull back to V} and the realization (see Corollary \ref{cor: realization of isoc}), we have a functor
	\begin{equation}\label{eq: dR functor}
	\Isoc^\dagger(Y_\pi/W^\varnothing)\rightarrow\MIC^\dagger(\cX/V^\sharp),
	\end{equation}
which preserves unipotent objects.
For an object $\sE\in\Isoc^\dagger(Y_\pi/W^\varnothing)$ we denote by $\sE_\dR$ the image of $\sE$ in $\MIC^\dagger(\cX/V^\sharp)$.

\begin{definition}\label{def: HK map}
	Let $\cX$ be a strictly semistable weak formal log scheme adic over $V^\sharp$ and $\pi\in V$ a uniformizer.
	We denote $Y_\pi:=\cX\times_{V^\sharp,\tau_\pi}k^0$.
	For an object $\sE\in\Isoc^\dagger(Y_\pi/W^\varnothing)$, we define the Hyodo--Kato cohomology of $\cX$ with coefficients in $\sE$ to be
		\[R\Gamma_\HK(\cX,\sE)_\pi:=R\Gamma_\HK(Y_\pi,\sE)\]
	equipped with the {\it normalized monodromy operator} $\bsN:=e^{-1}N$, where $e$ is the ramification index of $K$ over $\bbQ_p$.
	
	We take a local embedding datum of $Y_\pi$ over $\cS$ and let $\cZ_\bullet$ be the associated simplicial weak formal log scheme over $\cS$.
	For a branch of the $p$-adic logarithm $\log$ over $K$, let
		\[\Psi_{\pi,\log}\colon R\Gamma(\cZ_{\bullet,\bbQ},\sE_{\cZ_\bullet}\otimes\omega^\star_{\cZ_\bullet/W^\varnothing,\bbQ}[u])\rightarrow R\Gamma((\cZ_\bullet\times_{\cS,j_\pi}V^\sharp)_\bbQ,\sE_{\cZ_\bullet}\otimes\omega^\star_{\cZ_\bullet\times_{\cS,j_\pi}V^\sharp/V^\sharp,\bbQ})\]
	be the map defined by the natural surjection between differential forms and the association
		\[\Psi_{\pi,\log}(u^{[i]}):=\frac{(-\log \pi)^i}{i!}.\]
	This defines a morphism in $ D^b(\Mod_L)$
		\begin{equation}\label{eq: HK map}
		\Psi_{\pi,\log}\colon R\Gamma_\HK(\cX,\sE)_\pi\rightarrow R\Gamma_\rig(Y_\pi/V^\sharp,\sE)_\pi\cong R\Gamma_\dR(\cX,\sE_\dR),
		\end{equation}
	which we call the {\it Hyodo--Kato map} with respect to $\pi$ and $\log$.
\end{definition}

For the dependence of the choice of a branch of the $p$-adic logarithm, we have the following.

\begin{proposition}\label{prop: log}
	Let $\cX$ be a strictly semistable weak formal log scheme over $V^\sharp$, $\pi\in V$ a uniformizer, and $\sE\in\Isoc^\dagger(Y_\pi/W^\varnothing)$.
	For branches $\log$, $\log'$ of the $p$-adic logarithm, we have
		\begin{equation}\label{eq: change log}
		\Psi_{\pi,\log',K}=\Psi_{\pi,\log,K}\circ\exp\left(-\frac{\log'(\xi)}{\ord_p(\xi)}\cdot\bsN\right),
		\end{equation}
	where $\xi\in\frm\setminus\{0\}$ is an element satisfying $\log(\xi)=0$.
	Here the morphism $\exp\left(-\frac{\log'(\xi)}{\ord_p(\xi)}\cdot\bsN\right)$ makes sense by the same reason as \cite[Remark 4.4]{EY}.
	
	In particular, we have
		\begin{equation}\label{eq: log pi}
		\Psi_{\pi,\log,K}=\Psi_{\pi,\log_\pi,K}\circ\exp(-\log(\pi)\cdot N)
		\end{equation}
	for any branch of the $p$-adic logarithm $\log$ over $K$.
\end{proposition}

\begin{proof}
	The equality \eqref{eq: log pi} follows by
		\begin{eqnarray*}
		\Psi_{\pi,\log_\pi,K}\circ\exp(-\log(\pi)\cdot N)(u^{[i]})&=&\Psi_{\pi,\log_\pi,K}\left(\sum_{j=0}^i\frac{(-\log(\pi))^j}{j!}u^{[i-j]}\right)\\
		&=&\frac{(-\log(\pi))^i}{i!}\\
		&=&\Psi_{\pi,\log}(u^{[i]}).
		\end{eqnarray*}
	Then \eqref{eq: change log} follows from \eqref{eq: log pi}.
\end{proof}

Next we will give a base change result for the Hyodo--Kato map, and see the independence from the choice of a uniformizer.
Let $V'$ be a finite extension of $V$.
We use notations $K'$, $\frm'$, $k'$, $W'$, $\cS'$, $L'$, $e'$, $\boldsymbol{\mu}'$ for $V'$.
For a uniformizer $\pi'\in V'$, there uniquely exist units $v\in 1+\frm'$, $w\in\mathbf{\mu}'$ such that $\pi=vw\pi'^\ell$, where $\ell=e'e^{-1}$ is the ramification index of $K'$ over $K$.
Then there exists a commutative diagram
	\[\xymatrix{
	k'^0\ar[r]^-{\tau_{\pi'}}\ar[d]_-{\varrho_{\pi,\pi'}} & V'^\sharp\ar[d]^-{\varrho_{V'/V}}\\
	k^0\ar[r]^-{\tau_\pi} & V^\sharp,
	}\]
where the right vertical morphism $\varrho_{V'/V}$ is induced by the natural inclusion $V\hookrightarrow V'$, and the left vertical morphism $\varrho_{\pi,\pi'}$ is defined by the natural inclusion $k\hookrightarrow k'$ and $s\mapsto \overline{w}s'^\ell$.
Here $\overline{w}$ denotes the image of $w$ in $k'^\times$, and $s$ and $s'$ denote the canonical generators of the log structures of $k^0$ and $k'^0$, respectively.

Consider a commutative diagram
	\begin{equation}\label{eq: bc syn}
	\xymatrix{
	\cX'\ar[r]\ar[d]_-f & V'^\sharp\ar[d]^-{\varrho_{V'/V}}\\
	\cX\ar[r] & V^\sharp,
	}\end{equation}
where $\cX'$ and $\cX$ are strictly semistable weak formal log schemes adic over $V'^\sharp$ and $V^\sharp$, respectively.
Let $Y'_{\pi'}:=\cX'\times_{V'^\sharp,\tau_{\pi'}}k'^0$ and $Y_\pi:=\cX\times_{V^\sharp,\tau_\pi}k^0$.
Then $f$ induces a morphism $\overline{f}_{\pi,\pi'}\colon Y'_{\pi'}\rightarrow Y_\pi$ which is compatible with $\varrho_{\pi,\pi'}$.
The commutative diagram
	\[\xymatrix{
	Y'_{\pi'}\ar[d]_-{\overline{f}_{\pi,\pi'}}\ar[r] & k'^\varnothing \ar[d]\ar[r] & W'^\varnothing\ar[d]\\
	Y_\pi\ar[r] & k^\varnothing\ar[r] & W^\varnothing
	}\]
induces functors
	\begin{eqnarray}\label{eq: f pi pi'}
	\overline{f}^*_{\pi,\pi'}\colon \Isoc^\dagger(Y_\pi/W^\varnothing)\rightarrow \Isoc^\dagger(Y'_{\pi'}/W'^\varnothing),\\
	\nonumber \overline{f}^*_{\pi,\pi'}\colon F\Isoc^\dagger(Y_\pi/W^\varnothing)\rightarrow F\Isoc^\dagger(Y'_{\pi'}/W'^\varnothing),
	\end{eqnarray}
For an object $\sE\in\Isoc^\dagger(Y/W^\varnothing)$, we have an equality $(\overline{f}_{\pi,\pi'}^*\sE)_\dR=f^*\sE_\dR$.

\begin{proposition}\label{prop: unif}	
	Consider a diagram \eqref{eq: bc syn} and $\sE\in\Isoc^\dagger(Y_\pi/W^\varnothing)$.
	Then $f$ induces canonical morphisms
		\begin{eqnarray*}
		&&f^*_\HK\colon R\Gamma_\HK(\cX,\sE)_\pi\rightarrow R\Gamma_\HK(\cX',\overline{f}^*_{\pi,\pi'}\sE)_{\pi'},\\
		&&f^*_\dR\colon R\Gamma_\dR(\cX,\sE_\dR)\rightarrow R\Gamma_\dR(\cX',f^*\sE_\dR),
		\end{eqnarray*}
	which are compatible with the Hyodo--Kato maps in the sense that
		\begin{equation}\label{eq: HK bc}
		\Psi_{\pi',\log}\circ f_\HK^*=f_\dR^*\circ\Psi_{\pi,\log}
		\end{equation}
	holds for any branch of the $p$-adic logarithm.
	If a Frobenius structure $\Phi$ on $\sE$ is given, the morphism $f^*_\HK$ is compatible with Frobenius operators.
	
	In particular, the Hyodo--Kato map is independent of the choice of a uniformizer, up to canonical isomorphisms.
\end{proposition}

\begin{proof}
	The proof for the case of the trivial coefficients \cite[Proposition 4.6]{EY} is generalized without any difficulties.
	Here we give its detail for the benefit of readers.
	We denote $\sE':=\overline{f}^*_{\pi,\pi'}\sE$ to simplify notations.
	The map $f_\dR^*$ is defined obviously.
	Take local embedding data for $Y_\pi$ over $\cS$ and $Y'_{\pi'}$ over $\cS'$, and denote by $(Z_\bullet,\cZ_\bullet,i_\bullet)$ and $(Z'_\bullet,\cZ'_\bullet,i'_\bullet)$ the induced simplicial widenings.
	Let
		\[\widetilde{\varrho}_{\pi,\pi'}\colon\cS'\rightarrow\cS\]
	be the morphism defined by $s\mapsto ws'^\ell$.
	Let $\cZ''_{m,n}$ be the exactification of 
		\[Z''_{m,n}:=Z_m\times_{Y_\pi}Z'_n\hookrightarrow \cZ_m\times_\cS\cZ'_n,\]
	where the morphisms $Z'_n\rightarrow Y_\pi$ and $\cZ'_n\rightarrow \cS$ are defined through $\overline{f}_{\pi,\pi'}$ and $\widetilde{\varrho}_{\pi,\pi'}$, respectively.
	Then the natural projection induces a quasi-isomorphism
		\begin{equation}\label{eq: HK bisimp}
		R\Gamma_\HK(\cX',\sE')_{\pi'}=R\Gamma(\cZ'_{\bullet,\bbQ},\sE'_{\cZ'_\bullet}\otimes\omega^\star_{\cZ'_\bullet/W'^\varnothing,\bbQ}[u'])\xrightarrow{\cong}R\Gamma(\cZ''_{\bullet,\bullet,\bbQ},\sE'_{\cZ''_{\bullet,\bullet}}\otimes\omega^\star_{\cZ''_{\bullet,\bullet}/W'^\varnothing,\bbQ}[u']),
		\end{equation}
	where $[u']$ denotes the Kim--Hain construction defined by $du'^{[i]}:=-d\log s'\cdot u'^{[i-1]}$.
	Then $f_\HK$ is given by \eqref{eq: HK bisimp} and the map
		\[R\Gamma_\HK(\cX,\sE)_\pi=R\Gamma(\cZ_{\bullet,\bbQ},\sE_{\cZ_\bullet}\otimes\omega^\star_{\cZ_\bullet/W^\varnothing,\bbQ}[u])\rightarrow R\Gamma(\cZ''_{\bullet,\bullet,\bbQ},\sE'_{\cZ''_{\bullet,\bullet}}\otimes\omega^\star_{\cZ''_{\bullet,\bullet}/W'^\varnothing,\bbQ}[u'])\]
	defined by $u^{[i]}\mapsto \ell^iu'^{[i]}$. 
	It is straightforward to see that $f_\HK$ is compatible with differential, Frobenius, and normalized monodromy.
	
	To see the equality \eqref{eq: HK bc}, let $\widetilde{\cS}$ and $\widetilde{\cZ}_{m,n}$ be the exactifications of $k'^0\hookrightarrow\cS\times_{W^\varnothing}\cS'$ and $Z_{m,n}\hookrightarrow\cZ_m\times_{W^\varnothing}\cZ'_n$, respectively.
	Here the morphisms $k'^0\rightarrow \cS$ is defined through $\varrho_{\pi,\pi'}$.
	Then there exists an exact closed immersion $j_{\pi,\pi'}\colon V'^\sharp\hookrightarrow\widetilde{\cS}$ which makes the diagram
		\[\xymatrix{
		& V'^\sharp\ar[r]\ar[ld]_-{j_{\pi'}}\ar[d]^{j_{\pi,\pi'}} & V^\sharp\ar[d]^-{j_\pi}\\
		\cS' & \widetilde{\cS}\ar[l]\ar[r] & \cS
		}\]
	commutative, and the structure morphism $\widetilde{\cZ}_{m,n}\rightarrow W^\varnothing$ factors through $\widetilde{\cS}$.
	We define a quasi-isomorphism
		\[\gamma\colon R\Gamma(\widetilde{\cZ}_{\bullet,\bullet,\bbQ},\sE_{\widetilde{\cZ}_{\bullet,\bullet}}\otimes\omega^\star_{\widetilde{\cZ}_{\bullet,\bullet}/W^\varnothing,\bbQ}[u])\xrightarrow{\cong}R\Gamma(\widetilde{\cZ}_{\bullet,\bullet,\bbQ},\sE'_{\widetilde{\cZ}_{\bullet,\bullet}}\otimes\omega^\star_{\widetilde{\cZ}_{\bullet,\bullet}/W'^\varnothing,\bbQ}[u'])\]
	by
		\[\gamma(u^{[i]}):=\exp\left(\log\left(\frac{ws'^\ell}{s}\right)\cdot \frac{N}{\ell}\right)(\ell^iu'^{[i]})=\sum_{j=0}^i\frac{\left(\log\left(\frac{ws'^\ell}{s}\right)\right)^j}{j!}\ell^{i-j}u'^{[i-j]}.\]
	Note that
		\[\log\left(\frac{ws'^\ell}{s}\right)=\sum_{m\geq 1}\frac{(-1)^{m-1}}{m}\left(\frac{ws'^\ell}{s}-1\right)^m\]
	is well defined as global functions on $\widetilde{\cS}$ and $\widetilde{\cZ}_{m,n}$.
	It is straightforward to see that $\gamma$ is compatible with differentials, Frobenius, and normalized monodromy ($e^{-1}N$ on the left hand side and $e'^{-1}N$ on the right hand side).
	Let
		\begin{align*}
		\cX_m:=\cZ_m\times_{\cS,j_\pi}V^\sharp,&&\cX'_m:=\cZ'_m\times_{\cS',j_{\pi'}}V'^\sharp,&&\widetilde{\cX}_{m,n}:=\widetilde{\cZ}_{m,n}\times_{\widetilde{\cS},j_{\pi,\pi'}}V'^\sharp.
		\end{align*}
	We denote by $\sE_{\cX_m}$, $\sE'_{\cX'_m}$, and $\sE'_{\widetilde{\cX}_{m,n}}$ the modules with log connection on $\cX_{m,\bbQ}$, $\cX'_{m,\bbQ}$, and $\widetilde{\cX}_{m,n,\bbQ}$ induced by $\sE$ and $\sE'$.
	Then we have the following diagram:
		{\small
		\[\xymatrix{
		R\Gamma(\cZ_{\bullet,\bbQ},\sE_{\cZ_\bullet}\otimes\omega^\star_{\cZ_\bullet/W^\varnothing,\bbQ}[u])\ar[r]\ar@/^10mm/[rr]^-{\Psi_{\pi,\log}}\ar[d]\ar@{}[rd]|-{(*)}
		&R\Gamma(\widetilde{\cZ}_{\bullet,\bullet,\bbQ},\sE_{\widetilde{\cZ}_{\bullet,\bullet}}\otimes\omega^\star_{\widetilde{\cZ}_{\bullet,\bullet}/W^\varnothing,\bbQ}[u])\ar@/^3mm/[rd]_-\psi\ar[d]^-\gamma_-\cong
		&R\Gamma(\cX_{\bullet,\bbQ},\sE_{\cX_\bullet}\otimes\omega^\star_{\cX_\bullet/V^\sharp,\bbQ})\ar[d]\ar@{}[ld]|>>>>>>>>>>{(**)}\\
		R\Gamma(\cZ''_{\bullet,\bullet,\bbQ},\sE'_{\cZ''_{\bullet,\bullet}}\otimes\omega^\star_{\cZ''_{\bullet,\bullet}/W'^\varnothing,\bbQ}[u'])
		&R\Gamma(\widetilde{\cZ}_{\bullet,\bullet,\bbQ},\sE'_{\widetilde{\cZ}_{\bullet,\bullet}}\otimes\omega^\star_{\widetilde{\cZ}_{\bullet,\bullet}/W'^\varnothing,\bbQ}[u'])\ar[l]_-\cong\ar[r]_-{\psi'}
		& R\Gamma(\widetilde{\cX}_{\bullet,\bullet,\bbQ},\sE'_{\widetilde{\cX}_{\bullet,\bullet}}\otimes\omega^\star_{\widetilde{\cX}_{\bullet,\bullet}/V'^\sharp,\bbQ})\\
		R\Gamma(\cZ'_{\bullet,\bbQ},\sE'_{\cZ'_\bullet}\otimes\omega^\star_{\cZ'_\bullet/W'^\varnothing,\bbQ}[u'])\ar[ru]_-\cong\ar[rr]_-{\Psi_{\pi',\log}}\ar[u]^-\cong
		&
		&R\Gamma(\cX'_{\bullet,\bbQ},\sE'_{\cX_\bullet}\otimes\omega^\star_{\cX'_\bullet/V'^\sharp,\bbQ})\ar[u]^-\cong,
		}\]}
	where the morphisms $\psi$ and $\psi'$ are defined by
		\begin{align*}
		u^{[i]}\mapsto\frac{(-\log(\pi))^i}{i!}&&\text{and}&&u'^{[i]}\mapsto\frac{(-\log (\pi'))^i}{i!},
		\end{align*}
	as well as $\Psi_{\pi,\log}$ and $\Psi_{\pi',\log}$, respectively.
	The morphisms $f_\HK$ and $f_\dR$ are given as the compositions of the left and the right vertical morphisms.
	The commutativity of the triangles and the squares in the diagram except for $(*)$ and $(**)$ follows from the construction.
	The commutativity of the square $(*)$ follows from the fact that we have
		\[\log\left(\frac{ws'^\ell}{s}\right)=\log (1)=0\]
	on $\cZ''_{m,n,\bbQ}$.
	Noting that $\log(w)=0$, the commutativity of the triangle $(**)$ follows by
		\begin{eqnarray*}
		\psi'\circ\gamma(u^{[i]})&=&\psi'\left(\sum_{j=0}^i\frac{\left(\log\left(\frac{ws'^\ell}{s}\right)\right)^j}{j!}\ell^{i-j}u'^{[i-j]}\right)\\
		&=&\sum_{j=0}^i\frac{\ell^{i-j}\left(\log\left(\frac{w\pi'^\ell}{\pi}\right)\right)^j(-\log(\pi'))^{i-j}}{j!(i-j)!}\\
		&=&\sum_{j=0}^i\frac{\binom{i}{j}(-\log(v))^j(-\ell\log(\pi'))^{i-j}}{i!}\\
		&=&\frac{(-\log(v\pi'^\ell))^i}{i!}=\frac{(-\log(\pi))^i}{i!}=\psi(u^{[i]}).
		\end{eqnarray*}
	Now the commutativity of the above diagram shows the equality \eqref{eq: HK bc}.
\end{proof}

\begin{remark}
	By \cite[Proposition 1.10 and Remark 1.12]{HL}, for any (proper) strictly semistable weak formal log scheme $\cX$ adic over $V^\sharp$ and any finite extension $V'/V$ there exists a diagram \eqref{eq: bc syn} such that $\cX'_\bbQ\rightarrow\cX_\bbQ\times_{\Sp K}\Sp K'$ is an isomorphism.
	Then $f_\HK$ and $f_\dR$ are quasi-isomorphisms after tensoring with $L'$ and $K'$, respectively.
\end{remark}

\begin{proposition}\label{prop: HK qis}
	Let $\cX$ be a strictly semistable weak formal log scheme over $V^\sharp$ and $\sE\in\Isoc^\dagger(Y_\pi/W^\varnothing)^\unip$.
	The Hyodo--Kato map induces a quasi-isomorphism
		\[\Psi_{\pi,\log,K}\colon R\Gamma_\HK(Y_\pi,\sE)_K:=R\Gamma_\HK(Y_\pi,\sE)\otimes_LK\xrightarrow{\cong}R\Gamma_\dR(\cX,\sE_\dR).\]
\end{proposition}

\begin{proof}
	Since $\sE$ is unipotent, we may assume that $\sE=\sO_{Y_\pi/W^\varnothing}$.
	In this case, the assertion was proved in \cite[Proposition 4.9]{EY}.
\end{proof}

Lastly, we note that in the smooth case our Hyodo--Kato map recovers the base change map of non-logarithmic rigid cohomology.

\begin{proposition}\label{prop: smooth case}
	Consider a smooth weak formal scheme adic over $V$, and deote by $\cX$ and $\cX^\varnothing$ the weak formal log scheme given by endowing the pull-back log structure of $V^\sharp$ and the trivial log structure, respectively.
	Let $Y_\pi:=\cX\times_{V^\sharp,\tau_\pi}k^0$ and $Y^\varnothing:=\cX^\varnothing\times_{V^\varnothing}k^\varnothing$.
	Let $\sE\in\Isoc^\dagger(Y^\varnothing/W^\varnothing)$ be a unipotent object, and we denote its pull-back to $\Isoc^\dagger(Y_\pi/W^\varnothing)$ again by $\sE$.
	Then the monodromy operator on $H^n_\HK(\cX,\sE)_\pi$ is trivial, and there exists a canonical quasi-isomorphism
		\begin{equation}\label{eq: smooth}
		R\Gamma_\rig(Y^\varnothing/W^\varnothing,\sE)\xrightarrow{\cong}R\Gamma_\HK(\cX,\sE)_\pi,
		\end{equation}
	which makes the diagram
		\[\xymatrix{
		R\Gamma_\rig(Y^\varnothing/W^\varnothing,\sE)\ar[d]^-\cong\ar[r] & R\Gamma_\rig(Y^\varnothing/V^\varnothing,\sE)\ar[d]^-\cong\\
		R\Gamma_\HK(\cX,\sE)_\pi\ar[r]^{\Psi_{\pi,\log}} &R\Gamma_\dR(\cX,\sE_\dR)
		}\]
	commutative for any branch of the $p$-adic logarithm $\log$.
	Here the upper horizontal map is the base change map.
	When we consider a Frobenius structure on $\sE$, \eqref{eq: smooth} is compatible with Frobenius operators.
\end{proposition}

\begin{proof}
	Pull-back morphisms give a commutative diagram
		\begin{equation}\label{eq: diag smooth}\xymatrix{
		R\Gamma_\rig(Y^\varnothing/W^\varnothing,\sE)\ar[rr]\ar[d]\ar[rd]&&R\Gamma_\rig(Y^\varnothing/V^\varnothing,\sE)\ar[dd]^-\cong\\
		R\Gamma_\rig(Y_\pi/W^0,\sE) & R\Gamma_\rig(Y_\pi/W^\varnothing,\sE)\ar[l]\ar[ld]\ar[rd] &\\
		R\Gamma_\HK(Y_\pi,\sE)\ar[u]^-\cong\ar[rr]_-{\Psi_{\pi,\log}}& &R\Gamma_\rig(Y_\pi/V^\sharp,\sE)\\
		R\Gamma_\HK(\cX,\sE)_\pi\ar@{=}[u]&&R\Gamma_\dR(\cX,\sE_\dR).\ar@{=}[u]
		}\end{equation}
	To prove that \eqref{eq: smooth} is a quasi-isomorphism, it suffices to show that $R\Gamma_\rig(Y^\varnothing/W^\varnothing,\sE)\rightarrow R\Gamma_\rig(Y_\pi/W^0,\sE)$ is a quasi-isomrophism.
	Since $\sE$ is unipotent, it suffices to prove for the trivial coefficient.
	Take a local embedding datum for $Y^\varnothing$ over $W^\varnothing$ consisting of smooth weak formal schemes over $W$ equipped with the trivial log structures, and let $\cZ^\varnothing_\bullet$ be the associated simplicial weak formal log scheme over $W^\varnothing$.
	If we let $\cZ_m:=\cZ_m^\varnothing\times_{W^\varnothing}W^0$, then $R\Gamma_\rig(Y_\pi/W^0)$ is computed by $\cZ_\bullet$.
	Since we have $\omega^\star_{\cZ^\varnothing_\bullet/W^\varnothing}=\omega^\star_{\cZ_\bullet/W^0}$ by construction, we obtain the isomorphism $R\Gamma_\rig(Y^\varnothing/W^\varnothing)\cong R\Gamma_\rig(Y_\pi/W^0)$ as desired.
	
	In addition, by the commutativity of the diagram \eqref{eq: diag smooth}, we see that the natural map $H^n_\rig(Y_\pi/W^\varnothing,\sE)\rightarrow H_\HK^n(Y_\pi,\sE)$ is surjective for any $n$.
	This implies that the monodromy operator on $ H_\HK^n(Y_\pi,\sE)$ is trivial.
\end{proof}

%%%%%%%%%%%%%%%%%%%%%%%
\section{Syntomic coefficients}\label{sec: syn}
%%%%%%%%%%%%%%%%%%%%%%%
In this section, we introduce the category of syntomic coefficients.
We first recall some notion concerning filtered $(\varphi,N)$-modules.

\begin{definition}
	We define the category $\MF_K(\varphi,N)$ as follows:
	\begin{itemize}
	\item An object of $\MF_K(\varphi,N)$ is an object $M\in\Mod_L(\varphi,N)$ equipped with a separated exhaustive descending filtration $F^\bullet$ on $M_K:=M\otimes_LK$ by $K$-subspaces.
		We call $F^\bullet$ the {\it Hodge filtration} on $M$.
	\item A morphism $f\colon M'\rightarrow M$ in $\MF_K(\varphi,N)$ is a morphism in $\Mod_L(\varphi,N)$ such that the induced map $f_K\colon M'_K\rightarrow M_K$ preserves the Hodge filtrations.
	\end{itemize}
	
	We denote by $\MF^\fin_K(\varphi,N)$ the full subcategory of $\MF_K(\varphi,N)$ consisting of objects which have finite dimension and whose Frobenius operators are bijective.
	We call an object of $\MF^\fin_K(\varphi,N)$ a {\it filtered $(\varphi,N)$-module}.
\end{definition}

Let $M\in\MF^\fin_K(\varphi,N)$ and choose a basis $\{e_1,\ldots,e_m\}$ of $M$ over $L$.
We define a matrix $A=(a_{i,j})_{i,j}\in\GL_m(L)$ by $\varphi(e_i)=\sum_{j=1}^ma_{i,j}e_j$.
Then the valuation of $\det A$ is independent of the choice of a basis.
We call the number
	\begin{equation}
	t_N(M):=v_p(\det A)
	\end{equation}
the {\it Newton number} of $M$, where $v_p$ is the additive valuation normalized by $v_p(p)=1$.
In addition we define the {\it Hodge number} of $M$ to be
	\begin{equation}
	t_H(M):=\sum_{n\in\bbZ}n\cdot\dim_K\Gr_F^nM_K.
	\end{equation}

\begin{definition}\label{def: admissible MF}
	An object $M\in\MF^\fin_K(\varphi,N)$ is said to be {\it admissible} if the following conditions hold:
	\begin{itemize}
	\item $t_N(M)=t_H(M)$,
	\item For any subobject $M'\subset M$, $t_N(M')\geq t_H(M')$.
	\end{itemize}
	We denote by $\MF^\ad_K(\varphi,N)$ the full subcategory of $\MF^\fin_K(\varphi,N)$ consisting of admissible objects.
	In addition we denote by $\MF^\ad_K(\varphi)\subset\MF^\ad_K(\varphi,N)$ the full subcategory consisting of objects whose monodromy operators are trivial.
\end{definition}

\begin{remark}
	Properly, an object of $\MF^\fin_K(\varphi,N)$ should said to be admissible if it comes from a finite-dimensional representation of $\Gal(\ol K/K)$ over $\bbQ_p$, and the condition in Definition \ref{def: admissible MF} should be called weak admissibility.
	However, today the following theorem is well-known, that means admissibility and weak admissibility are equivalent to each other.
	So we abuse the terminology for short.
\end{remark}

\begin{theorem}[Colmez--Fontaine \cite{CF}, Colmez \cite{Co}, Fontaine \cite{Fo03}, Berger \cite{Berger2}, Kisin \cite{Kisin}]\label{thm: admissibility}
	There exist equivalences of categories
		\begin{eqnarray}
		\label{eq: st equiv}&&\MF_K^\ad(\varphi,N)\cong\Rep_{\bbQ_p}^\st(\Gal(\overline{K}/K)),\\
		\label{eq: crys equiv}&&\MF_K^\ad(\varphi)\cong\Rep_{\bbQ_p}^\crys(\Gal(\overline{K}/K)),
		\end{eqnarray}
	where categories on the right hand side are categories of semistable representations and crystalline representations (on $\bbQ_p$ of the absolute Galois group), respectively.
	Moreover they are neutral tannakian categories over $\bbQ_p$.
\end{theorem}

\begin{remark}
	The equivalence \eqref{eq: st equiv} in fact depends on the choice of the embedidng $B_\st\hookrightarrow B_\dR$ of Fontaine's $p$-adic period domains.
	More precisely it depends on the choice of a branch of the $p$-adic logarithm.
	The equivalence \eqref{eq: crys equiv} is canonical.
\end{remark}

We also note that admissibility is closed under extension, in the following sense.
A sequence in $\MF^\fin_K(\varphi,N)$
	\[0\rightarrow M'\rightarrow M\rightarrow M''\rightarrow 0\]
is said to be exact, if it is exact as a sequence of $L$-vector spaces and the induced maps $M'_K\rightarrow M_K$ and $M_K\rightarrow M''_K$ are strictly compatible with Hodge filtrations.
Note that this is equivalent to the condition that
	\[0\rightarrow F^nM'_K\rightarrow F^nM_K\rightarrow F^nM''_K\rightarrow 0\]
is exact for any $n\in\bbZ$.

\begin{proposition}[{\cite[Proposition 4.2.1 iii)]{Fo}}]\label{prop: numbers additive}
	For an exact sequence
		\[0\rightarrow M'\rightarrow M\rightarrow M''\rightarrow 0\]
	in $\MF_K^\fin(\varphi,N)$, we have
		\begin{align}
		t_N(M)=t_N(M')+t_N(M'')&&\text{and}&&t_H(M)=t_H(M')+t_H(M'').
		\end{align}
	Moreover, If two of $M$, $M'$, and $M''$ are admissible, then the other one is also admissible.
\end{proposition}

The Frobenius operators and monodromy operators on the tensor products and internal Hom are given by the same rule as in Definition \ref{def: phi N module}.
The Hodge filtrations are defined by
	\begin{eqnarray}
	\label{eq: tensor MF}&&F^n(M_K\otimes_KM'_K):=\sum_{j\in\bbZ}F^jM_K\otimes_K F^{n-j}M'_K,\\
	\label{eq: Hom MF}&&F^n\Hom_K(M'_K,M_K):=\left\{f\in\Hom_K(M'_K,M_K)\mid f(F^jM'_K)\subset F^{j+n}M_K\text{ for any $j\in\bbZ$}\right\}.
	\end{eqnarray}

Let $\pi\in V$ be a uniformizer.
In order to consider Hodge filtrations, we need the properness of strictly semistable log schemes, but allow horizontal divisor.

\begin{definition}
	Let $\cX$ be a proper strictly semistable weak formal log scheme adic over $V^\sharp$.
	We define the category of (unipotent) syntomic coefficients $\Syn_\pi(\cX/V^\sharp)$ as follows:
	\begin{itemize}
	\item An object of $\Syn_\pi(\cX/V^\sharp)$ is a triple $\frE=(\sE,\Phi,F^\bullet)$, where $(\sE,\Phi)\in F\Isoc^\dagger(Y_\pi/W^\varnothing)^\unip$ and $F^\bullet$ is a Hodge filtration on $\sE_\dR$.
	\item A morphism $f\colon(\sE',\Phi',F'^\bullet)\rightarrow(\sE,\Phi,F^\bullet)$ in $\Syn_\pi(\cX/V^\sharp)$ is a morphism $f\colon(\sE',\Phi')\rightarrow(\sE,\Phi)$ in $F\Isoc^\dagger(Y_\pi/W^\varnothing)^\unip$ such that the induced morphism $f_\dR\colon\sE'_\dR\rightarrow\sE_\dR$ preserves the Hodge filtrations.
	\end{itemize}
	We say a sequence
	\[0\rightarrow \frE'\rightarrow\frE\rightarrow\frE''\rightarrow 0\]
	in $\Syn_\pi(\cX/V^\sharp)$ is exact if the sequence
	\[0\rightarrow\sE'\rightarrow\sE\rightarrow\sE''\rightarrow 0\]
	of the underlying log overconvergent isocrystals is exact and the morphisms $\sE'_\dR\rightarrow\sE_\dR$ and $\sE_\dR\rightarrow\sE''_\dR$ are strictly compatible with the Hodge filtrations.
	This makes $\Syn_\pi(\cX/V^\sharp)$ an exact category.
\end{definition}

For two objects $\frE,\frE'\in\Syn_\pi(\cX/V^\sharp)$, we define the tensor product $\frE\otimes\frE'$ and internal Hom $\sheafhom(\frE',\frE)$ by those on $F\Isoc^\dagger(Y_\pi/W^\varnothing)^\unip$ and the Hodge filtrations given by a similar manner as \eqref{eq: tensor MF} and \eqref{eq: Hom MF}.

Consider a commutative diagram
	\begin{equation}\label{eq: diagram pull back syn}\xymatrix{
	\cX'\ar[r]\ar[d]_-f & V'^\sharp\ar[d]^-{\varrho_{V'/V}} \\
	\cX\ar[r] & V^\sharp,
	}\end{equation}
where $\cX'$ and $\cX$ are proper strictly semistable weak formal log schemes adic over $V'^\sharp$ and $V^\sharp$, respectively.
Then for uniformizers $\pi\in V$ and $\pi'\in V'$, the pull-back of log overconvergent $F$-isocrystals \eqref{eq: f pi pi'} induces a functor
	\begin{equation}\label{eq: Syn bc}
	f_{\pi,\pi'}^*\colon\Syn_\pi(\cX/V^\sharp)\rightarrow\Syn_{\pi'}(\cX'/V'^\sharp).
	\end{equation}

To conclude this section, we will introduce syntomic coefficients associated to filtered $(\varphi,N)$-modules.
We first give a naive definition.

\begin{definition}
	Let $\cX$ be a proper strictly semistable weak formal log scheme adic over $V^\sharp$, $\pi\in V$ a uniformizer, and $M\in\MF^{\mathrm{fin}}_K(\varphi,N)$.
	If we denote by $M^a$ the pseudo-constant log overconvergent $F$-isocrystal defined by the underlying $(\varphi,N)$-module of $M$,
	\begin{eqnarray*}
	M\otimes\cO_{\cX_\bbQ}&\xrightarrow{\cong}&(1+\sJ_\cX)\backslash(\sN^s_\cX\times(M\otimes_L\cO_{\cX_\bbQ}))=M^a_\dR\\
	m\otimes f&\mapsto&(\pi,m\otimes f)
	\end{eqnarray*}
	gives an isomorphism
	\begin{equation}\label{eq: MadR}
	(M_K\otimes_K\cO_{\cX_\bbQ},1\otimes d)\cong M^a_\dR
	\end{equation}
	in $\MIC^\dagger(\cX/V^\sharp)$.
	The Hodge filtration on $M_K$ defines a filtration on $M^a_\dR$ by \eqref{eq: MadR}.
	This gives an object of $\Syn_\pi(\cX/V^\sharp)$, which we denote by $M^a_\pi$.
	
	We denote by $\frO_\cX$ the constant syntomic coefficient associated to the unit object $L\in\MF^\ad_K(\varphi)$.
\end{definition}

However the functors $M\mapsto M^a_\pi$ are not compatible with the functor \eqref{eq: Syn bc}.
In order to give a compatible construction, we need to introduce the ``twist'' of Hodge filtrations. 
In addition, in order to balance against the normalization of monodromy operator on the Hyodo--Kato cohomology as in Definition \ref{def: HK map}, we need to multiply the monodromy operator by the ramification index.
These observations leads the following definition. 

\begin{definition}
	Let $\pi\in V$ be a uniformizer and $\log\colon K^\times\rightarrow K$ a branch of the $p$-adic logarithm.
	\begin{enumerate}
	\item For $M\in\MF^\fin_K(\varphi,N)$, we define an object $M(\pi,\log)\in\MF^\fin_K(\varphi,N)$ as follows:
	The underlying $\varphi$-module of $M(\pi,\log)$ is given by that of $M$.
	The monodromy operator $\wt{N}$ on $M(\pi,\log)$ is given by $\wt{N}:=eN$, where $e$ is the ramification index of $K$ over $\bbQ_p$.
	We define the Hodge filtration on $M(\pi,\log)_K=M_K$ by
	\[F^nM(\pi,\log)_K:=\exp(-\log\pi\cdot \wt{N})(F^nM_K).\]
	\item Let $\cX$ be a strictly semistable weak formal log scheme over $V^\sharp$.
		For $M\in\MF^\fin_K(\varphi,N)$, we call $M(\pi,\log)_\pi^a\in\Syn_\pi(\cX/V^\sharp)$ the  {\it pseudo-constant syntomic coefficient} associated to $M$ with respect to $\pi$ and $\log$.
		If $M\in\MF^\ad_K(\varphi)$, we call $M(\pi,\log)_\pi^a$ the {\it constant syntomic coefficient} associated to $M$.
	\end{enumerate}
\end{definition}

Then we have the following compatibility:

\begin{proposition}
	Consider a commutative diagram
	\[\xymatrix{
	\cX'\ar[r]\ar[d]_-f & V'^\sharp\ar[d]^-{\varrho_{V'/V}} \\
	\cX\ar[r] & V^\sharp,
	}\]
	as in \eqref{eq: diagram pull back syn}.
	Let $\pi\in V$ and $\pi'\in V'$ be uniformizers, and $\log\colon K^\times\rightarrow K$ a branch of the $p$-adic logarithm.
	For $M\in\MF^\fin_K(\varphi,N)$ we denote by $M':=M\otimes_LL'\in\MF^\fin_{K'}(\varphi,N)$ the object induced from $M$ obviously.
	Then we have a canonical isomorphism
	\[ f_{\pi,\pi'}^*M(\pi,\log)^a_\pi\cong M'(\pi',\log)^a_{\pi'}.\]
\end{proposition}

\begin{proof}
	Let $e$ and $e'$ be the ramification indexes of $K$ and $K'$ over $\bbQ_p$, respectively.
	Let $\ell:=e'-e$ be the ramification index of $K'$ over $K$.
	For $(Z,\cZ,i,h,\theta)\in\OC(Y'_{\pi'}/W'^\varnothing)$, we define an isomorphism $\eta_\cZ\colon(\ol f^*_{\pi,\pi'}M(\pi,\log)^a)_\cZ\xrightarrow{\cong} M'(\pi',\log)^a_\cZ$ to be
	\begin{eqnarray*}
	(1+\sJ_\cZ)\backslash(\sN_\cZ^{s^\ell}\times(M\otimes_L\cO_{\cZ_\bbQ}))&\rightarrow& (1+\sJ_\cZ)\backslash(\sN_\cZ^{s}\times(M'\otimes_{L'}\cO_{\cZ_\bbQ}))\\
	(\alpha^\ell,m\otimes f)&\mapsto&(\alpha,m\otimes f).
	\end{eqnarray*}
	This is well-defined because for any $g\in 1+\sJ_\cZ$ we have
	\begin{eqnarray*}
	\eta_\cZ(g^\ell\alpha^\ell,m\otimes f)&=&\eta_\cZ(\alpha^\ell,\exp(\log g^\ell\cdot eN)(m\otimes f))\\
	&=&(\alpha,\exp(\log g\cdot e'N)(m\otimes f))=(g\alpha,m\otimes f).
	\end{eqnarray*}
	The isomorphisms $\eta_\cZ$ induces an isomorphism
		\[\eta\colon\ol f^*_{\pi,\pi'}M(\pi,\log)^a\xrightarrow{\cong} M'(\pi',\log)^a\]
	in $F\Isoc^\dagger(Y_{\pi'}/W'^\varnothing)$.
	
	The composition
	\begin{equation}\label{eq: MadR isom}
	f^*(M(\pi,\log)_K\otimes_K\cO_{\cX_\bbQ})\xrightarrow{\cong}(\ol f^*_{\pi,\pi'}M(\pi,\log)^a)_\dR\xrightarrow[\eta_\dR]{\cong}M'(\pi',\log)^a_\dR\xleftarrow{\cong}M'(\pi',\log)_{K'}\otimes_{K'}\cO_{\cX'_\bbQ}
	\end{equation}
	sends $m\otimes f\mapsto\exp(\log(\pi/\pi'^\ell)\cdot eN)(m\otimes f)$.
	Since 
	\[\exp(-\log\pi'\cdot e'N)=\exp(\log(\pi/\pi'^\ell)\cdot eN)\circ\exp(-\log\pi\cdot eN),\]
	we see that \eqref{eq: MadR isom} preserves the Hodge filtrations.
\end{proof}

%%%%%%%%%%%%%%%%%%%%%%%%%%
%
\section{$p$-adic Hodge cohomology and syntomic cohomology with coefficients}\label{sec: syn coh}
%
%%%%%%%%%%%%%%%%%%%%%%%%%%
In this section, we will give the definition of the $p$-adic Hodge cohomology and the syntomic cohomology with syntomic coefficients.
Then we will see that the syntomic cohomology computes the extension group of syntomic coefficients.

Let $H\Mod_K$ be the exact category of $K$-vector spaces with descending separated exhaustive filtration.
(The notation $H$ stands for Hodge filtrations. We do not use $F$ since it is confusing with Frobenius structure.)

We first define the dg category of (admissible) semistable $p$-adic Hodge complexes following \cite{DN}.
Note that they considered Galois action in order to include potentially semistable $p$-adic Hodge complexes.
Their construction without considering Galois action works as well.

\begin{definition}[{\cite[\S 2.10]{DN}}]
	\begin{enumerate}
	\item We define the dg category $ \sD_{\st,K}$ to be the homotopy limit
			\begin{equation*}
			 \sD_{\st,K}:=\holim(\sD^b(\Mod^\fin_L(\varphi,N))\xrightarrow{\Xi_\HK}\sD^b(\Mod_K)\xleftarrow{\Xi_\dR}\sD^b(H\Mod_K)),
			\end{equation*}
		where $\sD^b$ denote the bounded derived dg categories, $\Xi_\dR$ is the forgetful functor, and $\Xi_\HK$ is given by forgetting $\varphi,N$ and tensoring with $K$.
		Note that an object of $\mathsf{M}\in \sD_{\st,K}$ consists of objects $\mathsf{M}_\HK\in \sD^b(\Mod^\fin_L(\varphi,N))$, $\mathsf{M}_\dR\in\sD^b(H\Mod_K)$, and a quasi-isomorphism $\alpha_{\mathsf{M}}\colon\Xi_\HK(\mathsf{M}_\HK)\xrightarrow{\cong}\Xi_\dR(\mathsf{M}_\dR)$.
		The cohomology groups $H^n(\mathsf{M}_\HK)\otimes_LK\cong H^n(\mathsf{M}_\dR)$ define an object of $\MF_K(\varphi,N)$, which we denote by $H^n(\mathsf{M})$.
	\item An object $\mathsf{M}=(\mathsf{M}_\HK,\mathsf{M}_\dR,\alpha_{\mathsf{M}})$ is said to be {\it admissible} if $\mathsf{M}_\dR$ is strict (i.e.\ the differentials are strictly compatible with the filtration) and the cohomology groups $H^n(\mathsf{M})$ are admissible.
		We denote by $\sD_{\st,K}^\ad$ the full subcategory of $\sD_{\st,K}$ consisting of admissible objects.
	\item We denote by $D_{\st,K}$ and $D_{\st,K}^\ad$ the homotopy categories of $\sD_{\st,K}$ and $\sD_{\st,K}^\ad$, respectively.
	\end{enumerate}
\end{definition}

\begin{theorem}[{\cite[Theorem 2.17]{DN}}]
	There exists a canonical equivalence of dg categories
		\begin{equation}
		 \sD^b(\MF^\ad_K(\varphi,N))\xrightarrow{\cong} \sD^\ad_{\st,K},
		\end{equation}
	and hence an equivalence between their homotopy categories
		\begin{equation}\label{eq: adm complex}
		D^b(\MF^\ad_K(\varphi,N))\xrightarrow{\cong}D^\ad_{\st,K}.
		\end{equation}
\end{theorem}

Next we give the definition of $p$-adic Hodge cohomology and the admissibility of syntomic coefficients.

\begin{definition}
	Let $\cX$ be a proper strictly semistable weak formal log scheme adic over $V^\sharp$ and $\frE=(\sE,\Phi,F^\bullet)\in\Syn_\pi(\cX/V^\sharp)$.
	We define the Hyodo--Kato cohomology of $\cX$ with coefficients in $\frE$ to be
		\[R\Gamma_\HK(\cX,\frE)_\pi:=R\Gamma_\HK(Y_\pi,(\sE,\Phi))\]
	endowed with the Frobenius operator and the normalized monodromy operator $\bsN=e^{-1}N$.	
	We also define the de Rham cohomology of $\cX$ with coefficients in $\frE$ to be
		\[R\Gamma_\dR(\cX,\frE):=R\Gamma_\dR(\cX,\sE_\dR)\]
	endowed with the Hodge filtration.
	By Proposition \ref{prop: HK qis} and Remark \ref{rem: dg eq},  the Hyodo--Kato map
		\[\Psi_{\pi,\log}\colon R\Gamma_\HK(\cX,\frE)_\pi\rightarrow R\Gamma_\dR(\cX,\frE)\]
	defines an object of $ D_{\st,K}$.
	We denote it by $R\Gamma_{p\mathrm{H}}(\cX,\frE)_{\pi,\log}$ and call the {\it $p$-adic Hodge cohomology} of $\cX$ with coefficients in $\frE$ with respect to $\pi$ and $\log$.
	
	For the trivial coefficient $\frO_\cX$, we denote $R\Gamma_{p\mathrm{H}}(\cX)_{\pi,\log}:=R\Gamma_{p\mathrm{H}}(\cX,\frO_\cX)_{\pi,\log}$.
\end{definition}

\begin{definition}	
	An object $\frE\in\Syn_\pi(\cX/V^\sharp)$ is said to be {\it admissible} if the $p$-adic Hodge cohomology $R\Gamma_{p\mathrm{H}}(\cX,\frE)_{\pi,\log}$ lies in $D_{\st,K}^\ad$ for some choice of $\log$.
	We denote by $\Syn_\pi^\ad(\cX/V^\sharp)$ the full subcategory of $\Syn_\pi(\cX/V^\sharp)$ consisting of admissible objects.
\end{definition}

\begin{proposition}\label{prop: adm log}
	If $\frE\in\Syn_\pi(\cX/V^\sharp)$ is admissible, then $R\Gamma_{p\mathrm{H}}(\cX,\frE)_{\pi,\log}$ is admissible for any choice of $\log$.
\end{proposition}

\begin{proof}
	The strictness of $R\Gamma_\dR(\cX,\frE)_{\pi,\log}$ is independent of the choice of $\log$ by definition.
	Recall that for two choices of $\log$ and $\log'$, we have an equality
		\[\Psi_{\pi,\log',K}=\Psi_{\pi,\log,K}\circ\exp(a\mathbf{N})\colon R\Gamma_\HK(\cX,\frE)_{\pi,K}\xrightarrow{\cong}R\Gamma_\dR(\cX,\frE)\]
	for some $a\in K$ by Proposition \ref{prop: log}.
	For a subobject $M\subset H^n_{\HK}(\cX,\frE)_\pi$ as a $(\varphi,N)$-module, we denote by $F^\bullet$ and $F'^\bullet$ the filtration on $M_K$ induced by the Hodge filtration on $H^n_\dR(\cX,\frE)$ via $\Psi_{\pi,\log}$ and $\Psi_{\pi,\log'}$, respectively.
	Noting that $\exp(a\mathbf{N})$ induces an automorphism on $M$ (as a vector space), we have
		\begin{eqnarray*}
		F^mM_K&:=&\Psi_{\pi,\log,K}^{-1}(F^mH^n_\dR(\cX,\frE))\cap M_K,\\
		F'^mM_K&:=&\Psi_{\pi,\log',K}^{-1}(F^mH^n_\dR(\cX,\frE))\cap M_K=\exp(-a\mathbf{N})\circ\Psi_{\pi,\log,K}^{-1}(F^mH^n_\dR(\cX,\frE))\cap M_K\\
		&=&\exp(-a\mathbf{N})(F^mM_K).
		\end{eqnarray*}
	Thus the Hodge numbers of $M_K$ with respect to $F^\bullet$ and $F'^\bullet$ coincides each other, and we see that $H^n_{p\mathrm{H}}(\cX,\frE)_{\pi,\log'}$ is admissible if and only if $H^n_{p\mathrm{H}}(\cX,\frE)_{\pi,\log}$ is so.
\end{proof}

Consider a commutative diagram
	\begin{equation*}
	\xymatrix{
	\cX'\ar[r]\ar[d]_-f & V'^\sharp\ar[d]^-{\varrho_{V'/V}} \\
	\cX\ar[r] & V^\sharp
	}\end{equation*}
as in \eqref{eq: diagram pull back syn}.
Namely $\cX'$ and $\cX$ are proper strictly semistable weak formal log schemes adic over $V'^\sharp$ and $V^\sharp$, respectively.
Let $\pi\in V$ and $\pi'\in V'$ be uniformizers, and $\log\colon K^\times\rightarrow K$ be a branch of the $p$-adic logarithm.
For an object $\frE\in\Syn_\pi(\cX/V^\sharp)$, the morphisms
	\begin{eqnarray*}
	&&f_\HK^*\colon R\Gamma_\HK(\cX,\frE)_\pi\rightarrow R\Gamma_\HK(\cX',f_{\pi,\pi'}^*\frE)_{\pi'},\\
	&&f_\dR^*\colon R\Gamma_\dR(\cX,\frE)\rightarrow R\Gamma_\dR(\cX',f_{\pi,\pi'}^*\frE)
	\end{eqnarray*}
in Proposition \ref{prop: unif} induce a morphism between $p$-adic Hodge cohomology
	\begin{equation}
	f_{p\mathrm{H}}^*\colon R\Gamma_{p\mathrm{H}}(\cX,\frE)_{\pi,\log}\rightarrow R\Gamma_{p\mathrm{H}}(\cX',f_{\pi,\pi'}^*\frE)_{\pi',\log}.
	\end{equation}

\begin{proposition}\label{prop: pH pi}
	Consider the diagram \eqref{eq: diagram pull back syn} as above, and assume that $f$ induces an isomorphism $\cX'_\bbQ\xrightarrow{\cong}\cX_\bbQ\times_{\Sp K}\Sp K'$.
	Then $f_{p\mathrm{H}}^*$ induces an isomorphism
		\[R\Gamma_{p\mathrm{H}}(\cX,\frE)_{\pi,\log}\otimes_LL'\xrightarrow{\cong}R\Gamma_{p\mathrm{H}}(\cX',f_{\pi,\pi'}^*\frE)_{\pi',\log}\]
	in $D_{\st,K'}$.
	In particular, $\frE$ is admissible if and only if $f_{\pi,\pi'}^*\frE$ is so.
\end{proposition}

\begin{proof}
	Since $R\Gamma_\dR(\cX,\frE)\otimes_KK'\rightarrow R\Gamma_\dR(\cX',f_{\pi,\pi'}^*\frE)$ is an isomorphism, we obtain the assertion.
\end{proof}

Next we will see that a syntomic coefficient is admissible if it can be written as an iterated extension of pseudo-constant syntomic coefficients.
We first consider the case of trivial coefficient, passing through crystalline construction.

\begin{lemma}\label{lem: trivial coeff}
	Assume that $\cX$ is the weak completion along the special fiber of a proper strictly semistable log scheme over $V^\sharp$.
	Then the trivial coefficient $\frO_\cX$ is admissible.
\end{lemma}

\begin{proof}
	The strictness of $R\Gamma_\dR(\cX)$ follows from \cite[Proposition 4.7.9]{Ts} or \cite[Theorem 1.3]{LP}.
	By \cite[Proposition 6.13]{EY}, the Hyodo--Kato map $\Psi_{\pi,\log_\pi}$ is compatible with the crystalline Hyodo--Kato map $\Psi_\pi^\crys$.
	Therefore the semistable conjecture which is now a theorem \cites{Ts,Fa,Ni,Bei2} implies that $R\Gamma_{p\mathrm{H}}(\cX)_{\pi,\log_\pi}$ is admissible.
\end{proof}

The $p$-adic Hodge cohomology of a pseudo-constant syntomic coefficient is computed as follows.
The following is a crucial property of the Hyodo--Kato map.

\begin{lemma}\label{lem: hk map of pconst}
	Let $\cX$ be a proper strictly semistable weak formal log scheme adic over $V^\sharp$ and $M\in\MF_K^\fin(\varphi,N)$.
	Then the following diagram commutes:
		\begin{equation}\label{eq: diag hk}
		\xymatrix{
		H_\HK^i(\cX,M^a_\pi)_{\pi,K}\ar[r]^-\cong\ar[d]^-{\Psi_{\pi,\log,K}} & M_K\otimes H^i_\HK(\cX)_{\pi,K}\ar[d]^-{\exp(-\log(\pi)\cdot N)\otimes\Psi_{\pi,\log,K}}\\
		H_\dR^i(\cX,M^a_\pi)\ar[r]^-\cong & M_K\otimes H^i_\dR(\cX).
		}\end{equation}
\end{lemma}

\begin{proof}
	As in the proof of Proposition \ref{prop: pseudo-const}, an element $\alpha\in H^n_\HK(\cX)_\pi$ is represented via \eqref{eq: HK Tot 1} by a family
	\[\Bigl(\sum_{j\geq 0}\alpha_{j,m,n}u^{[j]},\sum_{j\geq 0}\beta_{j,m,n}u^{[j]}\Bigr)_{m,n}\]
	with
	\begin{align*}
	\alpha_{j,m,n}\in\Gamma(\cZ_{m,n,\bbQ},\omega^{i-m}_{\cZ_{m,n}/W^\varnothing,\bbQ}),&&
	\beta_{j,m,n}\in\Gamma(\cZ_{m,n,\bbQ},\omega^{i-m-1}_{\cZ_{m,n}/W^\varnothing,\bbQ}),
	\end{align*}
	satisfying the conditions \eqref{eq: cocycle alpha} and \eqref{eq: cocycle beta}.
	
	Also we have seen that an element $\mu\otimes\alpha\in M\otimes H^i_\HK(\cX)_\pi$ corresponds to an element of $H^i_\HK(\cX,M_\pi^a)_\pi$ represented via \eqref{eq: HK Tot 2} by the family
	\[\Bigl(\sum_{j\geq 0}\gamma_{j,m,n}u^{[j]},\sum_{j\geq 0}\delta_{j,m,n}u^{[j]}\Bigr)_{m,n}\]
	with
	\begin{align*}
	\gamma_{j,m,n}=\sum_{\ell=0}^j\binom j{\ell}N^\ell(\mu)\otimes\alpha_{j-\ell,m,n},&&
	\delta_{j,m,n}=\sum_{\ell=0}^j\binom j{\ell}N^\ell(\mu)\otimes\beta_{j-\ell,m,n}.
	\end{align*}
	
	Then we have
	\begin{align*}
	&\Psi_{\pi,\log}\left(\Bigl(\sum_{j\geq 0}\gamma_{j,m,n}u^{[j]})_{m,n},\sum_{j\geq 0}\delta_{j,m,n}u^{[j]}\Bigr)_{m,n}\right)\\
	&=\Bigl(\sum_{j\geq0}\sum_{\ell=0}^j\frac{(-\log(\pi))^j}{j!}\binom{j}{\ell}N^\ell(\mu)\otimes\alpha_{j-\ell,m,n},\sum_{j\geq 0}\sum_{\ell=0}^j\frac{(-\log(\pi))^j}{j!}\binom{j}{\ell}N^\ell(\mu)\otimes\beta_{j-\ell,m,n}\Bigr)_{m,n}\\
	&=\Bigl(\sum_{\ell,\nu\geq 0}\frac{(-\log(\pi))^{\ell+\nu}}{\ell!\nu!}N^\ell(\mu)\otimes\alpha_{\nu,m,n},\sum_{\ell\nu\geq 0}\frac{(-\log(\pi))^{\ell+\nu}}{\ell!\nu!}N^\ell(\mu)\otimes\beta_{\nu,m,n}\Bigr)_{m,n},
	\end{align*}
	whose class in $H^i_\dR(\cX,M^a_\pi)$ corresponds to the class in $M_K\otimes H^i_\dR(\cX)$ represented by
		\begin{align*}
		&\sum_{\ell\geq 0}\frac{(-\log(\pi))^\ell}{\ell!}N^\ell(\mu)\otimes \Bigl(\sum_{\nu\geq 0}\frac{(-\log(\pi))^\nu}{\nu!}\alpha_{\nu,m,n},\sum_{\nu\geq 0}\frac{(-\log(\pi))^\nu}{\nu!}\beta_{\nu,m,n}\Bigr)_{m,n}\\
		&=\exp(-\log(\pi)\cdot N)(\mu)\otimes\Psi_{\pi,\log}\left(\Bigl(\sum_{j\geq 0}\alpha_{j,m,n}u^{[j]},\sum_{j\geq 0}\beta_{j,m,n}u^{[j]}\Bigr)_{j,m,n}\right).
		\end{align*}
	This shows the commutativity of \eqref{eq: diag hk}.
\end{proof}

\begin{proposition}\label{prop: pH tensor}
	Let $\cX$ be a proper strictly semistable weak formal log scheme adic over $V^\sharp$ and $M\in\MF_K^\fin(\varphi,N)$.
	Then we have a canonical isomorphism in $\MF_K^\fin(\varphi,N)$
		\begin{equation}\label{eq: hk tensor}
		H^i_{p\mathrm{H}}(\cX,M(\pi,\log)_\pi^a)_{\pi,\log}\cong M\otimes H^i_{p\mathrm{H}}(\cX)_{\pi,\log}.
		\end{equation}
\end{proposition}

\begin{proof}
	Immediately follows from Lemma \ref{lem: hk map of pconst}.
\end{proof}

\begin{corollary}\label{cor: example of adm}
	Assume that $\cX$ is the weak completion of a proper strictly semistable log scheme over $V^\sharp$.
	An object $\frE\in\Syn_\pi(\cX/V^\sharp)$ is admissible if it can be written as iterated extension of pseudo-constant syntomic coefficients.
\end{corollary}
	
\begin{proof}
	We first note that, for a distinguished triangle in the filtered derived category, if two of them are strict then another one is also strict.
	By this fact and Proposition \ref{prop: numbers additive} we may assume that $\frE=M(\pi,\log)^a_\pi$ is pseudo-constant.
	The strictness of $R\Gamma_\dR(\cX)$ (i.e. $E_1$-degeneration of the Hodge-de Rham spectral sequence for the trivial coefficient) gives
		\begin{equation}\label{eq: dim sum}
		\sum_{i\in \bbZ}\dim H^{n-i}(\cX_\bbQ,\omega^i_{\cX/V^\sharp,\bbQ})=\dim H^n_\dR(\cX).
		\end{equation}
	For $M(\pi,\log)^a_\pi$, we have
		\[H_\dR^n(\cX,M(\pi,\log)_\pi^a)\cong M(\pi,\log)_K\otimes H_\dR^n(\cX),\]
	and the $E_1$-term of the Hodge-de Rham spectral sequence $E_1^{i,j}\Rightarrow H^{i+j}_\dR(\cX,M(\pi,\log)^a_\pi)$ is given by
		\[E_1^{i,j}=H^{i+j}(\cX_\bbQ,\Gr^i_F(M(\pi,\log)_K\otimes\omega^\bullet_{\cX/V^\sharp,\bbQ}))=\bigoplus_{k\in \bbZ}\Gr_F^{i-k}M(\pi,\log)_K\otimes H^{i+j-k}(\cX_\bbQ,\omega^k_{\cX/V^\sharp,\bbQ}).\]
	By \eqref{eq: dim sum} we have
		\begin{eqnarray*}
		\sum_{i\in\bbZ}\sum_{k\in\bbZ}\dim \Gr_F^{i-k}M(\pi,\log)_K\otimes H^{n-k}(\cX_\bbQ,\omega^k_{\cX/V^\sharp,\bbQ})&=&(\dim M(\pi,\log)_K)\cdot(\dim H^n_\dR(\cX))\\
		&=&\dim H^n_\dR(\cX,M(\pi,\log)_\pi^a).
		\end{eqnarray*}
	This shows that the spectral sequence degenerates at $E_1$ and hence $R\Gamma_\dR(\cX,M(\pi,\log)_\pi^a)$ is strict.
	Moreover, by the isomorphism \eqref{eq: hk tensor} and Lemma \ref{lem: trivial coeff} we see that $H^n_{p\mathrm{H}}(\cX,M(\pi,\log)^a_\pi)_{\pi,\log_\pi}$ is admissible.
\end{proof}

As a corollary of Proposition \ref{prop: isoc on point} and Proposition \ref{prop: pH tensor}, we obtain the following.

\begin{corollary}\label{cor: syn for V}
	For any choice of $\pi$ and $\log$, the functor
		\begin{align*}
		\MF_K^\ad(\varphi,N)\rightarrow\Syn^\ad_\pi(V^\sharp/V^\sharp)&&
		\bigl(\text{resp. }\MF_K^\fin(\varphi,N)\rightarrow\Syn_\pi(V^\sharp/V^\sharp)\bigr)
		\end{align*}
	given by $M\mapsto M(\pi,\log)^a_\pi$ is an equivalence of categories with a quasi-inverse $H^0_{p\mathrm{H}}(\cX,\cdot)_{\pi,\log}$.
\end{corollary}

Let $\cX$ be a proper strictly semistable weak formal log scheme adic over $V^\sharp$ and $Y_\pi:=\cX\times_{V^\sharp,\tau_\pi}k^0$.
For an object $\frE=(\sE,\Phi,F^\bullet)\in\Syn_\pi(\cX/V^\sharp)$, we denote
	\[R\Gamma_\abs(\cX,\frE)_\pi:=R\Gamma_\abs(Y_\pi,(\sE,\Phi)).\]
By construction, the composition
	\[R\Gamma_\abs(\cX,\frE)_\pi\rightarrow R\Gamma_\HK(\cX,\frE)_\pi\xrightarrow{\Psi_{\pi,\log}} R\Gamma_\dR(\cX,\frE)\]
is independent of the choice of $\log$, which we denote by $\psi_\pi$.

Now we may define the syntomic cohomology as a $p$-adic absolute Hodge cohomology.

\begin{definition}
	For an object $\frE=(\sE,\Phi,F^\bullet)\in\Syn_\pi(\cX/V^\sharp)$, we define the {\it syntomic cohomology} $R\Gamma_\syn(\cX,\frE)_\pi$ of $\cX$ with coefficients in $\frE$ to be
		\begin{equation}\label{eq: syn coh}
		R\Gamma_\syn(\cX,\frE)_\pi:=\Cone\left(R\Gamma_\abs(\cX,\frE)_\pi\oplus F^0R\Gamma_\dR(\cX,\frE) \rightarrow R\Gamma_\dR(\cX,\frE)\right)[-1],
		\end{equation}	
	the cone of the map defined by $(x,y)\mapsto \psi_\pi (x)-y$.
	More precisely, by taking a local embedding $F$-datum for $Y_\pi$ over $\cS$, we obtain a simplicial weak formal log scheme $\cZ_\bullet$ with Frobenius lifts $\phi_\bullet$.
	For $m\in\bbN$, put $\cX_m:=\cZ_m\times_{\cS,j_\pi}V^\sharp$ and let $\cX'_m$ be the exactification of the diagonal embedding $Y_\pi\hookrightarrow\cX_m\times_{V^\sharp}\cX$.
	Then the morphism $R\Gamma_\rig(Y_\pi/W^\varnothing,\sE)\rightarrow R\Gamma_\dR(\cX,\frE)$ is expressed as the zigzag
		\begin{eqnarray*}R\Gamma(\cZ_{\bullet,\bbQ},\sE_{\cZ_\bullet}\otimes\omega^\star_{\cZ_\bullet/W^\varnothing,\bbQ})&\xrightarrow{\psi_\pi}& R\Gamma(\cX_{\bullet,\bbQ},\sE_{\cX_\bullet}\otimes\omega^\star_{\cX_\bullet/V^\sharp,\bbQ})\\
		&\xrightarrow{\cong}& R\Gamma(\cX'_{\bullet,\bbQ},\sE_{\cX'_\bullet}\otimes\omega^\star_{\cX'_\bullet/V^\sharp,\bbQ})\\
		&\xleftarrow{\cong}& R\Gamma(\cX_\bbQ,\sE_\dR\otimes\omega^\star_{\cX/V^\sharp,\bbQ}),
		\end{eqnarray*}
	and $R\Gamma_\abs(\cX,\frE)_\pi$ is represented by the cone
		\[R\Gamma_\abs(\cX,\frE)_\pi\cong\Cone\left(R\Gamma(\cZ_{\bullet,\bbQ},\sE_{\cZ_\bullet}\otimes\omega^\star_{\cZ_\bullet/W^\varnothing,\bbQ})\xrightarrow{\varphi_\bullet}R\Gamma(\cZ_{\bullet,\bbQ},\sE_{\cZ_\bullet}\otimes\omega^\star_{\cZ_\bullet/W^\varnothing,\bbQ})\right)[-1].\]
	Therefore the syntomic cohomology is precisely defined as
		\begin{equation}\label{eq: explicit syn}
		R\Gamma_\syn(\cX,\frE)_\pi:=\Cone\left(\begin{array}{c}
		R\Gamma(\cZ_{\bullet,\bbQ},\sE_{\cZ_\bullet}\otimes\omega^\star_{\cZ_\bullet/W^\varnothing,\bbQ})\oplus F^0R\Gamma(\cX_\bbQ,\sE_\dR\otimes\omega^\star_{\cX/V^\sharp,\bbQ})\\
		\downarrow\\
		R\Gamma(\cZ_{\bullet,\bbQ},\sE_{\cZ_\bullet}\otimes\omega^\star_{\cZ_\bullet/W^\varnothing,\bbQ})\oplus R\Gamma(\cX'_{\bullet,\bbQ},\sE_{\cX'_\bullet}\otimes\omega^\star_{\cX'_\bullet/V^\sharp,\bbQ})
		\end{array}\right)[-1],\end{equation}
	where the map in the cone is defined by $(x,y)\mapsto (\varphi_\bullet(x),\psi_\pi(x)-y)$.
	This is well-defined as an object of $ D^b(\Mod_{\bbQ_p})$.
	
	For the trivial coefficient $\frO_\cX$, we denote $R\Gamma_\syn(\cX)_\pi:=R\Gamma_\syn(\cX,\frO_\cX)_\pi$.
\end{definition}

For a commutative diagram \eqref{eq: diagram pull back syn} and $\frE=(\sE,\Phi,F^\bullet)\in\Syn_\pi(\cX/V^\sharp)$, the natural morphisms
	\begin{eqnarray*}
	\ol{f}_{\pi,\pi'}^*&\colon& R\Gamma_\rig(Y_\pi/W^\varnothing,(\sE,\Phi))\rightarrow R\Gamma_\rig(Y'_{\pi'}/W'^\varnothing,\ol{f}^*_{\pi,\pi'}(\sE,\Phi)),\\
	f_\dR^*&\colon& R\Gamma_\dR(\cX,\sE_\dR)\rightarrow R\Gamma_\dR(\cX',f^*\sE_\dR)
	\end{eqnarray*}
together induce a morphism between syntomic cohomology
	\begin{equation}
	f_\syn^*\colon R\Gamma_\syn(\cX,\frE)_\pi\rightarrow R\Gamma_\syn(\cX',f_{\pi,\pi'}^*\frE)_{\pi'}.
	\end{equation}
Considering the case $\cX'=\cX$, we see that the syntomic cohomology is independent of the choice of $\pi$ under the identification of syntomic coefficients via the canonical equivalence $\Syn_\pi(\cX/V^\sharp)\cong\Syn_{\pi'}(\cX/V^\sharp)$.

Finally we mention that the syntomic cohomology group is regarded as the extension group.

\begin{proposition}\label{prop: Ext syn}
	Let $\cX$ be a proper strictly semistable weak formal log scheme adic over $V^\sharp$.
	For any $\frE,\frE'\in\Syn_\pi(\cX/V^\sharp)$ we have canonical isomorphisms
		\begin{eqnarray*}
		\Hom_{\Syn_\pi(\cX/V^\sharp)}(\frE',\frE)&\xrightarrow{\cong}&H^0_\syn(\cX,\sheafhom(\frE',\frE))_\pi,\\
		\Ext^1_{\Syn_\pi(\cX/V^\sharp)}(\frE',\frE)&\xrightarrow{\cong}&H^1_\syn(\cX,\sheafhom(\frE',\frE))_\pi.
		\end{eqnarray*}
\end{proposition}

\begin{proof}
	Using an explicit description of the syntomic cohomology \eqref{eq: explicit syn}, one can prove our assertion in a similar way to the proof of Proposition \ref{prop: Ext abs HK}.
	Note that, the necessity of taking $F^0$ in \eqref{eq: syn coh} arises from the Griffith transversality of extension.
\end{proof}

Although the above proposition deals with only $0$-th and $1$-st cohomology groups, the following proposition suggests that the syntomic cohomology with admissible syntomic coefficient computes higher extension groups in a virtual abelian category extending admissible syntomic coefficients.

\begin{proposition}\label{prop: Leray ss}
	Let $\cX$ be a proper strictly semistable weak formal log scheme adic over $V^\sharp$.
	For $\frE\in\Syn^\ad_\pi(\cX)$, we regard $R\Gamma_{p\mathrm{H}}(\cX,\frE)_{\pi,\log}$ as an object of $ D^b(\MF_K^\ad(\varphi,N))$ via the equivalence \eqref{eq: adm complex}.
	Then there exist a canonical isomorphism
		\begin{equation}
		R\Gamma_\syn(\cX,\frE)_\pi\cong R\Hom_{\MF_K^\ad(\varphi,N)}(L,R\Gamma_{p\mathrm{H}}(\cX,\frE)_{\pi,\log})
		\end{equation}
	and a spectral sequence
		\begin{equation}
		E_2^{i,j}=\Ext^i_{\MF_K^\ad(\varphi,N)}(L,H^j_{p\mathrm{H}}(\cX,\frE)_{\pi,\log})\Rightarrow H^{i+j}_\syn(\cX,\frE)_\pi.
		\end{equation}
\end{proposition}

\begin{proof}
	Immediately follows from the definition of the syntomic cohomology and the computation of extension groups in \cite[Theorem 3.1.5]{EK} (see also \cite[Proposition 2.7]{DN}).
\end{proof}

\begin{remark}
	The proof of \cite[Proposition 2.7]{DN} has a non-small gap, because the morphism $\Hom_{K^b(\mathrm{DF}_K)}(M,T[i])\rightarrow H^i(\Ker(f_{M,T}))$ in their proof is in fact not an isomorphism.
	For example, let $K=\bbQ_p$, $i=1$, and $M=T=\bbQ_p$ regarded as a complex with only $0$-th component.
	Then we have $\Hom_{K^b(\mathrm{DF}_{\bbQ_p})}(\bbQ_p,\bbQ_p[1])=0$, however $H^1(\Ker(f_{\bbQ_p,\bbQ_p}))\cong\bbQ_p$.
\end{remark}

\section*{Acknowledgements}
I would like to thank Kenichi Bannai for giving me helpful comments and great environment to research.
I am grateful to Veronika Ertl for her continuous interest in this work and helpful discussions for related topics.
I would also like to thank and Ju-Feng Wu and Yoshinosuke Hirakawa for many helpful comments.
I really appreciate the anonymous referee a lot of helpful suggestions and pointing out failures.
Most of this work was done when the author was a PhD. student, so the author would also like to thank the Faculty of Science and Technology at Keio University, in particular members of Bannai laboratory.


\begin{bibdiv}
\begin{biblist}

\bib{Ban}{article}{
   author={Bannai, Kenichi},
   title={Rigid syntomic cohomology and $p$-adic polylogarithms},
   journal={J. Reine Angew. Math.},
   volume={529},
   date={2000},
   pages={205--237},
   issn={0075-4102},
   review={\MR{1799937}},
   doi={10.1515/crll.2000.097},
}

\bib{Ban2}{article}{
   author={Bannai, Kenichi},
   title={Syntomic cohomology as a $p$-adic absolute Hodge cohomology},
   journal={Math. Z.},
   volume={242},
   date={2002},
   number={3},
   pages={443--480},
   issn={0025-5874},
   review={\MR{1985460}},
   doi={10.1007/s002090100351},
}

\bib{Ban3}{article}{
   author={Bannai, Kenichi},
   title={On the $p$-adic realization of elliptic polylogarithms for
   CM-elliptic curves},
   journal={Duke Math. J.},
   volume={113},
   date={2002},
   number={2},
   pages={193--236},
   issn={0012-7094},
   review={\MR{1909217}},
   doi={10.1215/S0012-7094-02-11321-0},
}

\bib{BK}{article}{
   author={Bannai, Kenichi},
   author={Kings, Guido},
   title={$p$-adic elliptic polylogarithm, $p$-adic Eisenstein series and
   Katz measure},
   journal={Amer. J. Math.},
   volume={132},
   date={2010},
   number={6},
   pages={1609--1654},
   issn={0002-9327},
   review={\MR{2766179}},
}

\bib{BK2}{article}{
   author={Bannai, Kenichi},
   author={Kings, Guido},
   title={$p$-adic Beilinson conjecture for ordinary Hecke motives
   associated to imaginary quadratic fields},
   conference={
      title={Algebraic number theory and related topics 2009},
   },
   book={
      series={RIMS K\^{o}ky\^{u}roku Bessatsu, B25},
      publisher={Res. Inst. Math. Sci. (RIMS), Kyoto},
   },
   date={2011},
   pages={9--30},
   review={\MR{2868567}},
}

\bib{BKT}{article}{
   author={Bannai, Kenichi},
   author={Kobayashi, Shinichi},
   author={Tsuji, Takeshi},
   title={On the de Rham and $p$-adic realizations of the elliptic
   polylogarithm for CM elliptic curves},
   language={English, with English and French summaries},
   journal={Ann. Sci. \'{E}c. Norm. Sup\'{e}r. (4)},
   volume={43},
   date={2010},
   number={2},
   pages={185--234},
   issn={0012-9593},
   review={\MR{2662664}},
   doi={10.24033/asens.2119},
}

\bib{Bei2}{article}{author={Beilinson, A.},title={On the crystalline period map},journal={Camb. J. Math.},volume={1},date={2013},number={1},pages={1--51},issn={2168-0930},review={\MR{3272051}},doi={10.4310/CJM.2013.v1.n1.a1},}

\bib{Berger2}{article}{
   author={Berger, Laurent},
   title={\'{E}quations diff\'{e}rentielles $p$-adiques et $(\phi,N)$-modules
   filtr\'{e}s},
   language={French, with English and French summaries},
   note={Repr\'{e}sentations $p$-adiques de groupes $p$-adiques. I.
   Repr\'{e}sentations galoisiennes et $(\phi,\Gamma)$-modules},
   journal={Ast\'{e}risque},
   number={319},
   date={2008},
   pages={13--38},
   issn={0303-1179},
   isbn={978-2-85629-256-3},
   review={\MR{2493215}},
}

\bib{Ber2}{misc}{
	author={Berthelot, Pierre},
	title={Cohomologie rigide et cohomologie rigide \`{a} support propre, Premi\`{e}re partie},
	status={Pr\'{e}publication IRMAR 96--03},
	date={1996},
	note={https://perso.univ-rennes1.fr/pierre.berthelot},
}

\bib{BDR}{article}{
   author={Bertolini, Massimo},
   author={Darmon, Henri},
   author={Rotger, Victor},
   title={Beilinson-Flach elements and Euler systems I: Syntomic regulators
   and $p$-adic Rankin $L$-series},
   journal={J. Algebraic Geom.},
   volume={24},
   date={2015},
   number={2},
   pages={355--378},
   issn={1056-3911},
   review={\MR{3311587}},
   doi={10.1090/S1056-3911-2014-00670-6},
}

\bib{Bes}{article}{
   author={Besser, Amnon},
   title={Syntomic regulators and $p$-adic integration. I. Rigid syntomic
   regulators},
   booktitle={Proceedings of the Conference on $p$-adic Aspects of the
   Theory of Automorphic Representations (Jerusalem, 1998)},
   journal={Israel J. Math.},
   volume={120},
   date={2000},
   number={part B},
   part={part B},
   pages={291--334},
   issn={0021-2172},
   review={\MR{1809626}},
   doi={10.1007/BF02834843},
}

\bib{Bes2}{article}{
   author={Besser, Amnon},
   title={Syntomic regulators and $p$-adic integration. II. $K_2$ of curves},
   booktitle={Proceedings of the Conference on $p$-adic Aspects of the
   Theory of Automorphic Representations (Jerusalem, 1998)},
   journal={Israel J. Math.},
   volume={120},
   date={2000},
   number={part B},
   part={part B},
   pages={335--359},
   issn={0021-2172},
   review={\MR{1809627}},
   doi={10.1007/BF02834844},
}
	
\bib{Bes3}{article}{
   author={Besser, Amnon},
   title={On the syntomic regulator for $K_1$ of a surface},
   journal={Israel J. Math.},
   volume={190},
   date={2012},
   pages={29--66},
   issn={0021-2172},
   review={\MR{2956231}},
   doi={10.1007/s11856-011-0188-0},
}

\bib{Bes4}{article}{
   author={Besser, Amnon},
   title={On the regulator formulas of Bertolini, Darmon and Rotger},
   journal={Res. Math. Sci.},
   volume={3},
   date={2016},
   pages={Paper No. 26, 18},
   issn={2522-0144},
   review={\MR{3539566}},
   doi={10.1186/s40687-016-0073-x},
}

\bib{BBdJR}{article}{
   author={Besser, A.},
   author={Buckingham, P.},
   author={de Jeu, R.},
   author={Roblot, X.-F.},
   title={On the $p$-adic Beilinson conjecture for number fields},
   journal={Pure Appl. Math. Q.},
   volume={5},
   date={2009},
   number={1},
   pages={375--434},
   issn={1558-8599},
   review={\MR{2520465}},
   doi={10.4310/PAMQ.2009.v5.n1.a12},
}
	
	
\bib{BdJ}{article}{
   author={Besser, Amnon},
   author={de Jeu, Rob},
   title={The syntomic regulator for the $K$-theory of fields},
   language={English, with English and French summaries},
   journal={Ann. Sci. \'{E}cole Norm. Sup. (4)},
   volume={36},
   date={2003},
   number={6},
   pages={867--924 (2004)},
   issn={0012-9593},
   review={\MR{2032529}},
   doi={10.1016/j.ansens.2003.01.003},
}

\bib{BdJ2}{article}{
   author={Besser, Amnon},
   author={de Jeu, Rob},
   title={The syntomic regulator and $K_4$ of curves},
   journal={Pacific J. Math.},
   volume={260},
   date={2012},
   number={2},
   pages={305--380},
   issn={0030-8730},
   review={\MR{3001797}},
   doi={10.2140/pjm.2012.260.305},
}

\bib{Ch}{article}{
   author={Chiarellotto, Bruno},
   title={Weights in rigid cohomology applications to unipotent
   $F$-isocrystals},
   language={English, with English and French summaries},
   journal={Ann. Sci. \'{E}cole Norm. Sup. (4)},
   volume={31},
   date={1998},
   number={5},
   pages={683--715},
   issn={0012-9593},
   review={\MR{1643966}},
   doi={10.1016/S0012-9593(98)80004-9},
}

\bib{CCM}{article}{author={Chiarellotto, Bruno},author={Ciccioni, Alice},author={Mazzari, Nicola},title={Cycle classes and the syntomic regulator},journal={Algebra Number Theory},volume={7},date={2013},number={3},pages={533--566},issn={1937-0652},review={\MR{3095220}},doi={10.2140/ant.2013.7.533},}


\bib{CL1}{article}{
   author={Chiarellotto, Bruno},
   author={Le Stum, Bernard},
   title={$F$-isocristaux unipotents},
   language={French, with English and French summaries},
   journal={Compositio Math.},
   volume={116},
   date={1999},
   number={1},
   pages={81--110},
   issn={0010-437X},
   review={\MR{1669440}},
   doi={10.1023/A:1000602824628},
}

\bib{CL2}{article}{
   author={Chiarellotto, Bruno},
   author={Le Stum, Bernard},
   title={Pentes en cohomologie rigide et $F$-isocristaux unipotents},
   language={French, with English summary},
   journal={Manuscripta Math.},
   volume={100},
   date={1999},
   number={4},
   pages={455--468},
   issn={0025-2611},
   review={\MR{1734795}},
   doi={10.1007/s002290050212},
}

\bib{Co}{article}{
   author={Colmez, Pierre},
   title={Espaces de Banach de dimension finie},
   language={French, with English and French summaries},
   journal={J. Inst. Math. Jussieu},
   volume={1},
   date={2002},
   number={3},
   pages={331--439},
   issn={1474-7480},
   review={\MR{1956055}},
   doi={10.1017/S1474748002000099},
}

\bib{CF}{article}{
   author={Colmez, Pierre},
   author={Fontaine, Jean-Marc},
   title={Construction des repr\'{e}sentations $p$-adiques semi-stables},
   language={French},
   journal={Invent. Math.},
   volume={140},
   date={2000},
   number={1},
   pages={1--43},
   issn={0020-9910},
   review={\MR{1779803}},
   doi={10.1007/s002220000042},
}

\bib{DN}{article}{
   author={D\'{e}glise, Fr\'{e}d\'{e}ric},
   author={Nizio\l , Wies\l awa},
   title={On $p$-adic absolute Hodge cohomology and syntomic coefficients.
   I},
   journal={Comment. Math. Helv.},
   volume={93},
   date={2018},
   number={1},
   pages={71--131},
   issn={0010-2571},
   review={\MR{3777126}},
   doi={10.4171/CMH/430},
}

\bib{DG}{book}{
   author={Demazure, Michel},
   author={Gabriel, Pierre},
   title={Groupes alg\'{e}briques. Tome I: G\'{e}om\'{e}trie alg\'{e}brique, g\'{e}n\'{e}ralit\'{e}s,
   groupes commutatifs},
   language={French},
   note={Avec un appendice {\it Corps de classes local} par Michiel
   Hazewinkel},
   publisher={Masson \& Cie, \'{E}diteur, Paris; North-Holland Publishing Co.,
   Amsterdam},
   date={1970},
   pages={xxvi+700},
   review={\MR{0302656}},
}

\bib{EK}{misc}{
	author={Emerton, Matthew},
	author={Kisin, Mark},
	title={Extensions of crystalline representations},
	status={preprint},
}
	
\bib{EY1}{article}{
   author={Ertl, Veronika},
   author={Yamada, Kazuki},
   title={Comparison between rigid and crystalline syntomic cohomology for
   strictly semistable log schemes with boundary},
   journal={Rend. Semin. Mat. Univ. Padova},
   volume={145},
   date={2021},
   pages={213--291},
   issn={0041-8994},
   review={\MR{4261656}},
   doi={10.4171/rsmup/81},
}
	
\bib{EY}{misc}{
	author={Ertl, Veronika},
	author={Yamada, Kazuki},
	title={Rigid analytic reconstruction of Hyodo--Kato theory},
	date={2022},
	status={preprint, arXiv:1907.10964},
}

\bib{EY2}{misc}{
	author={Ertl, Veronika},
	author={Yamada, Kazuki}
	title={Poincar\'{e} duality for rigid analytic Hyodo--Kato theory},
	date={2023},
	status={preprint, arXiv:2009.09160},
}

\bib{Fa}{article}{
   author={Faltings, Gerd},
   title={Almost \'{e}tale extensions},
   note={Cohomologies $p$-adiques et applications arithm\'{e}tiques, II},
   journal={Ast\'{e}risque},
   number={279},
   date={2002},
   pages={185--270},
   issn={0303-1179},
   review={\MR{1922831}},
}

\bib{Fo}{article}{
   author={Fontaine, Jean-Marc},
   title={Modules galoisiens, modules filtr\'{e}s et anneaux de Barsotti-Tate},
   language={French},
   conference={
      title={Journ\'{e}es de G\'{e}om\'{e}trie Alg\'{e}brique de Rennes},
      address={Rennes},
      date={1978},
   },
   book={
      series={Ast\'{e}risque},
      volume={65},
      publisher={Soc. Math. France, Paris},
   },
   date={1979},
   pages={3--80},
   review={\MR{563472}},
}

\bib{Fo03}{article}{
   author={Fontaine, Jean-Marc},
   title={Presque $C_p$-repr\'{e}sentations},
   language={French, with English summary},
   note={Kazuya Kato's fiftieth birthday},
   journal={Doc. Math.},
   date={2003},
   number={Extra Vol.},
   pages={285--385},
   issn={1431-0635},
   review={\MR{2046603}},
}

\bib{Gros}{article}{
   author={Gros, Michel},
   title={R\'{e}gulateurs syntomiques et valeurs de fonctions $L\;p$-adiques.
   II},
   language={French},
   journal={Invent. Math.},
   volume={115},
   date={1994},
   number={1},
   pages={61--79},
   issn={0020-9910},
   review={\MR{1248079}},
   doi={10.1007/BF01231754},
}
		

\bib{GK1}{article}{author={Grosse-Kl\"{o}nne, Elmar},title={Rigid analytic spaces with overconvergent structure sheaf},journal={J. Reine Angew. Math.},volume={519},date={2000},pages={73--95},issn={0075-4102},review={\MR{1739729}},doi={10.1515/crll.2000.018},}

\bib{GK1.5}{article}{
   author={Grosse-Kl\"{o}nne, Elmar},
   title={Finiteness of de Rham cohomology in rigid analysis},
   journal={Duke Math. J.},
   volume={113},
   date={2002},
   number={1},
   pages={57--91},
   issn={0012-7094},
   review={\MR{1905392}},
   doi={10.1215/S0012-7094-02-11312-X},
}

\bib{GK3}{article}{
   author={Grosse-Kl\"{o}nne, Elmar},
   title={Frobenius and monodromy operators in rigid analysis, and
   Drinfel\cprime d's symmetric space},
   journal={J. Algebraic Geom.},
   volume={14},
   date={2005},
   number={3},
   pages={391--437},
   issn={1056-3911},
   review={\MR{2129006}},
   doi={10.1090/S1056-3911-05-00402-9},
}

\bib{HZ}{article}{
   author={Hain, Richard M.},
   author={Zucker, Steven},
   title={Unipotent variations of mixed Hodge structure},
   journal={Invent. Math.},
   volume={88},
   date={1987},
   number={1},
   pages={83--124},
   issn={0020-9910},
   review={\MR{877008}},
   doi={10.1007/BF01405093},
}

		
\bib{HL}{article}{
   author={Hartl, Urs},
   author={L\"{u}tkebohmert, Werner},
   title={On rigid-analytic Picard varieties},
   journal={J. Reine Angew. Math.},
   volume={528},
   date={2000},
   pages={101--148},
   issn={0075-4102},
   review={\MR{1801659}},
   doi={10.1515/crll.2000.087},
}
	
\bib{HK}{article}{
   author={Hyodo, Osamu},
   author={Kato, Kazuya},
   title={Semi-stable reduction and crystalline cohomology with logarithmic
   poles},
   note={P\'{e}riodes $p$-adiques (Bures-sur-Yvette, 1988)},
   journal={Ast\'{e}risque},
   number={223},
   date={1994},
   pages={221--268},
   issn={0303-1179},
   review={\MR{1293974}},
}
		
\bib{Ka}{article}{
   author={Kato, Kazuya},
   title={Logarithmic structures of Fontaine-Illusie},
   conference={
      title={Algebraic analysis, geometry, and number theory},
      address={Baltimore, MD},
      date={1988},
   },
   book={
      publisher={Johns Hopkins Univ. Press, Baltimore, MD},
   },
   date={1989},
   pages={191--224},
   review={\MR{1463703}},
}

\bib{Ke}{article}{
   author={Kedlaya, Kiran S.},
   title={Semistable reduction for overconvergent $F$-isocrystals. I.
   Unipotence and logarithmic extensions},
   journal={Compos. Math.},
   volume={143},
   date={2007},
   number={5},
   pages={1164--1212},
   issn={0010-437X},
   review={\MR{2360314}},
   doi={10.1112/S0010437X07002886},
}


\bib{KH}{article}{
   author={Kim, Minhyong},
   author={Hain, Richard M.},
   title={A de Rham-Witt approach to crystalline rational homotopy theory},
   journal={Compos. Math.},
   volume={140},
   date={2004},
   number={5},
   pages={1245--1276},
   issn={0010-437X},
   review={\MR{2081156}},
   doi={10.1112/S0010437X04000442},
}
		
\bib{Kisin}{article}{
   author={Kisin, Mark},
   title={Crystalline representations and $F$-crystals},
   conference={
      title={Algebraic geometry and number theory},
   },
   book={
      series={Progr. Math.},
      volume={253},
      publisher={Birkh\"{a}user Boston, Boston, MA},
   },
   date={2006},
   pages={459--496},
   review={\MR{2263197}},
   doi={10.1007/978-0-8176-4532-8\_7},
}
			
\bib{LM}{article}{author={Langer, Andreas},   author={Muralidharan, Amrita},   title={An analogue of Raynaud's theorem: weak formal schemes and dagger   spaces},   journal={M\"{u}nster J. Math.},   volume={6},   date={2013},   number={1},   pages={271--294},   issn={1867-5778},   review={\MR{3148213}},}

\bib{LP}{article}{
   author={Li, Shizhang},
   author={Pan, Xuanyu},
   title={Logarithmic de Rham comparison for open rigid spaces},
   journal={Forum Math. Sigma},
   volume={7},
   date={2019},
   pages={Paper No. e32},
   review={\MR{4016499}},
   doi={10.1017/fms.2019.27},
}
	
	
\bib{Me}{article}{   author={Meredith, David},   title={Weak formal schemes},   journal={Nagoya Math. J.},   volume={45},   date={1972},   pages={1--38},   issn={0027-7630},   review={\MR{330167}},}

		
\bib{MW}{article}{
   author={Monsky, P.},
   author={Washnitzer, G.},
   title={Formal cohomology. I},
   journal={Ann. of Math. (2)},
   volume={88},
   date={1968},
   pages={181--217},
   issn={0003-486X},
   review={\MR{248141}},
   doi={10.2307/1970571},
}


\bib{NN}{article}{   author={Nekov\'{a}\v{r}, Jan},   author={Nizio\l , Wies\l awa},   title={Syntomic cohomology and $p$-adic regulators for varieties over $p$-adic fields},   note={With appendices by Laurent Berger and Fr\'{e}d\'{e}ric D\'{e}glise},   journal={Algebra Number Theory},   volume={10},   date={2016},   number={8},   pages={1695--1790},   issn={1937-0652},   review={\MR{3556797}},   doi={10.2140/ant.2016.10.1695},}		


\bib{Ni}{article}{
   author={Nizio\l , Wies\l awa},
   title={Semistable conjecture via $K$-theory},
   journal={Duke Math. J.},
   volume={141},
   date={2008},
   number={1},
   pages={151--178},
   issn={0012-7094},
   review={\MR{2372150}},
   doi={10.1215/S0012-7094-08-14114-6},
}

\bib{Og}{article}{
   author={Ogus, Arthur},
   title={On the logarithmic Riemann-Hilbert correspondence},
   note={Kazuya Kato's fiftieth birthday},
   journal={Doc. Math.},
   date={2003},
   number={Extra Vol.},
   pages={655--724},
   issn={1431-0635},
   review={\MR{2046612}},
}
		
	
\bib{Shi1}{article}{
   author={Shiho, Atsushi},
   title={Crystalline fundamental groups. I. Isocrystals on log crystalline
   site and log convergent site},
   journal={J. Math. Sci. Univ. Tokyo},
   volume={7},
   date={2000},
   number={4},
   pages={509--656},
   issn={1340-5705},
   review={\MR{1800845}},
}
		
\bib{Shi2}{article}{
   author={Shiho, Atsushi},
   title={Crystalline fundamental groups. II. Log convergent cohomology and
   rigid cohomology},
   journal={J. Math. Sci. Univ. Tokyo},
   volume={9},
   date={2002},
   number={1},
   pages={1--163},
   issn={1340-5705},
   review={\MR{1889223}},
}


\bib{Shi3}{misc}{	author={Shiho, Atsushi},	title={Relative log convergent cohomology and relative rigid cohomology I},	status={preprint},	date={2008},	note={arXiv:0707.1742v2},}

\bib{Sp}{article}{
   author={Sprang, Johannes},
   title={The syntomic realization of the elliptic polylogarithm via the
   Poincar\'{e} bundle},
   journal={Doc. Math.},
   volume={24},
   date={2019},
   pages={1099--1134},
   issn={1431-0635},
   review={\MR{3992317}},
}

\bib{SZ}{article}{
   author={Steenbrink, Joseph},
   author={Zucker, Steven},
   title={Variation of mixed Hodge structure. I},
   journal={Invent. Math.},
   volume={80},
   date={1985},
   number={3},
   pages={489--542},
   issn={0020-9910},
   review={\MR{791673}},
   doi={10.1007/BF01388729},
}
	

\bib{Ts1}{article}{
   author={Tsuji, Takeshi},
   title={Syntomic complexes and $p$-adic vanishing cycles},
   journal={J. Reine Angew. Math.},
   volume={472},
   date={1996},
   pages={69--138},
   issn={0075-4102},
   review={\MR{1384907}},
   doi={10.1515/crll.1996.472.69},
}

\bib{Ts}{article}{
   author={Tsuji, Takeshi},
   title={$p$-adic \'{e}tale cohomology and crystalline cohomology in the
   semi-stable reduction case},
   journal={Invent. Math.},
   volume={137},
   date={1999},
   number={2},
   pages={233--411},
   issn={0020-9910},
   review={\MR{1705837}},
   doi={10.1007/s002220050330},
}

\bib{vdPS}{article}{
   author={van der Put, M.},
   author={Schneider, P.},
   title={Points and topologies in rigid geometry},
   journal={Math. Ann.},
   volume={302},
   date={1995},
   number={1},
   pages={81--103},
   issn={0025-5831},
   review={\MR{1329448}},
   doi={10.1007/BF01444488},
}


\bib{Wi}{book}{
   author={Wildeshaus, J\"{o}rg},
   title={Realizations of polylogarithms},
   series={Lecture Notes in Mathematics},
   volume={1650},
   publisher={Springer-Verlag, Berlin},
   date={1997},
   pages={xii+343},
   isbn={3-540-62460-0},
   review={\MR{1482233}},
   doi={10.1007/BFb0093051},
}

\end{biblist}
\end{bibdiv}
\end{document}